\documentclass[11pt, leqno]{article}
\usepackage{amsmath}
\usepackage{color}
\usepackage{amssymb}
\usepackage{latexsym}
\usepackage{enumerate}
\usepackage{amsthm}

\topmargin by -\headsep
\usepackage{fullpage}

     {\begin{list}{}%
             {\setlength{\leftmargin}{#1}}%
             \item[]%
     }
     {\end{list}}

\newtheorem{theorem}{Theorem}[section]
\newtheorem{definition}[theorem]{Definition}
\newtheorem{lemma}[theorem]{Lemma}
\newtheorem{corollary}[theorem]{Corollary}

\newtheorem{abbreviation}[theorem]{Abbreviation}
\newtheorem{observation}[theorem]{Observation}
\newtheorem{convention}[theorem]{Convention}
\newtheorem{lem}[theorem]{Lemma}

 \newtheorem{deff}[theorem]{Definition}

\newtheorem{rem}[theorem]{Remark}

\newtheorem{remark}[theorem]{Remark}

\newcommand*{\IRS}{\textbf{IRS}_\Omega}
\newcommand*{\IRSOP}{\textbf{IRS}_\Omega^\mathcal{P}}
\newcommand*{\IRSOE}{\textbf{IRS}_\Omega^\mathbb{E}}
\newcommand*{\IKP}{\textbf{IKP}}
\newcommand*{\IKPP}{\textbf{IKP}(\mathcal{P})}
\newcommand*{\IKPE}{\textbf{IKP}(\mathcal{E})}
\newcommand*{\dotin}{\mathrel{\dot{\in}}}
\newcommand*{\Ea}{\mathbb{E}_\alpha}
\newcommand*{\Eb}{\mathbb{E}_\beta}
\newcommand*{\Eg}{\mathbb{E}_\gamma}
\newcommand*{\Ed}{\mathbb{E}_\delta}

\newcommand*{\Pip}{\Pi^{\mathcal{P}}}
\newcommand*{\Deltaoe}{\Delta_0^\mathcal{E}}

\newcommand{\SRE}{(\Sigma^{\mathcal E}\mbox{-}Ref)}
\newcommand{\Sigmae}{\Sigma^{\mathcal E}}
\newcommand{\Pie}{\Pi^{\mathcal E}}
\newcommand{\DE}{\Delta}
\newcommand{\RI}{\Rightarrow}
\newcommand{\Ebb}{\mathbb{E}_{\bar\beta}}
\newcommand{\sbar}{\bar{s}}
\newcommand{\bbeta}{\bar{\beta}}

\newcommand{\goed}[1]{\ulcorner #1\urcorner}

\newcommand{\dbi}{\mbox{
${\mathbf\Delta}^{ 1}_{ 2}$--${\mathbf{CA}}+{\mathbf{BI}}$}}
\newcommand{\KP}{{\mathbf{KP}}}

\newcommand{\DI}{\displaystyle}

\newcommand{\OO}{\Omega}

\newcommand{\BI}{{\mathbf{BI}}}

\newcommand{\beq}{\begin{eqnarray}}
\newcommand{\eeq}{\end{eqnarray}}

\newcommand{\LL}{{\mathbf{L}}}

\newcommand{\inn}{\,{\in}\,}
\newcommand{\el}{\inn}
\newcommand{\Cut}{\text{(Cut)}}
\newcommand{\GA}{\Gamma}

\newcommand{\CH}{{\mathcal H}}

\newcommand{\DX}{{\mathfrak X}}

\newcommand{\In}{\in}

\newcommand{\lev}[1]{{\mid}\,{#1}\,{\mid}}

\newcommand{\om}{\omega}

\newcommand{\al}{\alpha}

\newcommand{\ga}{\gamma}

\newcommand{\SR}{(\Sigma\text{-Ref}_\Omega)}

\newcommand{\infone}[8]
{\rule[-0.5cm]{0cm}{1.5cm}
\begin{array}{c}
    \provx {{#1}} {{#2}} {{#3}} {{#4}}\rule{0cm}{5mm}\\
    \hline \rule{0cm}{5mm}
    \provx {{#5}} {{#6}} {{#7}} {{#8}}
 \end{array}}

 \newcommand{\infonel}[8]
{\rule[-0.5cm]{0cm}{1.5cm}
\begin{array}{l}
    \provx {{#1}} {{#2}} {{#3}} {{#4}}\rule{0cm}{5mm}\\
    \hline \rule{0cm}{5mm}
    \provx {{#5}} {{#6}} {{#7}} {{#8}}
 \end{array}}

 \newcommand{\infonetwo}[3]
{\rule[-0.5cm]{0cm}{1.5cm}
\begin{array}{l}
    {#1}\rule{0cm}{5mm}\\[0.15cm]
    {#2}\\
    \hline \rule{0cm}{5mm}
   {#3}
 \end{array}}

 \newcommand{\infonethree}[4]
{\rule[-0.5cm]{0cm}{1.5cm}
\begin{array}{l}
    {#1}\rule{0cm}{5mm}\\[0.15cm]
    {#2}\\[0.15cm]
    {#3}\\
    \hline \rule{0cm}{5mm}
   {#4}
 \end{array}}

  \newcommand{\infonefour}[5]
{\rule[-0.5cm]{0cm}{1.5cm}
\begin{array}{l}
    {#1}\rule{0cm}{5mm}\\[0.15cm]
    {#2}\\[0.15cm]
    {#3}\\[0.15cm]
    {#4}\\
    \hline \rule{0cm}{5mm}
   {#5}
 \end{array}}

\newcommand{\ifthree}[3]
{\rule[-0.5cm]{0cm}{1.5cm}
 \begin{array}{c}
     {{#1}} \qquad {{#2}}
 \rule{0cm}{5mm}
\\
    \hline  \rule{0cm}{5mm}
{{#3}}
 \end{array}}

\newcommand{\und}{\,\wedge\,}

\newcommand{\psio}[1]{\psi_{\Omega}(#1)}

\newcommand{\TI}{{\mathrm{TI}}}

\newcommand{\IZF}{{\mathbf{IZF}}}

\newcommand{\ZF}{{\mathbf{ZF}}}
\newcommand{\CZF}{{\mathbf{CZF}}}

\newcommand{\hoch}[2]{\phantom{}^{#1}{#2}}

\newcommand{\KPP}{{\mathbf{KP}}({\mathcal P})}
\newcommand{\Deltaop}{\Delta_0^{\mathcal P}}

\newcommand{\Sigmap}{\Sigma^{\mathcal P}}

\newcommand{\PC}{{\mathcal P}}

\newcommand{\subb}{\subseteq}
\newcommand{\Vb}[1]{{\mathbb V}_{#1}}
\newcommand{\Va}{{\mathbb V}_{\alpha}}
\newcommand{\SRP}{(\Sigma^{\mathcal P}\mbox{-}Ref)}

\newcommand{\EP}{{\mathbf{EP}}}

\mathchardef\str='1066
\def\negprov#1#2#3{
\setbox1=\hbox{\kern1.5pt$\scriptstyle#2$} \setbox4=\hbox{$\str$}
\def\zeichen{#1}
\ifx\zeichen\empty\setbox0=\hbox to 1em{}\else\setbox0=\hbox
{\kern1.5pt$\scriptstyle#1$}\fi \dimen1=\dp0 \ifdim \dimen1=0pt
\advance \dimen1 by 1.5ex \else \advance \dimen1 by 1.2ex
\fi\dimen3=2ex\dimen4=.5ex\ifdim \wd0<\wd1 \dimen2=\wd1 \else
\dimen2=\wd0
\fi\hbox{\hskip.5em$\kern-1.9pt\raise1pt\copy4\kern-\wd4\kern1.9pt\vrule
height\dimen3 depth\dimen4\raise\dimen1\copy0\hskip-1\wd0
\lower\ht1\copy1\hskip-1\wd1\vrule width\dimen2 height.7ex
depth-.6ex \hskip3pt minus1.5pt#3\hskip2pt plus2pt minus2pt$}}

\def\mod#1#2{
\def\zeichen{#1}
\hbox{\hskip 2pt plus3pt minus 2pt\vrule width.5pt height2ex
depth.5ex \vbox{\ifx\zeichen\empty\hbox to .75em{}\else
\hbox{\kern1.5pt $\scriptstyle#1$}\fi \kern2pt \hrule \kern1.7pt
\hrule\kern1.7pt} \hskip3pt minus 2pt$#2$}\hskip2pt plus3pt
minus2pt}

\def\notmod#1#2{\hbox{\hskip 2pt plus 3pt minus 3pt\vrule width.5pt
height2ex depth.5ex \vbox{\hbox{\kern1.5pt
$\scriptstyle#1$}\kern3pt \setbox0=\hbox{\kern2pt$\scriptstyle/$}
\hrule \kern-1.7pt \copy0 \kern-\ht0 \kern 1.7pt
\hrule\kern1.7pt}n \hskip3pt minus 2pt$#2$}\hskip2pt plus3pt
minus2pt}
\def\sq{\hbox{\rlap{$\sqcap$}$\sqcup$}}
\def\qed{\ifmmode\sq\else{\unskip\nobreak\hfil\penalty50\hskip1em\null
\nobreak\hfil\sq\parfillskip=0pt\finalhyphendemerits=0\endgraf}\fi\medskip}

\def\lleq{\hbox{\hskip3pt minus3pt\kern1pt\lower4pt
\vbox{\hbox{$\scriptstyle\ll$}
\kern-7pt\hbox{\kern1pt$\scriptstyle=$}}\hskip3pt minus 3pt}}

\mathchardef\res='1152 \mathchardef\qin='1062
\mathchardef\qprec='1036 \mathchardef\qless='474
\mathchardef\dpkt='72

\def\EnableBpAbbreviations{%
    \let\AX\Axiom
    \let\AXC\AxiomC
    \let\UI\UnaryInf
    \let\UIC\UnaryInfC
    \let\BI\BinaryInf
    \let\BIC\BinaryInfC
    \let\TI\TrinaryInf
    \let\TIC\TrinaryInfC
    \let\QI\QuaternaryInf
    \let\QIC\QuaternaryInfC
    \let\QuI\QuinaryInf
    \let\QuIC\QuinaryInfC
    \let\LL\LeftLabel
    \let\RL\RightLabel
    \let\DP\DisplayProof
}

\def\ScoreOverhang{4pt}         % How much underlines extend out
\def\ScoreOverhangLeft{\ScoreOverhang}
\def\ScoreOverhangRight{\ScoreOverhang}

\def\extraVskip{2pt}            % Extra space above and below lines
\def\ruleScoreFiller{\hrule}        % Horizontal rule filler.

\def\defaultScoreFiller{\ruleScoreFiller}  % Default horizontal filler.
\def\defaultBuildScore{\buildSingleScore}  % In \singleLine mode at start.

\def\defaultHypSeparation{\hskip.2in}   % Used if \insertBetweenHyps isn't given

\def\labelSpacing{3pt}      % Horizontal space separating labels and lines

\def\proofSkipAmount{\vskip.8ex plus.8ex minus.4ex}
\ifx\fCenter\undefined
\def\fCenter{\relax}
\fi

\def\theHypSeparation{\defaultHypSeparation}
\def\alwaysScoreFiller{\defaultScoreFiller} % Horizontal filler.
\def\alwaysBuildScore{\defaultBuildScore}
\def\theScoreFiller{\alwaysScoreFiller} % Horizontal filler.
\def\buildScore{\alwaysBuildScore}   %This command builds the score.
\def\hypKernAmt{0pt}    % Initial setting for kerning the hypotheses.

\def\defaultLeftLabel{}
\def\defaultRightLabel{}

\def\myTrue{Y}
\def\bottomAlignFlag{N}
\def\centerAlignFlag{N}
\def\defaultRootAtBottomFlag{Y}
\def\rootAtBottomFlag{Y}

\expandafter\ifx\csname newenvironment\endcsname\relax%
\def\makeatletter{\catcode`\@=11\relax}
\def\makeatother{\catcode`\@=12\relax}
\makeatletter
\def\newcount{\alloc@0\count\countdef\insc@unt}
\def\newdimen{\alloc@1\dimen\dimendef\insc@unt}
\def\newskip{\alloc@2\skip\skipdef\insc@unt}
\def\newbox{\alloc@4\box\chardef\insc@unt}
\makeatother
\else
\newenvironment{prooftree}%
{\begin{center}\proofSkipAmount \leavevmode}%
{\DisplayProof \proofSkipAmount \end{center} }
\fi

\def\thecur#1{\csname#1\number\theLevel\endcsname}

\newcount\theLevel    % This counter is the height of the stack.
\global\theLevel=0      % Initialized to zero
\newcount\myMaxLevel
\global\myMaxLevel=0
\newbox\myBoxA      % Temporary storage boxes
\newbox\myBoxB
\newbox\myBoxC
\newbox\myBoxD
\newbox\myBoxLL     % Boxes for the left label and the right label.
\newbox\myBoxRL
\newdimen\thisAboveSkip     %Internal use: amount to skip above line
\newdimen\thisBelowSkip     %Internal use: amount to skip below line
\newdimen\newScoreStart     % More temporary storage.
\newdimen\newScoreEnd
\newdimen\newCenter
\newdimen\displace
\newdimen\leftLowerAmt%     Amount to lower left label
\newdimen\rightLowerAmt%    Amount to lower right label
\newdimen\scoreHeight%      Score height
\newdimen\scoreDepth%       Score Depth
\newdimen\htLbox%
\newdimen\htRbox%
\newdimen\htRRbox%
\newdimen\htRRRbox%
\newdimen\htAbox%
\newdimen\htCbox%

\setbox\myBoxLL=\hbox{\defaultLeftLabel}%
\setbox\myBoxRL=\hbox{\defaultRightLabel}%

\def\allocatemore{%
    \ifnum\theLevel>\myMaxLevel%
        \expandafter\newbox\curBox%
        \expandafter\newdimen\curScoreStart%
        \expandafter\newdimen\curCenter%
        \expandafter\newdimen\curScoreEnd%
        \global\advance\myMaxLevel by1%
    \fi%
}

\def\prepAxiom{%
    \advance\theLevel by1%
    \edef\curBox{\thecur{myBox}}%
    \edef\curScoreStart{\thecur{myScoreStart}}%
    \edef\curCenter{\thecur{myCenter}}%
    \edef\curScoreEnd{\thecur{myScoreEnd}}%
    \allocatemore%
}

\def\Axiom$#1\fCenter#2${%
    \prepAxiom%
    \setbox\myBoxA=\hbox{$\mathord{#1}\fCenter\mathord{\relax}$}%
    \setbox\myBoxB=\hbox{$#2$}%
    \global\setbox\curBox=%
         \hbox{\hskip\ScoreOverhangLeft\relax%
        \unhcopy\myBoxA\unhcopy\myBoxB\hskip\ScoreOverhangRight\relax}%
    \global\curScoreStart=0pt \relax
    \global\curScoreEnd=\wd\curBox \relax
    \global\curCenter=\wd\myBoxA \relax
    \global\advance \curCenter by \ScoreOverhangLeft%
    \ignorespaces
}

\def\AxiomC#1{      % Note argument not in math mode
    \prepAxiom%
    \setbox\myBoxA=\hbox{#1}%
    \global\setbox\curBox =%
        \hbox{\hskip\ScoreOverhangLeft\relax%
                        \unhcopy\myBoxA\hskip\ScoreOverhangRight\relax}%
        \global\curScoreStart=0pt \relax
        \global\curScoreEnd=\wd\curBox \relax
        \global\curCenter=.5\wd\curBox \relax
        \global\advance \curCenter by \ScoreOverhangLeft%
    \ignorespaces
}

\def\prepUnary{%
    \ifnum \theLevel<1
        \errmessage{Hypotheses missing!}
    \fi%
    \edef\curBox{\thecur{myBox}}%
    \edef\curScoreStart{\thecur{myScoreStart}}%
    \edef\curCenter{\thecur{myCenter}}%
    \edef\curScoreEnd{\thecur{myScoreEnd}}%
}

\def\UnaryInf$#1\fCenter#2${%
    \prepUnary%
    \buildConclusion{#1}{#2}%
    \joinUnary%
    \resetInferenceDefaults%
    \ignorespaces%
}

\def\UnaryInfC#1{
    \prepUnary%
    \buildConclusionC{#1}%
    \joinUnary%
    \resetInferenceDefaults%
    \ignorespaces%
}

\def\prepBinary{%
    \ifnum\theLevel<2
        \errmessage{Hypotheses missing!}
    \fi%
    \edef\rcurBox{\thecur{myBox}}%   Set up names of right hypothesis
    \edef\rcurScoreStart{\thecur{myScoreStart}}%
    \edef\rcurCenter{\thecur{myCenter}}%
    \edef\rcurScoreEnd{\thecur{myScoreEnd}}%
    \advance\theLevel by-1%
    \edef\lcurBox{\thecur{myBox}}% Set up names of left hypothesis
    \edef\lcurScoreStart{\thecur{myScoreStart}}%
    \edef\lcurCenter{\thecur{myCenter}}%
    \edef\lcurScoreEnd{\thecur{myScoreEnd}}%
}

\def\BinaryInf$#1\fCenter#2${%
    \prepBinary%
    \buildConclusion{#1}{#2}%
    \joinBinary%
    \resetInferenceDefaults%
    \ignorespaces%
}

\def\BinaryInfC#1{%
    \prepBinary%
    \buildConclusionC{#1}%
    \joinBinary%
    \resetInferenceDefaults%
    \ignorespaces%
}

\def\prepTrinary{%
    \ifnum\theLevel<3
        \errmessage{Hypotheses missing!}
    \fi%
    \edef\rcurBox{\thecur{myBox}}%   Set up names of right hypothesis
    \edef\rcurScoreStart{\thecur{myScoreStart}}%
    \edef\rcurCenter{\thecur{myCenter}}%
    \edef\rcurScoreEnd{\thecur{myScoreEnd}}%
    \advance\theLevel by-1%
    \edef\ccurBox{\thecur{myBox}}% Set up names of center hypothesis
    \edef\ccurScoreStart{\thecur{myScoreStart}}%
    \edef\ccurCenter{\thecur{myCenter}}%
    \edef\ccurScoreEnd{\thecur{myScoreEnd}}%
    \advance\theLevel by-1%
    \edef\lcurBox{\thecur{myBox}}% Set up names of left hypothesis
    \edef\lcurScoreStart{\thecur{myScoreStart}}%
    \edef\lcurCenter{\thecur{myCenter}}%
    \edef\lcurScoreEnd{\thecur{myScoreEnd}}%
}

\def\TrinaryInf$#1\fCenter#2${%
    \prepTrinary%
    \buildConclusion{#1}{#2}%
    \joinTrinary%
    \resetInferenceDefaults%
    \ignorespaces%
}

\def\TrinaryInfC#1{%
    \prepTrinary%
    \buildConclusionC{#1}%
    \joinTrinary%
    \resetInferenceDefaults%
    \ignorespaces%
}

\def\prepQuaternary{%
    \ifnum\theLevel<4
        \errmessage{Hypotheses missing!}
    \fi%
    \edef\rrcurBox{\thecur{myBox}}%   Set up names of very right hypothesis
    \edef\rrcurScoreStart{\thecur{myScoreStart}}%
    \edef\rrcurCenter{\thecur{myCenter}}%
    \edef\rrcurScoreEnd{\thecur{myScoreEnd}}%
    \advance\theLevel by-1%
    \edef\rcurBox{\thecur{myBox}}%   Set up names of right hypothesis
    \edef\rcurScoreStart{\thecur{myScoreStart}}%
    \edef\rcurCenter{\thecur{myCenter}}%
    \edef\rcurScoreEnd{\thecur{myScoreEnd}}%
    \advance\theLevel by-1%
    \edef\ccurBox{\thecur{myBox}}% Set up names of center hypothesis
    \edef\ccurScoreStart{\thecur{myScoreStart}}%
    \edef\ccurCenter{\thecur{myCenter}}%
    \edef\ccurScoreEnd{\thecur{myScoreEnd}}%
    \advance\theLevel by-1%
    \edef\lcurBox{\thecur{myBox}}% Set up names of left hypothesis
    \edef\lcurScoreStart{\thecur{myScoreStart}}%
    \edef\lcurCenter{\thecur{myCenter}}%
    \edef\lcurScoreEnd{\thecur{myScoreEnd}}%
}

\def\QuaternaryInf$#1\fCenter#2${%
    \prepQuaternary%
    \buildConclusion{#1}{#2}%
    \joinQuaternary%
    \resetInferenceDefaults%
    \ignorespaces%
}

\def\QuaternaryInfC#1{%
    \prepQuaternary%
    \buildConclusionC{#1}%
    \joinQuaternary%
    \resetInferenceDefaults%
    \ignorespaces%
}

\def\joinQuaternary{% Construct the quarterary inference into a vbox.
    \setbox\myBoxA=\hbox{\theHypSeparation}%
    \lcurScoreEnd=\rrcurScoreEnd%
    \advance\lcurScoreEnd by\wd\rcurBox%
    \advance\lcurScoreEnd by\wd\lcurBox%
    \advance\lcurScoreEnd by\wd\ccurBox%
    \advance\lcurScoreEnd by3\wd\myBoxA%
    \displace=\lcurScoreEnd%
    \advance\displace by -\lcurScoreStart%
    \lcurCenter=.5\displace%
    \advance\lcurCenter by\lcurScoreStart%
    \ifx\rootAtBottomFlag\myTrue%
        \setbox\lcurBox=%
            \hbox{\box\lcurBox\unhcopy\myBoxA\box\ccurBox%
                      \unhcopy\myBoxA\box\rcurBox
                      \unhcopy\myBoxA\box\rrcurBox}%
    \else%
        \htLbox = \ht\lcurBox%
        \htAbox = \ht\myBoxA%
        \htCbox = \ht\ccurBox%
        \htRbox = \ht\rcurBox%
        \htRRbox = \ht\rrcurBox%
        \setbox\lcurBox=%
            \hbox{\lower\htLbox\box\lcurBox%
                  \lower\htAbox\copy\myBoxA\lower\htCbox\box\ccurBox%
                  \lower\htAbox\copy\myBoxA\lower\htRbox\box\rcurBox%
                  \lower\htAbox\copy\myBoxA\lower\htRRbox\box\rrcurBox}%
    \fi%
    \displace=\newCenter%
    \advance\displace by -.5\newScoreStart%
    \advance\displace by -.5\newScoreEnd%
    \advance\lcurCenter by \displace%
    \edef\curBox{\lcurBox}%
    \edef\curScoreStart{\lcurScoreStart}%
    \edef\curScoreEnd{\lcurScoreEnd}%
    \edef\curCenter{\lcurCenter}%
    \joinUnary%
}

\def\prepQuinary{%
    \ifnum\theLevel<5
        \errmessage{Hypotheses missing!}
    \fi%
    \edef\rrrcurBox{\thecur{myBox}}%   Set up names of very very right hypothesis
    \edef\rrrcurScoreStart{\thecur{myScoreStart}}%
    \edef\rrrcurCenter{\thecur{myCenter}}%
    \edef\rrrcurScoreEnd{\thecur{myScoreEnd}}%
    \advance\theLevel by-1%
    \edef\rrcurBox{\thecur{myBox}}%   Set up names of very right hypothesis
    \edef\rrcurScoreStart{\thecur{myScoreStart}}%
    \edef\rrcurCenter{\thecur{myCenter}}%
    \edef\rrcurScoreEnd{\thecur{myScoreEnd}}%
    \advance\theLevel by-1%
    \edef\rcurBox{\thecur{myBox}}%   Set up names of right hypothesis
    \edef\rcurScoreStart{\thecur{myScoreStart}}%
    \edef\rcurCenter{\thecur{myCenter}}%
    \edef\rcurScoreEnd{\thecur{myScoreEnd}}%
    \advance\theLevel by-1%
    \edef\ccurBox{\thecur{myBox}}% Set up names of center hypothesis
    \edef\ccurScoreStart{\thecur{myScoreStart}}%
    \edef\ccurCenter{\thecur{myCenter}}%
    \edef\ccurScoreEnd{\thecur{myScoreEnd}}%
    \advance\theLevel by-1%
    \edef\lcurBox{\thecur{myBox}}% Set up names of left hypothesis
    \edef\lcurScoreStart{\thecur{myScoreStart}}%
    \edef\lcurCenter{\thecur{myCenter}}%
    \edef\lcurScoreEnd{\thecur{myScoreEnd}}%
}

\def\QuinaryInf$#1\fCenter#2${%
    \prepQuinary%
    \buildConclusion{#1}{#2}%
    \joinQuinary%
    \resetInferenceDefaults%
    \ignorespaces%
}

\def\QuinaryInfC#1{%
    \prepQuinary%
    \buildConclusionC{#1}%
    \joinQuinary%
    \resetInferenceDefaults%
    \ignorespaces%
}

\def\joinQuinary{% Construct the quinary inference into a vbox.
    \setbox\myBoxA=\hbox{\theHypSeparation}%
    \lcurScoreEnd=\rrrcurScoreEnd%
    \advance\lcurScoreEnd by\wd\rrcurBox%
    \advance\lcurScoreEnd by\wd\rcurBox%
    \advance\lcurScoreEnd by\wd\lcurBox%
    \advance\lcurScoreEnd by\wd\ccurBox%
    \advance\lcurScoreEnd by4\wd\myBoxA%
    \displace=\lcurScoreEnd%
    \advance\displace by -\lcurScoreStart%
    \lcurCenter=.5\displace%
    \advance\lcurCenter by\lcurScoreStart%
    \ifx\rootAtBottomFlag\myTrue%
        \setbox\lcurBox=%
            \hbox{\box\lcurBox\unhcopy\myBoxA\box\ccurBox%
                      \unhcopy\myBoxA\box\rcurBox
                      \unhcopy\myBoxA\box\rrcurBox
                      \unhcopy\myBoxA\box\rrrcurBox}%
    \else%
        \htLbox = \ht\lcurBox%
        \htAbox = \ht\myBoxA%
        \htCbox = \ht\ccurBox%
        \htRbox = \ht\rcurBox%
        \htRRbox = \ht\rrcurBox%
        \htRRRbox = \ht\rrrcurBox%
        \setbox\lcurBox=%
            \hbox{\lower\htLbox\box\lcurBox%
                  \lower\htAbox\copy\myBoxA\lower\htCbox\box\ccurBox%
                  \lower\htAbox\copy\myBoxA\lower\htRbox\box\rcurBox%
                  \lower\htAbox\copy\myBoxA\lower\htRRbox\box\rrcurBox%
                  \lower\htAbox\copy\myBoxA\lower\htRRRbox\box\rrrcurBox}%
    \fi%
    \displace=\newCenter%
    \advance\displace by -.5\newScoreStart%
    \advance\displace by -.5\newScoreEnd%
    \advance\lcurCenter by \displace%
    \edef\curBox{\lcurBox}%
    \edef\curScoreStart{\lcurScoreStart}%
    \edef\curScoreEnd{\lcurScoreEnd}%
    \edef\curCenter{\lcurCenter}%
    \joinUnary%
}

\def\buildConclusion#1#2{% Build lower sequent w/ center at \fCenter position.
        \setbox\myBoxA=\hbox{$\mathord{#1}\fCenter\mathord{\relax}$}%
        \setbox\myBoxB=\hbox{$#2$}%
    \setbox\myBoxC =%
          \hbox{\hskip\ScoreOverhangLeft\relax%
        \unhcopy\myBoxA\unhcopy\myBoxB\hskip\ScoreOverhangRight\relax}%
    \newScoreStart=0pt \relax%
    \newCenter=\wd\myBoxA \relax%
    \advance \newCenter by \ScoreOverhangLeft%
    \newScoreEnd=\wd\myBoxC%
}

\def\buildConclusionC#1{% Build lower sequent w/o \fCenter present.
    \setbox\myBoxA=\hbox{#1}%
    \setbox\myBoxC =%
        \hbox{\hbox{\hskip\ScoreOverhangLeft\relax%
                        \unhcopy\myBoxA\hskip\ScoreOverhangRight\relax}}%
    \newScoreStart=0pt \relax%
    \newCenter=.5\wd\myBoxC \relax%
    \newScoreEnd=\wd\myBoxC%
        \advance \newCenter by \ScoreOverhangLeft%
}

\def\joinUnary{%Align and join \curBox and \myBoxC into a single vbox
    \global\advance\curCenter by -\hypKernAmt%
    \ifnum\curCenter<\newCenter%
        \displace=\newCenter%
        \advance \displace by -\curCenter%
        \kernUpperBox%
    \else%
        \displace=\curCenter%
        \advance \displace by -\newCenter%
        \kernLowerBox%
    \fi%
        \ifnum \newScoreStart < \curScoreStart %
        \global \curScoreStart = \newScoreStart \fi%
    \ifnum \curScoreEnd < \newScoreEnd %
        \global \curScoreEnd = \newScoreEnd \fi%
    \ifnum \curScoreStart<\wd\myBoxLL%
        \global\displace = \wd\myBoxLL%
        \global\advance\displace by -\curScoreStart%
        \kernUpperBox%
        \kernLowerBox%
    \fi%
    \buildScore%
    \buildScoreLabels%
    \ifx\rootAtBottomFlag\myTrue%
        \buildRootBottom%
    \else%
        \buildRootTop%
    \fi%
    \global \curScoreStart=\newScoreStart%
    \global \curScoreEnd=\newScoreEnd%
    \global \curCenter=\newCenter%
}

\def\buildRootBottom{%
    \global \setbox \curBox =%
        \vbox{\box\curBox%
            \vskip\thisAboveSkip \relax%
            \nointerlineskip\box\myBoxD%
            \vskip\thisBelowSkip \relax%
            \nointerlineskip\box\myBoxC}%
}

\def\buildRootTop{%
    \global \setbox \curBox =%
        \vbox{\box\myBoxC%
            \vskip\thisAboveSkip \relax%
            \nointerlineskip\box\myBoxD%
            \vskip\thisBelowSkip \relax%
            \nointerlineskip\box\curBox}%
}

\def\kernUpperBox{%
        \global\setbox\curBox =%
            \hbox{\hskip\displace\box\curBox}%
        \global\advance \curScoreStart by \displace%
        \global\advance \curScoreEnd by \displace%
        \global\advance\curCenter by \displace%
}

\def\kernLowerBox{%
        \global\setbox\myBoxC =%
            \hbox{\hskip\displace\unhbox\myBoxC}%
        \global\advance \newScoreStart by \displace%
        \global\advance \newScoreEnd by \displace%
        \global\advance\newCenter by \displace%
}

\def\joinBinary{% Construct the binary inference into a vbox.
    \setbox\myBoxA=\hbox{\theHypSeparation}%
    \lcurScoreEnd=\rcurScoreEnd%
    \advance\lcurScoreEnd by\wd\lcurBox%
    \advance\lcurScoreEnd by\wd\myBoxA%
    \displace=\lcurScoreEnd%
    \advance\displace by -\lcurScoreStart%
    \lcurCenter=.5\displace%
    \advance\lcurCenter by\lcurScoreStart%
    \ifx\rootAtBottomFlag\myTrue%
        \setbox\lcurBox=%
            \hbox{\box\lcurBox\unhcopy\myBoxA\box\rcurBox}%
    \else%
        \htLbox = \ht\lcurBox%
        \htAbox = \ht\myBoxA%
        \htRbox = \ht\rcurBox%
        \setbox\lcurBox=%
            \hbox{\lower\htLbox\box\lcurBox%
                  \lower\htAbox\box\myBoxA\lower\htRbox\box\rcurBox}%
    \fi%
    \displace=\newCenter%
    \advance\displace by -.5\newScoreStart%
    \advance\displace by -.5\newScoreEnd%
    \advance\lcurCenter by \displace%
    \edef\curBox{\lcurBox}%
    \edef\curScoreStart{\lcurScoreStart}%
    \edef\curScoreEnd{\lcurScoreEnd}%
    \edef\curCenter{\lcurCenter}%
    \joinUnary%
}

\def\joinTrinary{% Construct the trinary inference into a vbox.
    \setbox\myBoxA=\hbox{\theHypSeparation}%
    \lcurScoreEnd=\rcurScoreEnd%
    \advance\lcurScoreEnd by\wd\lcurBox%
    \advance\lcurScoreEnd by\wd\ccurBox%
    \advance\lcurScoreEnd by2\wd\myBoxA%
    \displace=\lcurScoreEnd%
    \advance\displace by -\lcurScoreStart%
    \lcurCenter=.5\displace%
    \advance\lcurCenter by\lcurScoreStart%
    \ifx\rootAtBottomFlag\myTrue%
        \setbox\lcurBox=%
            \hbox{\box\lcurBox\unhcopy\myBoxA\box\ccurBox%
                      \unhcopy\myBoxA\box\rcurBox}%
    \else%
        \htLbox = \ht\lcurBox%
        \htAbox = \ht\myBoxA%
        \htCbox = \ht\ccurBox%
        \htRbox = \ht\rcurBox%
        \setbox\lcurBox=%
            \hbox{\lower\htLbox\box\lcurBox%
                  \lower\htAbox\copy\myBoxA\lower\htCbox\box\ccurBox%
                  \lower\htAbox\copy\myBoxA\lower\htRbox\box\rcurBox}%
    \fi%
    \displace=\newCenter%
    \advance\displace by -.5\newScoreStart%
    \advance\displace by -.5\newScoreEnd%
    \advance\lcurCenter by \displace%
    \edef\curBox{\lcurBox}%
    \edef\curScoreStart{\lcurScoreStart}%
    \edef\curScoreEnd{\lcurScoreEnd}%
    \edef\curCenter{\lcurCenter}%
    \joinUnary%
}

\def\DisplayProof{%
    \ifnum \theLevel=1 \relax \else%x
        \errmessage{Proof tree badly specified.}%
    \fi%
    \edef\curBox{\thecur{myBox}}%
    \ifx\bottomAlignFlag\myTrue%
        \displace=0pt%
    \else%
        \displace=.5\ht\curBox%
        \ifx\centerAlignFlag\myTrue\relax
        \else%
                \advance\displace by -3pt%
        \fi%
    \fi%
    \leavevmode%
    \lower\displace\hbox{\copy\curBox}%
    \global\theLevel=0%
    \global\def\alwaysBuildScore{\defaultBuildScore}% Restore "always"
    \global\def\alwaysScoreFiller{\defaultScoreFiller}% Restore "always"
    \global\def\bottomAlignFlag{N}%
    \global\def\centerAlignFlag{N}%
    \resetRootPosition
    \resetInferenceDefaults%
    \ignorespaces
}

\def\buildSingleScore{% Make an hbox with a single score.
    \displace=\curScoreEnd%
    \advance \displace by -\curScoreStart%
    \global\setbox \myBoxD =%
        \hbox to \displace{\expandafter\xleaders\theScoreFiller\hfill}%
}

\def\buildDoubleScore{% Make an hbox with a double score.
    \buildSingleScore%
    \global\setbox\myBoxD=%
        \hbox{\hbox to0pt{\copy\myBoxD\hss}\raise2pt\copy\myBoxD}%
}

\def\buildNoScore{% Make an hbox with no score (raise a little anyway)
    \global\setbox\myBoxD=\hbox{\vbox{\vskip1pt}}%
}

\def\noLine{%
    \gdef\buildScore{\buildNoScore}% Set next score to this type
    \ignorespaces
}

\def\LeftLabel#1{%
    \global\setbox\myBoxLL=\hbox{{#1}\hskip\labelSpacing}%
    \ignorespaces
}
\def\RightLabel#1{%
    \global\setbox\myBoxRL=\hbox{\hskip\labelSpacing #1}%
    \ignorespaces
}

\def\buildScoreLabels{%
    \scoreHeight = \ht\myBoxD%
    \scoreDepth = \dp\myBoxD%
    \leftLowerAmt=\ht\myBoxLL%
    \advance \leftLowerAmt by -\dp\myBoxLL%
    \advance \leftLowerAmt by -\scoreHeight%
    \advance \leftLowerAmt by \scoreDepth%
    \leftLowerAmt=.5\leftLowerAmt%
    \rightLowerAmt=\ht\myBoxRL%
    \advance \rightLowerAmt by -\dp\myBoxRL%
    \advance \rightLowerAmt by -\scoreHeight%
    \advance \rightLowerAmt by \scoreDepth%
    \rightLowerAmt=.5\rightLowerAmt%
    \displace = \curScoreStart%
    \advance\displace by -\wd\myBoxLL%
    \global\setbox\myBoxD =%
        \hbox{\hskip\displace%
            \lower\leftLowerAmt\copy\myBoxLL%
            \box\myBoxD%
            \lower\rightLowerAmt\copy\myBoxRL}%
    \global\thisAboveSkip = \ht\myBoxLL%
    \global\advance \thisAboveSkip by -\leftLowerAmt%
    \global\advance \thisAboveSkip by -\scoreHeight%
    \ifnum \thisAboveSkip<0 %
        \global\thisAboveSkip=0pt%
    \fi%
    \displace = \ht\myBoxRL%
    \advance \displace by -\rightLowerAmt%
    \advance \displace by -\scoreHeight%
    \ifnum \displace<0 %
        \displace=0pt%
    \fi%
    \ifnum \displace>\thisAboveSkip %
        \global\thisAboveSkip=\displace%
    \fi%
    \global\thisBelowSkip = \dp\myBoxLL%
    \global\advance\thisBelowSkip by \leftLowerAmt%
    \global\advance\thisBelowSkip by -\scoreDepth%
    \ifnum\thisBelowSkip<0 %
        \global\thisBelowSkip = 0pt%
    \fi%
    \displace = \dp\myBoxRL%
    \advance\displace by \rightLowerAmt%
    \advance\displace by -\scoreDepth%
    \ifnum\displace<0 %
        \displace = 0pt%
    \fi%
    \ifnum\displace>\thisBelowSkip%
        \global\thisBelowSkip = \displace%
    \fi%
    \global\thisAboveSkip = -\thisAboveSkip%
    \global\thisBelowSkip = -\thisBelowSkip%
    \global\advance\thisAboveSkip by\extraVskip% Extra space above line
    \global\advance\thisBelowSkip by\extraVskip% Extra space below line
}

\def\resetInferenceDefaults{%
    \global\def\theHypSeparation{\defaultHypSeparation}%
    \global\setbox\myBoxLL=\hbox{\defaultLeftLabel}%
    \global\setbox\myBoxRL=\hbox{\defaultRightLabel}%
    \global\def\buildScore{\alwaysBuildScore}%
    \global\def\theScoreFiller{\alwaysScoreFiller}%
    \gdef\hypKernAmt{0pt}% Restore to zero kerning.
}

\def\resetRootPosition{%
    \global\edef\rootAtBottomFlag{\defaultRootAtBottomFlag}
}

\def\provx#1#2#3#4{
\setbox1=\hbox{\kern1.5pt$\scriptstyle#3$}
\def\zeichen{#2}
\ifx\zeichen\empty\setbox0=\hbox to .75em{}\else\setbox0=\hbox
{\kern1.5pt$\scriptstyle#2$}\fi
\dimen1=\dp0 \ifdim \dimen1=0pt
\advance \dimen1 by 1.5ex \else \advance \dimen1 by 1.2ex
\fi\dimen3=2ex\dimen4=.5ex\ifdim \wd0<\wd1 \dimen2=\wd1 \else \dimen2=\wd0
\fi\hbox{$#1\hskip 5pt minus5pt\vrule height\dimen3
depth\dimen4\raise\dimen1\copy0\hskip-1\wd0 \lower\ht1
\copy1\hskip-1\wd1\vrule width\dimen2 height.7ex depth-.6ex\hskip3pt
minus1.5pt#4\hskip2pt plus2pt minus2pt$}}

\def\prov#1#2#3{
\setbox1=\hbox{\kern1.5pt$\scriptstyle#2$}
\def\zeichen{#1}
\ifx\zeichen\empty\setbox0=\hbox to .75em{}\else\setbox0=\hbox
{\kern1.5pt$\scriptstyle#1$}\fi
\dimen1=\dp0
\ifdim \dimen1=0pt
\advance \dimen1 by 1.5ex \else \advance \dimen1 by 1.2ex
\fi\dimen3=2ex\dimen4=.5ex\ifdim \wd0<\wd1 \dimen2=\wd1 \else \dimen2=\wd0
\fi\hbox{\hskip0pt plus 4pt
$\vrule height\dimen3
depth\dimen4\raise\dimen1\copy0\hskip-1\wd0
\lower\ht1\copy1\hskip-1\wd1\vrule width\dimen2 height.7ex depth-.6ex
\hskip3pt minus1.5pt#3\hskip2pt plus2pt minus2pt$}}

\def\prv#1#2{
\setbox1=\hbox{\kern1.5pt$\scriptstyle#2$}
\ifx\zeichen\empty\setbox0=\hbox to .75em{}\else\setbox0=\hbox
{\kern1.5pt$\scriptstyle#1$}\fi
\dimen1=\dp0 \ifdim \dimen1=0pt
\advance \dimen1 by 1.5ex \else \advance \dimen1 by 1.2ex
\fi\dimen3=2ex\dimen4=.5ex\ifdim \wd0<\wd1 \dimen2=\wd1 \else \dimen2=\wd0
\fi\hbox{\hskip.5em$\vrule height\dimen3
depth\dimen4\raise\dimen1\copy0\hskip-1\wd0
\lower\ht1\copy1\hskip-1\wd1\vrule width\dimen2 height.7ex depth-.6ex
\hskip3pt minus1.5pt$}}

\mathchardef\str='1066
\def\negprov#1#2#3{
\setbox1=\hbox{\kern1.5pt$\scriptstyle#2$}
\setbox4=\hbox{$\str$}
\def\zeichen{#1}
\ifx\zeichen\empty\setbox0=\hbox to 1em{}\else\setbox0=\hbox
{\kern1.5pt$\scriptstyle#1$}\fi
\dimen1=\dp0
\ifdim \dimen1=0pt
\advance \dimen1 by 1.5ex \else \advance \dimen1 by 1.2ex
\fi\dimen3=2ex\dimen4=.5ex\ifdim \wd0<\wd1 \dimen2=\wd1 \else \dimen2=\wd0
\fi\hbox{\hskip.5em$\kern-1.9pt\raise1pt\copy4\kern-\wd4\kern1.9pt\vrule height\dimen3
depth\dimen4\raise\dimen1\copy0\hskip-1\wd0
\lower\ht1\copy1\hskip-1\wd1\vrule width\dimen2 height.7ex depth-.6ex
\hskip3pt minus1.5pt#3\hskip2pt plus2pt minus2pt$}}

\def\goed#1{\setbox5=\hbox{$#1$}\dimen1=.25em \dimen2=\dimen1 \advance \dimen2
by -1pt\hbox{\raise.65\ht5 \hbox{\vrule height.5\ht5 depth0pt width.4pt\vrule
height.5\ht5 width\dimen1 depth-.48\ht5}\kern-\dimen2\copy5\kern-\dimen2
\raise.65\ht5 \hbox{\vrule height .5\ht5 width\dimen1 depth-.48\ht5\vrule
height.5\ht5 depth 0pt width.4pt}\hskip4pt plus2pt minus2pt}}

\def\mod#1#2{
\def\zeichen{#1}
\hbox{\hskip 2pt plus3pt minus 2pt\vrule width.5pt height2ex depth.5ex
\vbox{\ifx\zeichen\empty\hbox to .75em{}\else
\hbox{\kern1.5pt $\scriptstyle#1$}\fi
\kern2pt
\hrule
\kern1.7pt
\hrule\kern1.7pt}
\hskip3pt minus 2pt$#2$}\hskip2pt
plus3pt minus2pt}

\def\notmod#1#2{\hbox{\hskip 2pt plus 3pt minus 3pt\vrule width.5pt
height2ex depth.5ex
\vbox{\hbox{\kern1.5pt $\scriptstyle#1$}\kern3pt
\setbox0=\hbox{\kern2pt$\scriptstyle/$}
\hrule
\kern-1.7pt
\copy0
\kern-\ht0
\kern 1.7pt
\hrule\kern1.7pt}n
\hskip3pt minus 2pt$#2$}\hskip2pt
plus3pt minus2pt}
\def\sq{\hbox{\rlap{$\sqcap$}$\sqcup$}}
\def\qed{\ifmmode\sq\else{\unskip\nobreak\hfil\penalty50\hskip1em\null
\nobreak\hfil\sq\parfillskip=0pt\finalhyphendemerits=0\endgraf}\fi\medskip}

\def\lleq{\hbox{\hskip3pt minus3pt\kern1pt\lower4pt
\vbox{\hbox{$\scriptstyle\ll$}
\kern-7pt\hbox{\kern1pt$\scriptstyle=$}}\hskip3pt minus 3pt}}

\mathchardef\res='1152
\mathchardef\qin='1062
\mathchardef\qprec='1036
\mathchardef\qless='474
\mathchardef\dpkt='72

 \newcommand{\Avarphi}{A}
 \newcommand{\Apsi}{B}
 \newcommand{\Avartheta}{C}
 \newcommand{\funci}[3]{\mathrm{Fun}({#1},{#2},{#3})}
\newcommand{\strictSigma}{\mbox{strict-}\Sigma}

\begin{document}
\title{Ordinal analysis of Intuitionistic Power and Exponentiation Kripke Platek Set Theory\footnote{Dedicated to Gerhard J\"ager on the occasion of his 60th birthday.}}
\author{Jacob Cook and Michael Rathjen\\
{\small Department of Pure Mathematics, University of Leeds}\\
{\small Leeds LS2 9JT, United Kingdom}\\
 {\tt \scriptsize mmjcoo@leeds.ac.uk $\;\;$
rathjen@maths.leeds.ac.uk}
}

%\date{}
\maketitle
\begin{abstract} Until the 1970s, proof theoretic investigations  were mainly concerned with theories of inductive definitions, subsystems of analysis and finite type systems.   With the pioneering work of Gerhard J\"ager in the late 1970s and early 1980s, the focus switched to set theories, furnishing ordinal-theoretic proof theory  with a uniform and elegant framework. More recently it was shown that these tools can even sometimes be adapted to the context of strong axioms such as the powerset axiom, where one does not attain complete cut elimination but can nevertheless
extract witnessing information and characterize the strength of the theory in terms of provable heights of the cumulative hierarchy. Here this technology is applied to intuitionistic Kripke-Platek set theories $\IKPP$ and $\IKPE$, where
the operation of  powerset and exponentiation, respectively, is allowed as a primitive in the separation and collection schemata. In particular, $\IKPP$ proves the powerset axiom whereas $\IKPE$ proves the  exponentiation axiom. The latter expresses that given any sets $A$ and $B$, the collection of all
functions from $A$ to $B$ is a set, too. While $\IKPP$ can be dealt with in a similar vein as its classical cousin, the treatment of $\IKPE$ posed considerable obstacles. One of them was that in the infinitary system the levels of terms become a moving target as they
cannot be assigned a fixed level in the formal cumulative hierarchy solely based on their syntactic structure.

 As adumbrated in an earlier paper, the results of this paper are an important tool in showing that several intuitionistic set theories with the collection axiom possess the existence property, i.e., if they prove an existential theorem then a witness can be provably described in the theory, one example being intuitionistic Zermelo-Fraenkel set theory with bounded separation.
\end{abstract}
\section{Introduction}
In his early work, Gerhard J\"ager laid the foundations for a direct proof-theoretic
treatment of set theories
(cf. \cite{j80,j82}) which then began to a large extent to supplant earlier work on theories of inductive definitions
and subsystems of analysis.
By and large, ordinal analyses for set theories are more uniform
and transparent than for the latter theories.
The primordial example of a set theory amenable to ordinal analysis is
Kripke-Platek set theory, $\KP$. It is an important theory for various reasons, one being that a great deal of set theory requires only
the axioms of $\KP$. Another reason is that
admissible sets, the transitive models of $\KP$, have been a major source of interaction between
model theory of infinitary languages, recursion theory and set theory (cf. \cite{ba}).
 $\KP$ arises from $\ZF$ by completely omitting the power set axiom and
restricting separation and collection to bounded formulae.
Many of the familiar subsystems of second order arithmetic can be viewed as reduced versions
of set theories based on the notion of admissible set. This applies for example to a fairly strong theory like $\Delta^1_2\mbox{-}\mathbf{CA}$ plus bar induction which is of the same strength as $\KP$ augmented
by an axiom saying that every set is contained in an admissible set, whose ordinal analysis is due to J\"ager and Pohlers \cite{jp}. By restricting or completely omitting induction principles in theories
of admissible sets, J\"ager was also able to give a unified proof-theoretic treatment
of many predicative theories in \cite{j86}.
Systematic accounts and surveys of admissible proof theory can be found in \cite{j82a,j86,bu93,po98,mi99realm} and other places.

Ordinal analyses of ever stronger theories have been obtained over the last 20 years. The strongest systems for which proof-theoretic ordinals have been determined are extensions of $\KP$ augmented by $\Sigma_1$-separation that correspond to subsystems of
second-order arithmetic with comprehension restricted to $\Pi^1_2$ comprehension or iterations thereof (cf. \cite{arai3,r94b,r2005}).
Thus it appears that it is currently impossible to furnish an ordinal analysis of any set theory which has the
power set axiom among its axioms as such a theory would dwarf the strength of second-order arithmetic.
It is, however, possible to relativize the techniques of ordinal analysis developed for
Kripke-Platek set theory to obtain useful information about Power Kripke-Platek set theory as shown in \cite{powerKP}.
The kind of information one can extract concerns  bounds for the transfinite iterations of the power set operation that are provable in the latter
theory. In this paper the method is applied to intuitionistic
 Kripke-Platek set theories $\IKPP$ and $\IKPE$, where
the operation of  powerset,  respectively, exponentiation, is allowed as a primitive in the separation and collection schemata. In particular, $\IKPP$ proves the powerset axiom whereas $\IKPE$ proves the  exponentiation axiom. The latter expresses that given any sets $A$ and $B$, the collection of all
functions from $A$ to $B$ is a set, too. While $\IKPP$ can be dealt with in a similar vein as its classical cousin in \cite{powerKP}, the treatment of $\IKPE$ posed considerable obstacles, one of them being that in the infinitary system terms
cannot be assigned a fixed level in the formal cumulative hierarchy.

  It was outlined in \cite{weak} that the results of this paper are an important tool for showing that several intuitionistic set theories with the collection axiom possess the existence property, i.e., if they prove an existential theorem then a witness can be provably described in the theory. One example for such a theory is intuitionistic Zermelo-Fraenkel set theory with bounded separation. Details will be presented in \cite{rathjen-EP}.

 \subsection{Intuitionistic set theories and the existence property}
 Intuitionistic theories are known to often possess very pleasing metamathematical properties
such as the disjunction property and the numerical existence property. While it is fairly easy to establish these properties
for arithmetical theories and theories with quantification over sets of natural numbers  or Baire space (e.g. second order arithmetic and function arithmetic), set theories with their transfinite hierarchies of sets of sets and the extensionality axiom can pose considerable technical challenges.

\begin{deff}\label{deff1}{\em
 Let $T$ be a theory whose language, $L(T)$,
 encompasses the language of set theory.
  $T$ has the {\em
existence property}, {\bf EP}, if whenever $T\vdash \exists x
A(x)$ holds for a formula $A(x)$ having at most the free
variable $x$, then  there is a formula $C(x)$ with exactly
$x$ free, so that $$T\vdash
\exists!x\,[C(x)\wedge A(x)].$$
}\end{deff}
A theory that does not have the existence property is  intuitionistic Zermelo-Fraenkel set theory, $\IZF$,
formulated with Collection, as was shown in \cite{fs}. Since the version of $\IZF$ with replacement in lieu of collection has the existence property, collection is clearly implicated in the failure of {\bf EP}. This prompted Beeson
 in \cite[IX.1]{beeson} to ask the following question:
\begin{quote}{\em  Does any reasonable set theory with collection have the existence property?}\end{quote}
An important theory that is closely related to Martin-L\"of type theory is Constructive Zermelo-Fraenkel set theory, $\CZF$ (cf. \cite{mar,book}). It has been shown in \cite{swan} that $\CZF$ lacks the $\EP$. While the proof is quite difficult, the failure of $\EP$ is perhaps not that surprising since $\CZF$ features an axiom, Subset Collection, that is a consequence of exponentiation and a choice principle called the presentation axiom.\footnote{Although sometimes even set theories with strong choice principles can have the $\EP$ (see \cite{rt}).} However, in \cite{weak} it was shown that three perhaps more natural versions of $\CZF$ possess the weak existence property, which requires a provably definable inhabited set of witnesses for every existential theorem. Tellingly, neither $\IZF$ nor $\CZF$ has the weak existence property (see \cite{swan} and \cite[Proposition 1.3]{weak}). The three versions of $\CZF$ shown to have the weak existence property   are $\CZF$ without subset collection ($\CZF^-$), $\CZF$ with exponentiation instead of subset collection ($\CZF_{\mathcal E}$), and $\CZF$ augmented by the powerset axiom
($\CZF_{\mathcal P}$). \cite{weak} also provided reductions of these three theories
%$\CZF^-$,  $\CZF_{\mathcal E}$, $\CZF_{\mathcal P}$
to  pertaining versions of intuitionistic Kripke-Platek set theories %and extensions thereof, respectively,
in such a  way that if the latter theories possessed the existence property for pertaining syntactically restricted classes
of existential theorems, then the former would possess the full {\bf EP}.
This gave rise to the strategy of first embedding these extended theories of intuitionistic Kripke-Platek set theory into infinitary proof systems and use techniques of ordinal analysis
to remove  those inferences which embody collection. The second step, then,  consists in showing that the infinitary systems have the term existence property, i.e., for each provable existential theorem there is a witnessing term. It will then ensue from the fact that the numerical existence property holds for $\CZF^-$,  $\CZF_{\mathcal E}$, and $\CZF_{\mathcal P}$ that the existence property holds for these theories, too. The numerical existence property
was verified in \cite{tklr} and also holds for $\CZF$, even when augmented by various choice principles \cite{tklrac}

\subsection{Intuitionistic power and exponentiation Kripke-Platek set theories}
We call a formula {\em bounded} or $\Delta_0$ if all its quantifiers are of the form $\forall x\in a$ and $\exists y\in b$.
 The axioms of classical $\KP$ consist of
{\em Extensionality, Pair, Union, Infinity, Bounded Separation}
 $$\exists x\,\forall u\left[u\in x\leftrightarrow(u\in
 a\,\wedge\,\Avarphi(u))\right]$$
 for all bounded formulae $\Avarphi(u)$, {\em Bounded Collection}
 $$\forall x\in a\,\exists y\,\Apsi(x,y)\,\to\,\exists z\,\forall x\in
 a\,\exists y\in z\,\Apsi(x,y)$$
 for all bounded formulae $\Apsi(x,y)$, and
 {\em Set Induction}
 $$\forall x\,\left[(\forall y\in x\,\Avartheta(y))\to
 \Avartheta(x)\right]\,\to\,\forall x\,\Avartheta(x)$$
 for all formulae $\Avartheta(x)$.

 We denote by $\IKP$ the version of $\KP$ where the underlying logic is intuitionistic logic.

We use subset bounded quantifiers $\exists x\subseteq y \;\ldots$ and $\forall x\subseteq y\;\ldots$ as abbreviations
for $\exists x(x\subseteq y\und \ldots)$ and $\forall x(x\subseteq y\to \ldots)$, respectively.

  We call a formula of ${\mathcal L}_\in$ $\Deltaop$ if all its quantifiers
are of the form $Q\,x\subseteq y$ or $Q\,x{\in}y$ where $Q$ is $\forall$ or $\exists$ and $x$ and $y$
are distinct variables.

Let $\funci fxy$ be a acronym for the bounded formula expressing that $f$ is a function with domain $x$ and co-domain $y$.
We use exponentiation bounded quantifiers $\exists f\in \hoch{x}{y}  \;\ldots$ and $\forall f\in \hoch{x}{y}\;\ldots$ as abbreviations
for $\exists f(\funci fxy\und \ldots)$ and $\forall x(\funci fxy\to \ldots)$, respectively.

\begin{deff}\label{kp1} {\em
The $\Delta_0^{\mathcal P}$-formulae are the smallest class of formulae containing the atomic formulae closed under $\wedge,\vee,\to,\neg$ and the quantifiers
  $$\forall x\in a,\;\exists x\in a,\; \forall x\subseteq a,\; \exists x\subseteq a.$$
  The $\Delta_0^{\mathcal E}$-formulae are the smallest class of formulae containing the atomic formulae closed under $\wedge,\vee,\to,\neg$ and the quantifiers
  $$\forall x\in a,\;\exists x\in a,\; \forall f\in \hoch{a}{b},\; \exists f\in \hoch{a}{b}.$$
  }\end{deff}

\begin{deff}\label{kp2}{\em $\IKPE$ has the same language and logic as $\IKP$. Its axioms are the following: Extensionality, Pairing, Union, Infinity, Exponentiation, Set Induction,
       $\Delta_0^{\mathcal E}$-Separation and $\Delta_0^{\mathcal E}$-Collection.

 $\IKPP$ has the same language and logic as $\IKP$. Its axioms are the following: Extensionality, Pairing, Union, Infinity, Powerset, Set Induction,
       $\Delta_0^{\mathcal P}$-Separation and $\Delta_0^{\mathcal P}$-Collection.
       \\[1ex]
  The  transitive classical models of $\IKPP$ have been termed {\bf power admissible} sets in \cite{friedaCount}.
  There is also a close connection between $\IKPP$ and versions of Martin-L\"of type theory with an impredicative
  type of propositions and the calculus of constructions (see \cite{ML}).
}\end{deff}

\begin{rem}\label{kp3}{\em
 Alternatively, $\IKPP$ can be obtained from $\IKP$ by adding  a function symbol $\PC$ for the powerset function as a primitive symbols to the language and the axiom
$$\forall y\,[y\in \PC(x)\leftrightarrow y\subseteq x]$$
and extending the schemes of $\Delta_0$-Separation and Collection to the $\Delta_0$-formulae of this new language.

Likewise, $\IKPE$ can be obtained from $\IKP$ by adding a primitive function symbol $\mathcal{E}$ for the exponentiation and
the pertaining axioms.}
\end{rem}

\begin{deff}\label{kp5} {\em The class of $\Sigma$-formulae in the strict sense, denoted by $\strictSigma$, is the smallest class of formulae containing the $\Delta_0$-formulae closed under $\wedge,\vee$ and the quantifiers
  $$\forall x\in a,\;\exists x\in a,\;\exists x.$$
The $\strictSigma^{\mathcal P}$-formulae are the smallest class of formulae containing the $\Delta_0^{\mathcal P}$-formulae closed under $\wedge,\vee$ and the quantifiers
  $$\forall x\in a,\;\exists x\in a,\; \forall x\subseteq a,\; \exists x\subseteq a,\;\exists x.$$
  The $\strictSigma^{\mathcal E}$-formulae are the smallest class of formulae containing the $\Delta_0^{\mathcal E}$-formulae closed under $\wedge,\vee$ and the quantifiers
  $$\forall x\in a,\;\exists x\in a,\;\forall f\in \hoch{a}{b},\; \exists f\in \hoch{a}{b},\;\exists x.$$
  }\end{deff}
  \noindent
  Later on we shall have occasion to introduce the wider class of $\Sigma$-formulae.
  A crucial property that singles out the $\strictSigma$-formulae is that every such formula is equivalent
  to a $\Sigma_1$-formula provably in $\IKP$, where a $\Sigma_1$-formula is a formula that starts with a single existential quantifier and thereafter continues with a $\Delta_0$-formula. A similar characterization holds for the
  $\strictSigma^{\mathcal P}$ and
  $\strictSigma^{\mathcal E}$-formulae.
  It is standard to show that $\IKPP$ proves $\strictSigma^{\mathcal P}$-Collection and $\IKPE$ proves
  $\strictSigma^{\mathcal E}$-Collection (for details in the case of $\IKP$ see \cite{mar,book}).

\subsection{Outline of the paper} The main objective of the paper is an in-depth presentation of the ordinal analyses
of the three theories $\IKP$,  $\IKPP$, and $\IKPE$. In the case of $\IKP$ this means characterising the {\em proof theoretic ordinal} in the sense of \cite{mi99realm}, this is done in such a way that we can also extract witness terms from cut-free derivations of existential statements in the infinitary system. In the cases of $\IKPP$ and $\IKPE$ we present a type of relativised ordinal analysis similar to that given in \cite{powerKP}, we characterise the number of iterations of the relevant hierarchy of sets (Von Neumann and Exponentiation respectively) that can be proven to exist within the theory. These details cannot be found in the existing research literature.
They are needed for term extraction and, moreover, have to be shown to be formalizable in $\CZF^-$, i.e., constructive Zermelo-Fraenkel
set theory without subset collection. Naturally, the first system to be analyzed is the simplest. Section 2 treats $\IKP$, so is basically a detailed rendering of the intuitionistic version of the classical analysis of $\KP$.
 However, this section is also used later since several parts of it are modular and thus transfer to the stronger systems without any changes.  Section 3 deals with $\IKPP$ and adapts the machinery
of classical \cite{powerKP} to the intuitionistic case. The system $\IKPE$ is the most difficult to handle and
is addressed in Section 4. A particular challenge is provided by the problem of assigning a rank to the formal
terms of the infinitary system. Ultimately this turned out to be impossible and we had to deal with it in a new way
by allowing for rank declarations as extra hypothesis.

A further reason to document these ordinal analyses in the literature is that the second author fosters hopes that this framework can be used to analyze much stronger theories such as ones with full negative separation. It appears that in striking difference to classical theories, negative separation does not block the
the methods of ordinal analysis that bring about the elimination of collection inferences in intuitionistic
derivations.

\subsection*{Acknowledgement} The work of the second author was supported by a Leverhulme Research Fellowship and the Engineering and Physical Sciences Research Council under grant number EP/K023128/1.

\section{The case of $\IKP$}
This section provides a detailed rendering of the ordinal analysis of Kripke-Platek set theory formulated with intuitionistic logic, $\IKP$. This is done in such a way that we are able to extract witness terms from the resulting cut-free derivations of $\Sigma$ sentences in the infinitary system. This results in a proof that $\IKP$ has the existence property for $\Sigma$ sentences, which in conjunction with results in \cite{weak} verifies that $\CZF^-$ has the full existence property. Many of the arguments in this section are modular and transfer over to the stronger systems analysed in subsequent sections with minimal changes.

\subsection{A sequent calculus formulation of \textbf{IKP}}
\begin{definition}{\em The language of \textbf{IKP} consists of free variables $a_0,a_1,...$, bound variables $x_0,x_1,...$, the binary predicate symbol $\in$ and the logical symbols $\neg,\vee, \wedge, \rightarrow, \forall, \exists$ as well as parentheses $),($.\\

\noindent The atomic formulas are those of the form $a\in b$. \\

\noindent The formulas of \textbf{IKP} are defined inductively by:
\begin{description}
\item[i)] All atomic formulas are formulas.
\item[ii)] If $A$ and $B$ are formulas then so are $\neg A$, $A\vee B$, $A\wedge B$ and $A\rightarrow B$.
\item[iii)] If $A(b)$ is a formula in which the bound variable $x$ does not occur, then $\forall x A(x)$, $\exists x A(x)$, $(\forall x\in a)A(x)$ and $(\exists x\in a)A(x)$ are also formulas.
\end{description}
\noindent Quantifiers of the form $\exists x$ and $\forall x$ will be called unbounded and those of the form $(\exists x\in a)$ and $(\forall x\in a)$ will be referred to as bounded quantifiers.\\

\noindent A $\Delta_0$-formula is one in which no unbounded quantifiers appear.\\

\noindent The expression $a=b$ is to be treated as an abbreviation for $(\forall x\in a)(x\in b)\wedge (\forall x\in b)(x\in a)$.\\

\noindent The derivations of \textbf{IKP} take place in a two-sided sequent calculus. The sequents derived are \textit{intuitionistic sequents} of the form $\Gamma\Rightarrow\Delta$ where $\Gamma$ and $\Delta$ are finite sets of formulas and $\Delta$ contains at most one formula. The intended meaning of $\Gamma\Rightarrow\Delta$ is that the conjunction of formulas in $\Gamma$ implies the formula in $\Delta$, or if $\Delta$ is empty, a contradiction. The expressions $\RI\DE$ and $\GA\RI$ are shorthand for $\emptyset\RI\DE$ and $\GA\RI\emptyset$ respectively. $\Gamma,A$ stands for $\Gamma\cup\{A\}$, $\Gamma\Rightarrow A$ stands for $\Gamma\Rightarrow \{A\}$, etc.

\noindent The axioms of \textbf{IKP} are:\\

\begin{tabular} {ll}
{\em {Logical axioms:}}  & $\Gamma, A \Rightarrow A$ \ for every
$\Delta_0$ formula $A$.\\
{\em {Extensionality:}}  & $\Gamma\Rightarrow a=b\wedge B(a)\rightarrow B(b)$ \ for every $\Delta_0$ formula $B(a)$.\\
{\em {Pair:}} & $\Gamma\Rightarrow\exists z(a\in z\wedge b\in z)$.\\
{\em {Union:}} & $\Gamma\Rightarrow\exists z(\forall y\in z)(\forall x\in y)(x\in z)$.\\
{\em {$\Delta_0$-Separation:}} & $\Gamma\Rightarrow\exists y[(\forall x\in y)(x\in a\wedge B(x))\wedge (\forall x\in a)(B(x)\rightarrow x\in y)]$\\
& for every $\Delta_0$-formula $B(a)$.\\
{\em {Set Induction:}} & $\Gamma\Rightarrow\forall x[(\forall y\in x F(y)\rightarrow F(x)]\rightarrow\forall x F(x)$ for any formula $F(a)$.\\
{\em {Infinity:}} & $\Gamma\Rightarrow\exists x[(\exists z\in x)(z\in x)\wedge(\forall y\in x)(\exists z\in x)(y\in z)]$.\\
{\em {$\Delta_0$-Collection:}} & $\Gamma\Rightarrow(\forall x\in a)\exists y G(x,y)\rightarrow\exists z(\forall x\in a)(\exists y\in z)G(x,y)$ for any $\Delta_0$-formula $G$.
\end{tabular}\\[0.4cm]
\noindent The rules of inference are
\begin{prooftree}
\AxiomC{$\Gamma, C\Rightarrow\Delta$}
\LeftLabel{$(\wedge L)$}
\RightLabel{for $C\in\{A,B\}$.}
\UnaryInfC{$\Gamma,A\wedge B\Rightarrow\Delta$}
\AxiomC{$\Gamma\Rightarrow A$}
\AxiomC{$\Gamma\Rightarrow B$}
\LeftLabel{$(\wedge R)$}
\BinaryInfC{$\Gamma\Rightarrow A\wedge B$}
\noLine
\BinaryInfC{}
\end{prooftree}
\begin{prooftree}
\AxiomC{$\Gamma, A\Rightarrow \Delta$}
\AxiomC{$\Gamma, B\Rightarrow \Delta$}
\LeftLabel{$(\vee L)$}
\BinaryInfC{$\Gamma,A\vee B\Rightarrow\Delta$}
\AxiomC{$\Gamma\Rightarrow C$}
\LeftLabel{$(\vee R)$}
\RightLabel{for $C\in\{A,B\}$.}
\UnaryInfC{$\Gamma\Rightarrow A\vee B$}
\noLine
\BinaryInfC{}
\end{prooftree}
\begin{prooftree}
\AxiomC{$\Gamma\Rightarrow A$}
\LeftLabel{$(\neg L)$}
\UnaryInfC{$\Gamma,\neg A\Rightarrow$}
\AxiomC{$\Gamma,A\Rightarrow $}
\LeftLabel{$(\neg R)$}
\UnaryInfC{$\Gamma\Rightarrow \neg A$}
\AxiomC{$\Gamma\Rightarrow $}
\LeftLabel{$(\perp)$}
\UnaryInfC{$\Gamma\Rightarrow A$}
\noLine
 \TrinaryInfC{} %\BinaryInfC{}
\end{prooftree}
\begin{prooftree}
\AxiomC{$\Gamma, B\Rightarrow \Delta$}
\AxiomC{$\Gamma\Rightarrow A$}
\LeftLabel{$(\rightarrow L)$}
\BinaryInfC{$\Gamma,A\rightarrow B\Rightarrow\Delta$}
\AxiomC{$\Gamma, A\Rightarrow B$}
\LeftLabel{$(\rightarrow R)$}
\UnaryInfC{$\Gamma\Rightarrow A\rightarrow B$}
\noLine
\BinaryInfC{}
\end{prooftree}
\begin{prooftree}
\Axiom$\fCenter\Gamma, a\in b\wedge F(a)\Rightarrow \Delta$
\LeftLabel{$(b\exists L)$}
\UnaryInf$\fCenter\Gamma,(\exists x\in b)F(x)\Rightarrow\Delta$
\Axiom$\fCenter\Gamma\Rightarrow a\in b\wedge F(a)$
\LeftLabel{$(b\exists R)$}
\UnaryInf$\fCenter\Gamma\Rightarrow(\exists x\in b)F(x)$
\noLine
\BinaryInf$\fCenter$
\end{prooftree}
\begin{prooftree}
\Axiom$\fCenter\Gamma, a\in b\rightarrow F(a)\Rightarrow \Delta$
\LeftLabel{$(b\forall L)$}
\UnaryInf$\fCenter\Gamma,(\forall x\in b)F(x)\Rightarrow\Delta$
\Axiom$\fCenter\Gamma\Rightarrow a\in b\rightarrow F(a)$
\LeftLabel{$(b\forall R)$}
\UnaryInf$\fCenter\Gamma\Rightarrow(\forall x\in b)F(x)$
\noLine
\BinaryInf$\fCenter$
\end{prooftree}
\begin{prooftree}
\Axiom$\fCenter\Gamma, F(a)\Rightarrow \Delta$
\LeftLabel{$(\exists L)$}
\UnaryInf$\fCenter\Gamma,\exists xF(x)\Rightarrow\Delta$
\Axiom$\fCenter\Gamma\Rightarrow F(a)$
\LeftLabel{$(\exists R)$}
\UnaryInf$\fCenter\Gamma\Rightarrow\exists xF(x)$
\noLine
\BinaryInf$\fCenter$
\end{prooftree}
\begin{prooftree}
\Axiom$\fCenter\Gamma, F(a)\Rightarrow \Delta$
\LeftLabel{$(\forall L)$}
\UnaryInf$\fCenter\Gamma,\forall xF(x)\Rightarrow\Delta$
\Axiom$\fCenter\Gamma\Rightarrow F(a)$
\LeftLabel{$(\forall R)$}
\UnaryInf$\fCenter\Gamma\Rightarrow\forall xF(x)$
\noLine
\BinaryInf$\fCenter$
\end{prooftree}
\begin{prooftree}
\AxiomC{$\Gamma\Rightarrow A$}
\AxiomC{$\Gamma, A\Rightarrow \Delta$}
\LeftLabel{(Cut)}
\BinaryInfC{$\Gamma\Rightarrow\Delta$}
\end{prooftree}
In each of the inferences $(b \exists L)$, $(\exists L)$ $(b\forall R)$ and $(\forall R)$ the variable $a$ is forbidden from occurring in the conclusion. Such a variable is known as the \textit{eigenvariable} of the inference.\\

\noindent The \textit{minor formulae} of an inference are those rendered prominently in its premises, the other formulae in the premises will be referred to as \textit{side formulae}. The \textit{principal} formula of an inference is the one rendered prominently in the conclusion. Note that in inferences where the principal formula is on the left, the principal formula can also be a side formula of that inference, when this happens we say that there has been a \textit{contraction}.
}\end{definition}

\subsection{An ordinal notation system}
Given below is a very brief description of how to carry out the construction of a primitive recursive ordinal notation system for the Bachmann-Howard ordinal.
\begin{definition}{\em Let $\Omega$ be a `big' ordinal, eg. $\aleph_1$. (In fact we could have chosen $\omega_1^{CK}$, see \cite{mi93}.) We define the sets $B^\Omega(\alpha)$ and ordinals $\psi_\Omega(\al)$ by transfinite recursion on $\al$ as follows
\begin{eqnarray} B^\Omega (\alpha)& =&\left\{\begin{array}{l}
\mbox{ \it closure of }
\{0,\OO\} \mbox{ \it under: }\\
+, (\xi,\eta\mapsto \varphi\xi\eta)\\
(\xi\longmapsto\psio{\xi})_{\xi<\alpha}
\end{array}\right. \\ \nonumber
&\phantom{AA}& \\
  \psio{\alpha}&\simeq
&\min\{\rho<\Omega:\;\rho\notin B(\alpha)\}.\end{eqnarray}
}
\end{definition}
\noindent It can be shown that $\psio\alpha$ is always defined and thus $\psio\alpha<\Omega.$ Moreover, it can also be shown that $B_\Omega(\al)\cap\OO=\psio\al.$\\

\noindent Let $\varepsilon_{\OO+1}$ be the least orinal $\eta>\OO$ such that $\omega^\eta=\eta$. The set $B^\Omega(\varepsilon_{\OO+1})$ gives rise to a primitive recursive ordinal notation system \cite{bu86} \cite{mi99}. The ordinal $\psio{\varepsilon_{\OO+1}}$ is known as the Bachmann-Howard ordinal. There are many slight variants in the specific ordinal functions used to build up a notation system for this ordinal, for example rather than `closing off' under the $\varphi$ function at each stage, we could have chosen $\omega$-exponentiation, all the systems turn out to be equivalent, in that they eventually `catch-up' with one another and the specific ordinal functions used can be defined in terms of one another. Here the functions $\varphi$ and $\psi$ are chosen as primitive since they correspond to the ordinal operations arising from the two main cut elimination theorems of the next section.

\subsection{The infinitary system \textbf{IRS}$_\Omega$}
The purpose of this section is to define an intuitionistic style infinitary system \textbf{IRS}$_\Omega$ within which we will be able to embed \textbf{IKP} and then extract useful information about \textbf{IKP} derivations.\\

\noindent Henceforth all ordinals will be assumed to belong to the primitive recursive ordinal representation system arising from $B^\Omega(\varepsilon_{\Omega+1})$.\\

\noindent The system is based around the constructible hierarchy up to level $\Omega$.
\begin{align*}
L_0&:=\emptyset\\
L_{\alpha+1}&=\{X\subseteq L_\alpha\:|\:\text{$X$ is definable over $L_\alpha$ in the language of \textbf{IKP} with parameters}\}\\
L_\lambda&:=\bigcup_{\xi<\lambda}L_\xi\quad\text{if $\lambda$ is a limit ordinal}
\end{align*}
\begin{definition}{\em
We inductively define the terms of \textbf{IRS}$_\Omega$. To each term $t$ we also assign an ordinal level $|t|$.
\begin{description}
\item[i)] For each $\alpha<\Omega$, $\mathbb{L}_\alpha$ is a term with $|\mathbb{L}_\alpha|:=\alpha$.
\item[ii)] If $F(a,b_1,...,b_n)$ is a formula of \textbf{IKP} with all free variables indicated and $s_1,...,s_n$ are \textbf{IRS}$_\Omega$ terms with levels less than $\alpha$, then
\begin{equation*}
[x\in\mathbb{L}_\alpha\:|\:F(x,s_1,...,s_n)^{\mathbb{L}_\alpha}]
\end{equation*}
is a term of level $\alpha$. Here $F^{\mathbb{L}_\alpha}$ indicates that all unbounded quantifiers in $F$ are restricted to $\mathbb{L}_\alpha$.
\end{description}
}\end{definition}
\noindent The formulae of \textbf{IRS}$_\Omega$ are of the form $F(s_1,...,s_n)$ where $F(a_1,...,a_n)$ is a formula of \textbf{IKP} with all free variables displayed and $s_1,...,s_n$ are \textbf{IRS}$_\Omega$-terms.\\

\noindent Note that the system \textbf{IRS}$_\Omega$ does not contain free variables. We can think of the universe made up of \textbf{IRS}$_\Omega$-terms as a formal, syntactical version of $L_\Omega$, unbounded quantifiers in \textbf{IRS}$_\Omega$-formulas can be thought of as ranging over $L_\Omega$.\\

\noindent For the remainder of this section \textbf{IRS}$_\Omega$-terms and $\IRS$-formulae will simply be referred to as terms and formulae.\\
\noindent A formula is said to be $\Delta_0$ if it contains no unbounded quantifiers.\\
\noindent We inductively (and simultaneously) define the class of $\Sigma$-formulae and the class of $\Pi$-formulae by the following clauses:\\
(i) Every $\Delta_0$-formula is a $\Sigma$ and a $\Pi$-formula.\\
(ii) If $A$ and $B$ are $\Sigma$-formulae ($\Pi$-formulae) then so are $A\vee B$, $A\wedge B$, $(\forall x\in s)A$, and $(\exists x\in s)A$.\\
(iii) If $A$ is a $\Sigma$-formula ($\Pi$-formula) then so is $\exists x A$ ($\forall x A$).\\
(iv) If $A$ is $\Pi$-formula and $B$ is a $\Sigma$-formula, then $A\to B$ and $\neg A$ are $\Sigma$-formulae
while $B\to A$ and $\neg B$ are $\Pi$-formulae.

 The strict $\Sigma$-formulae of Definition \ref{kp5} are $\Sigma$-formulae but the latter form a larger collection.
It's perhaps worth noting that in classical $\KP$ every $\Sigma$-formula is equivalent to a $\Sigma_1$-formula and every $\Pi$-formula is equivalent to a $\Pi_1$-formula, and therefore both are equivalent to strict versions. This, however, does not extend to $\IKP$. These formulae, though, satisfy well-known persistence properties.

\begin{lem}{\em For a formula $C$ and free variable $a$, let $C^{a}$ be the result of replacing each {\em unbounded} quantifier $\forall x$ and $\exists y$ in $C$ by $\forall x\in a$ and $\exists y\in a$, respectively. Suppose $A$ is a $\Sigma$-formula and $B$ is a $\Pi$-formula. Then the following are provable in $\IKP$:
\begin{itemize}
\item[(i)] $a\subseteq b\,\wedge\, A^{a}\to A^{b}$,
\item[(ii)] $a\subseteq b\,\wedge\, B^{b}\to B^{a}$.
\end{itemize}}
\end{lem}
\begin{proof} Straightforward by simultaneously induction on the buildup of $A$ and $B$. \end{proof}

\begin{abbreviation} {\em For $\diamond$ a binary propositional connective, $A$ a formula and $s,t$ terms with $\lev s<\lev t$ we define the following abbreviation:
\begin{align*}
s\dotin t\diamond A:=&A\quad\quad\quad\quad\text{if $t$ is of the form $\mathbb{L}_\alpha$}\\
:=&B(s)\diamond A\quad\text{if $t$ is of the form $[x\in\mathbb{L}_\alpha\:|\:B(x)]$}
\end{align*}
}\end{abbreviation}

\noindent Like in \textbf{IKP}, derivations in \textbf{IRS}$_\Omega$ take place in a two sided sequent calculus. Intuitionistic sequents of the form $\Gamma\Rightarrow\Delta$ are derived, where $\GA$ and $\DE$ are finite sets of formulae and at most one formula occurs in $\Delta$. $\Gamma,\Delta,\Lambda,...$ will be used as meta variables ranging over finite sets of formulae.\\

\noindent \textbf{IRS}$_\Omega$ has no axioms, although note that some of the rules can have an empty set of premises. The inference rules are as follows:

$$\begin{array}{ll}
(\!\in\!L)_\infty & \frac{\DI\Gamma, \;p\dotin t\wedge r=p\Rightarrow\Delta\;\mbox{ for all $\lev p<\lev t$}}
{\DI\Gamma, r\in t^{\phantom{I}}\Rightarrow\Delta\phantom{AAAAAAAAAAAA} }
\\[0.6cm]

(\!\in\!R)& \frac{\DI\Gamma\Rightarrow s\dotin t \;\wedge\; r =
s_{\phantom{A}}}{\DI\Gamma\Rightarrow s \!\in\! t^{\phantom{I}}}\;\; \mbox{ if } \lev{s}<\lev t\\[0.6cm]

(b\forall L) & \frac{\DI\Gamma,\, s\dotin t\rightarrow A(s)\Rightarrow\Delta}{\DI\Gamma,
(\forall x\in t)A(x)^{\phantom{I}}\Rightarrow\Delta }\;\; \mbox{ if } \lev{s}<\lev t\\[0.6cm]

(b\forall R)_{\infty}& \frac{\DI\Gamma\Rightarrow p\dotin t\rightarrow A(p)\;\mbox{ for all $\lev{p}<\lev t$}} {\DI\Gamma\Rightarrow (\forall x \in
t)A(x)^{\phantom{I}} \phantom{AAAAAAAAA} }\\[0.6cm]

(b\exists L)_{\infty}& \frac{\DI\Gamma,\, p\dotin t\wedge A(p)\Rightarrow\Delta\;\mbox{ for all $\lev{p}<\lev t$}} {\DI\Gamma,\, (\exists x \in
t) A(x)^{\phantom{I}}\Rightarrow\Delta \phantom{AAAAAAAAA} }\\[0.6cm]
\end{array}$$

$$\begin{array}{ll}
(b\exists R) & \frac{\DI\Gamma\Rightarrow s\dotin t\wedge A(s)}{\DI\Gamma\Rightarrow
(\exists x \in t)A(x)^{\phantom{I}} }\;\; \mbox{ if } \lev{s}<\lev t\\[0.6cm]

(\forall L)& {\frac{\DI\Gamma,\; A(s)\Rightarrow\Delta}{\DI\Gamma, \forall x
\,A(x)^{\phantom{I}}\Rightarrow\Delta}}\;\;
\\[0.6cm]
\end{array}$$

$$\begin{array}{ll}
(\forall R)_{\infty}& \frac{\DI\Gamma\Rightarrow A(p)\;\mbox{ for all $p$}}
{\DI\Gamma\Rightarrow \forall xA(x)^{\phantom{I}} \phantom{AAA}  }\\[0.6cm]

(\exists L)_{\infty}& \frac{\DI\Gamma,\; A(p)\Rightarrow\Delta\;\mbox{ for all $p$}}
{\DI\Gamma,\;\exists xA(x)^{\phantom{I}}\Rightarrow\Delta \phantom{AAA}  }\\[0.6cm]

(\exists R)& {\frac{\DI\Gamma\Rightarrow A(s)}{\DI\Gamma\Rightarrow \exists x
A(x)^{\phantom{I}}}}\;\;\\[0.6cm]

\SR &{\frac{\DI\Gamma\;\Rightarrow\; A\quad}{\DI\Gamma\Rightarrow\exists z A^{z}{\phantom{I}}^{\phantom{I}}}}\;\;
\mbox{ if }A\mbox{ is a }\Sigma\mbox{-formula,}\\[0.6cm]
\end{array}$$

\noindent as well as the rules $(\wedge L)$, $(\wedge R)$, ($\vee L)$, $(\vee R)$, $(\neg L)$, $(\neg R)$, $(\perp)$, $(\rightarrow L)$, $(\rightarrow R)$ and (Cut) which are defined identically to the rules of the same name in \textbf{IKP}.\\

\noindent In general we are unable to remove cuts from $\IRS$ derivations, one of the main obstacles to full cut elimination comes from $\SR$ since it breaks the symmetry of the other rules. However we can still perform cut elimination on certain derivations, provided they are of a very uniform kind. Luckily, certain embedded proofs from \textbf{IKP} will be of this form. In order to express uniformity in infinite proofs we draw on \cite{bu93}, where Buchholz developed a powerful method of describing such uniformity, called {\em operator control}.

\begin{deff}\label{operatorr}{\em Let  $$P(ON) = \{ X : X \mbox{ is a set of
ordinals}\}.$$ A class function { $$ {\mathcal H} : P(ON)
\to P(ON)$$} will be called an {\bf operator} if ${\mathcal H}$ satisfies the following conditions for all $X\el P(ON)$:
\begin{description}
\item[1.] $X\subseteq Y\RI\CH(X)\subseteq\CH(Y)$ (monotone)
\item[2.] $X\subseteq\CH(X)$ (inclusive)
\item[3.] $\CH(\CH(X))=\CH(X)$ (idempotent)
\item[4.]
  $0 \in {\mathcal H}(X)$ and $\Omega \in {\mathcal H}(X)$.
  \item[5.] If  $\alpha$ has Cantor normal form
$\omega^{\alpha_1}+\cdots +\omega^{\alpha_n}$, then
$$\alpha\!\in\!{\mathcal H}(X)\quad\text{iff}\quad
\alpha_1,...,\alpha_n\!\in\!{\mathcal H}(X).$$
\end{description}
The latter ensures that $\CH(X)$ will be closed under $+$ and
$\sigma\mapsto\om^{\sigma}$, and decomposition of its members
 into additive and multiplicative
components.\\

\noindent From now on $\al\in\CH$ and $\{\al_1,...,\al_n\}\subseteq\CH$ will be considered shorthand for $\al\in\CH(\emptyset)$ and $\{\al_1,...,\al_n\}\subseteq\CH(\emptyset)$ respectively.
}\end{deff}
\begin{definition}{\em
If $A$ is a formula let
\begin{equation*}
k(A):=\{\alpha\in ON\;:\;\text{the symbol $\mathbb{L}_\alpha$ occurs in $A$, subterms included}\}.
\end{equation*}
Likewise we define
\begin{equation*}
k(\{A_1,...,A_n\}):=k(A_1)\cup...\cup k(A_n)\quad\text{and}\quad k(\Gamma\Rightarrow\Delta):=k(\Gamma)\cup k(\Delta).
\end{equation*}
Now for $\CH$ an arbitrary operator, $s$ a term and $\mathfrak{X}$ a formula, set of formulae or a sequent we define
\begin{align*}
\CH[s](X):=&\CH(X\cup\{|s|\})\\
\CH[\mathfrak{X}](X):=&\CH(X\cup k(\mathfrak{X})).
\end{align*}
}\end{definition}
\begin{lem} {\em Let $\CH$ be an operator, $s$ a term and $\DX$ a formula, set of formulae or sequent.
\begin{itemize}
\item[(i)] For any $X,X'\In P(ON)$, if $X'\subseteq X$ then $\CH(X')\subseteq\CH(X)]$.
\item[(ii)]$\CH[s]$ and  $\CH[\DX]$ are operators.
\item[(iii)] If $k(\DX)\subseteq\CH(\emptyset)$ then $\CH[\DX]=\CH.$
\item[(iv)] If $\lev{s}\in \CH$ then $\CH[s]=\CH.$
\end{itemize}
}\end{lem}
\begin{proof}
This result is demonstrated in full in \cite{mi99}.
\end{proof}
\noindent We also need to keep track of the complexity of cuts appearing in derivations.
\begin{definition} {\em The \textit{rank} of a term or formula is determined by
\begin{description}
\item[1.] $rk(\mathbb{L}_\alpha):=\omega\cdot\alpha$
\item[2.] $rk([x\in\mathbb{L}_\alpha\;|\;F(x)]):=\text{max}\{\omega\cdot\alpha+1,rk(F(\mathbb{L}_0))+2\}$
\item[3.] $rk(s\in t):=\text{max}\{rk(s)+6,rk(t)+1\}$
\item[4.] $rk(A\wedge B)=rk(A\vee B)=rk(A\rightarrow B):=\text{max}\{rk(A)+1,rk(B)+1\}$
\item[5.] $rk(\neg A):=rk(A)+1$
\item[6.] $rk((\exists x\in t)A(x))=rk((\forall x\in t)A(x)):=\text{max}\{rk(t), rk(F(\mathbb{L}_0))+2\}$
\item[7.] $rk(\exists x A(x))=rk(\forall x A(x)):=\text{max}\{\Omega,rk(F(\mathbb{L}_0))+1\}$
\end{description}
}\end{definition}
\begin{observation}\label{rkobs}{\em \begin{description}
\item[i)]$rk(s)=\omega\cdot |s|+n$ for some $n<\omega$.
\item[ii)] If $A$ is $\Delta_0$, $rk(A)=\omega\cdot\text{max}(k(A))+m$ for some $m<\omega$.
\item[iii)] If $A$ contains unbounded quantifiers $rk(A)=\Omega+m$ for some $m<\omega$.
\item[iv)] $rk(A)<\Omega$ if and only if $A$ is $\Delta_0$.
\end{description}
}\end{observation}
\noindent There is plenty of leeway in defining the actual rank of a formula, basically we need to make sure the following lemma holds.
\begin{lemma}\label{rank}{\em
In every rule of $\IRS$ other than $\SR$ and (Cut), the rank of the minor formulae is strictly less than the rank of the principal formula.
}\end{lemma}
\begin{proof}
This result is demonstrated for a different set of propositional connectives in \cite{mi99}, the adapted proof to the intuitionistic system is similar.
\end{proof}
\begin{deff}[Operator controlled derivability for $\IRS$]{\em
Let $\CH$ be an operator and $\Gamma\Rightarrow\Delta$ an intuitionistic sequent of $\IRS$, we define the relation $\provx{\CH}{\alpha}{\rho}{\Gamma\Rightarrow\Delta}$ by recursion on $\alpha$.

\noindent We require always that $k(\Gamma\Rightarrow\Delta)\cup\{\alpha\}\subseteq\mathcal{H}$, this condition will not be repeated in the inductive clauses for each of the inference rules of $\IRS$ below. The column on the right gives the ordinal requirements for each of the inference rules.
$$ \begin{array}{lcr}

(\in\! L)_{\infty} & \infone{ {\mathcal H}[r]} {\alpha_{r}}
{\rho} {\Gamma, r\dotin t\wedge r= s\Rightarrow\Delta\mbox{ for all }\lev r<\lev t}
 {\mathcal H} {\alpha} {\rho} {\Gamma,s\in t\Rightarrow\Delta}
&\lev{r}\leq\alpha_r < \alpha \\[0.6cm]

(\in R) & \infone{\mathcal H} {\alpha_0} {\rho}
{\Gamma\;\Rightarrow r\dotin t\wedge r=s} {\mathcal H} {\alpha} {\rho} {\GA\Rightarrow s\in t}
&\begin{array}{r}\alpha_{0} < \alpha\\ \lev r<\lev t \\ \lev{r}<\al
\end{array}
\end{array}$$
$$\begin{array}{lcr}
(b\forall L) & \infone{\mathcal H} {\alpha_0} {\rho}
{\GA,s\dotin t\rightarrow A(s)\Rightarrow\Delta} {\mathcal H} {\alpha} {\rho} {\GA,(\forall x\in t)A(x)\Rightarrow\Delta}
&\begin{array}{r}\alpha_{0} < \alpha\\ \lev s<\lev t \\ \lev{s}<\al
\end{array} \\[0.6cm]

(b\forall R)_{\infty} & \infone{ {\mathcal H}[s]} {\alpha_{s}}
{\rho} {\GA\Rightarrow s\dotin t\rightarrow F(s)\mbox{ for all }\lev s<\lev t}
 {\mathcal H} {\alpha} {\rho} {\Gamma\Rightarrow(\forall x\In t)F(x)}
&\lev{s}\leq\alpha_{s} < \alpha
\\[0.6cm]

(b\exists L)_{\infty} & \infone{ {\mathcal H}[s]} {\alpha_{s}}
{\rho} {\GA,s\dotin t\wedge F(s)\Rightarrow\Delta\mbox{ for all }\lev s<\lev t}
 {\mathcal H} {\alpha} {\rho} {\Gamma,(\exists x\In t)F(x)\Rightarrow\Delta}
&\lev{s}\leq\alpha_{s} < \alpha
\\[0.6cm]

(b\exists R) & \infone{\mathcal H} {\alpha_0} {\rho}
{\GA\Rightarrow s\dotin t\wedge A(s)} {\mathcal H} {\alpha} {\rho} {\GA\Rightarrow(\exists x\in t)A(x)}
&\begin{array}{r}\alpha_{0} < \alpha\\ \lev s<\lev t \\ \lev{s}<\al
\end{array} \\[0.6cm]

\end{array}$$
$$\begin{array}{lcr}

(\forall L) & \infone{\mathcal H} {\alpha_0} {\rho}
{\GA, F(s)\Rightarrow\Delta} {\mathcal H} {\alpha} {\rho} {\GA,\forall x F(x)\Rightarrow\Delta}
&\begin{array}{r}\alpha_{0}+1 < \alpha\\  \lev{s}<\al
\end{array} \\[0.6cm]

(\forall R)_{\infty} & \infone{ {\mathcal H}[s]} {\alpha_{s}}
{\rho} {\GA\Rightarrow  F(s)\mbox{ for all } s}
 {\mathcal H} {\alpha} {\rho} {\Gamma\Rightarrow\forall x F(x)}
&\lev{s}<\alpha_{s}+1 < \alpha \\[0.6cm]

(\exists L)_{\infty} & \infone{ {\mathcal H}[s]} {\alpha_{s}}
{\rho} {\GA,  F(s)\Rightarrow\Delta\mbox{ for all } s}
 {\mathcal H} {\alpha} {\rho} {\Gamma,\exists x F(x)\Rightarrow\Delta}
&\lev{s}<\alpha_{s}+1 < \alpha \\[0.6cm]

(\exists R) & \infone{\mathcal H} {\alpha_0} {\rho}
{\GA\Rightarrow F(s)} {\mathcal H} {\alpha} {\rho} {\GA,\Rightarrow\exists x F(x)}
&\begin{array}{r}\alpha_{0}+1 < \alpha\\  \lev{s}<\al
\end{array}
\end{array}$$

$$\begin{array}{lcr}

\Cut & \ifthree{\provx{\mathcal H} {\alpha_0}{\rho}  {\Gamma,
B\Rightarrow\Delta}}{\provx{\mathcal H} {\alpha_1} {\rho}{\Gamma\Rightarrow B}}{\provx
{\mathcal H} {\alpha}{\rho} {\Gamma\Rightarrow\Delta}}
&\begin{array}{r}\alpha_{0} ,\al_1< \alpha\\
rk(B)<\rho\end{array} \\[0.6cm]

\SR &
\infone{\CH}{\al_0}{\rho}{\Gamma\Rightarrow A}{\CH}{\al}{\rho}{\Gamma\Rightarrow
\exists z\,A^z}
&\begin{array}{r} \al_0+1,\Omega<\al\\
 A\text{ is a $\Sigma$-formula}\end{array}
\end{array}$$
Lastly if $\Gamma\Rightarrow\Delta$ is the result of a propositional inference of the form $(\wedge L)$, $(\wedge R)$, ($\vee L)$, $(\vee R)$, $(\neg L)$, $(\neg R)$, $(\perp)$, $(\rightarrow L)$ or $(\rightarrow R)$, with premise(s) $\Gamma_i\Rightarrow\Delta_i$ then from $\provx{\CH}{\alpha_0}{\rho}{\Gamma_i\Rightarrow\Delta_i}$ (for each $i$) we may conclude $\provx{\CH}{\alpha}{\rho}{\Gamma\Rightarrow\Delta}$, provided $\alpha_0<\alpha$.
}\end{deff}
\begin{lemma}[Weakening and Persistence for $\IRS$] \label{weakpers} $\phantom{JJ}$
{\em \begin{description}\item[i)] If  $\Gamma_0\subseteq\GA$, $k(\Gamma)\subseteq\CH$, $\alpha_0\leq\alpha\in\CH$, $\rho_0\leq\rho$ and $\provx{\CH}{\alpha_0}{\rho_0}{\GA_0\Rightarrow\Delta}$ then
\begin{equation*}
\provx{\CH}{\al}{\rho}{\GA\Rightarrow\Delta}.
\end{equation*}
\item[ii)] If $\beta\geq\gamma\in\CH$ and $\provx{\CH}{\al}{\rho}{\GA,(\exists x\in\mathbb{L}_\beta)A(x)\Rightarrow\Delta}$ then $\provx{\CH}{\al}{\rho}{\GA,(\exists x\in\mathbb{L}_\gamma)A(x)\Rightarrow\Delta}.$
\item[iii)] If $\beta\geq\gamma\in\CH$ and $\provx{\CH}{\al}{\rho}{\GA\Rightarrow(\forall x\in\mathbb{L}_\beta)A(x)}$ then $\provx{\CH}{\al}{\rho}{\GA\Rightarrow(\forall x\in\mathbb{L}_\gamma)A(x)}$
\item[iv)] If $\gamma\in\CH$ and  $\provx{\CH}{\al}{\rho}{\GA,\exists xA(x)\Rightarrow\Delta}$ then $\provx{\CH}{\al}{\rho}{\GA,(\exists x\in\mathbb{L}_\gamma)A(x)\Rightarrow\Delta}.$
\item[v)] If $\gamma\in\CH$ and $\provx{\CH}{\al}{\rho}{\GA\RI\forall xA(x)}$ then $\provx{\CH}{\al}{\rho}{\GA\RI(\forall x\in\mathbb{L}_\gamma)A(x)}.$
\end{description}
}\end{lemma}
\begin{proof}
We show i), ii) and v).\\

\noindent i) is proved by an easy induction on $\al$.\\

\noindent ii) Is also proved using induction on $\al$, suppose $\beta\geq\gamma\in\CH(\emptyset)$ and $\provx{\CH}{\al}{\rho}{\GA,(\exists x\in\mathbb{L}_\beta)A(x)\Rightarrow\Delta}$. If $(\exists x\in\mathbb{L}_\beta)A(x)$ was not the principal formula of the last inference or the last inference was not $(b\exists L)_\infty$ then we may apply the induction hypotheses to it's premises followed by the same inference again. So suppose $(\exists x\in\mathbb{L}_\beta)A(x)$ was the principal formula of the last inference which was $(b\exists L)_\infty$, so we have
\begin{equation*}
\provx{\CH[s]}{\al_s}{\rho}{\GA,(\exists x\in\mathbb{L}_\beta)A(x),A(s)\Rightarrow\Delta}\quad\text{for all $|s|<\beta$, with $\alpha_s<\alpha$.}
\end{equation*}
From the induction hypothesis we obtain
\begin{equation*}
\provx{\CH[s]}{\al_s}{\rho}{\GA,(\exists x\in\mathbb{L}_\gamma)A(x),A(s)\Rightarrow\Delta}\quad\text{for all $|s|<\beta$, with $\alpha_s<\alpha$}
\end{equation*}
but since $\beta\geq\gamma$ this also holds for all $|s|<\gamma$. So by another application of $(b\exists L)_\infty$ we get
\begin{equation*}
\provx{\CH}{\al}{\rho}{\GA,(\exists x\in\mathbb{L}_\gamma)A(x)\Rightarrow\Delta}
\end{equation*}
as required.\\

\noindent For v) suppose $\provx{\CH}{\al}{\rho}{\Gamma\Rightarrow\forall xA(x)}$. The interesting case is where $\forall xA(x)$ was the principal formula of the last inference, which was $(\forall R)_\infty$, in this case we have
\begin{equation*}
\provx{\CH[s]}{\alpha_s}{\rho}{\Gamma\Rightarrow A(s)}\quad\text{for all $s$, with $\lev s<\alpha_s+1<\alpha$}.
\end{equation*}
So taking just the cases where $\lev s<\gamma$ and noting that in these cases $A(s)\equiv s\dotin\mathbb{L}_\gamma\rightarrow A(s)$, we may apply $(b\forall R)$ to obtain
\begin{equation*}
\provx{\CH}{\al}{\rho}{\GA\Rightarrow(\forall x\in\mathbb{L}_\gamma)A(x)}
\end{equation*}
as required.\\

\noindent The proofs of iii) and iv) may be carried out in a similar manner to those above.
\end{proof}
\subsection{Cut elimination for $\IRS$}
\begin{lemma}[Inversions of $\IRS$]\label{inversion}
\begin{description}
\item[i)] If $\provx{\mathcal{H}}{\alpha}{\rho}{\Gamma,A\wedge B\Rightarrow\Delta}$ then $\provx{\mathcal{H}}{\alpha}{\rho}{\Gamma,A, B\Rightarrow\Delta}$.
\item[ii)] If $\provx{\mathcal{H}}{\alpha}{\rho}{\Gamma\Rightarrow A\wedge B}$ then $\provx{\mathcal{H}}{\alpha}{\rho}{\Gamma\Rightarrow A}$ and $\provx{\mathcal{H}}{\alpha}{\rho}{\Gamma\Rightarrow B}$.
\item[iii)] If $\provx{\mathcal{H}}{\alpha}{\rho}{\Gamma,A\vee B\Rightarrow\Delta}$ then $\provx{\mathcal{H}}{\alpha}{\rho}{\Gamma,A\Rightarrow\Delta}$ and $\provx{\mathcal{H}}{\alpha}{\rho}{\Gamma,B\Rightarrow\Delta}$.
\item[iv)]  If $\provx{\mathcal{H}}{\alpha}{\rho}{\Gamma,A\rightarrow B\Rightarrow\Delta}$ then $\provx{\mathcal{H}}{\alpha}{\rho}{\Gamma,B\Rightarrow\Delta}$.
\item[v)] If $\provx{\mathcal{H}}{\alpha}{\rho}{\Gamma\Rightarrow A\rightarrow B}$ then $\provx{\mathcal{H}}{\alpha}{\rho}{\Gamma,A\Rightarrow B}$.
\item[vi)] If $\provx{\mathcal{H}}{\alpha}{\rho}{\Gamma\Rightarrow\neg A}$ then $\provx{\mathcal{H}}{\alpha}{\rho}{\Gamma, A\Rightarrow}$.
\item[vii)] If $\provx{\mathcal{H}}{\alpha}{\rho}{\Gamma,r\in t\Rightarrow\Delta}$ then $\provx{\mathcal{H}[s]}{\alpha}{\rho}{\Gamma,s\dotin t\wedge r=s\Rightarrow\Delta}$ for all $|s|<|t|$.
\item[viii)] If $\provx{\mathcal{H}}{\alpha}{\rho}{\Gamma,(\exists x\in t)A(x)\Rightarrow\Delta}$ then $\provx{\mathcal{H}[s]}{\alpha}{\rho}{\Gamma,s\dotin t\wedge A(s)\Rightarrow\Delta}$ for all $|s|<|t|$.
\item[ix)] If $\provx{\mathcal{H}}{\alpha}{\rho}{\Gamma\Rightarrow(\forall x\in t)A(x)}$ then $\provx{\mathcal{H}[s]}{\alpha}{\rho}{\Gamma\Rightarrow s\dotin t\rightarrow A(s)}$ for all $|s|<|t|$.
\item[x)] If $\provx{\mathcal{H}}{\alpha}{\rho}{\Gamma,\exists x A(x)\Rightarrow\Delta}$ then $\provx{\mathcal{H}[s]}{\alpha}{\rho}{\Gamma, A(s)\Rightarrow\Delta}$ for all $s$.
\item[xi)] If $\provx{\mathcal{H}}{\alpha}{\rho}{\Gamma,\Rightarrow\forall x A(x)}$ then $\provx{\mathcal{H}[s]}{\alpha}{\rho}{\Gamma\Rightarrow A(s)}$ for all $s$.
\end{description}
\end{lemma}
\begin{proof}
All proofs are by induction on $\alpha$, we treat three of the most interesting cases, iv), vi) and x).\\

\noindent iv) Suppose $\provx{\CH}{\al}{\rho}{\GA,A\rightarrow B\RI\DE}$, If the last inference was not $(\rightarrow L)$ or the principal formula of that inference was not $A\rightarrow B$ we may apply the induction hypothesis to the premises of that inference, followed by the same inference again. Now suppose $A\rightarrow B$ was the principal formula of the last inference, which was $(\rightarrow L)$. Thus, with the possible use of weakening, we have
\begin{align*}
\tag{1}&\provx{\CH}{\al_0}{\rho}{\GA,B,A\rightarrow B\RI\DE}&\text{for some $\al_0<\al$.}\\
\tag{2}&\provx{\CH}{\al_1}{\rho}{\GA,A\rightarrow B\RI A}&\text{for some $\al_1<\al$.}
\end{align*}
Applying the induction hypothesis to (1) yields $\provx{\CH}{\al_0}{\rho}{\GA,B\RI\DE}$ from which we may obtain the desired result by weakening.\\

\noindent vi) Now suppose $\provx{\CH}{\al}{\rho}{\GA\RI\neg A}$. If $\neg A$ was the principal formula of the last inference which was $(\neg R)$ then we have $\provx{\CH}{\al_0}{\rho}{\GA,A\RI}$ for some $\al_0<\al$, from which we may obtain the desired result by weakening.
If the last inference was $(\perp)$ then  $\provx{\CH}{\al_0}{\rho}{\GA\RI}$ for some $\al_0<\al$, from which we also obtain the desired result by weakening.
If the last inference was different to $(\neg R)$ and $(\perp)$ we may apply the induction hypothesis to the premises of that inference followed by the same inference again.\\

\noindent x) Finally suppose $\provx{\CH}{\al}{\rho}{\GA,\exists x A(x)\RI\DE}$. If $\exists xA(x)$ was the principal formula of the last inference which was $(\exists L)_\infty$ then we have
\begin{equation*}
\CH[s]\;\prov{\al_s}{\rho}{\GA,\exists xA(x),A(s)\RI\DE}\quad\text{with $\al_s<\al$ for each $s$.}
\end{equation*}
Applying the induction hypothesis yields
\begin{equation*}
\CH[s]\;\prov{\al_s}{\rho}{\GA,A(s)\RI\DE}
\end{equation*}
from which we get the desired result by weakening. If $\exists xA(x)$ was not the principal formula of the last inference or the last inference was not $(\exists L)_\infty$ then we may apply the induction hypothesis to the premises of that inference followed by the same inference again.
\end{proof}
\begin{lemma}[Reduction for $\IRS$]\label{reduction}{\em
Let $\rho:=rk(C)\neq\Omega$
\begin{equation*}
\text{If}\quad\provx{\mathcal{H}}{\alpha}{\rho}{\Gamma,C\Rightarrow\Delta}\quad\text{and}\quad\provx{\mathcal{H}}{\beta}{\rho}{\Xi\Rightarrow C}\quad\text{then}\quad\provx{\mathcal{H}}{\alpha\#\alpha\#\beta\#\beta}{\rho}{\Gamma,\Xi\Rightarrow\Delta}
\end{equation*}
}\end{lemma}
\begin{proof}
The proof is by induction on $\alpha\#\alpha\#\beta\#\beta$. Assume that
\begin{align*}
&\rho:=rk(C)\neq\Omega\tag{1}\\
&\provx{\mathcal{H}}{\alpha}{\rho}{\Gamma,C\Rightarrow\Delta}\tag{2}\\
&\provx{\mathcal{H}}{\beta}{\rho}{\Xi\Rightarrow C}\tag{3}
\end{align*}
If $C$ was not the principal formula of the last inference in both derivations then we may simply use the induction hypothesis on the premises and then the final inference again.\\

\noindent So suppose $C$ was the principal formula of the last inference in both (2) and (3). Note also that (1) gives us immediately that the last inference in (3) was \emph{not} $\SR$.\\

\noindent We treat three of the most interesting cases.\\

\noindent Case 1. Suppose $C\equiv r\in t$, thus we have
\begin{equation*}\tag{4}
\provx{\mathcal{H}[p]}{\alpha_p}{\rho}{\Gamma,C,p\dotin t\wedge r=p\Rightarrow\Delta}\quad\text{for all $\lev p<\lev t$ with $\alpha_p<\alpha$}
\end{equation*}
and
\begin{equation*}\tag{5}
\provx{\mathcal{H}}{\beta_0}{\rho}{\Xi\Rightarrow s\dotin t\wedge r=s}\quad\text{for some $\lev s<\lev t$ with $\beta_0<\beta$.}
\end{equation*}
Now from (5) we know that $\lev s\in\mathcal{H}$ and thus from (4) we have
\begin{equation*}\tag{6}
\provx{\mathcal{H}}{\alpha_s}{\rho}{\Gamma,C,s\dotin t\wedge r=s\Rightarrow\Delta}.
\end{equation*}
Applying the induction hypothesis to (6) and (3) yields
\begin{equation*}\tag{7}
\provx{\mathcal{H}}{\alpha_s\#\alpha_s\#\beta\#\beta}{\rho}{\Xi,\Gamma, s\dotin t\wedge r=s\Rightarrow\Delta}.
\end{equation*}
Finally a (Cut) applied to (5) and (7) yields
\begin{equation*}
\provx{\mathcal{H}}{\alpha\#\alpha\#\beta\#\beta}{\rho}{\Xi,\Gamma\Rightarrow\Delta}
\end{equation*}
as required.\\

\noindent Case 2. Now suppose $C\equiv(\forall x\in t)F(x)$ so we have
\begin{equation*}\tag{8}
\provx{\mathcal{H}}{\alpha_0}{\rho}{\Gamma,C,s\dotin t\rightarrow F(s)\Rightarrow\Delta}\quad\text{for some $\lev s<\lev t$ with $\alpha_0,\lev s<\alpha$}
\end{equation*}
and
\begin{equation*}\tag{9}
\provx{\mathcal{H}[p]}{\beta_p}{\rho}{\Xi\Rightarrow p\dotin t\rightarrow F(p)}\quad\text{for all $\lev p<\lev t$ with $\beta_p<\beta$.}
\end{equation*}
Now (8) gives $s\in\mathcal{H}$ and thus from (9) we have
\begin{equation*}\tag{10}
\provx{\mathcal{H}}{\beta_s}{\rho}{\Xi\Rightarrow s\dotin t\rightarrow F(s)}.
\end{equation*}
Applying the induction hypothesis to (3) and (8) gives
\begin{equation*}\tag{11}
\provx{\mathcal{H}}{\alpha_0\#\alpha_0\#\beta\#\beta}{\rho}{\Gamma,\Xi,s\dotin t\rightarrow F(s)\Rightarrow\Delta}.
\end{equation*}
Finally (Cut) applied to (10) and (11) yields the desired result.\\

\noindent Case 3. Now suppose $C\equiv A\rightarrow B$ so we have
\begin{align*}
&\provx{\mathcal{H}}{\alpha_0}{\rho}{\Gamma,C\Rightarrow A}\quad\text{with $\alpha_0<\alpha$}\tag{12}\\
&\provx{\mathcal{H}}{\alpha_1}{\rho}{\Gamma,C,B\Rightarrow \Delta}\quad\text{with $\alpha_1<\alpha$}\tag{13}\\
&\provx{\mathcal{H}}{\beta_0}{\rho}{\Xi,A\Rightarrow B}\quad\text{with $\beta_0<\beta$}\tag{14}
\end{align*}
The induction hypothesis applied to (12) and (3) gives
\begin{equation*}\tag{15}
\provx{\mathcal{H}}{\alpha_0\#\alpha_0\#\beta\#\beta}{\rho}{\Gamma,\Xi\Rightarrow A}.
\end{equation*}
Now an application of (Cut) to (15) and (14) gives
\begin{equation*}\tag{16}
\provx{\mathcal{H}}{\alpha_0\#\alpha\#\beta\#\beta}{\rho}{\Gamma,\Xi\Rightarrow B}.
\end{equation*}
Inversion (Lemma \ref{inversion} iv)) applied to (13) gives
\begin{equation*}\tag{17}
\provx{\mathcal{H}}{\alpha_1}{\rho}{\Gamma,B\Rightarrow\Delta}.
\end{equation*}
Finally a single application of (Cut) to (16) and (17) yields the desired result.
\end{proof}
\begin{theorem}[Predicative Cut Elimination for $\IRS$]\label{predce}{\em
Suppose $\provx{\mathcal{H}}{\alpha}{\rho+\omega^\beta}{\Gamma\Rightarrow\Delta}$, where $\Omega\notin[\rho,\rho+\omega^\beta)$ and $\beta\in\mathcal{H}$, then
\begin{equation*}
\provx{\mathcal{H}}{\varphi\beta\alpha}{\rho}{\Gamma\Rightarrow\Delta}.
\end{equation*}
Provided $\mathcal{H}$ is an operator closed under $\varphi$.
}\end{theorem}
\begin{proof}
The proof is by main induction on $\beta$ and subsidiary induction on $\alpha$.\\

\noindent If the last inference was anything other than (Cut) or was a cut of rank $<\rho$ then we may apply the subsidiary induction hypothesis to the premises and then re-apply the final inference. So suppose the last inference was (Cut) with cut-formula $C$ and $rk(C)\in[\rho,\rho+\omega^\beta)$. So we have
\begin{align*}
&\tag{1}\provx{\mathcal{H}}{\alpha_0}{\rho+\omega^\beta}{\Gamma,C\Rightarrow\Delta}\quad\text{with $\al_0<\al$,}\\
&\tag{2}\provx{\mathcal{H}}{\alpha_1}{\rho+\omega^\beta}{\Gamma\Rightarrow C}\quad\text{with $\al_1<\al$.}
\end{align*}
First  applying the subsidiary induction hypothesis to (1) and (2) gives
\begin{align*}
&\provx{\mathcal{H}}{\varphi\beta\alpha_0}{\rho}{\Gamma,C\Rightarrow\Delta}\tag{3}\\
&\provx{\mathcal{H}}{\varphi\beta\alpha_1}{\rho}{\Gamma,\Rightarrow C}.\tag{4}
\end{align*}
Now if $rk(C)=\rho$ then one application of the Reduction Lemma \ref{reduction} gives the desired result (once it is noted that $\varphi\beta\alpha_0\#\varphi\beta\alpha_0\#\varphi\beta\alpha_1\#\varphi\beta\alpha_1<\varphi\beta\alpha$ since $\varphi\beta\alpha$ is additive principal.)\\

\noindent Now let us suppose that $\beta>0$ and $rk(C)\in(\rho,\rho+\omega^\beta)$. Since $rk(C)<\rho+\omega^\beta$ we can find some $\beta_0<\beta$ and some $n<\omega$ such that
\begin{equation*}
rk(C)<\rho+n\cdot\omega^{\beta_0}.
\end{equation*}
Thus applying (Cut) to (3) and (4) gives
\begin{equation*}
\provx{\mathcal{H}}{\varphi\beta\alpha}{\rho+n\cdot\omega^{\beta_0}}{\Gamma\Rightarrow\Delta\,.}
\end{equation*}
Now by the main induction hypothesis we obtain
\begin{equation*}
\provx{\mathcal{H}}{\varphi\beta_0 (\varphi\beta\alpha)}{\rho+(n-1)\cdot\omega^{\beta_0}}{\Gamma\Rightarrow\Delta}
\end{equation*}
But by definition $\varphi\beta\alpha$ is a fixed point of the function $\varphi\beta_0 (\cdot)$ ie. $\varphi\beta_0 (\varphi\beta\alpha)=\varphi\beta\alpha$, so we have
\begin{equation*}
\provx{\mathcal{H}}{\varphi\beta\alpha}{\rho+(n-1)\cdot\omega^{\beta_0}}{\Gamma\Rightarrow\Delta\,.}
\end{equation*}
From here a further $(n - 1)$ applications of the main induction hypothesis yields the desired result.
\end{proof}
\begin{lemma}[Boundedness for $\IRS$]\label{boundedness}{\em
If $A$ is a $\Sigma$-formula, $B$ is a $\Pi$-formula, $\alpha\leq\beta<\Omega$ and $\beta\in\mathcal{H}$ then
\begin{description}
\item[i)] If $\provx{\mathcal{H}}{\alpha}{\rho}{\Gamma\Rightarrow A}$ then $\provx{\mathcal{H}}{\alpha}{\rho}{\Gamma\Rightarrow A^{\mathbb{L}_\beta}}$.
\item[ii)] If $\provx{\mathcal{H}}{\alpha}{\rho}{\Gamma,B\Rightarrow\Delta}$ then $\provx{\mathcal{H}}{\alpha}{\rho}{\Gamma,B^{\mathbb{L}_\beta}\Rightarrow\Delta\,.}$
\end{description}
}\end{lemma}
\begin{proof}
Suppose that $\provx{\mathcal{H}}{\alpha}{\rho}{\Gamma\Rightarrow A}$. We prove i) and ii) simultaneously by induction on $\alpha$.\\

\noindent First we look at i). If $A$ was not the principal formula of the last inference then we can simply use the induction hypothesis. If $A$ was the principal formula of the last inference and is of the form $\neg C$, $C\wedge D$, $C\vee D$, $C\rightarrow D$, $(\exists x\in t)C(x)$ or $(\forall x\in t)C(x)$, then again the result follows immediately from the induction hypothesis.\\

\noindent Note that the last inference cannot have been $(\forall R)_\infty$ or $\SR$ since $A$ is a $\Sigma$ formula and $\alpha<\Omega$. \\

\noindent So suppose $A\equiv \exists xC(x)$ and
\begin{equation*}
\provx{\mathcal{H}}{\alpha_0}{\rho}{\Gamma\Rightarrow C(s)}
\end{equation*}
for some $\alpha_0, \lev s<\alpha$. By induction hypothesis we obtain
\begin{equation*}
\provx{\mathcal{H}}{\alpha_0}{\rho}{\Gamma\Rightarrow C(s)^{\mathbb{L}_\beta}}.
\end{equation*}
Which may be written as
\begin{equation*}
\provx{\mathcal{H}}{\alpha_0}{\rho}{\Gamma\Rightarrow s\dotin\mathbb{L}_\beta\wedge C(s)^{\mathbb{L}_\beta}}.
\end{equation*}
Now an application of $(b\exists R)$ yields the desired result.\\

\noindent As part ii) is proved in a similar manner, we shall confine ourselves to the  case when the last inference was $(\to L)$ with principal formula $B$.
 So suppose $B\equiv C\to D$ and
\begin{equation*}
\provx{\mathcal{H}}{\alpha_0}{\rho}{\Gamma\Rightarrow C}\;\;\mbox{ and }\;\; \provx{\mathcal{H}}{\alpha_0}{\rho}{\Gamma,D\Rightarrow \Delta}
\end{equation*}
for some $\alpha_0, \lev s<\alpha$. By induction hypothesis we obtain
\begin{equation*}
\provx{\mathcal{H}}{\alpha_0}{\rho}{\Gamma\Rightarrow C^{\mathbb{L}_\beta}}\mbox{ and }
\provx{\mathcal{H}}{\alpha_0}{\rho}{\Gamma,D^{\mathbb{L}_\beta}\Rightarrow\Delta\,.}\end{equation*}
Now an application of $(\to L)$ yields the desired result.
\end{proof}

\begin{definition}\label{opdef}{\em
For each $\eta$ we define
\begin{align*}
\CH_\eta:&\:\mathcal{P}(B^\OO(\varepsilon_{\OO+1}))\longrightarrow\mathcal{P}(B^\OO(\varepsilon_{\OO+1}))\\[0.1cm]
\CH_\eta (X):&=\bigcap\{B^\OO(\alpha)\;:\;X\subseteq B^\OO(\alpha)\:\text{and}\:\eta<\alpha\}.
\end{align*}
}\end{definition}
\begin{lemma}\label{op1}{\em \begin{description}\item[i)] $\CH_\eta$ is an operator for each $\eta$.
\item[ii)] $\eta<\eta^\prime\:\Longrightarrow\CH_\eta(X)\subseteq\CH_{\eta^\prime}(X)$.
\item[iii)] If $\xi\in\CH_\eta (X)$ and $\xi<\eta+1$ then $\psio\xi\in\CH_\eta (X)$.
\end{description}
}\end{lemma}
\begin{proof}
This is proved in \cite{bu93}.
\end{proof}
\begin{lemma}\label{op2}{\em
Suppose $\eta\in\CH_\eta$ and let $\hat\beta:=\eta+\omega^{\Omega+\beta}$.
\begin{description}
\item[i)] If $\alpha\in\CH_\eta$ then $\hat\alpha,\psio{\hat\al}\in\CH_{\hat\al}$.
\item[ii)] If $\alpha_0\in\CH_\eta$ and $\al_0<\al$ then $\psio{\hat{\al_0}}<\psio{\hat\al}$.
\end{description}
}\end{lemma}
\begin{proof}
i) From $\al,\eta\in\CH_\eta=B^\OO(\eta+1)$ we get $\hat\al\in B^\OO(\eta+1)$ and hence $\hat\al\in B^\OO(\hat\al)$ by \ref{op1}ii). Thus $\psio{\hat{\alpha}}\in B^\OO(\hat\al +1)=\CH_{\hat\al}(\emptyset)$.\\

\noindent ii) Suppose that $\al>\al_0\in\CH_\eta$. By the argument above we get $\psio{\hat{\al_o}}\in B^\OO(\hat{\al_0}+1)\subseteq B^\OO(\hat\al)$, thus $\psio{\hat{\al_0}}<\psio{\hat{\al}}$.
\end{proof}
\begin{theorem}[Collapsing for $\IRS$]\label{collapsing}{\em Suppose that $\eta\in\CH_\eta$, $\Delta$ is a set of at most one $\Sigma$-formula and $\Gamma$ a finite set of $\Pi$-formulae. Then
 %with $\text{max}\{rk(A)\:|\:A\in\Gamma\}\leq\Omega$ then:
\begin{equation*}
\provx{\mathcal{H}_\eta}{\alpha}{\Omega+1}{\Gamma\Rightarrow\Delta}\quad\text{implies}\quad\provx{\mathcal{H}_{\hat{\alpha}}}{\psi_\Omega (\hat{\alpha})}{\psi_\Omega (\hat{\alpha})}{\Gamma\Rightarrow\Delta\,.}
\end{equation*}
}\end{theorem}
\begin{proof}
We proceed by induction on $\alpha$. If the last inference was propositional then the assertion follows easily from the induction hypothesis.\\

\noindent Case 1. Suppose the last inference was $(b\forall R)_\infty$, then $\Delta=\{(\forall x\in t)F(x)\}$ and
\begin{equation*}
\provx{\mathcal{H}_\eta[p]}{\alpha_p}{\Omega+1}{\Gamma\Rightarrow p\dotin t\rightarrow F(p)}\quad\text{for all $\lev p<\lev t$ with $\alpha_p<\alpha$.}
\end{equation*}
Since $k(t)\subseteq\mathcal{H}_\eta$, we know that $\lev t\in B(\eta+1)$ and thus $\lev t<\psi_\Omega (\eta+1)$. Thus $k(p)\subseteq\mathcal{H}_\eta$ for all $\lev p<\lev t$, so $\mathcal{H}_\eta[p]=\mathcal{H}_\eta$ for all such $p$. Since $p\dotin t\rightarrow F(p)$ is also a
$\Sigma$-formula we can invoke the induction hypothesis to give
\begin{equation*}
\provx{\mathcal{H}_{\hat{\alpha_p}}}{\psi_\Omega (\hat{\alpha_p})}{\psi_\Omega (\hat{\alpha_p})}{\Gamma,p\dotin t\Rightarrow F(p)}.
\end{equation*}
Since $\psi_\Omega (\hat{\alpha_p})+1<\psi_\Omega(\hat{\alpha})$ for all $p$, we may apply $(\rightarrow R)$ and then $(b\forall R)_\infty$ to obtain the desired result.\\

\noindent Case 2. Suppose the last inference was $(b\forall L)$ so $(\forall x\in t)F(x)\in\Gamma$ and
\begin{equation*}
\provx{\mathcal{H}_\eta}{\alpha_0}{\Omega+1}{\Gamma,s\dotin t\rightarrow F(s)\Rightarrow\Delta}\quad\text{for some $\lev s<\lev t$ with $\alpha_0<\alpha$.}
\end{equation*}
Noting that  $s\dotin t\rightarrow F(s)$ is itself a $\Pi$-formula, we may apply the induction hypothesis to give
\begin{equation*}
\provx{\mathcal{H}_{\hat{\alpha_0}}}{\psi_\Omega\hat{\alpha_0}}{\psi_\Omega\hat{\alpha_0}}{\Gamma,s\dotin t\rightarrow F(s)\Rightarrow\Delta}
\end{equation*}
from which we obtain the desired result using one application of $(b\forall L)$.\\

\noindent Case 3. $(b\exists L)_\infty$ and $(b\exists R)$ are similar to cases 1 and 2.
 \\

\noindent Case 4. Suppose the last inference was $(\exists R)$, so $\Delta=\{\exists xF(x)\}$ and
\begin{equation*}
\provx{\CH_\eta}{\alpha_0}{\Omega+1}{\Gamma\Rightarrow F(s)}\quad\text{for some $\lev s<\alpha$ and $\alpha_0<\alpha$}.
\end{equation*}
Since $F(s)$ is $\Sigma$ we may immediately apply the induction hypothesis to obtain
\begin{equation*}
\provx{\CH_{\hat{\alpha_0}}}{\psi_\Omega\hat{\alpha_0}}{\psi_\Omega\hat{\alpha_0}}{\Gamma\Rightarrow F(s)}.
\end{equation*}
Now since $\lev s\in\CH_\eta=B(\eta+1)$ we know that $\lev s<\psi_\Omega(\eta+1)<\psi_\Omega\hat\alpha$, so we may apply $(\exists R)$ to obtain the desired result.\\

\noindent Case 5. If the last inference was $(\forall L)$ we may argue in a similar fashion to case 4.\\

\noindent It cannot be the case that the last inference was $(\exists L)$ or $(\forall R)$ since $\Gamma$ contains only $\Pi$ formulae and $\Delta$ only $\Sigma$ formulae.\\

\noindent Case 6. Suppose the last inference was $(\Sigma$-Ref$_\Omega)$, so $\Delta=\{\exists zF^z\}$ for some $\Sigma$ formula $F$ and
\begin{equation*}
\provx{\mathcal{H}_\eta}{\alpha_0}{\Omega+1}{\Gamma\Rightarrow F}.
\end{equation*}
The induction hypothesis yields
\begin{equation*}
\provx{\mathcal{H}_{\hat{\alpha_0}}}{\psi_\Omega\hat{\alpha_0}}{\psi_\Omega\hat{\alpha_0}}{\Gamma\Rightarrow F\,.}
\end{equation*}
Now applying Boundedness \ref{boundedness} yields
\begin{equation*}
\provx{\mathcal{H}_{\hat{\alpha_0}}}{\psi_\Omega\hat{\alpha_0}}{\psi_\Omega\hat{\alpha_0}}{\Gamma\Rightarrow F^{\mathbb{L}_{\psi_\Omega(\hat{\alpha_0})}}\,.}
\end{equation*}
From which one application of $(\exists R)$ yields the desired result.\\

\noindent Case 7. Finally suppose the last inference was (Cut), then there is a formula $C$ with $rk(C)\leq\Omega$ and $\alpha_0<\alpha$ such that
\begin{align*}
&\provx{\mathcal{H}_\eta}{\alpha_0}{\Omega+1}{\Gamma,C\Rightarrow\Delta}\tag{1}\\
&\provx{\mathcal{H}_\eta}{\alpha_0}{\Omega+1}{\Gamma\Rightarrow C\,.}\tag{2}
\end{align*}
7.1 If $rk(C)<\Omega$ then $C$ contains only bounded quantification and as such is both $\Sigma$ and $\Pi$, thus we may apply the induction hypothesis to both (1) and (2) to give
\begin{align*}
&\provx{\mathcal{H}_{\hat{\alpha_0}}}{\psi_\Omega(\hat{\alpha_0})}{\psi_\Omega(\hat{\alpha_0})}{\Gamma,C\Rightarrow\Delta}\tag{3}\\
&\provx{\mathcal{H}_{\hat{\alpha_0}}}{\psi_\Omega(\hat{\alpha_0})}{\psi_\Omega(\hat{\alpha_0})}{\Gamma\Rightarrow C}.\tag{4}
\end{align*}
Since $k(C)\subseteq\mathcal{H}_\eta$ and $rk(C)<\Omega$, we have $rk(C)<\psi_\Omega(\eta+1)$, so we may apply (Cut) to (3) and (4) to obtain the desired result.\\

\noindent 7.2 If $rk(C)=\Omega$ then $C\equiv\exists xF(x)$ or $C\equiv\forall xF(x)$ with $F(\mathbb{L}_0)$ a $\Delta_0$ formula. The two cases are similar so for simplicity just the case where $C\equiv\exists xF(x)$ is considered.\\

\noindent We can begin by immediately applying the induction hypothesis to (2) since $C$ is a $\Sigma$ formula, giving
\begin{equation*}
\provx{\mathcal{H}_{\hat{\alpha_0}}}{\psi_\Omega(\hat{\alpha_0})}{\psi_\Omega(\hat{\alpha_0})}{\Gamma\Rightarrow C}.
\end{equation*}
Now applying Boundedness \ref{boundedness} yields
\begin{equation*}\tag{5}
\provx{\mathcal{H}_{\hat{\alpha_0}}}{\psi_\Omega(\hat{\alpha_0})}{\psi_\Omega(\hat{\alpha_0})}{\Gamma\Rightarrow C^{\mathbb{L}_{\psi_\Omega(\hat{\alpha_0})}}}.
\end{equation*}
Since $\psi_\Omega(\hat{\alpha_0})\in\mathcal{H}_{\hat{\alpha_0}}$ we may apply \ref{weakpers}ii) to (1) to obtain
\begin{equation*}
\provx{\mathcal{H}_{\hat{\alpha_0}}}{\alpha_0}{\Omega+1}{\Gamma,(\exists x\in\mathbb{L}_{\psi_\Omega(\hat{\alpha_0})})F(x)\Rightarrow\Delta}.
\end{equation*}
Now $(\exists x\in\mathbb{L}_{\psi_\Omega(\hat{\alpha_0})})F(x)$ is bounded and hence $\Pi$ so by the induction hypothesis we obtain
\begin{equation*}\tag{6}
\provx{\mathcal{H}_{\hat{\alpha_1}}}{\psi_\Omega(\alpha_1}{\psi_\Omega(\alpha_1)}{\Gamma,(\exists x\in\mathbb{L}_{\psi_\Omega(\hat{\alpha_0})})F(x)\Rightarrow\Delta}.
\end{equation*}
Where $\alpha_1:=\hat{\alpha_0}+\omega^{\Omega+\alpha_0}$. Since $\alpha_1<\eta+\omega^{\Omega+\alpha}:=\hat{\alpha}$ and $rk((\exists x\in\mathbb{L}_{\psi_\Omega(\hat{\alpha_0})})F(x))<\psi_\Omega(\alpha)$ we may apply (Cut) to (5) and (6) to complete the proof.
\end{proof}

\subsection{Embedding \textbf{IKP} into $\IRS$}
In this section we show how \textbf{IKP} derivations can be carried out in a very uniform manner within \textbf{IRS}$_\OO$. First some preparatory definitions.
\begin{definition}{\em \begin{description}
\item[i)] Given ordinals $\alpha_1,...,\alpha_n$. The expression $\omega^{\alpha_1}\#...\#\omega^{\alpha_n}$ denotes the ordinal
\begin{equation*}
\omega^{\alpha_{p(1)}}+...+\omega^{\alpha_{p(n)}}
\end{equation*}
where $p\!:\!\{1,...,n\}\mapsto \{1,...,n\}$ such that $\alpha_{p(1)}\geq...\geq\alpha_{p(n)}$. More generally $\alpha\#0:=0\#\alpha:=0$ and if $\alpha=_{NF}\omega^{\alpha_1}+...+\omega^{\alpha_n}$ and $\beta=_{NF}\omega^{\beta_1}+...+\omega^{\beta_m}$ then $\alpha\#\beta:=\omega^{\alpha_1}\#...\#\omega^{\alpha_n}\#\omega^{\beta_1}\#...\#\omega^{\beta_m}$.

\item[ii)] If $A$ is any \textbf{IRS}$_\Omega$-formula then $no(A):=\omega^{rk(A)}$ and if $\Gamma\Rightarrow\Delta$ is an \textbf{IRS}$_\Omega$-sequent containing formulas $\{A_1,...,A_n\}$, then $no(\Gamma\Rightarrow\Delta):=no(A_1)\#...\#no(A_n)$.

\item[iii)] $\Vdash\Gamma\Rightarrow\Delta$ will be used to abbreviate that
\begin{equation*}
\provx{\mathcal{H}[\Gamma\Rightarrow\Delta]}{no(\Gamma\Rightarrow\Delta)}{0}{\Gamma\Rightarrow\Delta}\quad\text{holds for any operator $\mathcal{H}$.}
\end{equation*}
\item[iv)] $\Vdash^\xi_\rho\Gamma\Rightarrow\Delta$ will be used to abbreviate that
\begin{equation*}
\provx{\mathcal{H}[\Gamma\Rightarrow\Delta]}{no(\Gamma\Rightarrow\Delta)\#\xi}{\rho}{\Gamma\Rightarrow\Delta}\quad\text{holds for any operator $\mathcal{H}$.}
\end{equation*}
\end{description}
}\end{definition}
\noindent We would like to be able to use $\Vdash$ as a calculus since it dispenses with a lot of superfluous notation, luckily under certain conditions this is possible.
\begin{lemma}\label{calc}{\em i) If $\GA\RI\DE$ follows from premises $\GA_i\RI\DE_i$ by an inference other than (Cut) or $\SR$ and without contractions then
\begin{equation*}
\text{if}\;\Vdash^\al_\rho\GA_i\RI\DE_i\quad\text{then}\quad\Vdash^\al_\rho\GA\RI\DE.
\end{equation*}
ii) If $\Vdash^\al_\rho \GA,A,B\RI\DE$ then $\Vdash^\al_\rho\GA,A\wedge B\RI\DE$.
}\end{lemma}
\begin{proof}
The first part follows from the additive principal nature of ordinals of the form $\omega^\al$ and Lemma \ref{rank}.\\

\noindent For the second part suppose $\Vdash^\al_\rho\GA,A,B\RI\DE$ which means we have
\begin{equation*}
\provx{\CH[\GA,A,B\RI\DE]}{no(\Gamma\Rightarrow\Delta)\#no(A)\#no(B)\#\al}{\rho}{\GA,A,B\RI\DE}.
\end{equation*}
Two applications of $(\wedge L)$ yields
\begin{equation*}
\provx{\CH[\GA,A,B\RI\DE]}{no(\Gamma\Rightarrow\Delta)\#no(A)\#no(B)\#\al+2}{\rho}{\GA,A\wedge B\RI\DE}.
\end{equation*}
It remains to note that $\CH[\GA,A,B\RI\DE]=\CH[\GA,A\wedge B\RI\DE]$ and
\begin{equation*}
no(A)\#no(B)+2=\omega^{rk(A)}\#\omega^{rk(B)}+2<\omega^{rk(A\wedge B)}=no(A\wedge B)
\end{equation*}
to complete the proof.
\end{proof}
\begin{lemma}\label{setuplemma} {\em For any \textbf{IRS}$_\Omega$ formulas $A,B$ and terms $s,t$ we have
\begin{description}
\item[i)]$\Vdash\Gamma,A\Rightarrow A\,.$
\item[ii)]$\Vdash s\in s\Rightarrow\,. $
\item[iii)]$\Vdash\Rightarrow s\subseteq s$    here $s\subseteq s$ is shorthand for $(\forall x\in s)(x\in s)$.
\item[iv)]$\Vdash\Rightarrow s\dotin t\rightarrow s\in t$ and $\Vdash s\dotin t\Rightarrow s\in t$, for $\lev s<\lev t$.
\item[v)]$\Vdash s=t\Rightarrow t=s\,.$
\item[vi)] If $\Vdash\Gamma,A\Rightarrow B$ then $\Gamma,s\dotin t\wedge A\Rightarrow s\dotin t\wedge B$ for $\lev s<\lev t$.
\item[vii)] If $\Vdash\Gamma,A,B\Rightarrow\Delta$ then $\Vdash\Gamma,s\dotin t\rightarrow A,s\dotin t\wedge B\Rightarrow\Delta$ for $\lev s<\lev t$.
\item[viii)] If $\lev s<\beta$ then $\Vdash\Rightarrow s\in\mathbb{L}_\beta\,.$
\end{description}
}\end{lemma}
\begin{proof}
i) By induction of $rk(A)$. We split into cases based on the form of the formula $A$.\\

\noindent Case 1. If $A\equiv (r\in t)$ then by the induction hypothesis we have
\begin{equation*}
\Vdash\Gamma, s\dotin t\wedge r=s\Rightarrow s\dotin t\wedge r=s\quad\text{for all $\lev s<\lev t$.}
\end{equation*}
The following is a template for \textbf{IRS}$_\Omega$ derivations.
\begin{prooftree}
\Axiom$\fCenter\Vdash s\dotin t\wedge r=s\Rightarrow s\dotin t\wedge r= s\quad\text{for all $\lev s<\lev t$}$
\LeftLabel{$(\in\! R)$}
\UnaryInf$\fCenter\Vdash s\dotin t\wedge r=s\Rightarrow r\in t\quad\text{for all $\lev s<\lev t$}$
\LeftLabel{$(\in\! L)_\infty$}
\UnaryInf$\fCenter\Vdash r\in t\Rightarrow r\in t$
\end{prooftree}
Case 2. If $A\equiv (\exists x\in t)F(x)$ then by the induction hypothesis we have
\begin{equation*}
\Vdash s\dotin t\wedge F(s)\Rightarrow s\dotin t\wedge F(s)\quad\text{for all $\lev s<\lev t$.}
\end{equation*}
We have the following template for \textbf{IRS}$_\Omega$ derivations.
\begin{prooftree}
\Axiom$\fCenter\Vdash s\dotin t\wedge F(s)\Rightarrow s\dotin t\wedge F(s)\quad\text{for all $\lev s<\lev t$}$
\LeftLabel{$(b\exists R)$}
\UnaryInf$\fCenter\Vdash s\dotin t\wedge F(s)\Rightarrow (\exists x\in t)F(x)\quad\text{for all $\lev s<\lev t$}$
\LeftLabel{$(b\exists L)_\infty$}
\UnaryInf$\fCenter\Vdash (\exists x\in t)F(x)\Rightarrow (\exists x\in t)F(x)$
\end{prooftree}
Case 3. All remaining cases can be proved in a similar fashion to above.\\

\noindent ii) The proof is by induction on $rk(s)$, inductively we get $\Vdash r\in r\Rightarrow$ for all $\lev r<\lev s$. Now if $s$ is of the form $\mathbb{L}_\alpha$, then $r\in r\equiv r\dotin s\rightarrow r\in r$ and we have the following template for \textbf{IRS}$_\Omega$ derivations.
\begin{prooftree}
\Axiom$\fCenter\Vdash r\in r\Rightarrow\quad\text{for all $\lev r<\lev s$}$
\LeftLabel{$(b\forall L)$}
\UnaryInf$\fCenter\Vdash(\forall x\in s)(x\in r)\Rightarrow\quad\text{for all $\lev r<\lev s$}$
\LeftLabel{$(\wedge L)$}
\UnaryInf$\fCenter\Vdash s=r\Rightarrow\quad\text{for all $\lev r<\lev s$}$
\LeftLabel{$(\in\! L)_\infty$}
\UnaryInf$\fCenter\Vdash s\in s\Rightarrow$
\end{prooftree}
Now if $s\equiv [x\in \mathbb{L}_\alpha\:|\:B(x)]$ then we have the following template for derivations in \textbf{IRS}$_\Omega$.
\begin{prooftree}
\Axiom$\fCenter\quad\quad i)$
\UnaryInf$\fCenter\Vdash B(r)\Rightarrow B(r)\quad\text{for all $\lev r<\lev s$}$
\Axiom$\fCenter\quad\text{Induction Hypothesis}$
\UnaryInf$\fCenter\Vdash r\in r\Rightarrow\quad\text{for all $\lev r<\lev s$}$
\LeftLabel{$(\rightarrow L)$}
\BinaryInf$\fCenter\Vdash B(r), B(r)\rightarrow r\in r\Rightarrow$
\LeftLabel{$(b\forall L)$}
\UnaryInf$\fCenter\Vdash B(r),(\forall x\in s)(x\in r)\Rightarrow$
\LeftLabel{$(\wedge L)$}
\UnaryInf$\fCenter\Vdash B(r),r=s\Rightarrow$
\LeftLabel{Lemma \ref{calc}ii)}
\UnaryInf$\fCenter\Vdash B(r)\wedge r=s\Rightarrow$
\LeftLabel{$(\in\! L)_\infty$}
\UnaryInf$\fCenter\Vdash s\in s\Rightarrow$
\end{prooftree}
iii) Again we use induction on $rk(s)$. Inductively we have $\Vdash\Rightarrow r\subseteq r$ for all $\lev r<\lev s$. If $s\equiv[x\in \mathbb{L}_\alpha\:|\:B(x)]$ then we have the following template for derivations in \textbf{IRS}$_\Omega$.
\begin{prooftree}
\Axiom$\fCenter\quad\quad i)$
\UnaryInf$\fCenter\Vdash B(r)\Rightarrow B(r)\quad\text{for all $\lev r<\lev s$}$
\Axiom$\fCenter\text{Induction Hypothesis}$
\UnaryInf$\fCenter\Vdash B(r)\Rightarrow r\subseteq r\quad\text{for all $\lev r<\lev s$}$
\LeftLabel{$(\wedge R)$}
\UnaryInf$\fCenter\Vdash B(r)\Rightarrow r=r$
\LeftLabel{$(\wedge R)$}
\BinaryInf$\fCenter\Vdash B(r)\Rightarrow B(r)\wedge r=r$
\LeftLabel{$(\in\! R)$}
\UnaryInf$\fCenter\Vdash B(r)\Rightarrow r\in s$
\LeftLabel{$(\rightarrow R)$}
\UnaryInf$\fCenter\Vdash \Rightarrow r\dotin s\rightarrow r\in s$
\LeftLabel{$(b\forall R)_\infty$}
\UnaryInf$\fCenter\Vdash\Rightarrow(\forall x\in s)(x\in s)$
\end{prooftree}
If $s\equiv\mathbb{L}_\alpha$ then we have the following template for derivations in \textbf{IRS}$_\Omega$.
\begin{prooftree}
\Axiom$\fCenter\text{Induction Hypothesis}$
\UnaryInf$\fCenter\Vdash\Rightarrow r\subseteq r\quad\text{for all $\lev r<\lev s$}$
\LeftLabel{$(\wedge R)$}
\UnaryInf$\fCenter\Vdash\Rightarrow r=r$
\LeftLabel{$(\in\! R)$}
\UnaryInf$\fCenter\Vdash\Rightarrow r\in s$
\LeftLabel{$(b\forall R)_\infty$}
\UnaryInf$\fCenter\Vdash\Rightarrow(\forall x\in s)(x\in s)$
\end{prooftree}
iv) Was shown whilst proving iii).\\

\noindent v) The following is a template for \textbf{IRS}$_\Omega$ derivations
\begin{prooftree}
\Axiom$\fCenter\quad i)$
\UnaryInf$\fCenter\Vdash s\subseteq t\Rightarrow s\subseteq t$
\LeftLabel{$(\wedge L)$}
\UnaryInf$\fCenter\Vdash s=t\Rightarrow s\subseteq t$
\Axiom$\fCenter\quad i)$
\UnaryInf$\fCenter\Vdash t\subseteq s\Rightarrow t\subseteq s$
\LeftLabel{$(\wedge L)$}
\UnaryInf$\fCenter\Vdash s=t\Rightarrow t\subseteq s$
\LeftLabel{$(\wedge R)$}
\BinaryInf$\fCenter\Vdash s=t\Rightarrow t=s$
\end{prooftree}
vi) Trivial for $t\equiv\mathbb{L}_\alpha$, now if $t\equiv[x\in\mathbb{L}_\alpha\:|\:C(x)]$ then we have the following template for \textbf{IRS}$_\Omega$ derivations.
\begin{prooftree}
\Axiom$\fCenter\Vdash\Gamma,A\Rightarrow B$
\LeftLabel{$(\wedge L)$}
\UnaryInf$\fCenter\Vdash\Gamma, C(s)\wedge A\Rightarrow B$
\Axiom$\fCenter\Vdash\Gamma, C(s)\Rightarrow C(s)$
\LeftLabel{$(\wedge L)$}
\UnaryInf$\fCenter\Vdash\Gamma, C(s)\wedge A\Rightarrow C(s)$
\LeftLabel{$(\wedge R)$}
\BinaryInf$\fCenter\Vdash\Gamma, C(s)\wedge A\Rightarrow C(s)\wedge B$
\end{prooftree}
vii) This is also trivial for $t\equiv\mathbb{L}_\alpha$ so suppose $t\equiv[x\in\mathbb{L}_\alpha\:|\:C(x)]$ and we have the following template for \textbf{IRS}$_\Omega$ derivations.
\begin{prooftree}
\Axiom$\fCenter\Vdash\Gamma,C(s)\Rightarrow C(s)$
\LeftLabel{$(\wedge L)$}
\UnaryInf$\fCenter\Vdash\Gamma, C(s)\wedge B\Rightarrow C(s)$
\Axiom$\fCenter\Vdash\Gamma, A,B\Rightarrow \Delta$
\LeftLabel{$(\wedge L)$}
\UnaryInf$\fCenter\Vdash\Gamma, A, C(s)\wedge B\Rightarrow\Delta$
\LeftLabel{$(\rightarrow L)$}
\BinaryInf$\fCenter\Vdash\Gamma, C(s)\rightarrow A, C(s)\wedge B\Rightarrow\Delta$
\end{prooftree}
viii) Suppose $\lev s<\beta$ then we have the following template for derivations in \text{IRS}$_\Omega$.
\begin{prooftree}
\Axiom$\fCenter\quad\quad iii)$
\UnaryInf$\fCenter\Vdash \Rightarrow s=s$
\LeftLabel{$(\in\! R)$}
\UnaryInf$\fCenter\Vdash\Rightarrow s\in\mathbb{L}_\beta$
\end{prooftree}
\end{proof}
\begin{lemma}\label{equality1}{\em
For any terms $s_1,...,s_n,t_1,...,t_n$ and any formula $A(s_1,...,s_n)$ we have
\begin{equation*}
\Vdash[s_1=t_1],...,[s_n=t_n],A(s_1,...,s_n)\Rightarrow A(t_1,...,t_n)
\end{equation*}
Where $[s_i=t_i]$ is shorthand for $s_i\subseteq t_i, t_i\subseteq s_i$.
}\end{lemma}
\begin{proof}
We proceed by induction on $rk(A(s_1,...,s_n))\#rk(A(t_1,...,t_n))$.\\

\noindent Case 1. Suppose $A(x_1,x_2)\equiv (x_1\in x_2)$, then for all $\lev s<\lev{s_2}$ and $\lev t<\lev{t_2}$ we have the following template for derivations in $\IRS$.
\begin{prooftree}
\Axiom$\fCenter\Vdash [s_1=t_1],[t=s],s_1=s\Rightarrow t_1=t$
\LeftLabel{Lemma \ref{calc}ii)}
\UnaryInf$\fCenter\Vdash [s_1=t_1], t=s,s_1=s\Rightarrow t_1=t$
\LeftLabel{\ref{setuplemma}vi)}
\UnaryInf$\fCenter\Vdash [s_1=t_1], t\dotin t_2\wedge t=s, s_1=s\Rightarrow t\dotin t_2\wedge t_1=t$
\LeftLabel{$(\in\! R)$}
\UnaryInf$\fCenter\Vdash [s_1=t_1], t\dotin t_2\wedge t=s,s_1=s\Rightarrow t_1\in t_2$
\LeftLabel{$(\in\! L)_\infty$}
\UnaryInf$\fCenter\Vdash [s_1=t_1], s\in t_2,s_1=s\Rightarrow t_1\in t_2$
\LeftLabel{\ref{setuplemma}vii)}
\UnaryInf$\fCenter\Vdash [s_1=t_1], s\dotin s_2\rightarrow s\in t_2,s\dotin s_2\wedge s_1=s\Rightarrow t_1\in t_2$
\LeftLabel{$(\forall L)$}
\UnaryInf$\fCenter\Vdash [s_1=t_1], (\forall x\in s_2)(x\in t_2),s\dotin s_2\wedge s_1=s\Rightarrow t_1\in t_2$
\LeftLabel{$(\in\! L)_\infty$}
\UnaryInf$\fCenter\Vdash [s_1=t_1], (\forall x\in s_2)(x\in t_2),s_1\in s_2\Rightarrow t_1\in t_2$
\LeftLabel{Lemma \ref{weakpers}i)}
\UnaryInf$\fCenter\Vdash [s_1=t_1], [s_2=t_2],s_1\in s_2\Rightarrow t_1\in t_2$
\end{prooftree}
Case 2. If $A(x_1,x_2)\equiv x_1\in x_1$ then the assertion follows by Lemma \ref{setuplemma}ii) and weakening.\\

\noindent Case 3. Suppose $A(x_1,...,x_n)\equiv (\exists y\in x_i)B(y,x_1,...,x_n)$, for simplicity let us suppose that $i=1$. Inductively for all $\lev r<\lev{s_1}$ we have
\begin{prooftree}
\Axiom$\fCenter\Vdash[s_1=t_1],...,[s_n=t_n], r\dotin s_1 \wedge B(r,s_1,...,s_n)\Rightarrow r\dotin t_1\wedge B(r,t_1,...,t_n)$
\LeftLabel{$(b\exists R)$}
\UnaryInf$\fCenter\Vdash[s_1=t_1],...,[s_n=t_n], r\dotin s_1 \wedge B(r,s_1,...,s_n)\Rightarrow (\exists y\in s_1)B(y,t_1,...,t_n)$
\LeftLabel{$(b\exists L)_\infty$}
\UnaryInf$\fCenter\Vdash[s_1=t_1],...,[s_n=t_n], (\exists y\in s_1)B(y,t_1,...,t_n)\Rightarrow (\exists y\in s_1)B(y,t_1,...,t_n)$
\end{prooftree}
Case 4. The bounded universal quantification case is dual to the bounded existential one.\\

\noindent Case 5. If $A(x_1,...,x_n)\equiv\exists yB(y,x_1,...,x_n)$ then inductively for all terms $r$ we have
\begin{equation*}
\Vdash[s_1=t_1],...,[s_n=t_n],B(r,s_1,...,s_n)\Rightarrow B(r,t_1,...,t_n)
\end{equation*}
subsequently applying $(\exists R)$ followed by $(\exists L)_\infty$ yields the desired result.\\

\noindent Case 6. The unbounded universal quantification case is dual to the unbounded existential one.\\

\noindent Case 7. All propositional cases follow immediately from the induction hypothesis.
\end{proof}
\begin{corollary}[Equality] \label{equality}{\em
For any \textbf{IRS}$_\Omega$-formula $A(s_1,...,s_n)$
\begin{equation*}
\Vdash\;\;\Rightarrow s_1=t_1\wedge...\wedge s_n=t_n\wedge A(s_1,...,s_n)\rightarrow A(t_1,...,t_n)\,.
\end{equation*}
}\end{corollary}
\begin{lemma}[Set Induction]\label{Set Induction}{\em
For any formula $F$
\begin{equation*}
\Vdash^{\omega^{rk(A)}}_0\:\Rightarrow\forall x[(\forall y\in x)F(y)\rightarrow F(x)]\rightarrow \forall xF(x).
\end{equation*}
Where $A:=\forall x[(\forall y\in x)F(y)\rightarrow F(x)]$.
}\end{lemma}
\begin{proof}
First we verify the following claim:
\begin{equation*}
\tag{*}\provx{\mathcal{H}[A,s]}{\omega^{rk(A)}\#\omega^{\lev{s}+1}}{0}{A\Rightarrow F(s)}\quad\text{for all $s$}.
\end{equation*}
The claim is verified by induction on $\lev s$, inductively suppose that
\begin{equation*}
\provx{\mathcal{H}[A,t]}{\omega^{rk(A)}\#\omega^{\lev{t}+1}}{0}{A\Rightarrow F(t)}\quad\text{holds for all $\lev t<\lev s$}.
\end{equation*}
If necessary we may apply $(\rightarrow R)$ to obtain
\begin{equation*}
\provx{\mathcal{H}[A,t,s]}{\omega^{rk(A)}\#\omega^{\lev t+1}+1}{0}{A\Rightarrow t\dotin s\rightarrow F(t)}.
\end{equation*}
Next applying $(b\forall R)_\infty$ yields
\begin{equation*}
\provx{\mathcal{H}[A,s]}{\omega^{rk(A)}\#\omega^{\lev s}+2}{0}{A\Rightarrow (\forall y\in s)F(y)}.
\end{equation*}
Also by Lemma \ref{setuplemma}i) we have
\begin{equation*}
\provx{\mathcal{H}[A,s]}{\omega^{rk(F(s))}\#\omega^{rk(F(s))}}{0}{F(s)\Rightarrow F(s)}.
\end{equation*}
Now one may note that $\omega^{rk(F(s))}\#\omega^{rk(F(s))}\leq\omega^{rk(F(s))+1}\leq\omega^{\text{max}(\OO,rk(F(\mathbb{L}_0))+3)}=\omega^{rk(A)}$ to see that by weakening we can conclude
\begin{equation*}
\provx{\mathcal{H}[A,s]}{\omega^{rk(A)\#\omega^{|s|}+2}}{0}{F(s)\Rightarrow F(s)}.
\end{equation*}
Hence using one application of $(\rightarrow L)$ we get
\begin{equation*}
\provx{\mathcal{H}[A,s]}{\omega^{rk(A)}\#\omega^{|s|}+3}{0}{A, (\forall y\in s)F(y)\rightarrow F(s)\Rightarrow F(s)}.
\end{equation*}
Applying $(b\forall L)$ yields
\begin{equation*}
\provx{\mathcal{H}[A,s]}{\omega^{rk(A)}\#\omega^{|s|}+4}{0}{A\Rightarrow F(s)}.
\end{equation*}
Thus the claim (*) is verified. A single application of $(\forall R)_\infty$ to (*) furnishes us with
\begin{equation*}
\provx{\mathcal{H}[A]}{\omega^{rk(A)}\#\OO}{0}{A\Rightarrow \forall xF(x)}.
\end{equation*}
Finally applying $(\rightarrow R)$ gives
\begin{equation*}
\Vdash^{\omega^{rk(A)}}_0\Rightarrow A\rightarrow \forall xF(x)
\end{equation*}
as required.
\end{proof}
\begin{lemma}[Infinity]\label{infinity}{\em
For any ordinal $\al>\omega$ we have
\begin{equation*}
\Vdash\;\;\Rightarrow(\exists x\in\mathbb{L}_\al)[(\exists z\in x)(z\in x)\wedge (\forall y\in x)(\exists z\in x)(y\in z)]\,.
\end{equation*}
}\end{lemma}
\begin{proof}
The following is a template for derivations in \textbf{IRS}$_\Omega$:
\begin{prooftree}
\Axiom$\fCenter\text{Lemma \ref{setuplemma} viii)}$
\UnaryInf$\fCenter\Vdash\Rightarrow\mathbb{L}_0\in\mathbb{L}_\omega$
\LeftLabel{$(b\exists R)$}
\UnaryInf$\fCenter\Vdash\Rightarrow (\exists z\in\mathbb{L}_\omega)(z\in\mathbb{L}_\omega)$
\Axiom$\fCenter\text{Lemma \ref{setuplemma} viii)}$
\UnaryInf$\fCenter\Vdash \Rightarrow s\in\mathbb{L}_\alpha\quad\text{for all $|s|<\alpha<\omega$}$
\LeftLabel{$(b\exists R)$}
\UnaryInf$\fCenter\Vdash \Rightarrow (\exists z\in\mathbb{L}_\omega)(s\in z)\quad\text{for all $|s|<\omega$}$
\LeftLabel{$(b\forall R)_\infty$}
\UnaryInf$\fCenter\Vdash \Rightarrow(\forall y\in\mathbb{L}_\omega)(\exists z\in\mathbb{L}_\omega)(y\in z)$
\LeftLabel{$(\wedge R)$}
\BinaryInf$\fCenter\Vdash \Rightarrow(\exists z\in\mathbb{L}_\omega)(z\in\mathbb{L}_\omega)\wedge(\forall y\in\mathbb{L}_\omega)(\exists z\in\mathbb{L}_\omega)(y\in z)$
\LeftLabel{$(b\exists R)$}
\UnaryInf$\fCenter\Vdash \Rightarrow(\exists x\in\mathbb{L}_\al)[(\exists z\in x)(z\in x)\wedge(\forall y\in x)(\exists z\in x)(y\in z)]$
\end{prooftree}
\end{proof}
\begin{lemma}[$\Delta_0$-Separation]\label{separation}{\em
Suppose $|s|,|t_1|,...,|t_n|<\lambda$ where $\lambda$ is a limit ordinal and $A(a,b_1,...,b_n)$ is a $\Delta_0$-formula of $\IKP$ with all free variables displayed, then
\begin{equation*}
\Vdash\;\;\Rightarrow(\exists y\in\mathbb{L}_\lambda)[(\forall x\in y)(x\in s\wedge A(x,t_1,...,t_n))\wedge (\forall x\in s)(A(x,t_1,...,t_n)\rightarrow x\in y)]\,.
\end{equation*}
}\end{lemma}
\begin{proof}
First let $\beta:=\text{max}\{|s|,|t_1|,...,|t_n|\}+1$ and note that $\beta<\lambda$ since $\lambda$ is a limit. Now let
\begin{equation*}
t:=[u\in\mathbb{L}_\beta\:|\:u\in s\wedge A(u,t_1,...,t_n)].
\end{equation*}
Let $B(x):=A(x,t_1,...,t_n)$, in what follows $r$ ranges over terms with $\lev r<\lev t$ and $p$ ranges over terms with $\lev p<\lev s$. We have the following two templates for derivations in \textbf{IRS}$_\Omega$:\\

\noindent Derivation (1)
\begin{prooftree}
\Axiom$\fCenter\text{Lemma \ref{setuplemma}i)}$
\UnaryInf$\fCenter\Vdash r\in s\wedge B(r)\Rightarrow r\in s\wedge B(r)$
\LeftLabel{$(\rightarrow R)$}
\UnaryInf$\fCenter\Vdash\Rightarrow r\dotin t\rightarrow (r\in s\wedge B(r))$
\LeftLabel{$(b\forall R)_\infty$}
\UnaryInf$\fCenter\Vdash\Rightarrow (\forall x\in t)(x\in s\wedge B(x))$
\end{prooftree}
Derivation (2)
\begin{prooftree}
\Axiom$\fCenter\text{Lemma \ref{setuplemma}iv)}$
\UnaryInf$\fCenter\Vdash p\dotin s, B(p)\Rightarrow p\in s$
\Axiom$\fCenter\text{Lemma \ref{setuplemma}i)}$
\UnaryInf$\fCenter\Vdash p\dotin s, B(p)\Rightarrow B(p)$
\LeftLabel{$(\wedge R)$}
\BinaryInf$\fCenter\Vdash p\dotin s, B(p)\Rightarrow p\in s\wedge B(p)$
\Axiom$\fCenter\text{Lemma \ref{setuplemma}iii)}$
\UnaryInf$\fCenter\Vdash \Rightarrow p\subseteq p$
\LeftLabel{$(\wedge R)$}
\UnaryInf$\fCenter\Vdash \Rightarrow p=p$
\BinaryInf$\fCenter\Vdash p\dotin s, B(p)\Rightarrow (p\in s\wedge B(p))\wedge p=p$
\LeftLabel{$(\in\! R)$}
\UnaryInf$\fCenter\Vdash p\dotin s, B(p)\Rightarrow p\in t$
\LeftLabel{$(\rightarrow R)$}
\UnaryInf$\fCenter\Vdash p\dotin s\Rightarrow B(p)\rightarrow p\in t$
\LeftLabel{$(\rightarrow R)$}
\UnaryInf$\fCenter\Vdash \Rightarrow p\dotin s\rightarrow (B(p)\rightarrow p\in t)$
\LeftLabel{$(b\forall R)_\infty$}
\UnaryInf$\fCenter\Vdash \Rightarrow (\forall x\in s) (B(x)\rightarrow x\in t)$
\end{prooftree}
Now applying $(\wedge R)$ to the conclusions of derivations (1) and (2) we obtain
\begin{equation*}
\Vdash\Rightarrow(\forall x\in t)(x\in s\wedge B(x))\wedge(\forall x\in s) (B(x)\rightarrow x\in t).
\end{equation*}
Finally note that $|t|=\beta<\lambda$ so we may apply $(b\exists R)$ to obtain
\begin{equation*}
\Vdash\Rightarrow(\exists y\in\mathbb{L}_\lambda)[(\forall x\in y)(x\in s\wedge B(x))\wedge(\forall x\in s) (B(x)\rightarrow x\in y)]
\end{equation*}
as required.
\end{proof}
\begin{lemma}[Pair]\label{pair}{\em
If $\lambda$ is a limit ordinal and $|s|,|t|<\lambda$, then
\begin{equation*}
\Vdash\;\;\Rightarrow(\exists z\in\mathbb{L}_\lambda)(s\in z\wedge t\in z)\,.
\end{equation*}
}\end{lemma}
\begin{proof}
Let $\delta:=\text{max}\{|s|,|t|\}+1$ and note that $\delta<\lambda$ since $\lambda$ is a limit. We have the following template for \textbf{IRS}$_\Omega$ derivations:
\begin{prooftree}
\Axiom$\fCenter\text{Lemma \ref{setuplemma}viii)}$
\UnaryInf$\fCenter\Vdash\Rightarrow s\in\mathbb{L}_\delta$
\Axiom$\fCenter\text{Lemma \ref{setuplemma}viii)}$
\UnaryInf$\fCenter\Vdash\Rightarrow t\in\mathbb{L}_\delta$
\LeftLabel{$(\wedge R)$}
\BinaryInf$\fCenter\Vdash\Rightarrow (s\in\mathbb{L}_\delta\wedge t\in\mathbb{L}_\delta)$
\LeftLabel{$(b\exists R)$}
\UnaryInf$\fCenter\Vdash\Rightarrow (\exists z\in\mathbb{L}_\lambda)(s\in z\wedge t\in z)$
\end{prooftree}
\end{proof}
\begin{lemma}[Union]\label{union}{\em
If $\lambda$ is a limit ordinal and $|s|<\lambda$ then
\begin{equation*}
\Vdash\;\;\Rightarrow(\exists z\in\mathbb{L}_\lambda)(\forall y\in s)(\forall x\in y)(x\in z)\,.
\end{equation*}
}\end{lemma}
\begin{proof}
Let $\alpha=|s|$, we have the following template for derivations in \textbf{IRS}$_\Omega$:
\begin{prooftree}
\Axiom$\fCenter\text{Lemma \ref{setuplemma}viii)}$
\UnaryInf$\fCenter\Vdash r\dotin s, q\dotin r\Rightarrow q\in\mathbb{L}_\alpha\quad\text{for all $|q|<|r|<\alpha$}$
\LeftLabel{$(\rightarrow R)$}
\UnaryInf$\fCenter\Vdash r\dotin s\Rightarrow q\dotin r\rightarrow q\in\mathbb{L}_\alpha$
\LeftLabel{$(b\forall R)_\infty$}
\UnaryInf$\fCenter\Vdash r\dotin s\Rightarrow (\forall x\in r)(x\in\mathbb{L}_\alpha)$
\LeftLabel{$(\rightarrow R)$}
\UnaryInf$\fCenter\Vdash \Rightarrow r\dotin s\rightarrow (\forall x\in r)(x\in\mathbb{L}_\alpha)$
\LeftLabel{$(b\forall R)_\infty$}
\UnaryInf$\fCenter\Vdash \Rightarrow(\forall y\in s)(\forall x\in y)(x\in\mathbb{L}_\alpha)$
\LeftLabel{$(b\exists R)$}
\UnaryInf$\fCenter\Vdash \Rightarrow(\exists z\in\mathbb{L}_\lambda)(\forall y\in s)(\forall x\in y)(x\in z)$
\end{prooftree}
\end{proof}
\begin{lemma}[$\DE_0$-Collection]\label{collection}{\em
For any $\Delta_0$ formula $F(x,y)$,
\begin{equation*}
\Vdash\;\;\Rightarrow(\forall x\in s)\exists yF(x,y)\rightarrow\exists z(\forall x\in s)(\exists y\in z)F(x,y)\,.
\end{equation*}
}\end{lemma}
\begin{proof}
Using Lemma \ref{setuplemma} we have
\begin{equation*}
\Vdash (\forall x\in s)\exists yF(x,y)\RI(\forall x\in s)\exists yF(x,y)\,.
\end{equation*}
Now let $\bar\CH:=\CH[(\forall x\in s)\exists yF(x,y)]$ and $\al:=no((\forall x\in s)\exists yF(x,y)\RI(\forall x\in s)\exists yF(x,y))$, by applying $\SR$ we obtain
\begin{equation*}
\bar\CH\;\prov{\al+1}{0}{(\forall x\in s)\exists yF(x,y)\RI\exists z(\forall x\in s)(\exists y\in z)F(x,y)}.
\end{equation*}
Applying $(\rightarrow R)$ gives
\begin{equation*}
\bar\CH\;\prov{\al+2}{0}{\RI(\forall x\in s)\exists yF(x,y)\rightarrow\exists z(\forall x\in s)(\exists y\in z)F(x,y)}.
\end{equation*}
It remains to note that
\begin{align*}
\al+2=\al=&no((\forall x\in s)\exists yF(x,y)\RI(\forall x\in s)\exists yF(x,y))+2\\
<&no(\RI(\forall x\in s)\exists yF(x,y)\rightarrow\exists z(\forall x\in s)(\exists y\in z)F(x,y))
\end{align*}
and $\bar\CH=\CH[\RI(\forall x\in s)\exists yF(x,y)\rightarrow\exists z(\forall x\in s)(\exists y\in z)F(x,y)]$ to complete the proof.
\end{proof}
\begin{theorem}\label{IKPembed}{\em
If $\textbf{IKP}\vdash\Gamma(\bar{a})\Rightarrow\Delta(\bar{a})$ where $\Gamma(\bar{a})\Rightarrow\Delta(\bar{a})$ is an intuitionistic sequent containing exactly the free variables $\bar{a}:=a_1,...,a_n$, then there is an $m<\omega$ (which we may compute from the \textbf{IKP}-derivation) such that
\begin{equation*}
{\mathcal{H}[\Gamma(\bar{s})\Rightarrow\Delta(\bar s)]}\:\prov{\Omega\cdot\omega^m}{\Omega+m}{\Gamma(\bar{s})\Rightarrow\Delta(\bar{s})}
\end{equation*}
for any \textbf{IRS}$_\Omega$ terms $\bar{s}:=s_1,...s_n$ and any operator $\CH$.
}\end{theorem}
\begin{proof}
Let $A$ be any $\IRS$ formula, note that by Observation \ref{rkobs}, we have $rk(A)\leq\Omega+l$ for some $l<\omega$. Therefore
\begin{equation*}
no(A)=\omega^{rk(A)}\leq\omega^{\Omega+l}=\omega^\Omega\cdot\omega^l=\Omega\cdot\omega^l\,.
\end{equation*}
Thus for any choice of terms $\bar{s}$ we have
\begin{equation*}
no(\Gamma(\bar{s})\Rightarrow\Delta(\bar{s}))\leq\Omega\cdot\omega^m\quad\text{for some $m<\omega$.}
\end{equation*}
\noindent The remainder of the proof is by induction on the derivation $\textbf{IKP}\vdash\Gamma(\bar{a})\Rightarrow\Delta(\bar{a})$.\\

\noindent If $\Gamma(\bar{a})\Rightarrow\Delta(\bar{a})$ is an axiom of \textbf{IKP} then the assertion follows by Lemmas \ref{equality}, \ref{Set Induction}, \ref{infinity}, \ref{separation}, \ref{pair}, \ref{union} or \ref{collection}. If $\Gamma(\bar{a})\Rightarrow\Delta(\bar{a})$ was the result of a propositional inference then we may apply the induction hypothesis to the premises and then the corresponding inference in $\IRS$. In order to cut down on notation we make the following abbreviation, let
\begin{equation*}
\bar{\CH}:=\CH[\Gamma(\bar{s})\Rightarrow\Delta(\bar{s})].
\end{equation*}
Case 1. Suppose that $\Gamma(\bar{a})\Rightarrow\Delta(\bar{a})$ was the result of the inference $(b\forall R)$, then $\Delta(\bar{s})=\{(\forall x\in s_i)F(x)\}$. The induction hypothesis furnishes us with an $k<\omega$ such that
\begin{equation*}
\bar\CH[p]\;\prov{\Omega\cdot\omega^k}{\Omega+k}{\Gamma(\bar{s})\Rightarrow p\in s_i\rightarrow F(p)}\quad\text{for all $\lev p<\lev s_i$.}
\end{equation*}
Now by Lemma \ref{inversion}v) we have
\begin{equation*}
\bar\CH[p]\;\prov{\Omega\cdot\omega^k}{\Omega+k}{\Gamma(\bar{s}), p\in s_i\Rightarrow F(p)\,.}
\end{equation*}
Also by \ref{setuplemma}iv) we have
\begin{equation*}
\Vdash p\dotin s_i\Rightarrow p\in s_i\,.
\end{equation*}
Applying (Cut) to these two yields
\begin{equation*}
\bar\CH[p]\;\prov{\Omega\cdot\omega^k+1}{\Omega+k}{\Gamma(\bar{s}), p\dotin s_i\Rightarrow F(p)\,.}
\end{equation*}
Now by $(\rightarrow\! R)$ we have
\begin{equation*}
\bar\CH[p]\;\prov{\Omega\cdot\omega^k+2}{\Omega+k}{\Gamma(\bar{s})\Rightarrow p\dotin s_i\rightarrow F(p)}.
\end{equation*}
Hence by $(b\forall R)_\infty$ we have
\begin{equation*}
\bar\CH\;\prov{\Omega\cdot\omega^{k+1}}{\Omega+k}{\Gamma(\bar{s})\Rightarrow (\forall x\in s_i)F(x)}
\end{equation*}
as required.\\

\noindent Case 2. Now suppose that $\Gamma(\bar{a})\Rightarrow\Delta(\bar{a})$ was the result of the inference $(b\forall L)$. So $(\forall x\in a_i)F(x)\in\Gamma(\bar{a})$ and we are in the following situation in \textbf{IKP}
\begin{prooftree}
\AxiomC{$\Gamma(\bar{a}),c\in a_i\rightarrow F(c)\Rightarrow\Delta(\bar{a})$}
\LeftLabel{$(b\forall L)$}
\UnaryInfC{$\Gamma(\bar{a})\Rightarrow\Delta(\bar{a})$}
\end{prooftree}
If $c$ is not a member of $\bar{a}$ then by the induction hypothesis we have an $m<\omega$ such that
\begin{equation*}\tag{1}
\bar\CH\;\prov{\Omega\cdot\omega^m}{\Omega+m}{\Gamma(\bar{s}), s_1\in s_i\rightarrow F(s_1)\Rightarrow\Delta(\bar{s})}.
\end{equation*}
Now if $c$ is a member of $\bar{a}$, for simplicity let us suppose that $c=a_1$. Inductively we can find an $m<\omega$ such that (1) is also satisfied.
First we verify the following claim:
\begin{equation*}\tag{2}
\Vdash\Gamma,(\forall x\in s_i)F(x)\Rightarrow s_1\in s_i\rightarrow F(s_1)\,.
\end{equation*}
2.1 Suppose $s_i$ is of the form $\mathbb{L}_\alpha$. The claim is verified by the following template for derivations in $\IRS$, here $r$ ranges over terms with $|r|<|s_i|$.
\begin{prooftree}
\Axiom$\fCenter\quad\text{Lemma \ref{equality1}}$
\UnaryInf$\fCenter\Vdash\Gamma, F(r), r\in s_i, r=s_1\Rightarrow F(s_1)$
\LeftLabel{$(b\forall L)$}
\UnaryInf$\fCenter\Vdash\Gamma, (\forall x\in s_i)F(x), r\in s_i, r=s_1\Rightarrow F(s_1)$
\LeftLabel{Lemma \ref{calc}ii)}
\UnaryInf$\fCenter\Vdash\Gamma, (\forall x\in s_i)F(x), r\in s_i\wedge r=s_1\Rightarrow F(s_1)$
\LeftLabel{$(\in\! L)_\infty$}
\UnaryInf$\fCenter\Vdash\Gamma, (\forall x\in s_i)F(x), s_1\in s_i\Rightarrow F(s_1)$
\LeftLabel{$(\rightarrow\! R)$}
\UnaryInf$\fCenter\Vdash\Gamma, (\forall x\in s_i)F(x)\Rightarrow s_1\in s_i\rightarrow F(s_1)$
\end{prooftree}
2.2 Now supposing $s_i$ is of the form $[x\in\mathbb{L}_\alpha\;|\;B(x)]$, we have the following template for derivations in $\IRS$, where $r$ and $p$ range over terms with level below $\lev{s_i}$.
\begin{prooftree}
\Axiom$\fCenter\quad\text{Lemma \ref{equality1}}$
\UnaryInf$\fCenter\Vdash p\dotin s_i, r=p, r=s_i\Rightarrow r\dotin s_i$
\LeftLabel{Lemma \ref{calc}ii)}
\UnaryInf$\fCenter\Vdash p\dotin s_i\wedge r=p, r=s_i\Rightarrow r\dotin s_i$
\LeftLabel{$(\in\!L)_\infty$}
\UnaryInf$\fCenter\Vdash r\in s_i, r=s_i\Rightarrow r\dotin s_i$
\Axiom$\fCenter\quad\text{Lemma \ref{equality1}}$
\UnaryInf$\fCenter\Vdash F(r), r\in s_i, r=s_1\Rightarrow F(s_1)$
\LeftLabel{$(\rightarrow\! L)$}
\BinaryInf$\fCenter\Vdash\Gamma, r\dotin s_i\rightarrow F(r), r\in s_i, r=s_1\Rightarrow F(s_1)$
\LeftLabel{$(b\forall L)$}
\UnaryInf$\fCenter\Vdash\Gamma, (\forall x\in s_i)F(x), r\in s_i, r=s_1\Rightarrow F(s_1)$
\LeftLabel{Lemma \ref{calc}ii)}
\UnaryInf$\fCenter\Vdash\Gamma, (\forall x\in s_i)F(x), r\in s_i\wedge r=s_1\Rightarrow F(s_1)$
\LeftLabel{$(\in\! L)_\infty$}
\UnaryInf$\fCenter\Vdash\Gamma, (\forall x\in s_i)F(x), s_1\in s_i\Rightarrow F(s_1)$
\LeftLabel{$(\rightarrow\! R)$}
\UnaryInf$\fCenter\Vdash\Gamma, (\forall x\in s_i)F(x)\Rightarrow s_1\in s_i\rightarrow F(s_1)$
\end{prooftree}
Now that the claim is verified we may apply (Cut) to (1) and (2) to obtain
\begin{equation*}
\bar\CH\;\prov{\Omega\cdot\omega^{m^\prime}}{\Omega+m^\prime}{\Gamma(\bar{s})\Rightarrow\Delta(\bar{s})}
\end{equation*}
where $\Omega+m^\prime:=\text{max}\{\Omega+m, rk(s_1\in s_i\rightarrow F(s_1))\}$, which is the desired result.\\

\noindent All other quantification cases are similar to Cases 1 and 2.\\

\noindent Finally suppose $\Gamma(\bar{a})\Rightarrow\Delta(\bar{a})$ was the result of (Cut). So we are in the following situation in \textbf{IKP}.
\begin{prooftree}
\AxiomC{$\Gamma(\bar{a}),F(\bar{a},\bar{c})\Rightarrow\Delta(\bar{a})$}
\AxiomC{$\Gamma(\bar{a})\Rightarrow F(\bar{a},\bar{c})$}
\BinaryInfC{$\Gamma(\bar{a})\Rightarrow\Delta(\bar{a})$}
\end{prooftree}
where $\bar{c}$ are the free variables occurring in $F(\bar a,\bar c)$ that are distinct from $\bar a$. By the induction hypothesis we can find $m_0,m_1<\omega$ such that
\begin{align*}
\bar\CH&\;\prov{\Gamma\cdot\omega^{m_0}}{\Omega+m_0}{\Gamma(\bar{s}), F(\bar{s},\overline{\mathbb{L}_0})\Rightarrow\Delta(\bar{s})}\\
\bar\CH&\;\prov{\Gamma\cdot\omega^{m_1}}{\Omega+m_1}{\Gamma(\bar{s})\Rightarrow F(\bar{s},\overline{\mathbb{L}_0})}.
\end{align*}
Note that $k(F(\bar{s},\overline{\mathbb{L}_0}))\subseteq\bar\CH$ so we may apply (Cut) to finish the proof.
\end{proof}
\subsection{An ordinal analysis of $\IKP$}
\begin{lemma}\label{conc}{\em
If $A$ is a $\Sigma$-sentence and $\IKP\vdash\;\RI A$, then there is some $m<\omega$, which we may compute explicitly from the derivation, such that
\begin{equation*}
\provx{\CH_\gamma}{\varphi(\psio{\gamma})(\psio{\gamma})}{0}{\RI\; A}\quad\text{where $\gamma:=\omega_m(\OO\cdot\omega^m)$.}
\end{equation*}
Here $\omega_0(\al):=\al$ and $\omega_{k+1}(\al):=\omega^{\omega_k(\al)}$.
}\end{lemma}
\begin{proof}
Suppose that $A$ is a $\Sigma$-sentence and that $\IKP\vdash\;\RI A$, then by Theorem \ref{IKPembed} there is some $1\leq m<\omega$ such that
\begin{equation*}
\tag{1}\provx{\CH_0}{\OO\cdot\omega^m}{\OO+m}{\RI\; A}.
\end{equation*}
By applying Predicative Cut Elimination \ref{predce} $(m-1)$ times we obtain
\begin{equation*}
\tag{2}\provx{\CH_0}{\omega_{m-1}(\OO\cdot\omega^m)}{\OO+1}{\RI\; A}.
\end{equation*}
Applying Collapsing \ref{collapsing} to (2) gives
\begin{equation*}
\tag{3}\provx{\CH_\gamma}{\psio{\gamma}}{\psio{\gamma}}{\RI\;A}\quad\text{where $\gamma:=\omega_m(\OO\cdot\omega^m)$.}
\end{equation*}
Finally by applying Predicative Cut Elimination \ref{predce} again we get
\begin{equation*}
\provx{\CH_\gamma}{\varphi(\psio{\gamma})(\psio{\gamma})}{0}{\RI\;A}
\end{equation*}
completing the proof.
\end{proof}
\begin{theorem}\label{conct}{\em
If $A\equiv\exists xC(x)$ is a $\Sigma$-sentence such that $\IKP\vdash\;\RI A$ then there is an ordinal term $\al<\psio{\varepsilon_{\OO+1}}$, which we may compute from the derivation, such that
\begin{equation*}
L_\alpha\models A.
\end{equation*}
Moreover, there is a specific $\IRS$ term $s$, with $\lev s<\al$, which we may compute explicitly from the $\IKP$ derivation, such that
\begin{equation*}
L_\al\models C(s).
\end{equation*}
}\end{theorem}
\begin{proof}
Suppose $\IKP\vdash\;\RI A$ for some $\Sigma$-sentence $A$, from Lemma \ref{conc} we may compute some $1\leq m<\omega$ such that
\begin{equation*}
\provx{\CH_\gamma}{\varphi(\psio{\gamma})(\psio{\gamma})}{0}{\RI\;A}\quad\text{where $\gamma:=\omega_m(\OO\cdot\omega^m)$.}
\end{equation*}
Let $\al:=\varphi(\psio{\gamma})(\psio{\gamma})$, applying Boundedness \ref{boundedness} we obtain
\begin{equation*}
\tag{2}\provx{\CH_\gamma}{\al}{0}{\RI\;A^{\mathbb{L}_\al}}.
\end{equation*}
Since the derivation (2) contains no instances of $\Cut$ or $\SR$ and the correctness of the remaining rules within $L_\al$ is easily verified by induction on the derivation, it may be seen that
\begin{equation*}
L_{\al}\models A.
\end{equation*}
For the second part of the theorem note that it must be the case that the final inference in (2) was $(b\exists R)$ and thus by the intuitionistic nature of $\IRS$ there must be some $s$, with $\lev s<\al$, such that
\begin{equation*}
\tag{3}\provx{\CH_\gamma}{\al}{0}{\RI C(s)^{\mathbb{L}_\al}}.
\end{equation*}
Thus
\begin{equation*}
\tag{4}L_\al\models C(s).
\end{equation*}
The remainder of the proof is by checking that each part of the embedding and cut elimination of the previous two sections may be carried out effectively, details will appear in \cite{rathjen-EP}.
\end{proof}
\begin{remark}{\em
In fact Theorem \ref{conct} can be verified within $\IKP$, this is not immediately obvious since we do not have access to induction up to $\psio{\varepsilon_{\OO+1}}$. However one may observe that in an infinitary proof of the form (3) above, no terms of level higher than $\alpha$ are used. By carrying out the construction of $\IRS$ just using ordinals from $B(\omega_{m+1}(\OO\cdot\omega^m))$ we get a restricted system, but a system still capable of carrying out the embedding and cut elimination necessary for the particular derivation of the sentence $A$. This can be done inside $\IKP$ since we do have access to induction up to $\psio{\omega_{m+1}(\OO\cdot\omega^{m+1})}$. It follows that $\IKP$ has the set existence property for $\Sigma$ sentences. More details will be found \cite{rathjen-EP}.
}\end{remark}

Finally it is also worth pointing out that we can improve on Theorem \ref{IKPembed}. Instead of just verifying $\Delta_0$-Collection in the infinitary system
(Lemma \ref{collection}) we could have shown  the embedding result for $\Sigma$-Reflection.
As result we get a new conservativity result.
\begin{theorem}\label{cons1} {\em $\IKP$ and $\IKP+\mbox{$\Sigma$-Reflection}$ prove the same $\Sigma$-sentences.
In particular if $\IKP\vdash A$ with $A$ a $\Sigma$-sentence, then $\IKP\vdash \exists x\,A^{x}$. }
\end{theorem}

\section{The case of $\IKPP$}
This section provides a relativised ordinal analysis for intuitionistic power Kripke-Platek set theory $\IKPP$. The relativised ordinal analysis for the classical version of the theory, $\KPP$, was carried out in \cite{powerKP}, the work in this section adapts the techniques from that paper to the intuitionistic case. We begin by defining an infinitary system $\IRSOP$, unlike in $\IRS$ the terms in $\IRSOP$ can contain sub terms of a higher level, or from higher up the Von-Neumann hierarchy in the intended interpretation. This reflects the impredicativity of the power set operation. Next we prove some cut elimination theorems, allowing us to transform infinite derivations of $\Sigma$ formulae into infinite derivations with only power-bounded cut formulae. The following section provides an embedding of $\IKPP$ into $\IRSOP$. The final section collates these results into a relativised ordinal analysis of $\IKPP$.
\subsection{A sequent calculus formulation of $\IKPP$}
\begin{definition}{\em
The formulas of $\IKPP$ are the same as those of $\IKP$ except we also allow \textit{subset bounded quantifiers} of the form
\begin{equation*}
(\forall x\subseteq a)A(x)\quad\text{and}\quad(\exists x\subseteq a)A(x).
\end{equation*}
These are treated as quantifiers in their own right, not abbreviations. In contrast, the formula $a\subseteq b$ is still viewed as an abbreviation for the formula $(\forall x\in a)(x\in b)$\\

\noindent Quantifiers $\forall x$, $\exists x$ will still be referred to as unbounded, whereas the other quantifiers (including the subset bounded ones) will be referred to as bounded.\\

\noindent A $\Delta_0^\mathcal{P}$-formula of $\IKPP$ is one that contains no unbounded quantifiers.\\

\noindent As with $\IKP$, the system $\IKPP$ derives intuitionistic sequents of the form $\Gamma\Rightarrow\Delta$ where at most one formula can occur in $\Delta$.\\

\noindent The axioms of $\IKPP$ are the following:\\

\begin{tabular} {ll}
{\em {Logical axioms:}}  & $\Gamma, A \Rightarrow A$ \ for every
$\Deltaop$--formula A.\\
{\em {Extensionality:}}   & $\Gamma\Rightarrow a\!=\!b\wedge B(a)\rightarrow
B(b)$
\ for every $\Deltaop$-formula $B(a)$.\\
{\em {Pair:}}   & $\Gamma\Rightarrow \exists
x[a\!\in\!x \wedge b\!\in\!x]$\\
{\em {Union:}}   & $\Gamma\Rightarrow \exists x(\forall y\!\in\!a)(\forall
z\!\in\!y)(z\!\in\!x)$ \\
$\Deltaop$ {\em--{Separation:}}  & $\Gamma\Rightarrow\exists y[(\forall x\in y)(x\in a\wedge B(x))\wedge (\forall x\in a)(B(x)\rightarrow x\in y)]$\\
& for every $\Deltaop$-formula $B(a)$.\\
$\Deltaop$ {\em--{Collection:}}   & $\Gamma\Rightarrow (\forall x\in a)\exists yG(x,y) \rightarrow\exists z(\forall x\in a)(\exists y\in z)G(x,y)$\\
 &for every $\Deltaop$--formula $G(a,b)$.\\
{\em {Set Induction:}}    & $\Gamma\Rightarrow \forall u\,[(\forall x\In u)\,G(x)\,\to\,G(u)]\,\to\,\forall u\,G(u)$\\
& for every formula $G(b)$. \\
{\em {Infinity:}} &  $\Gamma\Rightarrow \exists x\,[(\exists y\In x)\,y\in x\;\wedge\;(\forall y\in x)(\exists z\In x) \,y\in z]$.\\
{\em {Power Set:}} & $\Gamma\Rightarrow \exists z\,(\forall x\subb a) x\in z$.
\end{tabular}
\\[0.4cm]
The rules of $\IKPP$ are the same as those of $\IKP$ (extended to the new language containing subset bounded quantifiers), together with the following four rules:

\begin{prooftree}
\Axiom$\fCenter\Gamma, a\subseteq b\wedge F(a)\Rightarrow \Delta$
\LeftLabel{$(pb\exists L)$}
\UnaryInf$\fCenter\Gamma,(\exists x\subseteq b)F(x)\Rightarrow\Delta$
\Axiom$\fCenter\Gamma\Rightarrow a\subseteq b\wedge F(a)$
\LeftLabel{$(pb\exists R)$}
\UnaryInf$\fCenter\Gamma\Rightarrow(\exists x\subseteq b)F(x)$
\noLine
\BinaryInf$\fCenter$
\end{prooftree}
\begin{prooftree}
\Axiom$\fCenter\Gamma, a\subseteq b\rightarrow F(a)\Rightarrow \Delta$
\LeftLabel{$(pb\forall L)$}
\UnaryInf$\fCenter\Gamma,(\forall x\subseteq b)F(x)\Rightarrow\Delta$
\Axiom$\fCenter\Gamma\Rightarrow a\subseteq b\rightarrow F(a)$
\LeftLabel{$(pb\forall R)$}
\UnaryInf$\fCenter\Gamma\Rightarrow(\forall x\subseteq b)F(x)$
\noLine
\BinaryInf$\fCenter$
\end{prooftree}
As usual it is forbidden for the variable $a$ to occur in the conclusion of the rules $(pb\exists L)$ and $(pb\forall R)$, such a variable is referred to as the eigenvariable of the inference.
}\end{definition}
\subsection{The infinitary system $\IRSOP$}
The purpose of this section is to introduce an infinitary proof system $\IRSOP$. As before all ordinals will be assumed to be members of $B^\OO(\varepsilon_{\OO+1})$.
\begin{definition}{\em We define the $\IRSOP$ terms. To each  $\IRSOP$ term $t$ we also assign its ordinal level, $\lev t$.
\begin{description}
\item[1.] For each $\alpha<\Omega$, $\Va$ is an $\IRSOP$ term with $\lev{\Va}=\alpha$.
\item[2.]For each $\alpha<\Omega$, we have infinitely many free variables $a_0^\alpha,a_1^\alpha,a_2^\alpha,....$, with $\lev{a^{\alpha}_i}=\alpha$.
\item[3.] If $F(x,\bar{y}\,)$ is a $\Deltaop$-formula of $\IKPP$ (whose free variables are exactly those indicated) and  $\sbar\equiv s_1,..., s_n$ are $\IRSOP$ terms, then the formal
expression  $[x\In\Va \mid F(x,\sbar\,)]$ is an $\IRSOP$ term with $\lev{[x\In\Va \mid F(x,\sbar\,)]}:=\alpha$.
\end{description}
The $\IRSOP$ formulae are of the form $A(s_1,...,s_n)$, where $A(a_1,...,a_n)$ is a formula of $\IKPP$ with all free variables indicated and $s_1,...,s_n$ are $\IRSOP$ terms.\\

\noindent A formula $A(s_1,...,s_n)$ of $\IRSOP$ is $\Deltaop$ if $A(a_1,...,a_n)$ is a $\Deltaop$ formula of $\IKPP$.\\

\noindent The $\Sigmap$ formulae of $\IRSOP$ are the smallest collection containing the $\Deltaop$-formulae and containing $A\vee B$, $A\wedge B$, $(\forall x\in s)A$, $(\exists x\in s)A$, $(\forall x\subseteq s)A$, $(\exists x\subseteq s)A$, $\exists xA$, $\neg C$ and $C\to A$ whenever it contains $A$ and $B$
 and $C$ is a $\Pip$-formula. The $\Pip$-formulae are the smallest collection containing the $\Deltaop$ formulae and containing $A\vee B$, $A\wedge B$, $(\forall x\in s)A$, $(\exists x\in s)A$, $(\forall x\subseteq s)A$, $(\exists x\subseteq s)A$, $\forall xA$, $\neg D$ and $D\to A$ whenever it contains $A$ and $B$ and $D$ is a $\Sigmap$-formula. \\

\noindent The {\em axioms} of $\IRSOP$ are:\\

\begin{tabular} {ll}
(A1) & $\Gamma, A\Rightarrow A$ \ for $A$ in $\Deltaop$.\\
(A2) & $\Gamma\Rightarrow t=t$.\\
(A3) & $\Gamma, s_1= t_1,...,s_n= t_n, A(s_1,...,s_n)\Rightarrow A(t_1,...,t_n)$ \ for $A(s_1,\ldots,s_n)$ in $\Deltaop$.\\
(A4) & $\Gamma\Rightarrow s\in \Va$ if $\lev{s}<\alpha$.\\
(A5) & $\Gamma\Rightarrow s\subseteq \Va$ if $\lev{s}\leq\alpha$.\\
(A6) & $\Gamma, t\in [x\in \Va\mid F(x,\sbar)]\Rightarrow F(t, \sbar)$ \ for $F(t,\sbar)$ is $\Deltaop$ and $\lev{t}<\alpha$.\\
(A7) & $\Gamma,  F(t,\sbar)\Rightarrow t\in [x\in \Va\mid F(x,\sbar)]$ \ for $F(t,\sbar)$ is $\Deltaop$ and $\lev{t}<\alpha$.
\end{tabular}\\

\noindent The {\em inference rules} of $\IRSOP$ are:
$$\begin{array}{ll}
(b\forall L) & \frac{\DI\Gamma,\, s\In t\to F(s)\Rightarrow\Delta}{\DI\Gamma,
(\forall x \!\in\! t)F(x)^{\phantom{I}}\Rightarrow\Delta }\;\; \mbox{ if } \lev{s}<\lev t\\[0.6cm]

(b\forall R)_{\infty}& \frac{\DI\Gamma\Rightarrow s\In t\to F(s)\;\mbox{ for all $\lev{s}<\lev t$}} {\DI\Gamma\Rightarrow (\forall x \!\in\!
t) F(x)^{\phantom{I}} \phantom{AAAAAAAAA} }\\[0.6cm]

(b\exists L)_{\infty}& \frac{\DI\Gamma,\; s\In t\wedge F(s)\Rightarrow\Delta\;\mbox{ for all $\lev{s}<\lev t$}} {\DI\Gamma,\; (\exists x \!\in\!
t) F(x)^{\phantom{I}}\Rightarrow\Delta \phantom{AAAAAAAAA} }\\[0.6cm]

(b\exists R) & \frac{\DI\Gamma\Rightarrow s\In t\wedge F(s)}{\DI\Gamma\Rightarrow
(\exists x \!\in\! t)F(x)^{\phantom{I}} }\;\; \mbox{ if } \lev{s}<\lev t\\[0.6cm]

(pb\forall L)& {\frac{\DI\Gamma, s \subb t \to F(s)\Rightarrow\Delta}{\DI\Gamma,
(\forall x \subb t)F(x)^{\phantom{I}}\Rightarrow\Delta}}\;\; \mbox{ if } \lev{s}\leq\lev{t}
\\[0.6cm]

(pb\forall R)_{\infty}& \frac{\DI\Gamma\Rightarrow s \subb t \to F(s)\;\mbox{ for all $\lev s\leq \lev t$}}
 {\DI\Gamma\Rightarrow (\forall x \subb t) F(x)^{\phantom{I}} \phantom{AAAAAAAAA} } \\[0.6cm]

  \end{array}$$
$$\begin{array}{ll}

(pb\exists L)_{\infty}& \frac{\DI\Gamma,\; s \subb t \wedge F(s)\Rightarrow\Delta\;\mbox{ for all $\lev s\leq \lev t$}}
 {\DI\Gamma, (\exists x \subb t) F(x)^{\phantom{I}}\Rightarrow\Delta \phantom{AAAAAAAAA} } \\[0.6cm]

(pb\exists R)& {\frac{\DI\Gamma\Rightarrow s \subb t \wedge F(s)}{\DI\Gamma\Rightarrow
(\exists x \subb t)F(x)^{\phantom{I}}}}\;\; \mbox{ if } \lev{s}\leq\lev{t}\\[0.6cm]

(\forall L)& {\frac{\DI\Gamma, F(s)\Rightarrow\Delta}{\DI\Gamma,\;\forall x
\,F(x)^{\phantom{I}}\Rightarrow\Delta}}\\[0.6cm]

(\forall R)_{\infty}& \frac{\DI\Gamma\Rightarrow F(s)\;\mbox{ for all $s$}}
{\DI\Gamma\Rightarrow \forall x
\, F(x)^{\phantom{I}} \phantom{AAA}  }\\[0.6cm]

(\exists L)_{\infty}& \frac{\DI\Gamma,\; F(s)\Rightarrow\Delta\;\mbox{ for all $s$}}
{\DI\Gamma,\;\exists x
\, F(x)^{\phantom{I}}\Rightarrow\Delta \phantom{AAA}  }\\[0.6cm]

(\exists R)& {\frac{\DI\Gamma\Rightarrow F(s)}{\DI\Gamma\Rightarrow \exists x
\,F(x)^{\phantom{I}}}}\;\;
%\\[0.6cm]
 \end{array}$$
$$\begin{array}{ll}
(\in\! L)_{\infty}& \frac{\DI\Gamma, r\in t\wedge r=s\Rightarrow\Delta\;\mbox{ for all $\lev r<\lev t$}}
{\DI\Gamma, s\in\! t^{\phantom{I}}\Rightarrow\Delta\phantom{AAAAAAAAA} }
\\[0.6cm]

(\in\! R)& \frac{\DI\Gamma\Rightarrow r\in t \;\wedge\; r =
s}{\DI\Gamma, s \!\in\! t^{\phantom{I}}}\;\; \mbox{ if } \lev{r}<\lev t\\[0.6cm]

(\subseteq\! L)_{\infty}& \frac{\DI\Gamma,\; r\subseteq t\wedge r=s\Rightarrow\Delta\;\mbox{ for all $\lev r\leq \lev t$}}
{\DI\Gamma,\; s\subseteq\! t^{\phantom{I}}\Rightarrow\Delta\phantom{AAAAAAAAA} }
\\[0.6cm]

(\subseteq\! R)& {\frac{\DI\Gamma\Rightarrow r \subseteq
t\,\wedge\,r=s}{\DI\Gamma\Rightarrow s\subseteq t^{\phantom{I}}}}\;\; \mbox{ if } \lev{r}\leq \lev s
%\\[0.6cm]
\end{array}$$
$$\begin{array}{ll}
\Cut& {\frac{\DI\Gamma, A\Rightarrow\Delta \;\;\;\;\;\; \Gamma\Rightarrow A}{\DI\Gamma^{\phantom{I}}\Rightarrow\Delta}}\\[0.6cm]

\SRP &{\frac{\DI\Gamma\Rightarrow A}{\DI\Gamma\Rightarrow \,\exists z\, A^{z}{\phantom{I}}^{\phantom{I}}}}\;\;
\mbox{ if }A\mbox{ is a }\Sigmap\mbox{-formula,}
\end{array}$$
as well as the rules $(\wedge L)$, $(\wedge R)$, ($\vee L)$, $(\vee R)$, $(\neg L)$, $(\neg R)$, $(\perp)$, $(\rightarrow L)$, $(\rightarrow R)$ from $\IRS$.
As usual $A^z$ results from $A$ by restricting all unbounded quantifiers to $z$.
}\end{definition}
\begin{deff}{\em The {\it rank } of a formula is determined as
follows.}
\begin{enumerate}
\item $rk(s\In t):=max\{\lev s +1,\lev t +1\}$.
\item $rk((\exists x\In t)F(x)):=rk((\forall x\In t)F(x)):=
max\{\lev t,rk(F(\Vb{0}))+2\}.$
\item $rk((\exists x\subb t)F(x)):=rk((\forall x\subb t)F(x)):=
max\{\lev t +1,rk(F(\Vb{0}))+2\}.$
\item $rk(\exists x\,F(x)):=rk(\forall x\,F(x)):=
max\{\Omega,rk(F(\Vb{0}))+2\}.$
\item $rk(A\wedge B):=rk(A\vee B):=rk(A\rightarrow B):=max\{rk(A),rk(B)\}+1.$
\item $rk(\neg A):=rk(A)+1$
\end{enumerate}
\end{deff}
\noindent Note that the definition of rank for $\IRSOP$ formulae is much less complex than for $\IRS$, this is because we are only aiming for partial cut-elimination for this system. In general it will not be possible to remove cuts with $\Deltaop$ cut formulae. Note however that we still have $rk(A)<\OO$ if and only if $A$ is $\Deltaop$.\\

\noindent We also have the following useful lemma.
\begin{lemma}\label{prank}{\em
If $A$ is a formula of $\IRSOP$ with $rk(A)\geq\OO$ (ie. $A$ contains unbounded quantifiers), and $A$ was the result of an $\IRSOP$ inference other than $\SRP$ and (Cut) then the rank of the minor formulae of that inference is strictly less than $rk(A)$.
}\end{lemma}
\begin{definition}[Operator controlled derivability for $\IRSOP$]{\em
If $A(s_1,...,s_n)$ is a formula of $\IRSOP$ then let
\begin{equation*}
\lev{A(s_1,...s_n)}:=\{\lev{s_1},...,\lev{s_n}\}.
\end{equation*}
Likewise if $\Gamma\Rightarrow\Delta$ is an intuitionistic sequent of $\IRSOP$ containing formulas $A_1,...,A_n$, we define
\begin{equation*}
\lev{\Gamma\Rightarrow\Delta}:=\lev{A_1}\cup...\cup\lev{A_n}.
\end{equation*}
}\end{definition}
\begin{deff}{\em
Let $\CH$ be an operator and $\Gamma\Rightarrow\Delta$ an intuitionistic sequent of $\IRSOP$ formulae. We define the relation $\provx{\CH}{\alpha}{\rho}{\Gamma\Rightarrow\Delta}$ by recursion on $\alpha$.\\

\noindent If $\Gamma\Rightarrow\Delta$ is an {\it axiom} and $\lev{\Gamma\Rightarrow\Delta}\cup\{\alpha\}\subseteq\CH$, then $\provx{\CH}{\alpha}{\rho}{\Gamma\Rightarrow\Delta}$.\\

\noindent We require always that $\lev{\Gamma\Rightarrow\Delta}\cup\{\alpha\}\subseteq\mathcal{H}$ where $\Gamma\Rightarrow\Delta$ is the sequent in the conclusion, this condition will not be repeated in the inductive clauses pertaining to the inference rules of $\IRSOP$  given below. The column on the right gives the ordinal requirements for each of the inference rules.
$$ \begin{array}{lcr}

(\in\! L)_{\infty} & \infone{ {\mathcal H}[r]} {\alpha_{r}}
{\rho} {\Gamma, r\in t\wedge r= s\Rightarrow\Delta\mbox{ for all }\lev r<\lev t}
 {\mathcal H} {\alpha} {\rho} {\Gamma,s\in t\Rightarrow\Delta}
&\lev{r}\leq\alpha_r < \alpha \\[0.6cm]

(\in\! R) & \infone{\mathcal H} {\alpha_0} {\rho}
{\Gamma\;\Rightarrow r\in t\wedge r=s} {\mathcal H} {\alpha} {\rho} {\GA\Rightarrow s\in t}
&\begin{array}{r}\alpha_{0} < \alpha\\ \lev r<\lev t \\ \lev{r}<\al
\end{array} \\[0.6cm]

(\subb\! L)_{\infty} & \infone{ {\mathcal H}[r]} {\alpha_{r}}
{\rho} {\Gamma, r\subb t\wedge r= s\Rightarrow\Delta\mbox{ for all }\lev r\leq\lev t}
 {\mathcal H} {\alpha} {\rho} {\Gamma,s\subb t\Rightarrow\Delta}
&\lev{r}\leq\alpha_r < \alpha \\[0.6cm]

(\subb\! R) & \infone{\mathcal H} {\alpha_0} {\rho}
{\Gamma\;\Rightarrow r\subb t\wedge r=s} {\mathcal H} {\alpha} {\rho} {\GA\Rightarrow s\subb t}
&\begin{array}{r}\alpha_{0} < \alpha\\ \lev r\leq\lev t \\ \lev{r}<\al
\end{array}
\end{array}$$
$$\begin{array}{lcr}
(b\forall L) & \infone{\mathcal H} {\alpha_0} {\rho}
{\GA,s\in t\rightarrow A(s)\Rightarrow\Delta} {\mathcal H} {\alpha} {\rho} {\GA,(\forall x\in t)A(x)\Rightarrow\Delta}
&\begin{array}{r}\alpha_{0} < \alpha\\ \lev s<\lev t \\ \lev{s}<\al
\end{array} \\[0.6cm]

(b\forall R)_{\infty} & \infone{ {\mathcal H}[s]} {\alpha_{s}}
{\rho} {\GA\Rightarrow s\in t\rightarrow F(s)\mbox{ for all }\lev s<\lev t}
 {\mathcal H} {\alpha} {\rho} {\Gamma\Rightarrow(\forall x\In t)F(x)}
&\lev{s}\leq\alpha_{s} < \alpha
\\[0.6cm]

(b\exists L)_{\infty} & \infone{ {\mathcal H}[s]} {\alpha_{s}}
{\rho} {\GA,s\in t\wedge F(s)\Rightarrow\Delta\mbox{ for all }\lev s<\lev t}
 {\mathcal H} {\alpha} {\rho} {\Gamma,(\exists x\In t)F(x)\Rightarrow\Delta}
&\lev{s}\leq\alpha_{s} < \alpha
\\[0.6cm]

(b\exists R) & \infone{\mathcal H} {\alpha_0} {\rho}
{\GA\Rightarrow s\in t\wedge A(s)} {\mathcal H} {\alpha} {\rho} {\GA\Rightarrow(\exists x\in t)A(x)}
&\begin{array}{r}\alpha_{0} < \alpha\\ \lev s<\lev t \\ \lev{s}<\al
\end{array}\\[0.6cm]

(pb\forall L) & \infone{\mathcal H} {\alpha_0} {\rho}
{\GA,s\subb t\rightarrow A(s)\Rightarrow\Delta} {\mathcal H} {\alpha} {\rho} {\GA,(\forall x\subb t)A(x)\Rightarrow\Delta}
&\begin{array}{r}\alpha_{0} < \alpha\\ \lev s\leq\lev t \\ \lev{s}<\al
\end{array} \\[0.6cm]

(pb\forall R)_{\infty} & \infone{ {\mathcal H}[s]} {\alpha_{s}}
{\rho} {\GA\Rightarrow s\subb t\rightarrow F(s)\mbox{ for all }\lev s\leq\lev t}
 {\mathcal H} {\alpha} {\rho} {\Gamma\Rightarrow(\forall x\subb t)F(x)}
&\lev{s}\leq\alpha_{s} < \alpha
\\[0.6cm]

(pb\exists L)_{\infty} & \infone{ {\mathcal H}[s]} {\alpha_{s}}
{\rho} {\GA,s\subb t\wedge F(s)\Rightarrow\Delta\mbox{ for all }\lev s\leq\lev t}
 {\mathcal H} {\alpha} {\rho} {\Gamma,(\exists x\subb t)F(x)\Rightarrow\Delta}
&\lev{s}\leq\alpha_{s} < \alpha
\\[0.6cm]

(pb\exists R) & \infone{\mathcal H} {\alpha_0} {\rho}
{\GA\Rightarrow s\subb t\wedge A(s)} {\mathcal H} {\alpha} {\rho} {\GA\Rightarrow(\exists x\subb t)A(x)}
&\begin{array}{r}\alpha_{0} < \alpha\\ \lev s\leq\lev t \\ \lev{s}<\al
\end{array} \\[0.6cm]

(\forall L) & \infone{\mathcal H} {\alpha_0} {\rho}
{\GA, F(s)\Rightarrow\Delta} {\mathcal H} {\alpha} {\rho} {\GA,\forall x F(x)\Rightarrow\Delta}
&\begin{array}{r}\alpha_{0}+1 < \alpha\\  \lev{s}<\al
\end{array} \\[0.6cm]

(\forall R)_{\infty} & \infone{ {\mathcal H}[s]} {\alpha_{s}}
{\rho} {\GA\Rightarrow  F(s)\mbox{ for all } s}
 {\mathcal H} {\alpha} {\rho} {\Gamma\Rightarrow\forall x F(x)}
&\lev{s}<\alpha_{s}+1 < \alpha \\[0.6cm]

(\exists L)_{\infty} & \infone{ {\mathcal H}[s]} {\alpha_{s}}
{\rho} {\GA,  F(s)\Rightarrow\Delta\mbox{ for all } s}
 {\mathcal H} {\alpha} {\rho} {\Gamma,\exists x F(x)\Rightarrow\Delta}
&\lev{s}<\alpha_{s}+1 < \alpha \\[0.6cm]

(\exists R) & \infone{\mathcal H} {\alpha_0} {\rho}
{\GA\Rightarrow F(s)} {\mathcal H} {\alpha} {\rho} {\GA,\Rightarrow\exists x F(x)}
&\begin{array}{r}\alpha_{0}+1 < \alpha\\  \lev{s}<\al
\end{array}\\[0.6cm]

\end{array}$$
$$\begin{array}{lcr}

\Cut & \ifthree{\provx{\mathcal H} {\alpha_0}{\rho}  {\Gamma,
B\Rightarrow\Delta}}{\provx{\mathcal H} {\alpha_0} {\rho}{\Gamma\Rightarrow B}}{\provx
{\mathcal H} {\alpha}{\rho} {\Gamma\Rightarrow\Delta}}
&\begin{array}{r}\alpha_{0} < \alpha\\
rk(B)<\rho\end{array} \\[0.6cm]

\SRP &
\infone{\CH}{\al_0}{\rho}{\Gamma\RI A}{\CH}{\al}{\rho}{\Gamma\RI
\exists z\,A^z}
&\begin{array}{r} \al_0+1,\Omega<\al\\
 A\text{ is a $\Sigmap$-formula}\end{array}
\end{array}$$
Lastly if $\Gamma\Rightarrow\Delta$ is the result of a propositional inference of the form $(\wedge L)$, $(\wedge R)$, ($\vee L)$, $(\vee R)$, $(\neg L)$, $(\neg R)$, $(\perp)$, $(\rightarrow L)$ or $(\rightarrow R)$, with premise(s) $\Gamma_i\Rightarrow\Delta_i$ then from $\provx{\CH}{\alpha_0}{\rho}{\Gamma_i\Rightarrow\Delta_i}$ (for each $i$) we may conclude $\provx{\CH}{\alpha}{\rho}{\Gamma\Rightarrow\Delta}$, provided $\alpha_0<\alpha$.
}\end{deff}

\subsection{Cut elimination for $\IRSOP$}

\begin{lemma}[Weakening and Persistence for $\IRSOP$]$\phantom{a}$\\[-0.5cm] \label{pweakpers}{\em \begin{description}\item[i)] If $\Gamma_0\subseteq\GA$, $\lev{\Gamma}\subseteq\CH$, $\alpha_0\leq\alpha\in\CH$, $\rho_0\leq\rho$ and $\provx{\CH}{\alpha_0}{\rho_0}{\GA_0\Rightarrow\Delta}$ then
\begin{equation*}
\provx{\CH}{\al}{\rho}{\GA\Rightarrow\Delta.}
\end{equation*}
\item[ii)] If $\gamma\in\CH$ and $\provx{\CH}{\al}{\rho}{\GA,\exists xA(x)\Rightarrow\Delta}$ then $\provx{\CH}{\al}{\rho}{\GA,(\exists x\in\mathbb{V}_\gamma)A(x)\Rightarrow\Delta}$.
\item[iii)] If $\gamma\in\CH$ and $\provx{\CH}{\al}{\rho}{\GA\Rightarrow\forall xA(x)}$ then $\provx{\CH}{\al}{\rho}{\GA\Rightarrow(\forall x\in\mathbb{V}_\gamma)A(x)}.$
\end{description}
}\end{lemma}
\begin{proof}
All proofs are by induction on $\al$. We show ii), suppose $\gamma\in\CH$ and $\provx{\CH}{\al}{\rho}{\GA,\exists xA(x)\Rightarrow\Delta}$. The interesting case is where $\exists xA(x)$ was the principal formula of the last inference which was $(\exists L)_\infty$, in this case we have $\provx{\CH[s]}{\alpha_s}{\rho}{\Gamma,\exists xA(x), A(s)\Rightarrow\Delta}$ for each term $s$ with $\lev{s}<\alpha_s+1<\alpha$ (If $\exists xA(x)$ was not a side formula we can use part i) to make it one). By the induction hypothesis we obtain $\provx{\CH[s]}{\alpha_s}{\rho}{\Gamma,(\exists x\in\mathbb{V}_\gamma)A(x), A(s)\Rightarrow\Delta}$ for all $\lev s<\gamma$. By $(\wedge L)$ we get
\begin{equation*}
\provx{\CH[s]}{\alpha_s+1}{\rho}{\Gamma,(\exists x\in\mathbb{V}_\gamma)A(x), s\in\mathbb{V}_\gamma\wedge A(s)\Rightarrow\Delta.}
\end{equation*}
Hence we may apply $(b\exists L)_\infty$ to obtain $\provx{\CH}{\al}{\rho}{\GA,(\exists x\in\mathbb{V}_\gamma)A(x)\Rightarrow\Delta}$ as required.
\end{proof}
\begin{lemma}[Inversions of $\IRSOP$]\label{pinversion} {\em $\phantom{a}$\\[-0.5cm]
\begin{description}
\item[i)] If $\provx{\mathcal{H}}{\alpha}{\rho}{\Gamma,A\wedge B\Rightarrow\Delta}$ and $rk(A\wedge B)\geq\OO$ then $\provx{\mathcal{H}}{\alpha}{\rho}{\Gamma,A, B\Rightarrow\Delta}$.
\item[ii)] If $\provx{\mathcal{H}}{\alpha}{\rho}{\Gamma\Rightarrow A\wedge B}$ and $rk(A\wedge B)\geq\OO$ then $\provx{\mathcal{H}}{\alpha}{\rho}{\Gamma\Rightarrow A}$ and $\provx{\mathcal{H}}{\alpha}{\rho}{\Gamma\Rightarrow B}$.
\item[iii)] If $\provx{\mathcal{H}}{\alpha}{\rho}{\Gamma,A\vee B\Rightarrow\Delta}$ and $rk(A\vee B)\geq\OO$ then $\provx{\mathcal{H}}{\alpha}{\rho}{\Gamma,A\Rightarrow\Delta}$ and $\provx{\mathcal{H}}{\alpha}{\rho}{\Gamma,B\Rightarrow\Delta}$.
\item[iv)]  If $\provx{\mathcal{H}}{\alpha}{\rho}{\Gamma,A\rightarrow B\Rightarrow\Delta}$ and $rk(A\rightarrow B)\geq\OO$ then $\provx{\mathcal{H}}{\alpha}{\rho}{\Gamma,B\Rightarrow\Delta}$.
\item[v)] If $\provx{\mathcal{H}}{\alpha}{\rho}{\Gamma\Rightarrow A\rightarrow B}$ and $rk(A\rightarrow B)\geq\OO$ then $\provx{\mathcal{H}}{\alpha}{\rho}{\Gamma,A\Rightarrow B}$.
\item[vi)] If $\provx{\mathcal{H}}{\alpha}{\rho}{\Gamma\Rightarrow\neg A}$ and $rk(A)\geq\OO$ then $\provx{\mathcal{H}}{\alpha}{\rho}{\Gamma, A\Rightarrow}$.
\item[vii)] If $\provx{\mathcal{H}}{\alpha}{\rho}{\Gamma,(\exists x\in t)A(x)\Rightarrow\Delta}$ and $rk(A(\mathbb{V}_0))\geq\OO$ then $\provx{\mathcal{H}[s]}{\alpha}{\rho}{\Gamma,s\in t\wedge A(s)\Rightarrow\Delta}$ for all $|s|<|t|$.
\item[viii)] If $\provx{\mathcal{H}}{\alpha}{\rho}{\Gamma\Rightarrow(\forall x\in t)A(x)}$ and $rk(A(\mathbb{V}_0))\geq\OO$ then $\provx{\mathcal{H}[s]}{\alpha}{\rho}{\Gamma\Rightarrow s\in t\rightarrow A(s)}$ for all $|s|<|t|$.
\item[ix)] If $\provx{\mathcal{H}}{\alpha}{\rho}{\Gamma,(\exists x\subb t)A(x)\Rightarrow\Delta}$ and $rk(A(\mathbb{V}_0))\geq\OO$ then $\provx{\mathcal{H}[s]}{\alpha}{\rho}{\Gamma,s\subb t\wedge A(s)\Rightarrow\Delta}$ for all $|s|\leq|t|$.
\item[x)] If $\provx{\mathcal{H}}{\alpha}{\rho}{\Gamma\Rightarrow(\forall x\subb t)A(x)}$ and $rk(A(\mathbb{V}_0))\geq\OO$ then $\provx{\mathcal{H}[s]}{\alpha}{\rho}{\Gamma\Rightarrow s\subb t\rightarrow A(s)}$ for all $|s|\leq|t|$.
\item[xi)] If $\provx{\mathcal{H}}{\alpha}{\rho}{\Gamma,\exists x A(x)\Rightarrow\Delta}$ then $\provx{\mathcal{H}[s]}{\alpha}{\rho}{\Gamma, F(s)\Rightarrow\Delta}$ for all $s$.
\item[xii)] If $\provx{\mathcal{H}}{\alpha}{\rho}{\Gamma,\Rightarrow\forall x A(x)}$ then $\provx{\mathcal{H}[s]}{\alpha}{\rho}{\Gamma\Rightarrow F(s)}$ for all $s$.
\end{description}
}\end{lemma}
\begin{proof}
The proof is by induction on $\al$ and many parts are standard for many intuitionistic systems of a similar nature. We show viii) and ix).\\

\noindent viii) Suppose that $\provx{\CH}{\al}{\rho}{\GA\RI(\forall x\in t)A(x)}$ and $rk(A(\mathbb{V}_0))\geq\OO$. Since $A$ must contain an unbounded quantifier, the sequent $\GA\RI(\forall x\in t)A(x)$ cannot be an axiom. If the last inference was not $(b\forall R)_\infty$ then we may apply the induction hypothesis to the premises of that inference, and then the same inference again. Finally suppose the last inference was $(b\forall R)_\infty$ so we have
\begin{equation*}
\CH[s]\;\prov{\al_s}{\rho}{\GA\RI s\in t\rightarrow A(s)}\quad\text{for all $\lev s<\lev t$, with $\al_s<\al$.}
\end{equation*}
Applying weakening completes the proof of this case.\\

\noindent ix) Suppose that $\provx{\CH}{\al}{\rho}{\GA,(\exists x\subb t)A(x)\RI\DE}$ and $rk(A(\mathbb{V}_0))\geq\OO$. Since $A(x)$ contains an unbounded quantifier $\exists x\subb t)A(x)$ cannot be the \textit{active part} of an axiom, thus if $\GA,(\exists x\subb t)A(x)\RI\DE$ is an axiom then so is $\GA,s\subb t\wedge A(x)\RI\DE$ for any $\lev s\leq \lev t$. As in viii) the remaining interesting case is where $(\exists x\subb t)A(x)$ was the principal formula of the last inference, which was $(pb\exists L)_\infty$. In this case we have
\begin{equation*}
\CH[s]\;\prov{\al_s}{\rho}{\GA,(\exists x\subb t)A(x), s\subb t\wedge A(s)\RI\DE}\quad\text{for all $\lev s\leq\lev t$ with $\al_s<\al$.}
\end{equation*}
Now applying the induction hypothesis yields $\CH[s]\;\prov{\al_s}{\rho}{\GA, s\subb t\wedge A(s)\RI\DE}$, to which we may apply weakening to complete the proof of this case.
\end{proof}
\begin{lemma}[Reduction\,.]\label{preduction}{\em
If $rk(C):=\rho>\Omega$, $\provx{\CH}{\al}{\rho}{\GA,C\Rightarrow\Delta}$ and $\provx{\CH}{\beta}{\rho}{\Xi\Rightarrow C}$ then
\begin{equation*}
\provx{\CH}{\al\#\al\#\beta\#\beta}{\rho}{\GA,\Xi\Rightarrow\Delta}
\end{equation*}
}\end{lemma}
\begin{proof}
The proof is by induction on $\al\#\al\#\beta\#\beta$. The interesting case is where $C$ was the principal formula of both final inferences, notice that in this case the last inference cannot have been $\SRP$ since $rk(C)>\Omega$ and the conclusion of an application of $\SRP$ always has rank $\OO$. Thus the rest of the proof follows in the usual way by the symmetry of the rules and Lemmas \ref{prank} and \ref{pinversion}, we treat the case where $C\equiv(\forall x\subb t)A(x)$ and $C$ was the principal formula of both last inferences, so we have
\begin{align*}
\tag{1}&\provx{\CH}{\al}{\rho}{\GA,C\Rightarrow\Delta}&\\
\tag{2}&\provx{\CH}{\beta}{\rho}{\Xi\RI C}&\\
\tag{3}&\provx{\CH}{\al_0}{\rho}{\GA,C,s\subb t\rightarrow A(s)\RI\DE}&\text{with $\al_0,\lev s<\al$ and $\lev s\leq\lev t$.}\\
\tag{4}&\provx{\CH[p]}{\beta_p}{\rho}{\Xi\RI p\subb t\rightarrow A(p)}&\text{for all $\lev p\leq\lev t$ with $\lev p\leq\al_p<\al$.}
\end{align*}
From (3) we know that $s\in\CH$, so from (4) we get
\begin{equation*}
\tag{5}\provx{\CH}{\beta_s}{\rho}{\Xi\RI s\subb t\rightarrow A(s)}.
\end{equation*}
Applying the induction hypothesis to (2) and (3) yields
\begin{equation*}
\tag{6}\provx{\CH}{\al_0\#\al_0\#\beta\#\beta}{\rho}{\GA,s\subb t\rightarrow A(s)\RI\DE}.
\end{equation*}
Finally by applying $\Cut$ to (5) and (6), whilst noting that by Lemma \ref{prank} $rk(s\subb t\rightarrow A(s))<\rho$, we obtain
\begin{equation*}
\provx{\CH}{\al\#\al\#\beta\#\beta}{\rho}{\GA,\Xi\RI\DE}
\end{equation*}
as required.
\end{proof}
\begin{lemma}\label{pprecesetup}{\em If $\provx{\CH}{\al}{\OO+n+1}{\GA\RI\DE}$ then $\provx{\CH}{\omega^\al}{\OO+n}{\GA\RI\DE}$ for any $n<\omega$.
}\end{lemma}
\begin{proof}
The proof is by induction on $\al$, suppose $\provx{\CH}{\al}{\OO+n+1}{\GA\RI\DE}$. If $\GA\RI\DE$ is an axiom there is nothing to show. If $\GA\RI\DE$ was the result of an inference other that $\Cut$ or a cut with cut-rank $<\OO+n$ then we may apply the induction hypothesis to the premises of that inference and then the same inference again. So suppose the last inference was $\Cut$ with cut-formula $C$, and that $rk(C)=\OO+n$. So we have
\begin{align*}
\tag{1}&\provx{\CH}{\al_0}{\OO+n+1}{\GA,C\RI\DE}&\text{with $\al_0<\al$.}\\
\tag{2}&\provx{\CH}{\al_1}{\OO+n+1}{\GA\RI C}&\text{with $\al_1<\al$.}
\end{align*}
Applying the induction hypothesis to (1) and (2) gives
\begin{align*}
\tag{3}&\provx{\CH}{\omega^{\al_0}}{\OO+n}{\GA,C\RI\DE}&\\
\tag{4}&\provx{\CH}{\omega^{\al_1}}{\OO+n}{\GA\RI C}.&
\end{align*}
Now applying the Reduction Lemma \ref{preduction} to (3) and (4) provides us with
\begin{equation*}
\provx{\CH}{\omega^{\al_0}\#\omega^{\al_0}\#\omega^{\al_1}\#\omega^{\al_1}}{\OO+n}.
\end{equation*}
It remains to note that $\omega^{\al_0}\#\omega^{\al_0}\#\omega^{\al_1}\#\omega^{\al_1}<\omega^\al$ since $\omega^\al$ is additive principal, so we can complete the proof by weakening.\end{proof}
\begin{theorem}[Partial cut elimination for $\IRSOP$]\label{ppredce}{\em If $\provx{\CH}{\alpha}{\OO+n+1}{\Gamma\Rightarrow\Delta}$ then $\provx{\CH}{\omega_n(\alpha)}{\Omega+1}{\Gamma\Rightarrow\Delta}$ where $\omega_0(\beta):=\beta$ and $\omega_{k+1}(\beta):=\omega^{\omega_k(\beta)}$.
}\end{theorem}
\begin{proof}
The proof uses an easy induction on $n$ and the previous Lemma.
\end{proof}
\noindent Note that \ref{ppredce} is much weaker than the full predicative cut elimination result we obtained for $\IRS$ (Theorem \ref{predce}), this is because in general we cannot eliminate cuts with $\Deltaop$ cut-formulae from $\IRSOP$ derivations.
\begin{lemma}[Boundedness]\label{pboundedness}{\em
If $A$ is a $\Sigmap$-formula, $B$ is a $\Pip$-formula, $\alpha\leq\beta<\Omega$ and $\beta\in\mathcal{H}$ then
\begin{description}
\item[i)] If $\provx{\mathcal{H}}{\alpha}{\rho}{\Gamma\Rightarrow A}$ then $\provx{\mathcal{H}}{\alpha}{\rho}{\Gamma\Rightarrow A^{\mathbb{V}_\beta}}$.
\item[ii)] If $\provx{\mathcal{H}}{\alpha}{\rho}{\Gamma,B\Rightarrow\Delta}$ then $\provx{\mathcal{H}}{\alpha}{\rho}{\Gamma,B^{\mathbb{V}_\beta}\Rightarrow\Delta\,.}$
\end{description}
}\end{lemma}
\begin{proof}
The proofs are by induction on $\al$, we show ii), the proof of i) is similar. As with Lemma \ref{boundedness} the only interesting case is where $B$ was the principal formula of the last inference and $B$ is of the form $\forall xC(x)$. So we have
\begin{equation*}
\provx{\CH}{\alpha_0}{\rho}{\Gamma, B, C(s)\Rightarrow\Delta}\quad\text{for some $|s|<\al$ with $\alpha_0+1<\al$.}
\end{equation*}
Using the induction hypothesis we obtain
\begin{equation*}
\provx{\CH}{\alpha_0}{\rho}{\Gamma, B^{\mathbb{V}_\beta}, C(s)\Rightarrow\Delta}\quad\text{for some $|s|<\al$ with $\alpha_0+1<\al$.}
\end{equation*}
Now since $\Gamma,B^{\mathbb{V}_\beta}\Rightarrow s\in\mathbb{V}_\beta$ is an axiom, we have $\provx{\CH}{\alpha_0}{\rho}{\Gamma, B^{\mathbb{V}_\beta}\Rightarrow s\in\mathbb{V}_\beta}$, so by $(\rightarrow\!L)$ we obtain
\begin{equation*}
\provx{\CH}{\alpha_0+1}{\rho}{\Gamma, B^{\mathbb{V}_\beta}, s\in\mathbb{V}_\beta\rightarrow C(s)\Rightarrow\Delta}\quad\text{for some $|s|<\al$ with $\alpha_0+1<\al$.}
\end{equation*}
Finally an application of $(b\forall L)$ yields
\begin{equation*}
\provx{\mathcal{H}}{\alpha}{\rho}{\Gamma,B^{\mathbb{V}_\beta}\Rightarrow\Delta}
\end{equation*}
as required.
\end{proof}
\begin{theorem}[Collapsing]\label{pcollapsing}{\em Suppose that $\eta\in\CH_\eta$, $\Delta$ is a set of at most one $\Sigmap$-formula and $\Gamma$ a set of $\Pip$-formulae.
 %with $\text{max}\{rk(A)\:|\:A\in\Gamma\}\leq\Omega$
  Then
\begin{equation*}
\provx{\mathcal{H}_\eta}{\alpha}{\Omega+1}{\Gamma\Rightarrow\Delta}\quad\text{implies}\quad\provx{\mathcal{H}_{\hat{\alpha}}}{\psi_\Omega (\hat{\alpha})}{\psi_\Omega (\hat{\alpha})}{\Gamma\Rightarrow\Delta.}
\end{equation*}
Here $\hat\beta=\eta+\omega^{\Omega+\beta}$ and the operators $\CH_\xi$ are those defined in Definition \ref{opdef}.
}\end{theorem}
\begin{proof}Note first that from $\eta\in\CH_\eta$ and Lemma \ref{op2} we obtain
\begin{equation*}\tag{1}
\hat\al,\psio{\hat\al}\in\CH_{\hat\al}.
\end{equation*}
The proof is by induction on $\al$.\\

\noindent Case 0. If $\Gamma\Rightarrow\Delta$ is an axiom then the result follows immediately from (1).\\

\noindent Case 1. If the last inference was propositional then the assertion follows easily by applying the induction hypothesis and then the same inference again.\\

\noindent Case 2. Suppose the last inference was $(pb\forall R)_\infty$, then $\Delta=\{(\forall x\subseteq t)F(x)\}$ and
\begin{equation*}
\provx{\CH_\eta[p]}{\al_p}{\Omega+1}{\Gamma\Rightarrow p\subseteq t\rightarrow F(p)}\quad\text{for all $\lev p\leq\lev t$ with $\alpha_p<\alpha$.}
\end{equation*}
Since $\lev t \in\CH_\eta(\emptyset)=B^\OO(\eta+1)$ and $\lev t<\Omega$, we have $\lev t<\psio{\eta+1}$, thus $\lev p\in\CH_\eta$ for all $\lev p \leq\lev t$. So we have
\begin{equation*}
\provx{\CH_\eta}{\al_p}{\Omega+1}{\Gamma\Rightarrow p\subseteq t\rightarrow F(p)}.
\end{equation*}
Since $p\subseteq t\rightarrow F(p)$ is also in $\Sigmap$
 we may apply the induction hypothesis to obtain
\begin{equation*}
\provx{\CH_{\hat{\al_p}}}{\psio{\hat{\al_p}}}{\psio{\hat{\al_p}}}{\Gamma\Rightarrow p\subseteq t\to F(p)}\quad\text{for all $\lev p\leq\lev t$ with $\alpha_p<\alpha$.}
\end{equation*}
Now noting that $\psio{\hat{\al_p}}+1<\psio{\hat{\al}}$, by applying $(pb\forall R)_\infty$ we obtain the desired result. The cases where the last inference was $(b\forall R)_\infty$, $(pb\exists L)_\infty$, $(b\exists L)_\infty$, $(\in\! L)_\infty$ or $(\subseteq\! L)_\infty$ are similar. \\

\noindent Case 3. Now suppose the last inference was $(pb\forall L)$, so $(\forall x\subseteq t)F(x)\in\Gamma$ and
\begin{equation*}
\provx{\CH_\eta}{\al_0}{\OO+1}{\Gamma,s \subseteq t\rightarrow F(s)\Rightarrow\Delta}\quad\text{for some $\lev s\leq\lev t$ with $\al_0<\al$.}
\end{equation*}
 Noting that $s\subseteq t\rightarrow F(s)$ is in $\Pip$, too,  we may apply the induction hypothesis to obtain
\begin{equation*}
\provx{\CH_{\hat{\alpha_0}}}{\psio{\hat{\al_0}}}{\psio{\hat{\al_0}}}{\Gamma, s \subseteq t\rightarrow F(s)\Rightarrow\Delta}
\end{equation*}
to which we may apply $(pb\forall L)$ to complete this case. The cases where the last inference was $(b\forall L)$, $(pb\exists R)$, $(b\exists R)$, $(\in\! R)$ or $(\subseteq\! R)$ are similar.\\

\noindent Case 4. Now suppose the last inference was $(\forall L)$, so $\forall xA(x)\in\GA$ and
\begin{equation*}
\provx{\CH_\eta}{\al_0}{\Omega+1}{\GA,F(s)\Rightarrow\Delta}\quad\text{for some $\lev s<\al$ and $\al_0<\al$.}
\end{equation*}
Since $F(s)$ is $\Pip$ we may apply the induction hypothesis to obtain
\begin{equation*}
\provx{\CH_{\hat{\al_0}}}{\psio{\hat{\al_0}}}{\psio{\hat{\al_0}}}{\Gamma,F(s)\Rightarrow\Delta\,.}
\end{equation*}
Now since $\lev s\in\CH_\eta=B^\OO(\eta+1)$ we have $\lev s<\psio{\eta+1}<\psio{\hat\al}$. So we may apply $(\forall L)$ to complete the case. The case where the last inference was $(\exists R)$ is similar.\\

\noindent The rest of the proof is completely analogous to that of Theorem \ref{collapsing}, using boundedness for $\IRSOP$ (\ref{pboundedness}) instead of for $\IRS$.
\end{proof}
\subsection{Embedding $\IKPP$ into $\IRSOP$}
\begin{definition}
{\em As in the embedding section for the case of \textbf{IKP}, $\Vdash\Gamma\Rightarrow\Delta$ will be used to abbreviate that
\begin{equation*}
\provx{\mathcal{H}[\Gamma\Rightarrow\Delta]}{no(\Gamma\Rightarrow\Delta)}{0}{\Gamma\Rightarrow\Delta}\quad\text{holds for any operator $\mathcal{H}$.}
\end{equation*}
Also $\Vdash^\xi_\rho\Gamma\Rightarrow\Delta$ will be used to abbreviate that
\begin{equation*}
\provx{\mathcal{H}[\Gamma\Rightarrow\Delta]}{no(\Gamma\Rightarrow\Delta)\#\xi}{\rho}{\Gamma\Rightarrow\Delta}\quad\text{holds for any operator $\mathcal{H}$.}
\end{equation*}
Only this time we are referring to operator controlled derivability in $\IRSOP$.
}\end{definition}
\begin{lemma}\label{plog}{\em
For any formula $A$
\begin{equation*}
\Vdash A\Rightarrow A\,.
\end{equation*}
}\end{lemma}
\begin{proof}
We proceed by induction on the complexity of $A$. If $A$ is $\Deltaop$ then this is axiom (A1) of $\IRSOP$.\\

\noindent Suppose $A$ is of the form $\exists xF(x)$. Let $\alpha_s=\lev s + no(F(s)\Rightarrow F(s))$ and $\alpha= no(A\Rightarrow A)$, note that $\lev s<\alpha_s+1<\alpha_s+2<\alpha$ for all $s$. By the induction hypothesis we have
\begin{equation*}
\provx{\CH[F(s),s]}{\alpha_s}{0}{F(s)\Rightarrow F(s)}\quad\text{for all terms $s$ and for an arbitrary operator $\CH$.}
\end{equation*}
Now using $(\exists R)$ we get
\begin{equation*}
\provx{\CH[F(s),s]}{\alpha_s+1}{0}{F(s)\Rightarrow\exists xF(x)\,.}
\end{equation*}
Finally since $\CH[F(s),s](\emptyset)\subseteq\CH[\exists xF(x)][s](\emptyset)$ we may apply $(\exists L)_\infty$ to obtain the desired result. The other cases are similar.
\end{proof}
\begin{lemma}[Extensionality]\label{pextensionality}{\em
For any formula $A$ and any terms $s_1,...,s_n,t_1,...,t_n$
\begin{equation*}
\Vdash s_1=t_1,...,s_n=t_n,A(s_1,...,s_n)\Rightarrow A(t_1,...,t_n).
\end{equation*}
}\end{lemma}
\begin{proof}
If $A$ is $\Deltaop$ then this is an axiom. The remainder of the proof is by induction on $rk(A(s_1,...,s_n))$, note that $rk(A(s_1,...,s_n)=rk(A(t_1,...,t_n)$ since $A$ is not $\Deltaop$.\\

\noindent Case 1. Suppose $A(s_1,...,s_n)\equiv\exists xB(x,s_1,...,s_n)$, we know that $rk(B(r,s_1,...,s_n))<rk(A(s_1,...,s_n))$ for all $r$ by Lemma \ref{prank}, so by induction hypothesis we have
\begin{equation*}
\Vdash s_1=t_1,...,s_n=t_n,B(r,s_1,...,s_n)\Rightarrow B(r,t_1,...,t_n)\quad\text{for all terms $r$.}
\end{equation*}
Now successively applying $(\exists R)$ and then $(\exists L)_\infty$ yields the desired result.\\

\noindent Case 2. Now suppose $A(s_1,...,s_n)\equiv(\exists x\subb s_i)B(x,s_1,...,s_n)$. Since $A$ is not $\Deltaop$, $B$ must contain an unbounded quantifier, and thus by Lemma \ref{prank} $\Omega\leq rk(r\subb s_i\wedge B(r,s_1,...,s_n))<rk(A(s_1,...,s_n))$ for any $\lev r\leq \lev{s_i}$, thus by induction hypothesis we have
\begin{equation*}
\Vdash s_1=t_1,...,s_n=t_n,r\subb s_i\wedge B(r,s_1,...,s_n)\Rightarrow r\subb t_i\wedge B(r,t_1,...,t_n)\quad\text{for all $\lev r\leq\lev{s_i}$.}
\end{equation*}
Thus successively applying $(pb\exists R)$ and then $(pb\exists L)_\infty$ yields the desired result. The other cases are similar.
\end{proof}
\begin{lemma}[$\Deltaop$-Collection]\label{pdeltaocoll}{\em
For any $\Deltaop$ formula $F$
\begin{equation*}
\Vdash\Rightarrow (\forall x\in s)\exists y F(x,y)\rightarrow\exists z(\forall x\in s)(\exists y\in z)F(x,y).
\end{equation*}
}\end{lemma}
\begin{proof}
Lemma \ref{plog} provides us with
\begin{equation*}
\Vdash(\forall x\in s)\exists yF(x,y)\Rightarrow(\forall x\in s)\exists yF(x,y)\,.
\end{equation*}
Noting that $(\forall x\in s)\exists yF(x,y)$ is a $\Sigmap$ formula and that $rk((\forall x\in s)\exists yF(x,y))=\omega^{\Omega+2}$ we may apply $\SRP$ to obtain
\begin{equation*}
\bar{\CH}\;\prov{\omega^{\Omega+2}\cdot 2+2}{0}{(\forall x\in s)\exists y F(x,y)\Rightarrow\exists z(\forall x\in s)(\exists y\in z)F(x,y)}
\end{equation*}
where $\bar{\CH}=\CH[(\forall x\in s)\exists y F(x,y)]$ and $\CH$ is an arbitrary operator. Now applying $(\rightarrow R)$ we get
\begin{equation*}
\bar{\CH}\;\prov{\omega^{\Omega+2}\cdot 2+3}{0}{\Rightarrow(\forall x\in s)\exists y F(x,y)\rightarrow\exists z(\forall x\in s)(\exists y\in z)F(x,y).}
\end{equation*}
It remains to note that $\omega^{\Omega+2}\cdot 2+3<\omega^{\Omega+3}=no((\Rightarrow\forall x\in s)\exists y F(x,y)\rightarrow\exists z(\forall x\in s)(\exists y\in z)F(x,y))$ to see that the result is verified.
\end{proof}
\begin{lemma}[Set Induction]\label{pSet Induction}{\em
For any formula $F$
\begin{equation*}
\Vdash\Rightarrow\forall x[(\forall y\in x)F(y)\rightarrow F(x)]\rightarrow\forall xF(x).
\end{equation*}
}\end{lemma}
\begin{proof}
Let $\CH$ be an arbitrary operator and let $A:=\forall x[(\forall y\in x)F(y)\rightarrow F(x)]$. First we prove the following
\begin{equation*}
\text{Claim:}\quad\CH[A,s]\;\prov{\omega^{rk(A)}\#\omega^{\lev s +1}}{0}{A\Rightarrow F(s)}\quad\text{for all terms $s$.}
\end{equation*}
The claim is proved by induction on $\lev s$. By the induction hypothesis we have
\begin{equation*}
\CH[A,t]\;\prov{\omega^{rk(A)}\#\omega^{\lev t +1}}{0}{A\Rightarrow F(t)}\quad\text{for all $\lev t<\lev s$.}
\end{equation*}
Using weakening and then $(\rightarrow R)$ we get
\begin{equation*}
\CH[A,s,t]\;\prov{\omega^{rk(A)}\#\omega^{\lev t +1}+1}{0}{A\Rightarrow t\in s\rightarrow F(t)}\quad\text{for all $\lev t<\lev s$.}
\end{equation*}
Hence by $(b\forall R)_\infty$ we get
\begin{equation*}
\CH[A,s]\;\prov{\omega^{rk(A)}\#\omega^{\lev s}+2}{0}{A\Rightarrow (\forall x\in s)F(x)}
\end{equation*}
(the extra $+2$ is needed when $\lev s$ is not a limit). Now let $\eta_s:=\omega^{rk(A)}\#\omega^{\lev s}+2$. By Lemma \ref{plog} we have $\CH[A,s]\prov{\eta_s}{0}{F(s)\Rightarrow F(s)}$, so by $(\rightarrow L)$ we get
\begin{equation*}
\CH[A,s]\;\prov{\eta_s +1}{0}{A,(\forall y\in s)F(y)\rightarrow F(s)\Rightarrow F(s)}.
\end{equation*}
Finally by applying $(\forall L)$ we get
\begin{equation*}
\CH[A,s]\;\prov{\eta_s+3}{0}{A\Rightarrow F(s)},
\end{equation*}
since $\eta_s+3<\omega^{rk(A)}\#\omega^{\lev s +1}$ the claim is verified. Now by applying $(\forall R)_\infty$ we deduce from the claim that
\begin{equation*}
\CH[A]\;\prov{\omega^{rk(A)}+\Omega}{0}{A\Rightarrow \forall xF(x)}.
\end{equation*}
Hence by $(\rightarrow R)$ we obtain the desired result.
\end{proof}
\begin{lemma}[Infinity]\label{pinfinity}{\em
For any operator $\CH$ we have
\begin{equation*}
\provx{\CH}{\omega+2}{0}{\Rightarrow \exists x[(\exists y\in x)(y\in x)\wedge(\forall y\in x)(\exists z\in x)(y\in z)]}
\end{equation*}
}\end{lemma}
\begin{proof}
First note that for any $\lev s<\alpha$ we have $\CH\;\prov{0}{0}{s\in\mathbb{V}_\alpha}$ by virtue of axiom (A4). Let $\lev s = n<\omega$, we have the following derivation in $\IRSOP$:
\begin{prooftree}
\Axiom$\fCenter\provx{\CH}{0}{0}{\Rightarrow \mathbb{V}_{n+1}\in\mathbb{V}_\omega}$
\Axiom$\fCenter\provx{\CH}{0}{0}{\Rightarrow s\in\mathbb{V}_{n+1}}$
\LeftLabel{$(\wedge R)$}
\BinaryInf$\fCenter\provx{\CH}{1}{0}{\Rightarrow \mathbb{V}_{n+1}\in\mathbb{V}_\omega\wedge s\in\mathbb{V}_{n+1}}$
\LeftLabel{$(b\exists R)$}
\UnaryInf$\fCenter\provx{\CH}{n+2}{0}{\Rightarrow(\exists z\in\mathbb{V}_\omega)(s\in z)}$
\LeftLabel{$(\rightarrow R)$}
\UnaryInf$\fCenter\provx{\CH}{n+3}{0}{\Rightarrow s\in\mathbb{V}_\omega\rightarrow (\exists z\in\mathbb{V}_\omega)(s\in z)}$
\LeftLabel{$(b\forall R)_\infty$}
\UnaryInf$\fCenter\provx{\CH}{\omega}{0}{\Rightarrow (\forall y\in\mathbb{V}_\omega)(\exists z\in\mathbb{V}_\omega)(y\in z)}$
\Axiom$\fCenter\provx{\CH}{0}{0}{\Rightarrow \mathbb{V}_0\in\mathbb{V}_{\omega}}$
\LeftLabel{$(\wedge R)$}
\UnaryInf$\fCenter\provx{\CH}{1}{0}{\Rightarrow \mathbb{V}_0\in\mathbb{V}_{\omega}\wedge\mathbb{V}_0\in\mathbb{V}_{\omega}}$
\LeftLabel{$(b\exists R)$}
\UnaryInf$\fCenter\provx{\CH}{2}{0}{\Rightarrow (\exists z\in\mathbb{V}_{\omega})(z\in\mathbb{V}_{\omega})}$
\LeftLabel{$(\wedge R)$}
\BinaryInf$\fCenter\provx{\CH}{\omega+1}{0}{\Rightarrow (\forall y\in\mathbb{V}_\omega)(\exists z\in\mathbb{V}_\omega)(y\in z)\wedge (\exists z\in\mathbb{V}_{\omega})(z\in\mathbb{V}_{\omega})}$
\LeftLabel{$(\exists R)$}
\UnaryInf$\fCenter\provx{\CH}{\omega+2}{0}{\Rightarrow \exists x[(\forall y\in x)(\exists z\in x)(y\in z)\wedge(\exists z\in x)(z\in x)]}$
\end{prooftree}
\end{proof}
\begin{lemma}[$\Deltaop$-Separation]\label{pseparation}{\em
If $A(a,b,c_1,...,c_n)$ is a $\Deltaop$-formula of $\IKPP$ with all free variables indicated, $r,\bar{s}:=s_1,...,s_n$ are $\IRSOP$ terms and $\CH$ is an arbitrary operator then:
\begin{equation*}
\CH[r,\bar s]\;\prov{\alpha+7}{\rho}{\RI\exists y[(\forall x\in y)(x\in r\wedge A(x,r,\bar{s}))\wedge(\forall x\in r)(A(x,r,\bar{s})\rightarrow x\in y)]}
\end{equation*}
where $\alpha:=\lev r$ and $\rho:=\text{max}\{\lev r, \lev{s_1},...,\lev{s_n}\}+\omega$.
}\end{lemma}
\begin{proof}
First we define
\begin{equation*}
p:=[x\in\mathbb{V}_\alpha\;|\;x\in r\wedge A(x,r,\bar{s})]\quad\text{and}\quad\bar{\CH}:=\CH[r,\bar{s}].
\end{equation*}
For $t$ any term with $\lev t<\alpha$ the following are derivations in $\IRSOP$, first we have:
\begin{prooftree}
\Axiom$\fCenter\quad\text{Axiom (A1)}$
\UnaryInf$\fCenter\bar{\CH}\;\prov{0}{0}{t\in r\Rightarrow t\in r}$
\Axiom$\fCenter\quad\text{Axiom (A1)}$
\UnaryInf$\fCenter\bar{\CH}\;\prov{0}{0}{A(t,r,\bar{s})\Rightarrow A(t,r,\bar{s})}$
\LeftLabel{$(\wedge R)$}
\BinaryInf$\fCenter\bar{\CH}\;\prov{1}{0}{t\in r,A(t,r,\bar{s})\Rightarrow t\in r\wedge A(t,r,\bar{s})}$
\Axiom$\fCenter\quad\quad\text{Axiom (A7)}$
\UnaryInf$\fCenter\bar{\CH}\;\prov{0}{0}{t\in r\wedge A(t,r,\bar{s})\Rightarrow t\in p}$
\LeftLabel{$(\text{cut})$}
\BinaryInf$\fCenter\bar{\CH}\;\prov{2}{\rho}{t\in r,A(t,r,\bar{s})\Rightarrow t\in p}$
\LeftLabel{$(\rightarrow R)$}
\UnaryInf$\fCenter\bar{\CH}\;\prov{3}{\rho}{t\in r\Rightarrow A(t,r,\bar{s})\rightarrow t\in p}$
\LeftLabel{$(\rightarrow R)$}
\UnaryInf$\fCenter\bar{\CH}\;\prov{4}{\rho}{\Rightarrow t\in r\rightarrow (A(t,r,\bar{s})\rightarrow t\in p)}$
\LeftLabel{$(b\forall R)_\infty$}
\UnaryInf$\fCenter\bar{\CH}\;\prov{\alpha+5}{\rho}{\Rightarrow (\forall x\in r) (A(x,r,\bar{s})\rightarrow x\in p)}$
\end{prooftree}
Next we have:
\begin{prooftree}
\Axiom$\fCenter\quad\quad\text{Axiom (A6)}$
\UnaryInf$\fCenter\bar{\CH}\;\prov{0}{0}{t\in p\Rightarrow t\in r\wedge A(t,r,\bar{s})}$
\LeftLabel{$(\rightarrow R)$}
\UnaryInf$\fCenter\bar{\CH}\;\prov{1}{0}{\Rightarrow t\in p\rightarrow t\in r\wedge A(t,r,\bar{s})}$
\LeftLabel{$(b\forall R)_\infty$}
\UnaryInf$\fCenter\bar{\CH}\;\prov{\alpha+2}{0}{\Rightarrow (\forall x\in p)(x\in r\wedge A(x,r,\bar{s}))}$
\end{prooftree}
Now by applying $(\wedge R)$ followed by $(\exists R)$ to the conclusions of these two derivations we get
\begin{equation*}
\bar{\CH}\;\prov{\alpha+7}{\rho}{\RI\exists y[(\forall x\in y)(x\in r\wedge A(x,r,\bar{s}))\wedge(\forall x\in r)(A(x,r,\bar{s})\rightarrow x\in y)]}
\end{equation*}
as required.
\end{proof}
\begin{lemma}[Pair]\label{ppair}{\em
For any operator $\CH$ and any terms $s$ and $t$ we have
\begin{equation*}
\CH[s,t]\;\prov{\alpha+2}{0}{\Rightarrow\exists z(s\in z\wedge t\in z)}
\end{equation*}
Where $\alpha:=\text{max}(\lev s,\lev t)+1$.
}\end{lemma}
\begin{proof}
The following is a derivation in $\IRSOP$:
\begin{prooftree}
\Axiom$\fCenter\quad\text{Axiom (A4)}$
\UnaryInf$\fCenter\CH[s,t]\;\prov{0}{0}{\Rightarrow s\in\mathbb{V}_\alpha}$
\Axiom$\fCenter\quad\text{Axiom (A4)}$
\UnaryInf$\fCenter\CH[s,t]\;\prov{0}{0}{\Rightarrow t\in\mathbb{V}_\alpha}$
\LeftLabel{$(\wedge R)$}
\BinaryInf$\fCenter\CH[s,t]\;\prov{1}{0}{\Rightarrow s\in\mathbb{V}_\alpha\wedge  t\in\mathbb{V}_\alpha}$
\LeftLabel{$(\exists R)$}
\UnaryInf$\fCenter\CH[s,t]\;\prov{\alpha+2}{0}{\Rightarrow\exists z(s\in z\wedge t\in z)}$
\end{prooftree}
\end{proof}
\begin{lemma}[Union]\label{punion}{\em
For any operator $\CH$ and any term $s$ we have
\begin{equation*}
\CH[s]\;\prov{\beta+5}{0}{\Rightarrow\exists z(\forall y\in s)(\forall x\in y)(x\in z)}
\end{equation*}
where $\beta=\lev s$.
}\end{lemma}
\begin{proof}
Let $r$ and $t$ be terms such that $\lev r<\lev t<\beta$, we have the following derivation in $\IRSOP$:
\begin{prooftree}
\Axiom$\fCenter\quad\text{Axiom (A4)}$
\UnaryInf$\CH[s,t,r]\;\fCenter\prov{0}{0}{t\in s,r\in t\Rightarrow r\in\mathbb{V}_\beta}$
\LeftLabel{$(\rightarrow R)$}
\UnaryInf$\CH[s,t,r]\;\fCenter\prov{1}{0}{t\in s\Rightarrow r\in t\rightarrow r\in\mathbb{V}_\beta}$
\LeftLabel{$(b\forall R)_\infty$}
\UnaryInf$\CH[s,t]\;\fCenter\prov{\beta+2}{0}{t\in s\Rightarrow (\forall x\in t)(x\in\mathbb{V}_\beta)}$
\LeftLabel{$(\rightarrow R)$}
\UnaryInf$\CH[s,t]\;\fCenter\prov{\beta+3}{0}{\Rightarrow t\in s\rightarrow (\forall x\in t)(x\in\mathbb{V}_\beta)}$
\LeftLabel{$(b\forall R)_\infty$}
\UnaryInf$\CH[s]\;\fCenter\prov{\beta+4}{0}{\Rightarrow (\forall y\in s)(\forall x\in y)(x\in\mathbb{V}_\beta)}$
\LeftLabel{$(\exists R)$}
\UnaryInf$\CH[s]\;\fCenter\prov{\beta+5}{0}{\Rightarrow\exists z(\forall y\in s)(\forall x\in y)(x\in z)}$
\end{prooftree}
\end{proof}
\begin{lemma}[Powerset]\label{ppowerset}{\em
For any operator $\CH$ and any term $s$ we have
\begin{equation*}
\CH[s]\;\prov{\alpha+3}{0}{\Rightarrow\exists z(\forall x\subb s)(x\in z)}
\end{equation*}
where $\alpha=\lev s$.
}\end{lemma}
\begin{proof}
Let $t$ be any term with $\lev t<\alpha$, we have the following derivation in $\IRSOP$:
\begin{prooftree}
\Axiom$\fCenter\quad\text{Axiom (A4)}$
\UnaryInf$\CH[s,t]\;\fCenter\prov{0}{0}{t\subb s\Rightarrow t\in\mathbb{V}_{\alpha+1}}$
\LeftLabel{$(\rightarrow R)$}
\UnaryInf$\CH[s,t]\;\fCenter\prov{1}{0}{\Rightarrow t\subb s\rightarrow t\in\mathbb{V}_{\alpha+1}}$
\LeftLabel{$(pb\forall R)_\infty$}
\UnaryInf$\CH[s]\;\fCenter\prov{\alpha+2}{0}{\Rightarrow (\forall x\subb s)(x\in\mathbb{V}_{\alpha+1})}$
\LeftLabel{$(\exists R)$}
\UnaryInf$\CH[s]\;\fCenter\prov{\alpha+3}{0}{\Rightarrow \exists z(\forall x\subb s)(x\in z)}$
\end{prooftree}
\end{proof}
\begin{theorem}\label{IKPPembed}{\em
If $\IKPP\vdash\Gamma(\bar{a})\Rightarrow\Delta(\bar a)$ where $\Gamma(\bar{a})\Rightarrow\Delta(\bar a)$ is an intuitionistic sequent containing exactly the free variables $\bar{a}=a_1,...,a_n$, then there exists an $m<\omega$ (which we may calculate from the derivation) such that
\begin{equation*}
\CH[\bar{s}]\;\prov{\Omega\cdot\omega^m}{\Omega+m}{\Gamma(\bar{s})\Rightarrow\Delta(\bar s)}
\end{equation*}
for any operator $\CH$ and any $\IRSOP$ terms $\bar s=s_1,...,s_n$.
}\end{theorem}
\begin{proof}
Note that the rank of $\IRSOP$ formulas is always $<\Omega+\omega$ and thus the norm of $\IRSOP$ sequents is always $<\omega^{\Omega+\omega}=\Omega\cdot\omega^\omega$. The proof is by induction on the $\IKPP$ derivation. If $\Gamma(\bar{a})\Rightarrow\Delta(\bar b)$ is an axiom of $\IKPP$ then the result follows by one of Lemmas \ref{plog}, \ref{pextensionality} \ref{pdeltaocoll}, \ref{pSet Induction}, \ref{pinfinity}, \ref{pseparation}, \ref{ppair}, \ref{punion} and \ref{ppowerset}. Let $\bar\CH:=\CH[\bar{s}]$.\\

\noindent Case 1. Suppose the last inference of the $\IKPP$ derivation was $(pb\exists L)$ then $(\exists x\subb a_i)F(x)\in\Gamma(\bar{a})$ and from the induction hypothesis we obtain a $k$ such that
\begin{equation*}
\bar\CH[p]\;\prov{\Omega\cdot\omega^k}{\Omega+k}{\Gamma(\bar{s}),p\subb s_i\wedge F(p)\Rightarrow\Delta(\bar s)}
\end{equation*}
for all $\lev p\leq\lev{s_i}$ (using weakening if necessary). Thus we may apply $(pb\exists L)_\infty$ to obtain the desired result.\\

\noindent Case 2. Now suppose the last inference was $(pb\exists R)$ then $\Delta(\bar a)=\{(\exists x\subb a_i)F(x)\}$ and we are in the following situation in $\IKPP$:
\begin{prooftree}
\Axiom $\fCenter\vdash\Gamma(\bar{a})\Rightarrow c\subb a_i\wedge F(c)$
\LeftLabel{$(pb\exists R)$}
\UnaryInf$\fCenter\vdash\Gamma(\bar{a})\Rightarrow(\exists x\subb a_i)F(x)$
\end{prooftree}
2.1 If $c$ is not a member of $\bar a$ then by the induction hypothesis we have a $k<\omega$ such that
\begin{equation*}
\bar\CH\;\prov{\Omega\cdot\omega^k}{\Omega+k}{\Gamma(\bar{s})\Rightarrow\mathbb{V}_0\subb s_i\wedge F(\mathbb{V}_0)\,.}
\end{equation*}
Hence we can apply $(pb\exists R)$ to complete this case.\\

\noindent 2.2 Now suppose $c$ is a member of $\bar a$ for simplicity let us suppose that $c=a_1$. Inductively we can find a $k<\omega$ such that
\begin{equation*}\tag{1}
\bar\CH\;\prov{\Omega\cdot\omega^k}{\Omega+k}{\Gamma(\bar{s})\Rightarrow s_1\subb s_i\wedge F(s_1)\,.}
\end{equation*}
Next we verify the following
\begin{equation*}\tag{2}
\text{claim:}\quad\Vdash^\omega\Gamma(\bar s), s_1\subb s_i\wedge F(s_1)\Rightarrow(\exists x\subb s_i)F(x).
\end{equation*}
Owing to axiom (A1) we have
\begin{equation*}
\tag{3}\bar\CH[r]\;\prov{0}{0}{r\subb s_i\RI r\subb s_i}\quad\text{for all $\lev r\leq\lev{s_i}$}.
\end{equation*}
Also by Lemma \ref{pextensionality} we have
\begin{equation*}
\tag{4}\Vdash\GA[\sbar],r=s_1,F(s_1)\RI F(r)\quad\text{for all $\lev r\leq\lev{s_i}$}.
\end{equation*}
Now let $\gamma_r=no(\GA[\sbar],r=s_1,F(s_1)\RI F(r))$. Applying $(\wedge R)$ to (3) and (4) provides
\begin{equation*}
\bar\CH[r]\;\prov{\gamma_r+1}{0}{\GA(\sbar),r\subb s_i,r=s_1,F(s_1)\RI r\subb s_i\wedge F(r)}.
\end{equation*}
Using $(pb\exists R)$ we may conclude
\begin{equation*}
\bar\CH[r]\;\prov{\gamma_r+2}{0}{\GA(\sbar),r\subb s_i,r=s_1,F(s_1)\RI (\exists x\subb s_i)F(x)}.
\end{equation*}
Now two applications of $(\wedge L)$ gives us
\begin{equation*}
\bar\CH[r]\;\prov{\gamma_r+4}{0}{\GA(\sbar),r\subb s_i\wedge r=s_1,F(s_1)\RI (\exists x\subb s_i)F(x)}.
\end{equation*}
Now applying $(\subseteq L)_\infty$ provides
\begin{equation*}
\bar\CH\;\prov{\gamma+5}{0}{\GA(\sbar),s_1\subb s_i,F(s_1)\RI (\exists x\subb s_i)F(x)}
\end{equation*}
where $\gamma=\sup_{\lev r\leq\lev{s_i}}\gamma_r$. Finally, by applying $(\wedge L)$ a further two times we can conclude
\begin{equation*}
\bar\CH\;\prov{\gamma+7}{0}{\GA(\sbar),s_1\subb s_i\wedge F(s_1)\RI (\exists x\subb s_i)F(x)}.
\end{equation*}
Via some ordinal arithmetic it can be observed that
\begin{equation*}
\gamma+7\leq no(\GA(\sbar),s_1\subb s_i\wedge F(s_1)\RI (\exists x\subb s_i)F(x))\#\omega,
\end{equation*}
so the claim is verified.\\

\noindent To complete this case we may now apply $\Cut$ to (1) and (2).\\

\noindent All other cases are similar to those above, or may be treated in a similar manner to Theorem \ref{IKPembed}.
\end{proof}

\subsection{A relativised ordinal analysis of $\IKPP$}
\noindent A major difference to the case of $\IKP$ is that we don't immediately have the soundness of cut-reduced $\IRSOP$ derivations of $\Sigmap$-formulae within the appropriate segment of the Von-Neumann Hierarchy. This is partly due to the fact that we don't have a term for each element of the hierarchy (this can be seen from a simple cardinality argument). In fact we do still have soundness for certain derivations within $V_{\psio{\varepsilon_{\OO+1}}}$, which is demonstrated in the next lemma, where we must make essential use of the free variables in $\IRSOP$. First we need the notion of an assignment. Let $VAR_{\mathcal{P}}$ be the set of free variables of $\IRSOP$. A variable assignment is a function
\begin{equation*}
v:VAR_{\mathcal{P}}\longrightarrow V_{\psio{\varepsilon_{\OO+1}}}
\end{equation*}
such that $v(a_i^\al)\in V_{\al+1}$ for each $i$. $v$ canonically lifts to all terms as follows
\begin{align*}
v(\mathbb{V}_\al)\:&=\:V_\al\\
v(\{x\in \mathbb{V}_\al\;|\;F(x, s_1,...,s_n)\})\:&=\:\{x\in V_\al\;|F(x,v(s_1),...,v(s_n))\}.
\end{align*}
Moreover it can be seen that $v(s)\in V_{\lev s+1}$ and thus $v(s)\in V_{\psio{\varepsilon_{\OO+1}}}$ for all terms $s$.
\begin{theorem}[Soundness for $\IRSOP$]\label{psoundness}{\em Suppose $\GA[s_1,...,s_n]$ is a finite set of $\Pip$ formulae with max$\{rk(A)\;|A\in\GA\}\leq\OO$, $\DE[s_1,...,s_n]$ a set containing at most one $\Sigmap$ formula and
\begin{equation*}
\provx{\CH}{\al}{\rho}{\GA[\sbar]\RI\DE[\sbar]}\quad\text{for some operator $\CH$ and some $\al,\rho<\OO$.}
\end{equation*}
Then for any assignment $v$,
\begin{equation*}
V_{\psio{\varepsilon_{\OO+1}}}\models\bigwedge\GA[v(s_1),...,v(s_n)]\rightarrow\bigvee\DE[v(s_1),...,v(s_n)]
\end{equation*}
where $\bigwedge\Gamma$ and $\bigvee\DE$ stand for the conjunction of formulas in $\GA$ and the disjunction of formulas in $\DE$ respectively, by convention $\bigwedge\emptyset = \top$ and $\bigvee\emptyset=\bot$.
}\end{theorem}
\begin{proof}
The proof is by induction on $\al$. Note that the derivation $\provx{\CH}{\al}{\rho}{\GA[\sbar]\RI\DE[\sbar]}$ contains no inferences of the form $(\forall R)_\infty$, $(\exists L)_\infty$ or $\SRP$ and all cuts have $\Deltaop$ cut formulae. All axioms of $\IRSOP$ can be observed to be sound with respect to the interpretation.\\

\noindent First we treat the case where the last inference was $(pb\forall L)$ so we have
\begin{equation*}
\provx{\CH}{\al_0}{\rho}{\GA[\sbar], t\subb s_i\rightarrow F(t,\sbar)\RI\DE[\sbar]}\quad\text{for some $\al_0,\lev t<\al$, with $\lev t\leq\lev{s_i}$.}
\end{equation*}
Since max$\{rk(A)\;|\;A\in\GA\}\leq\OO$, it follows that $t\subb s_i\rightarrow F(t,\sbar)$ is a $\Deltaop$ formula. So we may apply the induction hypothesis to obtain
\begin{equation*}
V_{\psio{\varepsilon_{\OO+1}}}\models\bigwedge\GA[v(\sbar)]\wedge [v(t)\subb v(s_i)\rightarrow F(v(t),v(\sbar))] \rightarrow\bigvee\DE[v(\sbar)],
\end{equation*}
where $v(\sbar):=v(s_1),...,v(s_n)$. From here the desired result follows by regular logical semantics.\\

\noindent Now suppose the last inference was $(pb\forall R)_\infty$, so we have
\begin{equation*}
\tag{1}\provx{\CH}{\al_t}{\rho}{\GA[\sbar]\RI t\subb s_i\rightarrow F(t,\sbar)}\quad\text{for all $\lev t\leq\lev{s_i}$ with $\al_t<\al$.}
\end{equation*}
In particular this means we have
\begin{equation*}
\tag{2}\provx{\CH}{\al_0}{\rho}{\GA[\sbar]\RI a^\beta_j\subb s_i\rightarrow F(a^\beta_j,\sbar)}\quad\text{for some $\al_0<\al$.}
\end{equation*}
Here $\beta:=\lev{s_i}$ and $j$ is chosen such that $a^\beta_j$ does not occur in any of the terms $s_1,...,s_n$. If $F$ contains an unbounded quantifier we may use inversion for $\IRSOP$ \ref{pinversion}v) to obtain
\begin{equation*}
\tag{3}\provx{\CH}{\al_0}{\rho}{\GA[\sbar], a^\beta_j\subb s_i\RI F(a^\beta_j,\sbar)}\quad\text{for some $\al_0<\al$.}
\end{equation*}
So we may apply the induction hypothesis to get
\begin{equation*}
\tag{4}V_{\psio{\varepsilon_{\OO+1}}}\models\bigwedge\GA[v(\sbar)],v(a^\beta_j)\subb v(s_i)\rightarrow F(v(a^\beta_j),v(\sbar))
\end{equation*}
for all variable assignments $v$. Thus by the choice of $a^\beta_j$ we have
\begin{equation*}
\tag{5}V_{\psio{\varepsilon_{\OO+1}}}\models\bigwedge\GA[v(\sbar)]\rightarrow (\forall x\subb v(s_i))F(x,v(\sbar))
\end{equation*}
as required. If $F$ is $\Deltaop$ then we may immediately apply the induction hypothesis to (2) to obtain
\begin{equation*}
\tag{6}V_{\psio{\varepsilon_{\OO+1}}}\models\bigwedge\GA[v(\sbar)]\rightarrow [v(a^\beta_j)\subb v(s_i)\rightarrow F(v(a^\beta_j),v(\sbar))]
\end{equation*}
for all variable assignments $v$, again by the choice of $a^\beta_i$ we obtain the desired result. All other cases may be treated in a similar manner to the two above.
\end{proof}
\begin{lemma}\label{pconc}{\em
Suppose $\IKPP\vdash\;\RI A$ for some $\Sigmap$ sentence $A$, then there is an $m<\omega$, which we may compute from the derivation, such that
\begin{equation*}
\provx{\CH_\sigma}{\psio\sigma}{\psio\sigma}{\RI\;A}\quad\text{where $\sigma:=\omega_m(\OO\cdot\omega^m)$.}
\end{equation*}
}\end{lemma}
\begin{proof}
Suppose $\IKPP\vdash\;\RI A$ for some $\Sigmap$ sentence $A$, then by Theorem \ref{IKPPembed} we can explicitly find some $m<\omega$ such that
\begin{equation*}
\provx{\CH_0}{\OO\cdot\omega^m}{\OO+m}{\RI A}.
\end{equation*}
Applying Partial cut elimination \ref{ppredce} we have
\begin{equation*}
\provx{\CH_0}{\omega_{m-1}(\OO\cdot\omega^m)}{\OO+1}{\RI A}.
\end{equation*}
Now using Collapsing \ref{pcollapsing} we obtain
\begin{equation*}
\provx{\CH_\sigma}{\psio\sigma}{\psio\sigma}{\RI A}\quad\text{where $\sigma:=\omega_m(\OO\cdot\omega^m)$,}
\end{equation*}
completing the proof.
\end{proof}
\noindent Note that we cannot eliminate \textit{all} cuts from the derivation since we don't have full predicative cut elimination for $\IRSOP$ as we do for $\IRS$.\\

\begin{theorem}\label{pconct}{\em If $A$ is a $\Sigmap$-sentence and $\IKPP\vdash\;\RI A$ then there is some ordinal term $\al<\psio{\varepsilon_{\OO+1}}$, which we may compute from the derivation, such that
\begin{equation*}
V_\al\models A.
\end{equation*}
}\end{theorem}
\begin{proof}
From Lemma \ref{pconc} we obtain some $m<\omega$ such that
\begin{equation*}
\tag{1}\provx{\CH_\sigma}{\psio\sigma}{\psio\sigma}{\RI A}\quad\text{where $\sigma:=\omega_m(\OO\cdot\omega^m)$.}
\end{equation*}
Let $\al:=\psio{\sigma}$. Applying Boundedness \ref{pboundedness} to (1) we obtain
\begin{equation*}
\tag{2}\provx{\CH_\sigma}{\al}{\al}{\RI A^{\mathbb{V}_\al}}.
\end{equation*}
Now applying Theorem \ref{psoundness} to (2) we obtain
\begin{equation*}
V_{\psio{\varepsilon_{\OO+1}}}\models A^{V_\al}
\end{equation*}
and thus
\begin{equation*}
V_\al\models A
\end{equation*}
as required.
\end{proof}
\begin{remark}{\em
Suppose $A\equiv \exists xC(x)$ is a $\Sigmap$ sentence and $\IKPP\vdash\;\RI A$. As well as the ordinal term $\al$ given by Theorem \ref{pconct}, it is possible to determine (making essential use of the intuitionistic nature of $\IRSOP$) a term $s$, with $\lev s<\al$, such that
\begin{equation*}
V_\al\models C(s).
\end{equation*}
This proof is somewhat more complex than in the case of $\IKP$ since the proof tree corresponding to (2) above can still contain cuts with $\Deltaop$ cut formulae. \\

\noindent Moreover, in order to show that $\IKPP$ has the existence property, the embedding and cut elimination for a given finite derivation of a $\Sigmap$ sentence, needs to be carried out \textit{inside} $\IKPP$. In order to do this it needs to be shown that from the finite derivation we can calculate some ordinal term $\gamma<\varepsilon_{\OO+1}$ such that the embedding and cut elimination for that derivation can still be performed inside $\IRSOP$ with the term structure restricted to $B(\gamma)$. \\

\noindent These proofs will appear in \cite{rathjen-EP}.
}\end{remark}

Like in the case of $\IKP$ we also arrive at a conservativity result.

\begin{theorem}\label{cons2}{\em $\IKPP+\mbox{$\Sigmap$-Reflection}$ is conservative over $\IKPP$ for $\Sigmap$-sentences.}
\end{theorem}

\section{The case of $\IKPE$}

This final section provides a relativised ordinal analysis for intuitionistic exponentiation Kripke-Platek set theory $\IKPE$. Given sets $a$ and $b$, set-exponentiation allows the formation of the set $^ab$, of all functions from $a$ to $b$. A problem that presents itself in this case is that it is not clear how to formulate a term structure in such a way that we can read off a terms level in the pertinent `exponentiation hierarchy' from that terms syntactic structure. Instead we work with a term structure similar to that used in $\IRSOP$, and a terms level becomes a dynamic property {\em inside} the infinitary system. Making this work in a system for which we can prove all the necessary embedding and cut-elimination theorems turned out to be a major technical hurdle. The end result of the section is a characterisation of $\IKPE$ in terms of provable height of the exponentiation hierarchy, this machinery will also be used in a later paper by Rathjen \cite{rathjen-EP}, to show that $\CZF^\mathcal{E}$ has the full existence property.

\subsection{A sequent calculus formulation of $\IKPE$}
\begin{definition}
{\em The formulas of $\IKPE$ are the same as those of $\IKP$ except we also allow \textit{exponentiation bounded quantifiers} of the form
\begin{equation*}
(\forall x\in\;\!^ab)A(x)\quad\text{and}\quad(\exists x\in\;\!^ab)A(x).
\end{equation*}
These are treated as quantifiers in their own right, not abbreviations. The formula "$\text{fun}(x,a,b)$" is defined below. It's intuitive meaning is "$x$ is a function from $a$ to $b$".
\begin{align*}
\text{fun}(x,a,b):=&\:x\subseteq a\times b\:\wedge\: (\forall y\in a)(\exists z\in b)((y,z)\in x)\\
&\:\wedge(\forall y\in a)(\forall z_1\in b)(\forall z_2\in b)[((y,z_1)\in x\wedge(y,z_2)\in x)\rightarrow z_1=z_2]
\end{align*}
Quantifiers $\forall x$, $\exists x$ will be referred to as unbounded, whereas the other quantifiers (including the exponentiation bounded ones) will be referred to as bounded.\\

\noindent A $\Delta_0^\mathcal{E}$-formula of $\IKPE$ is one that contains no unbounded quantifiers.\\

\noindent As with $\IKP$, the system $\IKPE$ derives intuitionistic sequents of the form $\Gamma\Rightarrow\Delta$ where $\Gamma$ and $\Delta$ are finite sets of formulae and $\Delta$ contains at most one formula.\\

\noindent The axioms of $\IKPE$ are given by:\\

\begin{tabular} {ll}
{\em {Logical axioms:}}  & $\Gamma, A, \Rightarrow A$ \ for every
$\Deltaoe$--formula A.\\
{\em {Extensionality:}}   & $\Gamma\Rightarrow a\!=\!b\wedge B(a)\rightarrow
B(b)$
\ for every $\Deltaoe$-formula $B(a)$.\\
{\em {Pair:}}   & $\Gamma\Rightarrow \exists
x[a\!\in\!x \wedge b\!\in\!x]$\\
{\em {Union:}}   & $\Gamma\Rightarrow \exists x(\forall y\!\in\!a)(\forall
z\!\in\!y)(z\!\in\!x)$ \\
{\em {Infinity:}} &  $\Gamma\Rightarrow \exists x\,[(\exists y\In x)\,y\in x\;\wedge\;(\forall y\in x)(\exists z\In x) \,y\in z]$.\\
$\Deltaoe$ {\em--{Separation:}}   & $\Gamma\Rightarrow \exists x ((\forall y\in x)(y\in a\wedge A(y))\wedge(\forall y\in a)(A(y)\rightarrow y\in x))$\\
& for every $\Deltaoe$ formula $A(b)$. \\
$\Deltaoe$ {\em--{Collection:}}   & $\Gamma\Rightarrow (\forall x\in a) \exists yB(x,y)\rightarrow\exists z(\forall x\in a)(\exists y\in z)B(x,y)$\\
& for every $\Deltaoe$ formula $B(b,c)$. \\
{\em {Set Induction:}}    & $\Gamma\Rightarrow \forall u\,[(\forall x\In u)\,G(x)\,\to\,G(u)]\,\to\,\forall u\,G(u)$\\
& for every formula $G(b)$. \\
{\em {Exponentiation:}} & $\Gamma\Rightarrow \exists z\,(\forall x\in\;\! ^ab)(x\in z)$.
\end{tabular}
\\[0.4cm]
The rules of $\IKPE$ are the same as those of $\IKP$ (extended to the new language containing exponentiation bounded quantifiers), together with the following four rules:

\begin{prooftree}
\Axiom$\fCenter\Gamma, \text{fun}(c,a,b)\wedge F(c)\Rightarrow \Delta$
\LeftLabel{$(\mathcal{E}b\exists L)$}
\UnaryInf$\fCenter\Gamma,(\exists x\in\;\!^ab)F(x)\Rightarrow\Delta$
\Axiom$\fCenter\Gamma\Rightarrow \text{fun}(c,a,b)\wedge F(c)$
\LeftLabel{$(\mathcal{E}b\exists R)$}
\UnaryInf$\fCenter\Gamma\Rightarrow(\exists x\in\;\!^ab)F(x)$
\noLine
\BinaryInf$\fCenter$
\end{prooftree}
\begin{prooftree}
\Axiom$\fCenter\Gamma, \text{fun}(c,a,b)\rightarrow F(c)\Rightarrow \Delta$
\LeftLabel{$(\mathcal{E}b\forall L)$}
\UnaryInf$\fCenter\Gamma,(\forall x\in\;\!^ab)F(x)\Rightarrow\Delta$
\Axiom$\fCenter\Gamma\Rightarrow \text{fun}(c,a,b)\rightarrow F(c)$
\LeftLabel{$(\mathcal{E}b\forall R)$}
\UnaryInf$\fCenter\Gamma\Rightarrow(\forall x\in\;\!^ab)F(x)$
\noLine
\BinaryInf$\fCenter$
\end{prooftree}
As usual it is forbidden for the variable $a$ to occur in the conclusion of the rules $(\mathcal{E}b\exists L)$ and $(\mathcal{E}b\forall R)$, such a variable is referred to as the eigenvariable of the inference.}
\end{definition}

\subsection{The infinitary system $\IRSOE$}
The purpose of this section is to introduce an infinitary system $\IRSOE$ within which we will be able to embed $\IKPE$. As with the von Neumann hierarchy built by iterating the power set operation through the ordinals, one may define an Exponentiation-hierarchy as follows
\begin{align*}
E_0&:=\emptyset\\
E_1&:=\{\emptyset\}\\
E_{\alpha+2}&:=\{X\;|\;X\text{ is definable over $\langle E_{\al+1},\in\rangle$ with parameters}\}\\
&\:\;\cup\{f\;|\; \text{fun}(f,a,b)\text{ for some $a,b\in E_\al$.}\}\\
E_\lambda&:=\bigcup_{\beta<\lambda} E_\beta\quad\text{for $\lambda$ a limit ordinal.}\\
E_{\lambda+1}&:=\{X\;|\;X\text{ is definable over $\langle E_{\al+1},\in\rangle$ with parameters}\}\quad\text{for $\lambda$ a limit ordinal.}
\end{align*}
\begin{lemma}\label{etran}{\em
If $y\in E_{\al+1}$ and $x\in y$ then $x\in E_\al$.
}\end{lemma}
\begin{proof}
The proof is by induction on $\al$. If $y$ is a set definable over $\langle E_\al,\in\rangle$ with parameters, the members of $y$, including $x$, must be members of $E_\al$.\\

\noindent Now suppose $\al=\beta+1$ and $y\in E_{\al+1}$ is a function $y:p\rightarrow q$ for two sets $p,q\in E_\beta$. Since $x\in y$, it follows that $x$ is of the form $(x_0,x_1)$ with $x_0\in p$ and $x_1\in q$, we use the standard definition of ordered pair so
\begin{equation*}
\tag{1}(x_0,x_1):=\{\{x_0,x_1\},\{x_0\}\}
\end{equation*}
We must now verify the following claim:
\begin{equation*}
\tag{*}\{x_0\},\{x_1\},\{x_0,x_1\}\in E_\beta.
\end{equation*}
If $\beta=\gamma+1$ then by the induction hypothesis applied to $x_0\in p\in E_\beta$ and $x_1\in q\in E_\beta$ we get $x_0,x_1\in E_\gamma$ and thus $\{x_0\},\{x_1\},\{x_0,x_1\}\in E_\beta$ as required.\\

\noindent If $\beta$ is a limit then by the induction hypothesis and the construction of the $E$ hierarchy at limit ordinals, we know that $s_0\in E_{\beta_0}$ and $s_1\in E_{\beta_1}$ for some $\beta_0,\beta_1<\beta$, thus $\{s_0\},\{s_1\},\{s_0,s_1\}\in E_{\text{max}(\beta_0,\beta_1)+1}$ which completes the proof of (*).\\

\noindent From (*) and (1) it is clear that $(s_0,s_1)\in E_{\beta+1}$ as required.
\end{proof}
\noindent The idea of $\IRSOE$ is to build an infinitary system for reasoning about the $E$ hierarchy.
\begin{definition}{\em
The terms of $\IRSOE$ are defined as follows
\begin{description}
\item[1.] $\Ea$ is an $\IRSOE$ term for each $\al<\Omega.$
\item[2.] $a^\al_i$ is an $\IRSOE$ term for each $\al<\Omega$ and each $i<\omega$, these terms will be known as free variables.
\item[3.] If $F(a,\bar b)$ is a $\Deltaoe$ formula of $\IKPE$ containing exactly the free variables indicated, and $t, \bar s:= s_1,...,s_n$ are $\IRSOE$ terms then
\begin{equation*}
[x\in t\;|\;F(x,\bar s)]
\end{equation*}
is also a term of $\IRSOE$.
\end{description}
Observe that $\IRSOE$ terms do not come with `levels' as in the other infinitary systems. This is because it is not clear how to immediately read off the location of a given term within the $E$ hierarchy, just from the syntactic information available within that term.\\

\noindent The formulas of $\IRSOE$ are of the form $F(s_1,...,s_n)$, where $F(a_1,...,a_n)$ is a formula of $\IKPE$ with all free variables indicated and $s_1,...,s_n$ are $\IRSOE$ terms. The formula $A(s_1,...,s_n)$ is said to be $\Deltaoe$ if $A(a_1,...,a_n)$ is a $\Deltaoe$ formula of $\IKPE$. The $\Sigma^\mathcal{E}$ formulae are the smallest collection containing the $\Deltaoe$ formulae such that $A\wedge B$, $A\vee B$,
$(\forall x\in t)A$, $(\exists x\in t)A$, $(\exists x\in\;\!^ab)A$, $(\forall x\in\;\!^ab)A$, $\exists xA$,
$\neg C$, and $C\to A$ are in $\Sigma^\mathcal{E}$ whenever $A,B$ are in $\Sigma^\mathcal{E}$ and
  $C$ is in $\Pi^\mathcal{E}$. Dually, the $\Pi^\mathcal{E}$ formulae are the smallest collection containing the $\Deltaoe$ formulae such that $A\wedge B$, $A\vee B$,
$(\forall x\in t)A$, $(\exists x\in t)A$, $(\exists x\in\;\!^ab)A$, $(\forall x\in\;\!^ab)A$, $\forall xA$,
$\neg C$, and $C\to A$ are in $\Pi^\mathcal{E}$ whenever $A,B$ are in $\Pi^\mathcal{E}$ and
  $C$ is in $\Sigma^\mathcal{E}$.

\noindent The axioms of $\IRSOE$ are given by \\

\begin{tabular} {ll}
(E1) & $\Gamma, A\Rightarrow A$ \ for every
$\Deltaoe$--formula A.\\
(E2) & $\Gamma\Rightarrow t=t$ \ for every $\IRSOE$ term $t$.\\
(E3) & $\Gamma, \bar s\!=\!\bar t, B(\bar s)\Rightarrow
B(\bar t)$
\ for every $\Deltaoe$-formula $B(\bar s)$.\\
(E4) & $\Gamma\Rightarrow \Eb\in\Ea$ \ for all $\beta<\al<\Omega$\\
(E5) & $\Gamma\Rightarrow a_i^\beta\in\Ea$ \ for all $i\in\omega$ and $\beta<\al<\Omega$\\
(E6) & $\Gamma, t\in \Ea, s\in t\Rightarrow s\in \Ea$ \ for all $\al<\Omega$\\
(E7) & $\Gamma, t\in \mathbb{E}_{\al+1}, s\in t\Rightarrow s\in \Ea$ \ for all $\al<\Omega$\\
(E8) & $\Gamma, s\in t, F(s,\bar p)\Rightarrow s\in[x\in t\;|\;F(x,\bar p)]$\\
(E9) & $\Gamma, s\in[x\in t\;|\;F(x,\bar p)]\Rightarrow s\in t\wedge F(s,\bar p)$\\
(E10) & $\Gamma, s\in \Ea,t\in\Eb, \text{fun}(p,s,t)\Rightarrow p\in\Eg$ \ for all $\gamma\geq\text{max}(\al,\beta)+2$.\\
(E11) & $\Gamma, t\in\Eb,\bar p\in\mathbb{E}_{\bar\al}\Rightarrow [x\in t\;|\;F(x,\bar p)]\in\Eg$ \ for all $\gamma\geq\text{max}(\beta,\bar\al)$.
\end{tabular}
}\end{definition}
\begin{definition}{\em
For a formula $A(a_1,...,a_n)$ of $\IKPE$ containing exactly the free variables $\bar{a}:=a_1,...,a_n$ and any  $\IRSOE$ terms $\bar{s}:=s_1,...,s_n$, we define the {\em$\bar{\beta}$-rank} $\|A(\bar{s})\|_{\bar{\beta}}$ where $\bar{\beta}:=\beta_1,...,\beta_n$ are any ordinals $<\Omega$. The definition is made by recursion on the build up of the formula $A$.
\begin{description}
\item{i)} $\|s\in t\|_{\beta_1,\beta_2}:=\text{max}(\beta_1,\beta_2)$
\item{ii)} $\|(\exists x\in t)F(x,\bar{s})\|_{\gamma,\bar{\beta}}:=\|(\forall x\in t)F(x,\bar{s})\|_{\gamma,\bar{\beta}}:=\text{max}(\gamma,\|F(\mathbb{E}_0,\bar{s})\|_{0,\bar{\beta}}+2)$
\item{iii)} $\|(\exists x\in\;\!\! ^st)F(x,\bar{p})\|_{\gamma,\delta,\bar{\beta}}:=\|(\forall x\in\;\!\! ^st)F(x,\bar{p})\|_{\gamma,\delta,\bar{\beta}}\\ :=\text{max}(\gamma+\omega,\delta+\omega,\|F(\mathbb{E}_0,\bar{p})\|_{0,\bar{\beta}}+2)$
\item{iv)} $\|\exists xF(x,\bar{s})\|_{\bar{\beta}}:=\|\forall xF(x,\bar{s})\|_{\bar{\beta}}:=\text{max}(\Omega,\|F(\mathbb{E}_0,\bar{s})\|_{0,\bar{\beta}}+2)$
\item{v)} $\|A\wedge B\|_{\bar{\beta}}:=\|A\vee B\|_{\bar{\beta}}:=\|A\rightarrow B\|_{\bar{\beta}}:=\text{max}(\|A\|_{\bar{\beta}}, \|B\|_{\bar{\beta}})+1$
\item{vi)} $\|\neg A\|_{\bar{\beta}}:=\| A\|_{\bar{\beta}}+1$
\end{description}
We define the {\em rank} of $A(\bar{s})$ by
\begin{equation*}
rk(A(\bar{s})):=\|A(\bar{s})\|_{\bar{0}}\,.
\end{equation*}
}\end{definition}
\begin{observation}{\em $\phantom{A}$
\begin{description}
\item{i)} $\|A(\bar{s})\|_{\bar{\beta}}<\Omega\quad\text{if and only if}\: A\:\text{is}\: \Deltaoe\,.$
\item{ii)} If $A$ contains unbounded quantifiers then $rk(A(\bar{s}))=\|A(\bar{s})\|_{\bar{\beta}}$ for all $\bar{s}$ and $\bar\beta$.
\end{description}
}\end{observation}
\begin{definition}[Operator Controlled Derivability in $\IRSOE$] {\em
$\IRSOE$ derives intuitionistic sequents of the form $\Gamma\Rightarrow\Delta$ where $\Gamma$ and $\Delta$ are finite sets of $\IRSOE$ formulae and $\Delta$ contains at most one formula. For $\mathcal{H}$ an operator and $\alpha,\rho$ ordinals we define the relation $\provx{\mathcal{H}}{\al}{\rho}{\Gamma\Rightarrow\Delta}$ by recursion on $\al$.\\

\noindent If $\Gamma\Rightarrow\Delta$ is an axiom and $\al\in\mathcal{H}$ then $\provx{\mathcal{H}}{\al}{\rho}{\Gamma\Rightarrow\Delta}$.\\

\noindent It is always required that $\al\in\CH$, this requirement is not repeated for each inference rule below.

$$ \begin{array}{llr}

(\mathbb{E}\text{-Lim})_{\infty} & \infonel{ {\mathcal H}[\delta]} {\alpha_{\delta}}
{\rho} {\Gamma, s\in\mathbb{E}_\delta\Rightarrow\Delta\mbox{ for all } \delta< \gamma}
 {\mathcal H} {\alpha} {\rho} {\Gamma,s\in \mathbb{E}_\gamma\Rightarrow\Delta}
&\begin{array}{r}\gamma\text{ a limit}\\ \alpha_\delta<\al \\ \gamma\in\CH
\end{array} \\[1cm]

(b\forall L) & \infonethree{\provx{\mathcal H} {\alpha_0} {\rho}
{\GA,s\in t\rightarrow A(s)\Rightarrow\Delta}}{\provx{\mathcal H} {\alpha_1} {\rho}
{\GA\Rightarrow t\in\Eb}}{\provx{\mathcal H} {\alpha_2} {\rho}
{\GA\Rightarrow s\in\Eg}}{ \provx{\mathcal H} {\alpha} {\rho} {\GA,(\forall x\in t)A(x)\Rightarrow\Delta}}
&\begin{array}{r}\alpha_{0},\al_1,\al_2 < \alpha\\ \beta,\gamma\in\CH\\ \gamma<\al \\ \gamma\leq\beta
\end{array} \\[1.5cm]

(b\forall R)_{\infty} & \infonetwo{\provx{\mathcal H}{\alpha_0}{\rho}{\GA\Rightarrow s\in t\rightarrow F(s)\mbox{ for all } s}}{\provx{\CH}{\al_1}{\rho}{\Gamma\Rightarrow t\in\Eb}}
{\provx{\mathcal H} {\alpha} {\rho} {\Gamma\Rightarrow(\forall x\In t)F(x)}}
&\begin{array}{r}\alpha_{0},\al_1 < \alpha\\ \beta\in\CH\\ \beta<\al
\end{array} \\[1.5cm]

(b\exists L)_{\infty} & \infonetwo{\provx{\mathcal H}{\alpha_0}{\rho}{\GA, s\in t\wedge F(s)\Rightarrow\Delta\mbox{ for all } s}}{\provx{\CH}{\al_1}{\rho}{\Gamma\Rightarrow t\in\Eb}}
{\provx{\mathcal H} {\alpha} {\rho} {\Gamma,(\exists x\In t)F(x)\Rightarrow\Delta}}
&\begin{array}{r}\alpha_{0},\al_1 < \alpha\\ \beta\in\CH\\ \beta<\al
\end{array} \\[1.5cm]

\end{array}$$
$$\begin{array}{llr}

(b\exists R) & \infonethree{\provx{\mathcal H} {\alpha_0} {\rho}
{\GA\Rightarrow s\in t\wedge A(s)}}{\provx{\mathcal H} {\alpha_1} {\rho}
{\GA\Rightarrow t\in\Eb}}{\provx{\mathcal H} {\alpha_2} {\rho}
{\GA\Rightarrow s\in\Eg}}{ \provx{\mathcal H} {\alpha} {\rho} {\GA\Rightarrow(\exists x\in t)A(x)}}
&\begin{array}{r}\alpha_{0},\al_1,\al_2 < \alpha\\ \beta,\gamma\in\CH\\ \gamma<\al \\ \gamma\leq\beta
\end{array} \\[1.5cm]

(\mathcal{E}b\forall L) & \infonefour{\provx{\mathcal H} {\alpha_0} {\rho}
{\GA,\text{fun}(p,s,t)\rightarrow A(p)\Rightarrow\Delta}}{\provx{\mathcal H} {\alpha_1} {\rho}
{\GA\Rightarrow s\in\Eb}}{\provx{\mathcal H} {\alpha_2} {\rho}
{\GA\Rightarrow t\in\Eg}}{\provx{\mathcal H} {\alpha_3} {\rho}
{\GA\Rightarrow p\in\Ed}}{ \provx{\mathcal H} {\alpha} {\rho} {\GA,(\forall x\in\;\!\!^st)A(x)\Rightarrow\Delta}}
&\begin{array}{r}\alpha_{0},\al_1,\al_2,\al_3 < \alpha\\ \beta,\gamma,\delta\in\CH\\ \delta<\al \\ \delta\leq\text{max}(\beta,\gamma)+2
\end{array} \\[2cm]

(\mathcal{E}b\forall R)_{\infty} & \infonethree{\provx{\mathcal H}{\alpha_0}{\rho}{\GA\Rightarrow \text{fun}(p,s,t)\rightarrow F(p)\mbox{ for all } p}}{\provx{\CH}{\al_1}{\rho}{\Gamma\Rightarrow s\in\Eb}}{\provx{\CH}{\al_2}{\rho}{\Gamma\Rightarrow t\in\Eg}}
{\provx{\mathcal H} {\alpha} {\rho} {\Gamma\Rightarrow(\forall x\in\;\!\!^st)F(x)}}
&\begin{array}{r}\alpha_{0},\al_1,\al_2 < \alpha\\ \beta,\gamma\in\CH\\ \text{max}(\beta,\gamma)+2\leq\al
\end{array} \\[1.5cm]

(\mathcal{E}b\exists L)_{\infty} & \infonethree{\provx{\mathcal H}{\alpha_0}{\rho}{\GA, \text{fun}(p,s,t)\wedge F(p)\Rightarrow\Delta \mbox{ for all } p}}{\provx{\CH}{\al_1}{\rho}{\Gamma\Rightarrow s\in\Eb}}{\provx{\CH}{\al_2}{\rho}{\Gamma\Rightarrow t\in\Eg}}
{\provx{\mathcal H} {\alpha} {\rho} {\Gamma,(\exists x\in\;\!\!^st)F(x)\Rightarrow\Delta}}
&\begin{array}{r}\alpha_{0},\al_1,\al_2 < \alpha\\ \beta,\gamma\in\CH\\ \text{max}(\beta,\gamma)+2\leq\al
\end{array} \\[1.5cm]

(\mathcal{E}b\exists R) & \infonefour{\provx{\mathcal H} {\alpha_0} {\rho}
{\GA\Rightarrow\text{fun}(p,s,t)\wedge A(p)}}{\provx{\mathcal H} {\alpha_1} {\rho}
{\GA\Rightarrow s\in\Eb}}{\provx{\mathcal H} {\alpha_2} {\rho}
{\GA\Rightarrow t\in\Eg}}{\provx{\mathcal H} {\alpha_3} {\rho}
{\GA\Rightarrow p\in\Ed}}{ \provx{\mathcal H} {\alpha} {\rho} {\GA\Rightarrow(\exists x\in\;\!\!^st)A(x)}}
&\begin{array}{r}\alpha_{0},\al_1,\al_2,\al_3 < \alpha\\ \beta,\gamma,\delta\in\CH\\ \delta<\al \\ \delta\leq\text{max}(\beta,\gamma)+2
\end{array} \\[2cm]

(\forall L) & \infonetwo{\provx{\mathcal H} {\alpha_0} {\rho}
{\GA, F(s)\Rightarrow\Delta}}{\provx{\CH}{\al_1}{\rho}{\Gamma\Rightarrow s\in\Eb}}{\provx{\mathcal H}{\alpha}{\rho}{\GA,\forall x F(x)\Rightarrow\Delta}}
&\begin{array}{r}\alpha_{0}+3,\al_1+3 < \alpha\\  \beta<\al\\\beta\in\CH
\end{array} \\[0.6cm]

(\forall R)_{\infty} & \infonel{ {\mathcal H}[\beta]} {\alpha_{\beta}}
{\rho} {\GA,s\in\Eb\Rightarrow  F(s)\mbox{ for all } s\text{ and all $\beta<\Omega$}}
 {\mathcal H} {\alpha} {\rho} {\Gamma\Rightarrow\forall x F(x)}
&\beta<\alpha_{\beta}+3 < \alpha

\end{array}$$
$$\begin{array}{llr}

(\exists L)_{\infty} & \infonel{ {\mathcal H}[\beta]} {\alpha_{\beta}}
{\rho} {\GA,s\in\Eb,F(s) \Rightarrow  \Delta\mbox{ for all } s\text{ and all $\beta<\Omega$}}
 {\mathcal H} {\alpha} {\rho} {\Gamma\Rightarrow\forall x F(x)}
&\beta<\alpha_{\beta}+3 < \alpha \\[1cm]

(\exists R) & \infonetwo{\provx{\mathcal H} {\alpha_0} {\rho}
{\GA\Rightarrow F(s)}}{\provx{\CH}{\al_1}{\rho}{\Gamma\Rightarrow s\in\Eb}}{\provx{\mathcal H}{\alpha}{\rho}{\GA\Rightarrow\exists xF(x)}}
&\begin{array}{r}\alpha_{0}+3,\al_1+3 < \alpha\\  \beta<\al\\\beta\in\CH
\end{array} \\[0.6cm]

\SRE &
\infonel{\CH}{\al_0}{\rho}{\Gamma\Rightarrow A}{\CH}{\al}{\rho}{\Gamma\Rightarrow
\exists z\,A^z}
&\begin{array}{r} \al_0+1,\Omega<\al\\
 A\text{ is a $\Sigmae$-formula}\end{array}\\[1cm]

 \Cut & \infonethree{\provx{\mathcal H}{\alpha_0}{\rho} {\Gamma,
A(s_1,...,s_n)\Rightarrow\Delta}}{\provx{\mathcal H} {\alpha_1} {\rho}{\Gamma\Rightarrow A(s_1,...,s_n)}}{\provx{\CH}{\al_2}{\rho}{\Gamma\Rightarrow s_i\in\mathbb{E}_{\beta_i}}\mbox{ $i=1,...,n$ }}{\provx
{\mathcal H}{\alpha}{\rho}{\Gamma\Rightarrow\Delta}}
&\begin{array}{r}\alpha_{0},\al_1,\al_2< \alpha\\
\|A(\bar{s})\|_{\bar\beta}<\rho\\
\bar\beta\in\CH\end{array}
\end{array}$$
Lastly if $\Gamma\Rightarrow\Delta$ is the result of a propositional inference of the form $(\wedge L)$, $(\wedge R)$, ($\vee L)$, $(\vee R)$, $(\neg L)$, $(\neg R)$, $(\perp)$, $(\rightarrow L)$ or $(\rightarrow R)$, with premise(s) $\Gamma_i\Rightarrow\Delta_i$ then from $\provx{\CH}{\alpha_0}{\rho}{\Gamma_i\Rightarrow\Delta_i}$ (for each $i$) we may conclude $\provx{\CH}{\alpha}{\rho}{\Gamma\Rightarrow\Delta}$, provided $\alpha_0<\alpha$.
}\end{definition}
\begin{convention}{\em
In cases where terms $\Ea$ and $a_i^\al$ occur directly as witnesses in existential rules or in cut formulae we will omit the extra premise declaring the terms location in the $\mathbb{E}$ term hierarchy since
\begin{equation*}
\Ea\in\mathbb{E}_{\al+1}\quad\text{and}\quad a_i^\al\in\mathbb{E}_{\al+1}
\end{equation*}
are axioms (E4) and (E5) respectively. It must still be checked that $\al\in\CH$ however.
}\end{convention}
\subsection{Cut elimination for $\IRSOE$}
\begin{lemma}[Inversions of $\IRSOE$]\label{einversion}{\em If max$(rk(A),rk(B))\geq\Omega$ then we have the usual propositional inversions for intuitionistic systems:
\begin{description}
\item[i)] If $\provx{\mathcal{H}}{\alpha}{\rho}{\Gamma,A\wedge B\Rightarrow\Delta}$ then $\provx{\mathcal{H}}{\alpha}{\rho}{\Gamma,A, B\Rightarrow\Delta}$.
\item[ii)] If $\provx{\mathcal{H}}{\alpha}{\rho}{\Gamma\Rightarrow A\wedge B}$ then $\provx{\mathcal{H}}{\alpha}{\rho}{\Gamma\Rightarrow A}$ and $\provx{\mathcal{H}}{\alpha}{\rho}{\Gamma\Rightarrow B}$.
\item[iii)] If $\provx{\mathcal{H}}{\alpha}{\rho}{\Gamma,A\vee B\Rightarrow\Delta}$ then $\provx{\mathcal{H}}{\alpha}{\rho}{\Gamma,A\Rightarrow\Delta}$ and $\provx{\mathcal{H}}{\alpha}{\rho}{\Gamma,B\Rightarrow\Delta}$.
\item[iv)]  If $\provx{\mathcal{H}}{\alpha}{\rho}{\Gamma,A\rightarrow B\Rightarrow\Delta}$ then $\provx{\mathcal{H}}{\alpha}{\rho}{\Gamma,B\Rightarrow\Delta}$.
\item[v)] If $\provx{\mathcal{H}}{\alpha}{\rho}{\Gamma\Rightarrow A\rightarrow B}$ then $\provx{\mathcal{H}}{\alpha}{\rho}{\Gamma,A\Rightarrow B}$.
\end{description}
If $rk(A)\geq\Omega$ we have the following additional inversions:
\begin{description}
\item[vi)] If $\provx{\mathcal{H}}{\alpha}{\rho}{\Gamma\Rightarrow\neg A}$ then $\provx{\mathcal{H}}{\alpha}{\rho}{\Gamma, A\Rightarrow}$.
\item[vii)] If $\provx{\CH}{\al}{\rho}{\Gamma\Rightarrow(\forall x\in t)A(x)}$ then $\provx{\CH}{\al}{\rho}{\GA\Rightarrow s\in t\rightarrow A(s)}$ for all terms $s$.
\item[viii)] If $\provx{\CH}{\al}{\rho}{\Gamma,(\exists x\in t)A(x)\Rightarrow\Delta}$ then $\provx{\CH}{\al}{\rho}{\GA, s\in t\wedge A(s)\Rightarrow\Delta}$ for all terms $s$.
\item[ix)] If $\provx{\CH}{\al}{\rho}{\Gamma\Rightarrow(\forall x\in\;\!\!^st)A(x)}$ then $\provx{\CH}{\al}{\rho}{\GA\Rightarrow \text{fun}(p,s,t)\rightarrow A(p)}$ for all terms $p$.
\item[x)] If $\provx{\CH}{\al}{\rho}{\Gamma,(\exists x\in\;\!\!^st)A(x)\Rightarrow\Delta}$ then $\provx{\CH}{\al}{\rho}{\GA, \text{fun}(p,s,t)\wedge A(p)\Rightarrow\Delta}$ for all terms $p$.
\end{description}
Finally we have the following persistence properties:
\begin{description}
\item[xi)]If $\gamma\in\CH\cap\Omega$ and $\provx{\CH}{\al}{\rho}{\GA\Rightarrow\forall xA(x)}$ then $\provx{\CH}{\al}{\rho}{\GA\Rightarrow(\forall x\in\Eg)A(x)}$.
\item[xii)]If $\gamma\in\CH\cap\Omega$ and $\provx{\CH}{\al}{\rho}{\GA,\exists xA(x)\Rightarrow\Delta}$ then $\provx{\CH}{\al}{\rho}{\GA,(\exists x\in\Eg)A(x)\Rightarrow\Delta}$.
\end{description}
}\end{lemma}
\begin{proof}
All proofs are by induction on $\al$, i) to vi) are standard for intuitionistic systems of this type.\\

\noindent For viii) suppose that $\provx{\CH}{\al}{\rho}{\Gamma,(\exists x\in t)A(x)\Rightarrow\Delta}$ and $rk(A(\mathbb{E}_0))\geq\Omega$. $(\exists x\in t)A(x)$ cannot have been the "active component" of an axiom, so if $\GA,(\exists x\in t)A(x)\Rightarrow\Delta$ is an axiom then so is $\GA, s\in t\wedge A(s)\Rightarrow\Delta$. Now if $(\exists x\in t)A(x)$ was not the principal formula of the last inference we may apply the induction hypothesis to the premises of that inference followed by the same inference again. Finally if $(\exists x\in t)A(x)$ was the principal formula of the last inference and the last inference was $(b\exists L)_\infty$ so we have
\begin{equation*}
\provx{\CH}{\al_0}{\rho}{\GA,(\exists x\in t)A(x),s\in t\wedge A(s)\Rightarrow\Delta}\quad\text{for all terms $s$ and for some $\al_0<\al$.}
\end{equation*}
Applying the induction hypothesis followed by weakening yields
\begin{equation*}
\provx{\CH}{\al}{\rho}{\GA,s\in t\wedge A(s)\Rightarrow\Delta}\quad\text{for all terms $s$}
\end{equation*}
as required.\ The proofs of vii), xi) and x) are similar.\\

\noindent For xi) suppose $\provx{\CH}{\al}{\rho}{\GA\Rightarrow\forall xA(x)}$ and $\gamma\in\CH\cap\Omega$. $\Gamma\Rightarrow\forall xA(x)$ cannot be an axiom. If the last inference was not $(\forall R)_\infty$ then we may apply the induction hypothesis to its premises and then the same inference again. So suppose the last inference was $(\forall R)_\infty$ in which case we have the premise
\begin{equation*}
\provx{\CH[\delta]}{\al_\delta}{\rho}{\GA,s\in\Ed\Rightarrow A(s)}\quad\text{for all $s$ and all $\delta<\Omega$, with $\delta<\al_\delta+3<\al$.}
\end{equation*}
In particular since $\gamma\in\CH$ we have
\begin{equation*}
\provx{\CH}{\al_\gamma}{\rho}{\GA,s\in\Eg\Rightarrow A(s)}\quad\text{for all $s$ with $\gamma<\al_\gamma+3<\al$.}
\end{equation*}
So by $(\rightarrow\! R)$ we have
\begin{equation*}
\provx{\CH}{\al_\gamma+1}{\rho}{\GA\Rightarrow s\in\Eg\rightarrow A(s)}\quad\text{for all $s$}\,.
\end{equation*}
Now since $\Rightarrow\Eg\in\mathbb{E}_{\gamma+1}$ is an instance of axiom (E4), $\gamma\in\CH$ and $\gamma<\al$ we may apply $(b\forall R)$ to obtain
\begin{equation*}
\provx{\CH}{\al}{\rho}{\GA\Rightarrow (\forall x\in\Eg)A(x)}
\end{equation*}
as required. The proof of xii) is similar.
\end{proof}
\begin{lemma}[Reduction for $\IRSOE$] \label{ereduction}Suppose $rk(C(\bar{s})):=\rho>\Omega$ where $C(\bar{a})$ is an $\IKPE$ formula with all free variables displayed. If
\begin{align*}
&\provx{\CH}{\al}{\rho}{\GA\Rightarrow C(\bar{s})}\\
&\provx{\CH}{\beta}{\rho}{\GA,C(\bar{s})\Rightarrow\Delta}\\
&\provx{\CH}{\gamma_i}{\rho}{\GA\Rightarrow s_i\in\mathbb{E}_{\eta_i}}\quad\text{with $\eta_i\in\CH\cap\Omega$ for each $1\leq i\leq n$.}
\end{align*}
then
\begin{equation*}
\provx{\CH}{\al\#\al\#\beta\#\beta\#\gamma}{\rho}{\GA\Rightarrow\Delta}\quad\text{where $\gamma:=\text{max}_{i=1,...,n}(\gamma_i)$}
\end{equation*}
\end{lemma}
\begin{proof}
The proof is by induction on $\al\#\al\#\beta\#\beta\#\gamma$. Assume that
\begin{align*}
&\tag{1}rk(C(\bar{s})):=\rho>\Omega\\
&\tag{2}\provx{\CH}{\al}{\rho}{\GA\Rightarrow C(\bar{s})}\\
&\tag{3}\provx{\CH}{\beta}{\rho}{\GA,C(\bar{s})\Rightarrow\Delta}\\
&\tag{4}\provx{\CH}{\gamma_i}{\rho}{\GA\Rightarrow s_i\in\mathbb{E}_{\eta_i}}\quad\text{for each $1\leq i\leq n$ and for some $\eta_i\in\CH\cap\Omega$.}
\end{align*}
Since $rk(C(\bar{s})):=\rho>\Omega$, $C$ cannot be the `active part' of an axiom, hence if (2) or (3) are axioms of $\IRSOE$ then so is $\GA\Rightarrow\Delta$.\\

\noindent If $C(\bar s)$ was not the principal formula of the last inference in either (2) or (3) then we may apply the induction hypothesis to the premises of that inference and then the same inference again.\\

\noindent So suppose $C(\bar s)$ was the principal formula of the last inference in both (2) and (3). Since the conclusion of a $\SRE$ inference always has rank $\Omega$ and $rk(C(\bar{s})):=\rho>\Omega$ we may conclude that the last inference of (2) was not $\SRE$.\\

\noindent Case 1. Suppose $C(\bar s)\equiv(\exists x\in s_i)F(x,\bar s)$, thus we have
\begin{align*}
&\tag{5}\provx{\CH}{\al_0}{\rho}{\GA\Rightarrow r\in s_i\wedge F(r,\bar s)}\quad&\al_0<\al\\
&\tag{6}\provx{\CH}{\al_1}{\rho}{\GA\Rightarrow s_i\in\Ed}\quad&\al_1<\al\:\text{and}\:\delta\in\CH\\
&\tag{7}\provx{\CH}{\al_2}{\rho}{\GA\Rightarrow r\in\mathbb{E}_\xi}\quad&\xi,\al_2<\al\:,\:\xi\in\CH(\emptyset)\:\text{and}\:\xi\leq\delta\\
&\tag{8}\provx{\CH}{\beta_0}{\rho}{\GA,C(\bar s), p\in s_i\wedge F(p,\bar s)\Rightarrow\Delta}\quad&\text{for all $p$ and $\beta_0<\beta$}\\
&\tag{9}\provx{\CH}{\beta_1}{\rho}{\GA,C(\bar s)\Rightarrow s_i\in\mathbb{E}_{\delta^\prime}}\quad&\delta^\prime,\beta_1<\beta\:\text{and}\:\delta^\prime\in\CH(\emptyset)
\end{align*}
From (8) we obtain
\begin{equation*}
\tag{10}\provx{\CH}{\beta_0}{\rho}{\GA,C(\bar s),r\in s_i\wedge F(r,\bar{s})\Rightarrow\Delta}.
\end{equation*}
Applying the induction hypothesis to (2), (4) and (10) yields
\begin{equation*}
\tag{11}\provx{\CH}{\al\#\al\#\beta_0\#\beta_0\#\gamma}{\rho}{\GA,r\in s_i\wedge F(r, \bar s)\Rightarrow\Delta}.
\end{equation*}
Note that
\begin{align*}
\Omega<rk(r\in s_i\wedge F(r,\bar s))&=rk(F(r,\bar s))+1\\
&<rk(F(r,\bar s))+2\\
&=rk(C(\bar s)):=\rho.
\end{align*}
So we may apply $\Cut$ to (4),(5),(7) and (11) giving
\begin{equation*}
\provx{\CH}{\al\#\al\#\beta\#\beta\#\gamma}{\rho}{\GA\Rightarrow\Delta}
\end{equation*}
as required. The case where $C(\bar s)\equiv (\forall x\in s_i)F(x,\bar s)$ is similar.\\

\noindent Now suppose $C(\bar s)\equiv (\forall x\in\;\!\!^{s_i}s_j)F(x,\bar s)$, so we have
\begin{align*}
\tag{12}&\provx{\CH}{\al_0}{\rho}{\GA\Rightarrow\text{fun}(p,s_i,s_j)\rightarrow F(p,\bar s)} \quad&\text{for all $p$ and $\al_0<\al$}\\
\tag{13}&\provx{\CH}{\al_1}{\rho}{\GA\Rightarrow s_i\in\Ed}\quad&\text{$\al_1<\al$ and $\delta\in\CH(\emptyset)$}
\end{align*}
\vspace{-0.715cm}
\begin{align*}
\tag{14}&\provx{\CH}{\al_2}{\rho}{\GA\Rightarrow s_j\in\mathbb{E}_{\delta^\prime}}\quad&\text{$\al_2<\al$, $\delta^\prime\in\CH(\emptyset)$ and max$(\delta,\delta^\prime)+2\leq\al$}
\end{align*}
\vspace{-0.715cm}
\begin{align*}
\tag{15}&\hspace{-0.8cm}\provx{\CH}{\beta_0}{\rho}{\GA,C(\bar s),\text{fun}(r,s_i,s_j)\rightarrow F(r,\bar{s})\Rightarrow\Delta}\quad&\text{$\beta_0<\beta$}
\end{align*}
\vspace{-0.715cm}
\begin{align*}
\tag{16}&\provx{\CH}{\beta_1}{\rho}{\GA,C(\bar s)\Rightarrow r\in\mathbb{E}_\xi}\quad&\text{$\xi<\beta$, $\xi\in\CH(\emptyset)$ and $\beta_1<\beta$}\\
\tag{17}&\provx{\CH}{\beta_2}{\rho}{\GA,C(\bar s)\Rightarrow s_i\in\mathbb{E}_\zeta}\quad&\text{$\zeta\in\CH(\emptyset)$ and $\beta_2<\beta$}\\
\tag{18}&\provx{\CH}{\beta_3}{\rho}{\GA,C(\bar s)\Rightarrow s_j\in\mathbb{E}_{\zeta^\prime}}\quad&\text{$\zeta^\prime\in\CH(\emptyset)$, $\beta_3<\beta$ and $\xi\leq\text{max}(\zeta,\zeta^\prime)+2$}
\end{align*}
As an instance of (12) we have
\begin{equation*}
\tag{19}\provx{\CH}{\al_0}{\rho}{\GA\Rightarrow\text{fun}(r,s_i,s_j)\rightarrow F(r,\bar s)}.
\end{equation*}
Applying the induction hypothesis to (2), (4) and (15) gives
\begin{equation*}
\tag{20}\provx{\CH}{\al\#\al\#\beta_0\#\beta_0\#\gamma}{\rho}{\GA,\text{fun}(r,s_i,s_j)\rightarrow F(r,\bar{s})\Rightarrow\Delta}.
\end{equation*}
Furthermore the induction hypothesis applied to (2),(4) and (16) gives
\begin{equation*}
\tag{21}\provx{\CH}{\al\#\al\#\beta_1\#\beta_1\#\gamma}{\rho}{\GA\Rightarrow r\in\mathbb{E}_\xi}.
\end{equation*}
Note that
\begin{align*}
\Omega<rk(\text{fun}(r,s_i,s_j)\rightarrow F(r,\bar s))&=rk(F(r,\bar s))+1\\
&<rk(F(r,\bar s))+2=rk(C(\bar s))
\end{align*}
so we may apply $\Cut$ to (4), (19), (20), (21) to give
\begin{equation*}
\tag{22}\provx{\CH}{\al\#\al\#\beta\#\beta\#\gamma}{\rho}{\GA\Rightarrow\Delta}
\end{equation*}
as required.\\

\noindent The case where $C(\bar s)\equiv(\exists x\in \;\!\!^{s_i}s_j)F(x,\bar s)$ is similar.\\

\noindent Case 3. Now suppose that $C(\bar s)\equiv\forall x F(x,\bar s)$, so we have
\begin{align*}
\tag{23}&\provx{\CH[\delta]}{\al_\delta}{\rho}{\GA,p\in\Ed\Rightarrow F(p,\bar s)}\quad&\text{for all $p$ and all $\delta<\Omega$ with $\al_\delta+3<\al$}\\
\tag{24}&\provx{\CH}{\beta_0}{\rho}{\GA,C(\bar s), F(r,\bar s)}\Rightarrow\Delta\quad&\text{with $\beta_0+3<\beta$}\\
\tag{25}&\provx{\CH}{\beta_1}{\rho}{\GA,C(\bar s)\Rightarrow r\in\mathbb{E}_\xi}\quad&\text{with $\xi<\beta$, $\xi\in\CH(\emptyset)$ and $\beta_1+3<\beta$.}
\end{align*}
Since $\xi\in\CH(\emptyset)$, from (23) we obtain
\begin{equation*}
\tag{26}\provx{\CH}{\al_\xi}{\rho}{\GA,r\in\mathbb{E}_\xi\Rightarrow F(r,\bar s)\,.}
\end{equation*}
Applying the induction hypothesis to (2), (4) and (24) gives
\begin{equation*}
\tag{27}\provx{\CH}{\al\#\al\#\beta_0\#\beta_0\#\gamma}{\rho}{\GA,F(r,\bar s)\Rightarrow\Delta}.
\end{equation*}
Again applying the induction hypothesis to (2), (4) and (25) gives
\begin{equation*}
\tag{28}\provx{\CH}{\al\#\al\#\beta_1\#\beta_1\#\gamma}{\rho}{\GA\Rightarrow r\in\mathbb{E}_\xi}.
\end{equation*}
Now a $\Cut$ applied to (26) and (28) yields
\begin{equation*}
\tag{29}\provx{\CH}{\al\#\al\#\beta\#\beta_1\#\gamma}{\rho}{\GA\Rightarrow F(r,\bar s)}.
\end{equation*}
Note that
\begin{equation*}
\Omega\leq rk(F(r,\bar s))<rk(F(r,\bar s))+2=rk(C)=\rho\,.
\end{equation*}
So a $\Cut$ applied to (4), (27), (28) and (29) yields
\begin{equation*}
\tag{30}\provx{\CH}{\al\#\al\#\beta\#\beta\#\gamma}{\rho}{\GA\Rightarrow\Delta}
\end{equation*}
as required.\\

\noindent The case where $C(\bar s)\equiv\exists xF(x,\bar s)$ is similar.\\

\noindent In the cases where $C\equiv A\wedge B, A\vee B, A\rightarrow B\text{ or } \neg A$ we may argue as with other intuitionistic systems of a similar nature.
\end{proof}
\begin{theorem}[Cut Elimination I]\label{epredce}{\em If $\provx{\CH}{\al}{\Omega+n+1}{\GA\Rightarrow\Delta}$ then $\provx{\CH}{\omega_n(\al)}{\Omega+1}{\GA\RI\DE}$ for all $n<\omega$,
where  $\omega_0(\al)=\al$ and $\omega_{n+1}(\al)=\omega^{\omega_n(\al)}$.
}\end{theorem}
\begin{proof}
By main induction on $n$ and subsidiary induction on $\al$. The interesting case is where the last inference was $\Cut$, with cut formula $A(\bar s)$ such that $rk(A(\bar s))=\Omega+n$ and $\bar s=s_1,...,s_m$ are the only terms occurring $A(\bar s)$. In this case we have
\begin{align*}
\tag{1}&\provx{\CH}{\al_0}{\Omega+n+1}{\GA\RI A(\bar s)}&\text{with $\al_0<\al$}\\
\tag{2}&\provx{\CH}{\al_1}{\Omega+n+1}{\GA, A(\bar s)\RI\DE}&\text{with $\al_1<\al$}\\
\tag{3}&\provx{\CH}{\al_2}{\Omega+n+1}{\GA\RI s_i\in\mathbb{E}_{\beta_i}}&\text{with $\al_2<\al$ and $\beta_i\in\CH$ for each $i=1,...,m$.}
\end{align*}
Applying the subsidiary induction hypothesis to (1), (2) and (3) gives
\begin{align*}
\tag{4}&\provx{\CH}{\omega^{\al_0}}{\Omega+n}{\GA\RI A(\bar s)}&\text{with $\al_0<\al$}\\
\tag{5}&\provx{\CH}{\omega^{\al_1}}{\Omega+n}{\GA, A(\bar s)\RI\DE}&\text{with $\al_1<\al$}\\
\tag{6}&\provx{\CH}{\omega^{\al_2}}{\Omega+n}{\GA\RI s_i\in\mathbb{E}_{\beta_i}}&\text{with $\al_2<\al$ and $\beta_i\in\CH$ for each $i=1,...,m$.}
\end{align*}
Now applying the Reduction Lemma \ref{ereduction} to (4), (5) and (6) gives
\begin{equation*}
\tag{7}\provx{\CH}{\omega^{\al_0}\#\omega^{\al_0}\#\omega^{\al_1}\#\omega^{\al_1}\#\omega^{\al_2}}{\Omega+n}{\GA\RI\DE}.
\end{equation*}
Note that $\omega^{\al_0}\#\omega^{\al_0}\#\omega^{\al_1}\#\omega^{\al_1}\#\omega^{\al_2}<\omega^\al$ so by weakening we have
\begin{equation*}
\tag{8}\provx{\CH}{\omega^\al}{\Omega+n}{\GA\RI\DE}.
\end{equation*}
Finally applying the main induction hypothesis gives
\begin{equation*}
\provx{\CH}{\omega_n(\al)}{\Omega+1}{\GA\RI\DE}
\end{equation*}
as required.
\end{proof}
\begin{lemma}\label{ehier}
{\em If $\gamma\leq\beta<\Omega$ with $\beta,\gamma\in\CH(\emptyset)$ and $\provx{\CH}{\al}{\rho}{\GA\RI s\in\Eg}$ then
\begin{equation*}
\provx{\CH}{\al+2}{\rho^*}{\GA\RI s\in\Eb}
\end{equation*}
where $\rho^*:=\text{max}(\rho,\beta+1)$.
}\end{lemma}
\begin{proof}
If $\gamma=\beta$ the result follows by weakening, so suppose $\gamma<\beta$. Assume that
\begin{equation*}
\tag{1}\provx{\CH}{\al}{\rho}{\GA\RI s\in\Eg}.
\end{equation*}
Now as instances of axioms (E4) and (E6) respectively we have
\begin{align*}
\tag{2}&\provx{\CH}{0}{0}{\GA\Rightarrow\Eg\in\Eb}\\
\tag{3}&\provx{\CH}{0}{0}{\GA,s\in\Eg,\Eg\in\Eb\RI s\in\Eb}.
\end{align*}
Applying $\Cut$ to (2) and (3) yields
\begin{equation*}
\tag{4}\provx{\CH}{1}{\beta+2}{\GA,s\in\Eg\RI s\in\Eb}.
\end{equation*}
Now applying a second $\Cut$ to (1) and (4) supplies us with
\begin{equation*}
\provx{\CH}{\al+2}{\rho^*}{\GA\RI s\in\Eb}
\end{equation*}
as required.
\end{proof}
\begin{lemma}[Boundedness]\label{eboundedness}{\em Suppose $\al\leq\beta<\OO$, $\beta\in\CH$, $A$ is a $\Sigmae$-formula and $B$ is a $\Pie$ formula then:
\begin{description}
\item{i)} If $\provx{\CH}{\al}{\rho}{\GA\RI A}$ then $\provx{\CH}{\al}{\rho^*}{\GA\RI A^{\Eb}}$.
\item{ii)} If $\provx{\CH}{\al}{\rho}{\GA,B\RI \DE}$ then $\provx{\CH}{\al}{\rho^*}{\GA, B^{\Eb}\RI\DE}$.
\end{description}
Where we set $\rho^*:=\text{max}(\rho,\beta+1)$.
}\end{lemma}
\begin{proof}
By induction on $\al$. The interesting case of i) is where $A\equiv\exists xC(x)$ and $A$ was the principal formula of the last inference which was $(\exists R)$. Note that since $\al<\OO$ the last inference cannot have been $\SRE$. So we have
\begin{align*}
\tag{1}&\provx{\CH}{\al_0}{\rho}{\GA\RI C(r)}&\text{with $\al_0+3<\al$.}\\
\tag{2}&\provx{\CH}{\al_1}{\rho}{\GA\RI r\in\Eg}&\text{with $\al_1<\al$, $\gamma\in\CH$ and $\gamma<\al$.}
\end{align*}
Since $\gamma<\al$ we also know that $\gamma<\beta$ so using Lemma \ref{ehier} we get
\begin{equation*}
\tag{3}\provx{\CH}{\al_1+2}{\rho^*}{\GA\RI r\in\Eb}.
\end{equation*}
Now by applying the induction hypothesis to (1) we get
\begin{equation*}
\tag{4}\provx{\CH}{\al_0}{\rho}{\GA\RI C(r)^{\Eb}}.
\end{equation*}
$(\wedge R)$ applied to (3) and (4) yields
\begin{equation*}
\tag{5}\provx{\CH}{\text{max}(\al_0+1,\al_1+3)}{\rho^*}{\GA\RI r\in\Eb \wedge C(r)^{\Eb}}.
\end{equation*}
Now since $\GA\RI\Eb\in\mathbb{E}_{\beta+1}$ is an axiom we may apply $(b\exists R)$ to (2) and (5) giving
\begin{equation*}
\provx{\CH}{\al}{\rho^*}{\GA\RI(\exists x\in\Eb)C(x)^{\Eb}}
\end{equation*}
as required.\\

\noindent Now for ii) the interesting case is where $B$ was the principal formula of the last inference which was $(b\forall L)$, thus $B\equiv\forall xC(x)$. So we have
\begin{align*}
\tag{6}&\provx{\CH}{\al_0}{\rho}{\GA,B,C(s)\RI\DE}&\text{with $\al_0<\al$,}\\
\tag{7}&\provx{\CH}{\al_1}{\rho}{\GA,B\RI s\in\Eg}&\text{with $\al_1+3<\al$, $\gamma\in\CH$ and $\gamma<\al$.}
\end{align*}
Applying the induction hypothesis twice to (6) and once to (7) we get
\begin{align*}
\tag{8}&\provx{\CH}{\al_0}{\rho}{\GA,B^{\Eb},C(s)^{\Eb}\RI\DE}&\text{with $\al_0<\al$,}\\
\tag{9}&\provx{\CH}{\al_1}{\rho}{\GA,B^{\Eb}\RI s\in\Eg}&\text{with $\al_1+3<\al$, $\gamma\in\CH$ and $\gamma<\al$.}
\end{align*}
Now since $\gamma<\al$ we also know that $\gamma<\beta$ so by applying Lemma \ref{ehier} to (9) we get
\begin{equation*}
\tag{10}\provx{\CH}{\al_1+2}{\rho^*}{\GA, B^{\Eb}\RI s\in\Eb}.
\end{equation*}
Applying $(\rightarrow L)$ to (8) and (10) supplies us with
\begin{equation*}
\tag{11}\provx{\CH}{\text{max}(\al_0+1,\al_1+3)}{\rho^*}{\GA,B^{\Eb},s\in\Eb\rightarrow C(s)^{\Eb}\RI\DE}.
\end{equation*}
Now applying $(b\forall L)$ to (11), (9) and $\RI\Eb\in\mathbb{E}_{\beta+1}$ which is an instance of axiom (E4), we obtain
\begin{equation*}
\provx{\CH}{\al}{\rho^*}{\GA,B^{\Eb}\RI\DE}
\end{equation*}
completing the proof.
\end{proof}
\begin{theorem}[Cut Elimination II; Collapsing]\label{ecollapsing}{\em Suppose $\eta\in\CH_\eta$, $\DE$ is a set of at most one $\Sigmae$ formula and $\GA$ is a finite set of $\Pie$ formulae.
 Then
\begin{equation*}
\provx{\CH_\eta}{\al}{\OO+1}{\GA\RI\DE}\quad\text{implies}\quad\provx{\CH_{\hat\al}}{\psio{\hat\al}}{\psio{\hat\al}}{\GA\RI\DE},
\end{equation*}
where $\hat\al:=\eta+\omega^\al$.
}\end{theorem}
\begin{proof}
The proof is by induction on $\al$. Note that since $\eta\in\CH_\eta$ we know from Lemma \ref{op2} that
\begin{equation*}
\hat\al,\psio{\hat\al}\in\CH_{\hat\al}.
\end{equation*}
Case 1. If $\GA\RI\DE$ is an axiom the result follows easily.\\

\noindent Case 2. If $\GA\RI\DE$ was the result of a propositional inference we may apply the induction hypothesis to the premises of that inference, and then the same inference again.\\

\noindent Case 3. Suppose the last inference was $(\mathbb{E}\text{-Lim})$, then $s\in\Eg$ is a formula in $\GA$ for some limit ordinal $\gamma$ and
\begin{equation*}
\provx{\CH_\eta[\delta]}{\al_\delta}{\OO+1}{\GA,s\in\Ed\RI\DE}\quad\text{for all $\delta<\gamma$ with $\al_\delta<\al$.}
\end{equation*}
Since $\gamma\in\CH_\eta(\emptyset)=B^\OO(\eta+1)$ and $\gamma<\OO$ we know that $\gamma<\psio{\eta+1}$ and thus $\delta\in\CH_\eta$ for all $\delta<\gamma$. So we have
\begin{equation*}
\provx{\CH_\eta}{\al_\delta}{\OO+1}{\GA,s\in\Ed\RI\DE}\quad\text{for all $\delta<\gamma$ with $\al_\delta<\al$.}
\end{equation*}
Now applying the induction hypothesis provides
\begin{equation*}
\provx{\CH_{\hat\al}}{\psio{\hat{\al_\delta}}}{\psio{\hat{\al_\delta}}}{\GA,s\in\Ed\RI\DE}\quad\text{for all $\delta<\gamma$ with $\al_\delta<\al$.}
\end{equation*}
Now since $\psio{\hat{\al_\delta}}<\psio{\hat\al}$ we may apply $(\mathbb{E}\text{-Lim})$ to get the desired result.\\

\noindent Case 4. Suppose the last inference was $(b\forall L)$, then $(\forall x\in t)F(x)\in\GA$ and:
\begin{align*}
\tag{1}&\provx{\CH_\eta}{\al_0}{\OO+1}{\GA,s\in t\rightarrow F(s)\RI\DE}&\text{with $\al_0<\al$.}\\
\tag{2}&\provx{\CH_\eta}{\al_1}{\OO+1}{\GA\RI t\in\Eb}&\text{$\beta\in\CH_\eta(\emptyset)$ and $\al_1<\al$.}\\
\tag{3}&\provx{\CH_\eta}{\al_2}{\OO+1}{\GA\RI s\in\Eg}&\text{$\gamma\in\CH_\eta(\emptyset)$, $\gamma,\al_2<\al$ and $\gamma\leq\beta$.}
\end{align*}
Since  $s\in t\rightarrow F(s)$ is also a $\Pi^{\mathcal E}$-formula we may immediately apply the induction hypothesis to (1), (2) and (3) giving
\begin{align*}
\tag{4}&\provx{\CH_{\hat\al}}{\psio{\hat{\al_0}}}{\psio{\hat{\al}}}{\GA,s\in t\rightarrow F(s)\RI\DE}\\
\tag{5}&\provx{\CH_{\hat\al}}{\psio{\hat{\al_1}}}{\psio{\hat{\al}}}{\GA\RI t\in\Eb}\\
\tag{6}&\provx{\CH_{\hat\al}}{\psio{\hat{\al_2}}}{\psio{\hat{\al}}}{\GA\RI s\in\Eg.}
\end{align*}
Since $\gamma\in\CH_\eta$ we know that $\gamma<\psio{\eta+1}$ and thus $\gamma\in\CH_{\hat\al}$ and $\gamma<\psio{\hat\al}$. Moreover $\psio{\al_i}<\psio\al$ for $i=0,1,2$ so we may apply $(b\forall L)$ to complete this case. The case where the last inference was $(b\exists R)$ is treated in a similar manner.\\

\noindent Case 5. Suppose the last inference was $(b\forall R)_\infty$, then $\DE=\{(\forall x\in t)F(x)\}$ and
\begin{align*}
\tag{7}&\provx{\CH_\eta}{\al_0}{\OO+1}{\GA\RI s\in t\rightarrow F(s)}&\text{for all $s$, with $\al_0<\al$,}\\
\tag{8}&\provx{\CH_\eta}{\al_1}{\OO+1}{\GA\RI t\in\Eb}&\text{with $\al_1,\beta<\al$ and $\beta\in\CH_\eta$.}
\end{align*}
Applying the induction hypothesis to (7) and (8) yields
\begin{align*}
\tag{9}&\provx{\CH_{\hat\al}}{\psio{\hat{\al_1}}}{\psio{\hat{\al}}}{\GA\RI t\in\Eb}\\
\tag{10}&\provx{\CH_{\hat\al}}{\psio{\hat{\al_0}}}{\psio{\hat{\al}}}{\GA\RI s\in t\to F(s).}
\end{align*}
Note that since $\beta\in\CH_\eta$ we know that $\beta<\psio{\eta+1}<\psio{\hat\al}$, thus applying  $(b\forall R)_\infty$ to (10) (noting that $\psio{\hat{\al_0}}+1<\psio{\hat\al}$) gives the desired result. The case where the last inference was $(b\exists L)_\infty$ is treated in a similar manner.\\

\noindent Case 6. Now suppose the last inference was $(\mathcal{E}b\exists L)_\infty$, so $(\exists x\in\;\!\!^st)F(x)\in\GA$ and:
\begin{align*}
\tag{11}&\provx{\CH_\eta}{\al_0}{\OO+1}{\GA,\text{fun}(p,s,t)\wedge F(p)\RI\DE}&\text{for all $p$, with $\al_0<\al$.}
\end{align*}
\vspace{-1cm}
\begin{align*}
\tag{12}&\provx{\CH_\eta}{\al_1}{\OO+1}{\GA\RI s\in\Eb}&\text{with $\beta\in\CH_\eta$ and $\al_1<\al$.}\\
\tag{13}&\provx{\CH_\eta}{\al_2}{\OO+1}{\GA\RI t\in\Eg}&\text{with $\al_2<\al$, $\gamma\in\CH_\eta$ and $\text{max}(\beta,\gamma)+2\leq\al$.}
\end{align*}
By assumption $\text{fun}(p,s,t)\wedge F(p)$ is a $\Pie$ formula so we may apply the induction hypothesis to (11), (12) and (13) giving:
\begin{align*}
\tag{14}&\provx{\CH_{\hat\al}}{\psio{\hat{\al_0}}}{\psio{\hat\al}}{\GA,\text{fun}(p,s,t)\wedge F(p)\RI\DE}&\text{for all $p$;}\\
\tag{15}&\provx{\CH_{\hat\al}}{\psio{\hat{\al_1}}}{\psio{\hat\al}}{\GA\RI s\in\Eb}&\\
\tag{16}&\provx{\CH_{\hat\al}}{\psio{\hat{\al_2}}}{\psio{\hat\al}}{\GA\RI t\in\Eg}.&
\end{align*}
Since $\psio{\hat{\al_i}}<\psio{\hat\al}$ for $i=0,1,2$ and $\beta,\gamma\in\CH_\eta$ means that $\text{max}(\beta,\gamma)+2<\psio{\eta+1}<\psio{\hat\al}$ we may apply $(\mathcal{E}b\exists L)_\infty$ to (14), (15) and (16) to complete this case. The case where the last inference was $(\mathcal{E}b\forall R)_\infty$ may be treated in a similar manner.\\

\noindent Case 7. Now suppose the last inference was $(\mathcal{E}b\forall R)$, so $\DE=\{(\exists x\in\;\!\!^st)F(x)\}$ and we have:
\begin{align*}
\tag{17}&\provx{\CH_\eta}{\al_0}{\OO+1}{\GA\RI\text{fun}(p,s,t)\wedge F(p)}&\text{for all $p$ with $\al_0<\al$;}\\
\tag{18}&\provx{\CH_\eta}{\al_1}{\OO+1}{\GA\RI s\in\Eb}&\text{with $\beta\in\CH_\eta(\emptyset)$ and $\al_1<\al$;}\\
\tag{19}&\provx{\CH_\eta}{\al_2}{\OO+1}{\GA\RI t\in\Eg}&\text{with $\gamma\in\CH_\eta(\emptyset)$ and $\al_2<\al$;}\\
\tag{20}&\provx{\CH_\eta}{\al_3}{\OO+1}{\GA\RI p\in\Ed}&\text{$\al_3,\delta<\al$, $\delta\in\CH_\eta(\emptyset)$ and $\delta\leq\text{max}(\beta,\gamma)+2$.}
\end{align*}
Since $\text{fun}(p,s,t)\wedge F(p)$ is a $\Sigmae$ formula we can apply the induction hypothesis to (17), (18), (19) and (20) followed by $(\mathcal{E}b\forall R)$, in a similar manner to Case 4. The case where the last inference was $(\mathcal{E}b\forall L)$ can also be treated in a similar manner.\\

\noindent Now suppose the last inference was $(\forall L)$, so $\forall xF(x)\in\GA$ and
\begin{align*}
\tag{21}&\provx{\CH_\eta}{\al_0}{\OO+1}{\GA,F(s)\RI\DE}&\text{with $\al_0+3<\al$,}\\
\tag{22}&\provx{\CH_\eta}{\al_1}{\OO+1}{\GA\RI s\in\Eb}&\text{$\beta,\al_1+3<\al$ and $\beta\in\CH_\eta(\emptyset)$.}
\end{align*}
Since $F(s)$ is $\Pie$ we may immediately apply the induction hypothesis to (21) and (22) giving
\begin{align*}
\tag{23}&\provx{\CH_{\hat{\al}}}{\psio{\hat{\al_0}}}{\psio{\hat\al}}{\GA,F(s)\RI\DE}&\\
\tag{24}&\provx{\CH_{\hat{\al}}}{\psio{\hat{\al_1}}}{\psio{\hat\al}}{\GA\RI s\in\Eb}.&
\end{align*}
Now since $\beta\in\CH_\eta$ we know that $\beta<\psio{\eta+1}<\psio{\hat\al}$ hence we may apply $(\forall L)$ to (23) and (24) to complete this case. The case where the last inference was $(\exists R)$ can be treated in a similar manner.\\

\noindent Case 9. Now suppose the last inference was $\SRE$, so $\DE=\{\exists z A^z\}$ where $A$ is a $\Sigmae$ formula and
\begin{align*}
\tag{25}&\provx{\CH_\eta}{\al_0}{\OO+1}{\GA\RI A}&\text{with $\al_0+1,\OO<\al$.}
\end{align*}
We may immediately apply the induction hypothesis to (25) giving
\begin{equation*}
\tag{26}\provx{\CH_{\hat{\al_0}}}{\psio{\hat{\al_0}}}{\psio{\hat{\al_0}}}{\GA\RI A}.
\end{equation*}
Applying Boundedness \ref{eboundedness}i) to (26) provides
\begin{equation*}
\tag{27}\provx{\CH_{\hat{\al_0}}}{\psio{\hat{\al_0}}}{\psio{\hat{\al_0}}+2}{\GA\RI A^{\mathbb{E}_{\psio{\hat{\al_0}}}}}.
\end{equation*}
Now as an instance of axiom (E4) we have
\begin{equation*}
\tag{28}\provx{\CH_{\hat{\al_0}}}{0}{0}{\RI \mathbb{E}_{\psio{\hat{\al_0}}}\in\mathbb{E}_{\psio{\hat{\al_0}}+1}}.
\end{equation*}
Since $\psio{\hat{\al_0}}+1\in\CH_{\hat\al}$ and $\psio{\hat{\al_0}}+1<\psio{\hat\al}$ we may apply $(\exists R)$ to (27) and (28) to complete the case.\\

\noindent Now suppose the last inference was $\Cut$, so that we have:
\begin{align*}
\tag{29}&\provx{\CH_\eta}{\al_0}{\OO+1}{\GA\RI A(\bar s)}&\text{with $\al_0<\al$;}\\
\tag{30}&\provx{\CH_\eta}{\al_1}{\OO+1}{\GA,A(\bar s)\RI \DE}&\text{with $\al_1<\al$;}\\
\tag{31}&\provx{\CH_\eta}{\al_2}{\OO+1}{\GA\RI s_i\in\mathbb{E}_{\beta_i}}&\text{with $\al_2<\al$, $\bar\beta\in\CH_\eta$ and $\|A(\bar s)\|_{\bar\beta}\leq\OO$.}
\end{align*}
Subcase 10.1: If $\|A(\bar s)\|_{\bar\beta}<\OO$ it follows from $\bar\beta\in\CH_\eta=B^\OO(\eta+1)$ that $\|A(\bar s)\|_{\bar\beta}\in B^\OO(\eta+1)$ and thus  $\|A(\bar s)\|_{\bar\beta}<\psio{\eta+1}<\psio{\hat\al}$. Also $A$ is $\Deltaoe$, thus we may apply the induction hypothesis to (29), (30) and (31) followed by $\Cut$ to complete this (sub)case.\\

\noindent Subcase 10.2: Now suppose $\|A(\bar s)\|_{\bar\beta}=\OO$. Then either $A\equiv\forall x F(x)$ or $A\equiv\exists xF(x)$ with $F$ a $\Deltaoe$ formula. The two cases are dual, we assume that the former is the case. Thus $A$ is $\Pie$, so we may apply the induction hypothesis to (30) giving
\begin{equation*}
\tag{32}\provx{\CH_{\hat{\al_1}}}{\psio{\hat{\al_1}}}{\psio{\hat{\al_1}}}{\GA,A(\bar s)\RI\DE\,.}
\end{equation*}
Applying Boundedness \ref{eboundedness}ii) to (32) yields
\begin{equation*}
\tag{33}\provx{\CH_{\hat\al}}{\text{max}(\psio{\hat{\al_0}},\psio{\hat{\al_1}})}{\psio{\hat{\al_1}}}{\GA,A(\bar s)^{\mathbb{E}_{\psio{\hat{\al_0}}}}\RI\DE}.
\end{equation*}
Now applying \ref{einversion}xi) (persistence) to (29) gives
\begin{equation*}
\tag{34}\provx{\CH_{\hat{\al_0}}}{\al_0}{\OO+1}{\GA\RI A(\bar s)^{\mathbb{E}_{\psio{\hat{\al_0}}}}}.
\end{equation*}
Noting that $A(\bar s)^{\mathbb{E}_{\psio{\hat{\al_0}}}}$ is $\Deltaoe$ we may apply the induction hypothesis to (34) giving
\begin{equation*}
\tag{35}\provx{\CH_{\al^*}}{\psio{\al^*}}{\psio{\al^*}}{\GA\RI A(\bar s)^{\mathbb{E}_{\psio{\hat{\al_0}}}}}.
\end{equation*}
where $\al^*:=\hat{\al_0}+\omega^{\OO+\al_0}$. Now applying the induction hypothesis to (31) gives
\begin{equation*}
\tag{36}\provx{\CH_{\hat{\al_2}}}{\psio{\hat{\al_2}}}{\psio{\hat{\al_2}}}{\GA\RI s_i\in\mathbb{E}_{\beta_i}}.
\end{equation*}
Now as an instance of axiom (E4) we have
\begin{equation*}
\tag{37}\provx{\CH_{\hat\al}}{0}{0}{\RI \mathbb{E}_{\psio{\hat{\al_0}}}\in\mathbb{E}_{\psio{\hat{\al_0}}+1}}.
\end{equation*}
Since $\bar\beta\in B^\OO(\eta+1)$ we get
\begin{equation*}
\|A(\bar s)^{\mathbb{E}_{\psio{\hat{\al_0}}}}\|_{\bar\beta,\psio{\hat{\al_0}}+1}=\psio{\hat{\al_0}}+1<\psio{\hat\al}.
\end{equation*}
It remains to note that
\begin{equation*}
\al^*=\eta+\omega^{\OO+\al_0}+\omega^{\OO+\al_0}<\eta+\omega^{\OO+\al}=\hat\al
\end{equation*}
and thus $\psio{\al^*}<\psio\al$. So we may apply $\Cut$ to (33),(35),(36) and (37) to conclude
\begin{equation*}
\provx{\CH_{\hat{\al}}}{\psio{\hat\al}}{\psio{\hat\al}}{\GA\RI\DE}
\end{equation*}
as required.
\end{proof}
\subsection{Embedding $\IKPE$ into $\IRSOE$.}
\begin{definition}{\em
If $\GA[\bar a]\RI\DE[\bar a]$ is an intuitionistic sequent of $\IKPE$ with exactly the free variables $\bar a=a_1,...,a_n$ and containing the formulas $A_1(\bar a),...,A_m(\bar a)$ then
\begin{equation*}
no_{\bar\beta}(\GA[\bar s]\RI\DE[\bar s]):=\omega^{\|A_1\|_{\bar\beta}}\#...\#\omega^{\|A_m\|_{\bar\beta}}.
\end{equation*}
For terms $\sbar:=s_1,...s_n$ and ordinals $\bbeta:=\beta_1,...,\beta_n$ the expression $\sbar\in\Ebb$ will be considered shorthand for $s_1\in\mathbb{E}_{\beta_1},...,s_n\in\mathbb{E}_{\beta_n}$\\

\noindent The expression $\Vdash\GA[\bar s]\RI\DE[\bar s]$ will be considered shorthand for
\begin{equation*}
{\CH[\bar{\beta}]}\;\prov{no_{\bar\beta}(\GA[\bar s]\RI\DE[\bar s])}{0}{\bar s\in\mathbb{E}_{\bar\beta},\GA[\bar s]\RI\DE[\bar s]}.
\end{equation*}
For any operator $\CH$ and any ordinals $\bar\beta<\OO$.\\

\noindent The expression $\Vdash^\al_\rho \GA[\bar s]\RI\DE[\bar s]$ will be considered shorthand for
\begin{equation*}
{\CH[\bar{\beta}]}\;\prov{no_{\bar\beta}(\GA[\bar s]\RI\DE[\bar s])\#\al}{\rho}{\bar s\in\mathbb{E}_{\bar\beta},\GA[\bar s]\RI\DE[\bar s]}.
\end{equation*}
For any operator $\CH$ and any ordinals $\bar\beta<\OO$.\\

\noindent As might be expected $\Vdash^\al\GA[\bar s]\RI\DE[\bar s]$ and $\Vdash_\rho\GA[\bar s]\RI\DE[\bar s]$ will be considered shorthand for $\Vdash^\al_0\GA[\bar s]\RI\DE[\bar s]$ and $\Vdash_\rho^0\GA[\bar s]\RI\DE[\bar s]$ respectively.
}\end{definition}
\begin{lemma}\label{elogic}{\em For any formula $A(\bar a)$ of $\IKPE$ containing exactly the free variables displayed and any $\IRSOE$ terms $\bar s=s_1,...,s_n$
\begin{equation*}
\Vdash_\OO A(\bar s)\RI A(\bar s)\,.
\end{equation*}
}\end{lemma}
\begin{proof}
By induction on the construction of the formula $A$. If $A$ is $\Deltaoe$ then this is an instance of axiom (E1).\

\noindent Suppose $A(\bar s)\equiv\forall xF(x,\bar s)$. For each $\ga<\OO$ we define
\begin{equation*}
\al^\gamma:=\gamma+no_{\gamma,\bar\beta}(F(t,\bar s)\RI F(t,\bar s)),
\end{equation*}
note that
\begin{equation*}
\gamma<\al^\gamma<\al^\gamma+8<no_{\bar\beta}(A(\bar s)\RI A(\bar s)).
\end{equation*}
By axiom (E1) we have
\begin{equation*}
\tag{1}{\CH[\gamma,\bar\beta]}\;\prov{0}{0}{t\in\Eg\RI t\in\Eg}\quad\text{for all $t$ and all $\gamma<\OO$.}
\end{equation*}
Now from the induction hypothesis we have
\begin{equation*}
\tag{2}\CH[\gamma,\bar\beta]\;\prov{\al^\gamma}{\OO}{\bar s\in\mathbb{E}_{\bar\beta}, t\in\Eg, F(t,\bar s)\RI F(t,\bar s)}\quad\text{for all $t$ and all $\gamma<\OO$.}
\end{equation*}
It is worth noting that this use of the induction hypothesis is where we really need cuts of $\bar\beta$-rank arbitrarily high in $\OO$. Applying $(\forall L)$ to (1) and (2) yields
\begin{equation*}
\CH[\gamma,\bar\beta]\;\prov{\al^\gamma+4}{\OO}{\bar s\in\mathbb{E}_{\bar\beta},t\in\Eg, A(\bar s)\RI F(t,\bar s)}
\end{equation*}
to which we may apply $(\forall R)_\infty$ to get the desired result.\\

\noindent Case 2. Now suppose $A\equiv(\forall x\in s_i)F(x,\bar s)$. From the induction hypothesis we have
\begin{equation*}
\tag{3}\CH[\delta,\bar\beta]\;\prov{(\omega^{\|F(t,\bar s)\|_{\delta,\bar\beta}})\cdot 2}{\OO}{t\in\Ed,\bar s\in\mathbb{E}_{\bar\beta},F(t,\bar s)\RI F(t,\bar s)}\quad\text{for all $t$ and all $\delta<\OO$.}
\end{equation*}
In particular when $\delta=\beta_i$ in (3) we have
\begin{equation*}
\tag{4}\CH[\delta,\bar\beta]\;\prov{\al_0}{\OO}{t\in\mathbb{E}_{\beta_i},\bar s\in\mathbb{E}_{\bar\beta},F(t,\bar s)\RI F(t,\bar s)}
\end{equation*}
where $\al_0:=(\omega^{\|F(t,\bar s)\|_{\beta_i,\bar\beta}})\cdot 2$. Now as an instance of axiom (E6) we have
\begin{equation*}
\tag{5}\CH[\bar\beta]\;\prov{0}{0}{s_i\in\mathbb{E}_{\beta_i},t\in s_i\RI t\in\mathbb{E}_{\beta_i}}\,.
\end{equation*}
Now applying $\Cut$ to (4) and (5) yields
\begin{equation*}
\tag{6}\CH[\bar\beta]\;\prov{\al_0+1}{\OO}{\bar s\in\mathbb{E}_{\bar\beta},t\in s_i,F(t,\bar s)\RI F(t,\bar s)}.
\end{equation*}
As an instance of axiom (E1) we have
\begin{equation*}
\tag{7}\CH[\bar\beta]\;\prov{0}{0}{t\in s_i\RI t\in s_i}.
\end{equation*}
Applying $(\rightarrow L)$ to (6) and (7) yields
\begin{equation*}
\tag{8}\CH[\bar\beta]\;\prov{\al_0+2}{\OO}{\bar s\in\mathbb{E}_{\bar\beta},t\in s_i,t\in s_i\rightarrow F(t,\bar s)\RI F(t,\bar s)}.
\end{equation*}
An application of $(b\forall L)$ to (5) and (8) provides
\begin{equation*}
\CH[\bar\beta]\;\prov{\al_0+3}{\OO}{\bar s\in\mathbb{E}_{\bar\beta},t\in s_i,(\forall x\in s_i)F(x,\bar s)\RI F(t,\bar s)}.
\end{equation*}
Finally using $(\rightarrow R)$ followed by $(b\forall R)_\infty$ and noting that $\al_0+5<no_{\bar\beta}(A(\bar s)\RI A(\bar s))$ we get the desired result.\\

\noindent Case 3. Suppose that $A\equiv (\exists x\in\;\!\!^{s_i}s_j)F(x,\bar s)$. From the induction hypothesis we know that
\begin{equation*}
\tag{9}\CH[\bar\beta,\delta]\;\prov{(\omega^{\|F(t,\bar s)\|_{\delta,\bar\beta}})\cdot 2}{\OO}{\bar s\in\mathbb{E}_{\bar\beta},t\in\Ed,F(t,\bar s)\RI F(t,\bar s)}\quad\text{for all $t$ and all $\delta<\OO$.}
\end{equation*}
In particular when $\delta=\gamma:=\text{max}(\beta_i,\beta_j)+2$ we have
\begin{equation*}
\tag{10}\CH[\bar\beta]\;\prov{\al_0}{\OO}{\bar s\in\mathbb{E}_{\bar\beta}, t\in\Eg, F(t,\bar s)\RI F(t,\bar s)}\quad\text{for all $t$,}
\end{equation*}
where $\al_0:=(\omega^{\|F(t,\bar s)\|_{\bar\beta,\gamma}})\cdot 2$. Now as an instance of axiom (E10) we have
\begin{equation*}
\tag{11}\CH[\bar\beta]\;\prov{0}{0}{\sbar\in\Ebb,\text{fun}(t,s_i,s_j)\RI t\in\Eg}.
\end{equation*}
Applying $\Cut$ to (10) and (11) gives
\begin{equation*}
\tag{12}\CH[\bbeta]\;\prov{\al_0+1}{\OO}{\sbar\in\Ebb,\text{fun}(t,s_i,s_j),F(t,\bar s)\RI F(t,\sbar)}.
\end{equation*}
As an instance of axiom (E1) we have
\begin{equation*}
\tag{13}\CH[\bar\beta]\;\prov{0}{0}{\text{fun}(t,s_i,s_j)\RI\text{fun}(t,s_i,s_j)}
\end{equation*}
Applying $(\wedge R)$ to (12) and (13) gives
\begin{equation*}
\tag{14}\CH[\bbeta]\;\prov{\al_0+2}{\OO}{\sbar\in\Ebb,\text{fun}(t,s_i,s_j),F(t,\sbar)\RI\text{fun}(t,s_i,s_j)\wedge F(t,\sbar)}.
\end{equation*}
Now applying $(\mathcal{E}b\exists R)$ to (11) and (14) yields
\begin{equation*}
\tag{15}\CH[\bbeta]\;\prov{\al_0+3}{\OO}{\sbar\in\Ebb,\text{fun}(t,s_i,s_j),F(t,\sbar)\RI(\exists x\in\;\!\!^{s_i}s_j) F(x,\sbar)}.
\end{equation*}
Two applications of $(\wedge L)$ gives
\begin{equation*}
\tag{16}\CH[\bbeta]\;\prov{\al_0+5}{\OO}{\sbar\in\Ebb,\text{fun}(t,s_i,s_j)\wedge F(t,\sbar)\RI(\exists x\in\;\!\!^{s_i}s_j) F(x,\sbar)}.
\end{equation*}
Finally using $(\mathcal{E}b\exists L)_\infty$ gives
\begin{equation*}
\tag{17}\CH[\bbeta]\;\prov{\al_0+6}{\OO}{\sbar\in\Ebb,(\exists x\in\;\!\!^{s_i}s_j) F(x,\sbar)\RI(\exists x\in\;\!\!^{s_i}s_j) F(x,\sbar)}.
\end{equation*}
It remains to note that $\al_0+6<no_{\bbeta}(A(\sbar)\RI A(\sbar))$ to complete this case.\\

\noindent All other cases are either propositional, for which the proof is standard or may be regarded as dual to one of the three above.
\end{proof}
\begin{lemma}[Extensionality]\label{eextensionality}{\em
For any formula $A(\bar a)$ of $\IKPE$ (not necessarily with all free variables displayed) and any $\IRSOE$ terms $\sbar:=s_1,...,s_n,\bar t:=t_1,...,t_n$ we have
\begin{equation*}
\Vdash_\OO\sbar=\bar t,A(\sbar)\RI A(\bar t)
\end{equation*}
where $\sbar=\bar t$ is shorthand for $s_1=t_1,...,s_n=t_n$.
}\end{lemma}
\begin{proof}
If $A(\sbar)$ is $\Deltaoe$ then this is an instance of axiom (E3). The remainder of the proof is by induction on $rk(A(\sbar))$, note that since $A$ is assumed to contain an unbounded quantifier
\begin{equation*}
rk(A)=\|A(\sbar)\|_{\bbeta}\geq\OO\quad\text{for any ordinals $\bbeta<\OO$.}
\end{equation*}
Case 1. Suppose $A(\bar s)\equiv\forall xF(x,\sbar)$. By the induction hypothesis we have
\begin{equation*}
\CH[\bbeta,\bar\gamma,\delta]\;\prov{no_{\bbeta,\bar\gamma,\delta}(\sbar=\bar t, F(r,\bar s)\RI F(r,\bar t))}{\OO}{\sbar\in\Ebb,\bar t\in\mathbb{E}_{\bar\gamma},r\in\Ed,\sbar=\bar t, F(r,\sbar)\RI F(r,\bar t)}
\end{equation*}
for all $r$ and all $\delta<\OO$. For ease of reading we suppress the other terms possibly occurring in $F(r,\sbar)$ and the assumptions about their locations in the $\mathbb{E}$ hierarchy since these do not affect the proof. By virtue of axiom (E1) we have
\begin{equation*}
\CH[\bbeta,\bar\gamma,\delta]\;\prov{0}{0}{ r\in\Ed\RI r\in\Ed}.
\end{equation*}
Hence we may apply $(\forall L)$ to obtain
\begin{equation*}
\CH[\bbeta,\bar\gamma,\delta]\;\prov{\al_\delta}{\OO}{\sbar\in\Ebb, t\in\mathbb{E}_{\bar\gamma},\sbar=\bar t, r\in\Ed,\forall xF(x,\sbar)\RI F(r,\bar t)}
\end{equation*}
where $\al_\delta:=\delta+no_{\bbeta,\bar\gamma,\delta}(\sbar=\bar t,F(r,\sbar)\RI F(r,\bar t))+1$. Note that
\begin{equation*}
\al_\delta+3<no_{\bbeta,\bar\gamma}(\bar s=\bar t,A(\sbar)\RI A(\bar t))=:\al.
\end{equation*}
Hence we may apply $(\forall R)_\infty$ to obtain
\begin{equation*}
\CH[\bbeta,\bar\gamma]\;\prov{\al}{\OO}{\sbar\in\Ebb,\bar t\in\mathbb{E}_{\bar\gamma},\sbar=\bar t, A(\bar s)\RI A(\bar t)}
\end{equation*}
as required.\\

\noindent Case 2. Now suppose $A(\bar s)\equiv (\forall x\in\;\!\!^{s_i}s_j)F(x,\sbar)$. Using the induction hypothesis we have
\begin{equation*}
\tag{1}\CH[\bbeta,\bar\gamma,\delta]\;\prov{\al_0}{\OO}{\sbar\in\Ebb,\bar t\in\mathbb{E}_{\bar\gamma},r\in\Ed,\bar s=\bar t,F(r,\bar s)\RI F(r,\bar t)}
\end{equation*}
for any term $r$ and any $\delta<\OO$, where $\al_0=no_{\bbeta,\bar\gamma,\delta}(\sbar=\bar t,F(r,\sbar)\RI F(r,\bar t))$. At this point we set $\delta=\text{max}(\beta_i,\beta_j)+2$, note that $\delta\in\CH[\bbeta,\bar\gamma]$. By virtue of axiom (E1) we have
\begin{equation*}
\tag{2}\CH[\bbeta,\bar\gamma]\;\prov{0}{0}{\text{fun}(r,s_i,s_j)\RI\text{fun}(r,s_i,s_j).}
\end{equation*}
Hence by $(\rightarrow L)$ we get
\begin{align*}
\tag{3}\CH[\bbeta,\bar\gamma]\;&\prov{\al_0+1}{\OO}{\sbar\in\Ebb,\bar t\in\mathbb{E}_{\bar\gamma}},r\in\Ed,\bar s=\bar t,\\
&\quad\quad\quad\text{fun}(r,s_i,s_j)\rightarrow F(r,\bar s),\text{fun}(r,s_i,s_j)\RI F(r,\bar t).
\end{align*}
As an instance of axiom (E10) we have
\begin{equation*}
\tag{4}\CH[\bbeta,\bar\gamma]\;\prov{0}{0}{\sbar\in\Ebb,\text{fun}(r,s_i,s_j)\RI r\in\Ed}.
\end{equation*}
An application of $\Cut$ to (3) and (4) yields
\begin{align*}
\tag{5}\CH[\bbeta,\bar\gamma]\;&\prov{\al_0+2}{\OO}{\sbar\in\Ebb,\bar t\in\mathbb{E}_{\bar\gamma}},\bar s=\bar t,\\
&\quad\quad\:\text{fun}(r,s_i,s_j)\rightarrow F(r,\bar s),\text{fun}(r,s_i,s_j)\RI F(r,\bar t).
\end{align*}
Now applying $(\mathcal{E}b\forall L)$ to (4) and (5) gives
\begin{equation*}
\tag{6}\CH[\bbeta,\bar\gamma]\;\prov{\al_0+3}{\OO}{\sbar\in\Ebb,\bar t\in\mathbb{E}_{\bar\gamma},\sbar=\bar t, (\forall x\in\;\!\!^{s_i}s_j)F(x,\sbar),\text{fun}(r,s_i,s_j)\RI F(r,\bar t).}
\end{equation*}
Note that $\al_0\geq\OO$ since $F$ is not $\Deltaoe$, so we don't have to worry about the condition $\delta<\al_0+3$. Now as an instance of axiom $(E3)$ we have
\begin{equation*}
\tag{7}\CH[\bbeta,\bar\gamma]\;\prov{0}{0}{\sbar=\bar t,\text{fun}(r,t_i,t_j)\RI\text{fun}(r,s_i,s_j)}.
\end{equation*}
Also axiom (E10) gives rise to
\begin{equation*}
\tag{8}\CH[\bbeta,\bar\gamma]\;\prov{0}{0}{\bar t\in\mathbb{E}_{\bar\gamma},\text{fun}(r,t_i,t_j)\RI r\in\mathbb{E}_\eta}\quad\text{where $\eta=\text{max}(\gamma_i,\gamma_j)+2$.}
\end{equation*}
Applying $\Cut$ to (6),(7) and (8) gives
\begin{equation*}
\tag{9}\CH[\bbeta,\bar\gamma]\;\prov{\al_0+4}{\OO}{\sbar\in\Ebb,\bar t\in\mathbb{E}_{\bar\gamma},\sbar=\bar t, (\forall x\in\;\!\!^{s_i}s_j)F(x,\sbar),\text{fun}(r,t_i,t_j)\RI F(r,\bar t).}
\end{equation*}
Now $(\rightarrow R)$ gives
\begin{equation*}
\tag{10}\CH[\bbeta,\bar\gamma]\;\prov{\al_0+5}{\OO}{\sbar\in\Ebb,\bar t\in\mathbb{E}_{\bar\gamma},\sbar=\bar t, (\forall x\in\;\!\!^{s_i}s_j)F(x,\sbar)\RI\text{fun}(r,t_i,t_j)\rightarrow F(r,\bar t).}
\end{equation*}
Finally we may apply $(\mathcal{E}b\forall R)_\infty$, noting that $\al_0+6<no_{\bbeta,\bar\gamma}(\sbar=\bar t,A(\sbar)\RI A(\bar t))$ to complete this case.\\

\noindent Note that it could also be the case that $A(\sbar)\equiv(\forall x\in\;\!\!^pq)F(x,\sbar)$ where $p$ and/or $q$ is not a member of $\sbar$. The following case is an example of this kind of thing.\\

\noindent Case 3. Suppose $A(\sbar)\equiv(\exists x\in p)F(x,\sbar,p)$, where $p$ is not present in $\sbar$. By the induction hypothesis we have
\begin{equation*}
\tag{11}\CH[\bbeta,\bar\gamma,\delta]\;\prov{\al_0}{\OO}{\sbar\in\Ebb,\bar t\in\mathbb{E}_{\bar\gamma},p\in\Ed, r\in\Ed,\sbar=\bar t, F(r,\sbar,p)\RI F(r,\bar t,p)}
\end{equation*}
where $\al_0:= no_{\bbeta,\bar\gamma,\delta,\delta}(\sbar=\bar t,F(r,\sbar,p)\RI F(r,\bar t, p))$.  As an instance of axiom (E1) we have
\begin{equation*}
\tag{12}\CH[\bbeta,\bar\gamma,\delta]\;\prov{0}{0}{r\in p\RI r\in p}.
\end{equation*}
Applying $(\wedge R)$ to (11) and (12) yields
\begin{equation*}
\tag{13}\CH[\bbeta,\bar\gamma,\delta]\;\prov{\al_0+1}{\OO}{\sbar\in\Ebb,\bar t\in\mathbb{E}_{\bar\gamma},p\in\Ed, r\in\Ed,\sbar=\bar t, F(r,\sbar,p),r\in p\RI r\in p\wedge F(r,\bar t,p)}.
\end{equation*}
As an instance of axiom (E6) we have
\begin{equation*}
\tag{14}\CH[\bbeta,\bar\gamma,\delta]\;\prov{0}{0}{p\in\Ed, r\in p\RI r\in\Ed}.
\end{equation*}
$\Cut$ applied to (12) and (13) gives
\begin{equation*}
\tag{15}\CH[\bbeta,\bar\gamma,\delta]\;\prov{\al_0+2}{\OO}{\sbar\in\Ebb,\bar t\in\mathbb{E}_{\bar\gamma},p\in\Ed,\sbar=\bar t, F(r,\sbar,p),r\in p\RI r\in p\wedge F(r,\bar t,p)}.
\end{equation*}
Now $(b\exists R)$ gives
\begin{equation*}
\tag{16}\CH[\bbeta,\bar\gamma,\delta]\;\prov{\al_0+3}{\OO}{\sbar\in\Ebb,\bar t\in\mathbb{E}_{\bar\gamma},p\in\Ed,\sbar=\bar t, F(r,\sbar,p),r\in p\RI A(\sbar)}.
\end{equation*}
Two applications of $(\wedge L)$ gives
\begin{equation*}
\tag{17}\CH[\bbeta,\bar\gamma,\delta]\;\prov{\al_0+5}{\OO}{\sbar\in\Ebb,\bar t\in\mathbb{E}_{\bar\gamma},p\in\Ed,\sbar=\bar t, r\in p\wedge F(r,\sbar,p)\RI A(\sbar)}.
\end{equation*}
To which we may apply $(b\exists L)$ to complete this case.\\

\noindent All other cases are similar to one of those above.
\end{proof}
\begin{lemma}[Set induction]\label{eSet Induction}{\em For any formula $F(a)$ of $\IKPE$ we have
\begin{equation*}
\Vdash_{\OO}\RI\forall x[(\forall y\in x)F(y)\rightarrow F(x)]\rightarrow\forall xF(x).
\end{equation*}
}\end{lemma}

\begin{proof}
Let $\CH$ be an arbitrary operator and let
\begin{equation*}
A:=\forall x[(\forall y\in x)F(y)\rightarrow F(x)].
\end{equation*}
Let $\bar p$ be the terms other than $s$ that occur in $F(s)$, sub-terms not included. Let $\bar\CH:=\CH[\bbeta]$ where $\bbeta$ is an arbitrary choice of ordinals $<\OO$. In the remainder of the proof we shall just write $\bar\CH\;\prov{\al}{\rho}{\GA\RI\DE}$ instead of $\CH[\bbeta]\;\prov{\al}{\rho}{\bar p\in\Ebb,\GA\RI\DE}$, since $\bar p\in\Ebb$ will always remain a side formula in the derivation.\\

\noindent Claim:
\begin{equation*}
\tag{*}\bar\CH[\gamma]\;\prov{\omega^{rk(A)}\#\omega^{\gamma+1}}{\OO}{A,s\in\Eg\RI F(s)}\quad\text{for all $\gamma<\OO$ and all terms $s$.}
\end{equation*}
Note that since $A$ contains an unbounded quantifier $rk(A)=no_{\bbeta}(A)$. We prove the claim by induction on $\gamma$. Thus the induction hypothesis supplies us with
\begin{equation*}
\tag{1}\bar\CH[\delta]\;\prov{\omega^{rk(A)}\#\omega^{\delta+1}}{\OO}{A,t\in\Ed\RI F(t)}\quad\text{for all $\delta<\gamma$ and all terms $t$.}
\end{equation*}
So by weakening we have
\begin{equation*}
\tag{2}\bar\CH[\gamma,\delta]\;\prov{\omega^{rk(A)}\#\omega^{\delta+1}}{\OO}{A,s\in\Eg,t\in s,t\in\Ed\RI F(t)}.
\end{equation*}
Case 1. Suppose $\gamma=\gamma_0+1$, so a special case of (2) becomes
\begin{equation*}
\tag{3}\bar\CH[\gamma]\;\prov{\omega^{rk(A)}\#\omega^\gamma}{\OO}{A,s\in\Eg,t\in s, t\in\mathbb{E}_{\gamma_0}\RI F(t)}.
\end{equation*}
As an instance of axiom (E7) we have
\begin{equation*}
\tag{4}\bar\CH[\gamma]\;\prov{0}{0}{s\in\Eg,t\in s\RI t\in\mathbb{E}_{\gamma_0}}.
\end{equation*}
Applying $\Cut$ to (3) and (4) yields
\begin{equation*}
\tag{5}\bar\CH[\gamma]\;\prov{\omega^{rk(A)}\#\omega^{\gamma}+1}{\OO}{A,s\in\Eg,t\in s\RI F(t)}.
\end{equation*}
$(\rightarrow R)$ followed by $(b\forall R)_\infty$ provides
\begin{equation*}
\tag{6}\bar\CH[\gamma]\;\prov{\omega^{rk(A)}\#\omega^{\gamma}+3}{\OO}{A,s\in\Eg\RI(\forall x\in s)F(x)}.
\end{equation*}
Now from Lemma \ref{elogic} we have
\begin{equation*}
\tag{7}\bar\CH[\gamma]\;\prov{no_{\bbeta,\gamma}(F(s)\RI F(s))}{\OO}{s\in\Eg,F(s)\RI F(s)}.
\end{equation*}
Since $no_{\bbeta,\gamma}(F(s)\RI F(s))<\omega^{rk(A)}$, by $(\rightarrow L)$ we get
\begin{equation*}
\tag{8}\bar\CH[\gamma]\;\prov{\omega^{rk(A)}\#\omega^{\gamma}+4}{\OO}{A,s\in\Eg,(\forall x\in s)F(x)\rightarrow F(s)\RI F(s)}.
\end{equation*}
To which we may apply $(\forall L)$ giving
\begin{equation*}
\tag{9}\bar\CH[\gamma]\;\prov{\omega^{rk(A)}\#\omega^{\gamma+1}}{\OO}{A,s\in\Eg\RI F(s)}
\end{equation*}
as required.\\

\noindent Case 2. Now suppose $\gamma$ is a limit ordinal. Applying $(\mathbb{E}\text{-Lim})$ to (2) provides us with
\begin{equation*}
\tag{10}\bar\CH[\gamma]\;\prov{\omega^{rk(A)}\#\omega^{\gamma}}{\OO}{A,s\in\Eg,t\in s,t\in\Eg\RI F(t)}.
\end{equation*}
As an instance of axiom (E6) we have
\begin{equation*}
\tag{11}\bar\CH[\gamma]\;\prov{0}{0}{s\in\Eg,t\in s\RI t\in\Eg}.
\end{equation*}
An application of $\Cut$ to (10) and (11) yields
\begin{equation*}
\tag{12}\bar\CH[\gamma]\;\prov{\omega^{rk(A)}\#\omega^{\gamma}+1}{\OO}{A,s\in\Eg,t\in s\RI F(t)}.
\end{equation*}
The remainder of this case can proceed exactly as in Case 1 from (5) onwards. Thus the claim (*) is verified.\\

\noindent Finally applying $(\forall R)_\infty$ to (*) gives
\begin{equation*}
\bar\CH\;\prov{\omega^{rk(A)}\#\OO}{\OO}{A\RI\forall xF(x)}.
\end{equation*}
Finally noting that $\omega^{rk(A)}\#\OO<no_{\bbeta}(A\rightarrow\forall xF(x))$ we may apply $(\rightarrow R)$ to complete the proof.
\end{proof}
\begin{lemma}[Infinity]\label{einfinity}{\em For any operator $\CH$ we have
\begin{equation*}
\provx{\CH}{\omega+4}{\omega}{\RI\exists x[(\forall y\in x)(\exists z\in x)(y\in z)\wedge(\exists y\in x)(y\in x)]}.
\end{equation*}
}\end{lemma}
\begin{proof}
Firstly note that by Definiton \ref{operatorr} $1,\omega\in\CH$. We have the following derivation trees in $\IRSOE$.
\begin{prooftree}
\Axiom$\fCenter\quad\text{Axiom (E6)}$
\UnaryInf$\fCenter\provx{\CH}{0}{0}{s\in\mathbb{E}_n, \mathbb{E}_n\in\mathbb{E}_{n+1}\RI s\in\mathbb{E}_{n+1}}$
\Axiom$\fCenter\quad\text{Axiom (E4)}$
\UnaryInf$\fCenter\provx{\CH}{0}{0}{\RI \mathbb{E}_n\in\mathbb{E}_{n+1}}$
\LeftLabel{$\Cut$}
\BinaryInf$\fCenter\provx{\CH}{1}{n+3}{s\in\mathbb{E}_n\RI s\in\mathbb{E}_{n+1}}$
\Axiom$\fCenter\quad\text{Axiom (E4)}$
\UnaryInf$\fCenter\provx{\CH}{0}{0}{\RI \mathbb{E}_{n+1}\in\mathbb{E}_{\omega}}$
\LeftLabel{$(\wedge R)$}
\BinaryInf$\fCenter\provx{\CH}{2}{n+3}{s\in\mathbb{E}_n\RI  \mathbb{E}_{n+1}\in\mathbb{E}_{\omega}\wedge s\in\mathbb{E}_{n+1}}$
\LeftLabel{$(b\exists R)$}
\UnaryInf$\fCenter\provx{\CH}{n+3}{n+3}{s\in\mathbb{E}_n\RI  (\exists z\in\mathbb{E}_\omega)(s\in z)}$
\LeftLabel{$(\mathbb{E}\text{-Lim})$}
\UnaryInf$\fCenter\provx{\CH}{\omega}{\omega}{s\in\mathbb{E}_\omega\RI  (\exists z\in\mathbb{E}_\omega)(s\in z)}$
\LeftLabel{$(\rightarrow R)$}
\UnaryInf$\fCenter\provx{\CH}{\omega+1}{\omega}{\RI s\in\mathbb{E}_\omega\rightarrow  (\exists z\in\mathbb{E}_\omega)(s\in z)}$
\LeftLabel{$(b\forall R)_\infty$}
\UnaryInf$\fCenter\provx{\CH}{\omega+2}{\omega}{\RI (\forall y\in\mathbb{E}_\omega)(\exists z\in\mathbb{E}_\omega)(y\in z)}$
\end{prooftree}

\begin{prooftree}
\Axiom$\fCenter\quad\text{Axiom (E4)}$
\UnaryInf$\fCenter\provx{\CH}{0}{0}{\RI\mathbb{E}_0\in\mathbb{E}_\omega}$
\LeftLabel{$(\wedge R)$}
\UnaryInf$\fCenter\provx{\CH}{1}{0}{\RI\mathbb{E}_0\in\mathbb{E}_\omega\wedge \mathbb{E}_0\in\mathbb{E}_\omega}$
\LeftLabel{$(b\exists R)$}
\UnaryInf$\fCenter\provx{\CH}{2}{0}{\RI(\exists y\in\mathbb{E}_\omega)(y\in\mathbb{E}_\omega)}$
\end{prooftree}
Applying $(\wedge R)$ followed by $(b\exists R)$ to the conclusions of the two proof trees above yields the required result.
\end{proof}
\begin{lemma}[$\Deltaoe$-Separation]\label{esep}{\em
For any $\Deltaoe$ formula $A(a,\bar b)$  of $\IKPE$ containing exactly the free variables $a,\bar b=b_1,...,b_n$, any $\IRSOE$ terms $r,s_1,...,s_n$ and any operator $\CH$:
\begin{equation*}
\CH[\gamma,\bbeta]\;\prov{\al+7}{0}{\sbar\in\Ebb,r\in\Eg\RI\exists x[(\forall y\in x)(y\in r\wedge A(y,\sbar))\wedge(\forall y\in r)(A(y,\sbar)\rightarrow y\in x)]}
\end{equation*}
where $\al=\text{max}(\bbeta,\gamma)$.
}\end{lemma}
\begin{proof}
First let
\begin{equation*}
p:=[x\in r\;|\;A(x,\sbar)].
\end{equation*}
As an instance of axiom (E11) we have
\begin{equation*}
\tag{1}\CH[\gamma,\bbeta]\;\prov{0}{0}{\sbar\in\Ebb,r\in\Eg\RI p\in\Ea}.
\end{equation*}
Moreover we have the following derivations in $\IRSOE$:
\begin{prooftree}
\Axiom$\fCenter\quad\text{Axiom (E9)}$
\UnaryInf$\fCenter\provx{\CH}{0}{0}{\sbar\in\Ebb,r\in\Eg, t\in p\RI t\in r\wedge A(t,\sbar)}$
\LeftLabel{$(\rightarrow R)$}
\UnaryInf$\fCenter\provx{\CH}{1}{0}{\sbar\in\Ebb,r\in\Eg \RI t\in p\rightarrow t\in r\wedge A(t,\sbar)}$
\Axiom$\fCenter\quad\text{(1)}$
\LeftLabel{$(b\forall R)_\infty$}
\BinaryInf$\fCenter\provx{\CH}{\al+2}{0}{\sbar\in\Ebb,r\in\Eg \RI(\forall y\in p)(y\in r\wedge A(y,\sbar))}$
\end{prooftree}
\begin{prooftree}
\Axiom$\fCenter\quad\text{Axiom (E8)}$
\UnaryInf$\fCenter\provx{\CH}{0}{0}{\sbar\in\Ebb,r\in\Eg, t\in r, A(t,\sbar)\RI t\in p}$
\LeftLabel{$(\rightarrow R)$}
\UnaryInf$\fCenter\provx{\CH}{1}{0}{\sbar\in\Ebb,r\in\Eg,t\in r \RI A(t,\sbar)\rightarrow t\in p}$
\LeftLabel{$(\rightarrow R)$}
\UnaryInf$\fCenter\provx{\CH}{2}{0}{\sbar\in\Ebb,r\in\Eg\RI t\in r \rightarrow(A(t,\sbar)\rightarrow t\in p)}$
\LeftLabel{$(b\forall R)_\infty$}
\UnaryInf$\fCenter\provx{\CH}{\gamma+3}{0}{\sbar\in\Ebb,r\in\Eg\RI (\forall y\in r)(A(y,\sbar)\rightarrow y\in p)}$
\end{prooftree}
Now applying $(\wedge R)$ to (1) and the conclusions of the two proof trees above, followed by an application of $(\exists R)$ yields the desired result.
\end{proof}
\begin{lemma}[Pair]\label{epair}{\em For any operator $\CH$, and $\IRSOE$ terms $s,t$ and any ordinals $\beta,\gamma<\OO$:
\begin{equation*}
\CH[\beta,\gamma]\;\prov{\al+6}{\al+2}{s\in\Eb,t\in\Eg\RI\exists z(s\in z\wedge t\in z)}
\end{equation*}
where $\al:=\text{max}(\beta,\gamma)$.
}\end{lemma}
\begin{proof}
If $\beta=\gamma$ the proof is straightforward, without loss of generality let us assume $\beta>\gamma$. As instances of axioms (E6) and (E4) we have
\begin{align*}
\tag{1}&\CH[\beta,\gamma]\;\prov{0}{0}{t\in\Eg,\Eg\in\Eb\RI t\in\Eb}\\
\tag{2}&\CH[\beta,\gamma]\;\prov{0}{0}{\RI\Eg\in\Eb}.
\end{align*}
Applying $\Cut$ gives
\begin{equation*}
\tag{3}\CH[\beta,\gamma]\;\prov{1}{\beta+2}{t\in\Eg\RI t\in\Eb}.
\end{equation*}
By axiom (E1) we have
\begin{equation*}
\tag{4}\CH[\beta,\gamma]\;\prov{0}{0}{s\in\Eb\RI s\in\Eb}.
\end{equation*}
Applying $(\wedge R)$ to (3) and (4) provides
\begin{equation*}
\tag{5}\CH[\beta,\gamma]\;\prov{2}{\beta+2}{s\in\Eb,t\in\Eb\RI s\in\Eb\wedge t\in\Eb},
\end{equation*}
to which we may apply $(\exists R)$ giving
\begin{equation*}
\CH[\beta,\gamma]\;\prov{\beta+6}{\beta+2}{s\in\Eb,t\in\Eg\RI\exists z(s\in z\wedge t\in z)},
\end{equation*}
as required.
\end{proof}
\begin{lemma}[Union]\label{eunion}{\em For any operator $\CH$, $\IRSOE$ term $s$ and any $\beta<\OO$ we have
\begin{equation*}
\provx{\CH[\beta]}{\beta+9}{\beta+2}{s\in\Eb\RI\exists z[(\forall y\in s)(\forall x\in y)(x\in z)]}.
\end{equation*}
}\end{lemma}
\begin{proof}
We have the following template for derivations in $\IRSOE$.
\begin{prooftree}
\Axiom$\fCenter\quad\text{Axiom (E6)}$
\UnaryInf$\fCenter\provx{\CH[\beta]}{0}{0}{t\in\Eb, r\in t\RI r\in\Eb}$
\Axiom$\fCenter\quad\text{Axiom (E6)}$
\UnaryInf$\fCenter\provx{\CH[\beta]}{0}{0}{s\in\Eb, t\in s\RI t\in\Eb}$
\LeftLabel{$\Cut$}
\BinaryInf$\fCenter\provx{\CH[\beta]}{1}{\beta+2}{s\in\Eb,t\in s, r\in t\RI r\in\Eb}$
\LeftLabel{$(\rightarrow R)$}
\UnaryInf$\fCenter\provx{\CH[\beta]}{2}{\beta+2}{s\in\Eb,t\in s\RI r\in t\rightarrow r\in\Eb}$
\LeftLabel{$(b\forall R)_\infty$}
\UnaryInf$\fCenter\provx{\CH[\beta]}{\beta+3}{\beta+2}{s\in\Eb,t\in s\RI (\forall x\in t)(x\in\Eb)}$
\LeftLabel{$(\rightarrow R)$}
\UnaryInf$\fCenter\provx{\CH[\beta]}{\beta+4}{\beta+2}{s\in\Eb\RI t\in s\rightarrow (\forall x\in t)(x\in\Eb)}$
\LeftLabel{$(b\forall R)_\infty$}
\UnaryInf$\fCenter\provx{\CH[\beta]}{\beta+5}{\beta+2}{s\in\Eb\RI (\forall y\in s)(\forall x\in y)(x\in\Eb)}$
\LeftLabel{$(\exists R)$}
\UnaryInf$\fCenter\provx{\CH[\beta]}{\beta+9}{\beta+2}{s\in\Eb\RI \exists z(\forall y\in s)(\forall x\in y)(x\in z).}$
\end{prooftree}
\end{proof}
\begin{lemma}[$\Deltaoe$-Collection]\label{ecoll}{\em Let $F(a,b,\bar c)$ be any $\Deltaoe$ formula of $\IKPE$ containing exactly the free variables displayed then for any $\sbar=s_1,...,s_n$
\begin{equation*}
\Vdash_\OO\;\;\RI(\forall x\in s_i)\exists yF(x,y,\sbar)\rightarrow\exists z(\forall x\in s_i)(\exists y\in z)F(x,y,\sbar).
\end{equation*}
}\end{lemma}
\begin{proof}
Since $F$ is $\Deltaoe$ we have
\begin{equation*}
no_{\bbeta}((\forall x\in s_i)\exists yF(x,y,\sbar))=\omega^{\OO+2}.
\end{equation*}
Hence by Lemma \ref{elogic} we have
\begin{equation*}
\CH[\bbeta]\;\prov{\omega^{\OO+2}\cdot 2}{\OO}{\sbar\in\Ebb,(\forall x\in s_i)\exists yF(x,y,\sbar)\RI(\forall x\in s_i)\exists yF(x,y,\sbar)}.
\end{equation*}
Applying $\SRE$ gives
\begin{equation*}
\CH[\bbeta]\;\prov{\omega^{\OO+2}\cdot 2+2}{\OO}{\sbar\in\Ebb,(\forall x\in s_i)\exists yF(x,y,\sbar)\RI\exists z(\forall x\in s_i)(\exists y\in z)F(x,y,\sbar)}.
\end{equation*}
By $(\rightarrow R)$ we get
\begin{equation*}
\CH[\bbeta]\;\prov{\omega^{\OO+2}\cdot 2+3}{\OO}{\sbar\in\Ebb\RI(\forall x\in s_i)\exists yF(x,y,\sbar)\rightarrow\exists z(\forall x\in s_i)(\exists y\in z)F(x,y,\sbar)}.
\end{equation*}
Finally since $\omega^{\OO+2}\cdot 2+3<\omega^{\OO+3}$ we may conclude
\begin{equation*}
\Vdash_\OO\RI(\forall x\in s_i)\exists yF(x,y,\sbar)\rightarrow\exists z(\forall x\in s_i)(\exists y\in z)F(x,y,\sbar)
\end{equation*}
as required.
\end{proof}
\begin{lemma}[Exponentiation]\label{eexponentiation}{\em For any terms $s,t$ any $\beta,\gamma<\OO$ and any operator $\CH$
\begin{equation*}
\CH[\beta,\gamma]\;\prov{\delta+4}{\delta+3}{s\in\Eb, t\in\Eg\RI\exists z(\forall x\in\;\!\!^st)(x\in z)}
\end{equation*}
where $\delta:=\text{max}(\beta,\gamma)+2$.
}\end{lemma}
\begin{proof}
First let
\begin{equation*}
p:=[x\in\Ed\;|\;\text{fun}(x,s,t)].
\end{equation*}
As an instance of axiom (E10) we have
\begin{equation*}
\tag{1}\CH[\beta,\gamma]\;\prov{0}{0}{s\in\Eb,t\in\Eg,\text{fun}(q,s,t)\RI q\in\Ed}\quad\text{for all $q$.}
\end{equation*}
Also axiom (E8) provides
\begin{equation*}
\tag{2}\CH[\beta,\gamma]\;\prov{0}{0}{q\in\Ed,\text{fun}(q,s,t)\RI q\in p}\quad\text{for all $q$.}
\end{equation*}
Applying $\Cut$ to (1) and (2) provides
\begin{equation*}
\tag{3}\CH[\beta,\gamma]\;\prov{1}{\delta+2}{s\in\Eb,t\in\Eg,\text{fun}(q,s,t)\RI q\in p}\quad\text{for all $q$.}
\end{equation*}
Now by $(\rightarrow R)$ we have
\begin{equation*}
\tag{4}\CH[\beta,\gamma]\;\prov{2}{\delta+2}{s\in\Eb,t\in\Eg\RI\text{fun}(q,s,t)\rightarrow q\in p}\quad\text{for all $q$.}
\end{equation*}
Thus we may use $(\mathcal{E}b\forall R)_\infty$ giving
\begin{equation*}
\tag{5}\CH[\beta,\gamma]\;\prov{\delta+1}{\delta+2}{s\in\Eb,t\in\Eg\RI(\forall x\in\;\!\!^st)(x\in p)}\quad\text{for all $q$.}
\end{equation*}
As instances of axioms (E11) and (E4) we also have
\begin{align*}
\tag{6}&\CH[\beta,\gamma]\;\prov{0}{0}{s\in\Eb,t\in\Eg,\Ed\in\mathbb{E}_{\delta+1}\RI p\in\mathbb{E}_{\delta+1}}\\
\tag{7}&\CH[\beta,\gamma]\;\prov{0}{0}{\RI\Ed\in\mathbb{E}_{\delta+1}.}
\end{align*}
We may apply $\Cut$ to (6) and (7) to obtain
\begin{equation*}
\tag{8}\CH[\beta,\gamma]\;\prov{1}{\delta+3}{s\in\Eb,t\in\Eg\RI p\in\mathbb{E}_{\delta+1}.}
\end{equation*}
Finally by applying $(\exists R)$ to (5) and (8) we get
\begin{equation*}
\CH[\beta,\gamma]\;\prov{\delta+4}{\delta+3}{s\in\Eb,t\in\Eg\RI\exists z(\forall x\in\;\!\!^st)(x\in z)}
\end{equation*}
as required.
\end{proof}
\begin{theorem}\label{IKPEembed}{\em If $\IKPE\vdash\GA[\bar a]\RI\DE[\bar a]$ with $\bar a$ the only free variables occurring in the intuitionistic sequent $\GA[\bar a]\RI\DE[\bar a]$. Then there is a $k<\omega$ such that for any $\IRSOE$ terms $\sbar$, any $\bbeta<\OO$ and any operator $\CH$
\begin{equation*}
\CH[\bbeta]\;\prov{\OO\cdot\omega^k}{\OO+k}{\sbar\in\Ebb,\GA[\sbar]\RI\DE[\sbar]}.
\end{equation*}
}\end{theorem}
\begin{proof}
The proof is by induction on the $\IKPE$ derivation. If $\GA[\bar a]\RI\DE[\bar a]$ is an axiom of $\IKPE$ then the result follows by one of lemmas \ref{elogic}, \ref{eextensionality}, \ref{eSet Induction}, \ref{einfinity}, \ref{esep}, \ref{epair}, \ref{eunion}, \ref{ecoll} and \ref{eexponentiation}.\\

\noindent Case 1. Suppose the last inference was $(\mathcal{E}b\exists L)$, then $(\exists x\in\;\!\!^{a_i}a_j)F(x)\in\GA[\bar a]$ and the final inference looks like
\begin{prooftree}
\Axiom$\fCenter\GA[\bar a],\text{fun}(b,a_i,a_j)\wedge F(b)\RI\DE[\bar a]$
\LeftLabel{$(\mathcal{E}b\exists L)$}
\UnaryInf$\fCenter\GA[\bar a]\RI\DE[\bar a]$
\end{prooftree}
where $b$ does not occur in $\bar a$. By the induction hypothesis we have a $k_0$ such that
\begin{equation*}
\tag{1}\CH[\bbeta,\gamma]\;\prov{\Omega\cdot\omega^{k_0}}{\OO+{k_0}}{\sbar\in\Ebb,p\in\Eg,\GA[\sbar],\text{fun}(p,s_i,s_j)\wedge F(p)\RI\DE[\sbar]}
\end{equation*}
for all $p$ and all $\gamma<\OO$. Let us choose the special case of (1) where $\gamma:=\text{max}(\beta_i,\beta_j)+2$ and note that for this choice of $\gamma$, $\CH[\bbeta,\gamma]=\CH[\bbeta]$. Now $\text{fun}(p,s_i,s_j)\RI\text{fun}(p,s_i,s_j)$ is an axiom due to (E1) and by Lemma \ref{elogic} we have $\Vdash_\OO F(p)\RI F(p)$ so applying $(\wedge R)$ gives
\begin{equation*}
\tag{2}\Vdash_\OO\text{fun}(p,s_i,s_j),F(p)\RI\text{fun}(p,s_i,s_j)\wedge F(p).
\end{equation*}
Applying $\Cut$ to (1) and (2) provides
\begin{equation*}
\tag{3}\CH[\bbeta]\;\prov{\Omega\cdot\omega^{k_1}}{\OO+k_1}{\sbar\in\Ebb,p\in\Eg,\GA[\sbar],\text{fun}(p,s_i,s_j), F(p)\RI\DE[\sbar]}.
\end{equation*}
Now as an instance of axiom (E10) we have
\begin{equation*}
\tag{4}{\CH[\bbeta]}\;\prov{0}{0}{\sbar\in\Ebb,\text{fun}(p,s_i,s_j)\RI p\in\Eg}.
\end{equation*}
So $\Cut$ to (3) and (4) gives
\begin{equation*}
\tag{5}\CH[\bbeta]\;\prov{\Omega\cdot\omega^{k_1}+1}{\OO+k_1}{\sbar\in\Ebb,\GA[\sbar],\text{fun}(p,s_i,s_j),F(p)\RI\DE[\sbar]}.
\end{equation*}
To which we may apply $(\wedge L)$ twice followed by $(\mathcal{E}b\exists L)_\infty$ to complete the case.\\

\noindent Case 2. Suppose the last inference was $(\mathcal{E}b\exists R)$ then $\DE[\bar a]=\{(\exists x\in\;\!\!^{a_i}a_j)F(x)\}$ and the final inference looks like
\begin{prooftree}
\Axiom$\fCenter\GA[\bar a]\RI\text{fun}(b,a_i,a_j)\wedge F(b)$
\LeftLabel{$(\mathcal{E}b\exists R)$}
\UnaryInf$\fCenter\GA[\bar a]\RI(\exists x\in\;\!\!^{a_i}a_j)F(x)$
\end{prooftree}
Suppose $b$ is a member of $\bar a$, without loss of generality let us suppose that $b\equiv a_1$, so by the induction hypothesis we have a $k_0<\omega$ such that
\begin{equation*}
\tag{6}\CH[\bbeta]\;\prov{\OO\cdot\omega^{k_0}}{\OO+k_0}{\sbar\in\Ebb,\GA[\sbar]\RI\text{fun}(s_1,s_i,s_j)\wedge F(s_1).}
\end{equation*}
If $b$ is not a member of $\bar a$ we can also conclude (6) by the induction hypothesis. As an instance of axiom (E1) we have $\text{fun}(s_1,s_i,s_j)\RI\text{fun}(s_1,s_i,s_j)$ to which we may apply $(\wedge L)$ giving
\begin{equation*}
\tag{7}\CH[\bbeta]\;\prov{1}{0}{\text{fun}(s_1,s_i,s_j)\wedge F(s_1)\RI\text{fun}(s_1,s_i,s_j)}.
\end{equation*}
Now applying $\Cut$ to (6) and (7) yields
\begin{equation*}
\tag{8}\CH[\bbeta]\;\prov{\OO\cdot\omega^{k_0}+1}{\OO+k_0}{\sbar\in\Ebb,\GA[\sbar]\RI\text{fun}(s_1,s_i,s_j).}
\end{equation*}
Axiom (E10) gives us
\begin{equation*}
\tag{9}\CH[\bbeta]\;\prov{0}{0}{\sbar\in\Ebb,\text{fun}(s_1,s_i,s_j)\RI s_1\in\Ed}\quad\text{where $\delta:=\text{max}(\beta_i,\beta_j)+2$.}
\end{equation*}
So applying $\Cut$ to (8) and (9) gives
\begin{equation*}
\tag{10}\CH[\bbeta]\;\prov{\OO\cdot\omega^{k_0}+1}{\OO+k_0}{\sbar\in\Ebb,\GA[\sbar]\RI s_1\in\Ed.}
\end{equation*}
Finally we may apply $(\mathcal{E}b\exists R)$ to (6) and (10) to complete this case.\\

\noindent Case 3. Now suppose the last inference was $(\mathcal{E}b\forall L)$, so $(\forall x\in\;\!\!^{a_i}a_j)F(x)\in\GA[\bar a]$ and the final inference looks like
\begin{prooftree}
\Axiom$\fCenter\GA[\bar a],\text{fun}(b,a_i,a_j)\rightarrow F(b)\RI\DE[\bar a]$
\LeftLabel{$(\mathcal{E}b\forall L)$}
\UnaryInf$\fCenter\GA[\bar a]\RI\DE[\bar a].$
\end{prooftree}
If $b$ is present in $\bar a$, without loss of generality let us suppose $b\equiv a_1$, regardless of whether $b$ is present in $\bar a$, by the induction hypothesis we have a $k_0<\omega$ such that
\begin{equation*}
\tag{11}\CH[\bbeta]\;\prov{\OO\cdot\omega^{k_0}}{\OO+k_0}{\sbar\in\Ebb,p\in\Eg,\GA[\sbar],\text{fun}(p,s_i,s_j)\rightarrow F(p)\RI\DE[\sbar]}.
\end{equation*}
The problem here is that $\beta_1$ may be greater than $\text{max}(\beta_i,\beta_j)+2$ meaning we cannot immediately apply $(\mathcal{E}b\forall L)$, moreover unlike in case 2 it is not possible to derive $\sbar\in\Ebb,\GA[\sbar]\RI\text{fun}(s_1,s_i,s_j)$. Instead we verify the following claim:
\begin{equation*}
\tag{*}\Vdash_\OO\GA[\sbar],(\forall x\in\;\!\!^{s_i}s_j)F(x)\RI\text{fun}(s_1,s_i,s_j)\rightarrow F(s_1)
\end{equation*}
To prove the claim we first note that as an instance of axiom (E10) we have
\begin{equation*}
\tag{12}\CH[\bbeta]\;\prov{0}{0}{\sbar\in\Ebb,\text{fun}(s_1,s_i,s_j)\RI s_1\in\Ed\quad\text{where $\delta:=\text{max}(\beta_i,\beta_j)+2$.}}
\end{equation*}
Then we have the following template for derivations in $\IRSOE$.
\begin{prooftree}
\Axiom$\fCenter\quad\quad\quad\text{(E1)}$
\UnaryInf$\fCenter\Vdash\text{fun}(s_1,s_i,s_j)\RI\text{fun}(s_1,s_i,s_j)$
\Axiom$\fCenter\text{Lemma \ref{elogic}}$
\UnaryInf$\fCenter\Vdash_\OO F(s_1)\RI F(s_1)$
\LeftLabel{$(\rightarrow L)$}
\BinaryInf$\fCenter\Vdash_\OO\text{fun}(s_1,s_i,s_j)\rightarrow F(s_1),\text{fun}(s_1,s_i,s_j)\RI F(s_1)$
\Axiom$\fCenter\text{(12)}$
\LeftLabel{$(\mathcal{E}b\forall L)$}
\BinaryInf$\fCenter\Vdash_\OO(\forall x\in\;\!\!^{s_i}s_j)F(x),\text{fun}(s_1,s_i,s_j)\RI F(s_1)$
\LeftLabel{$(\rightarrow R)$}
\UnaryInf$\fCenter\Vdash_\OO(\forall x\in\;\!\!^{s_i}s_j)F(x)\RI\text{fun}(s_1,s_i,s_j)\rightarrow F(s_1)$
\end{prooftree}
Thus the claim is verified. Now we may complete the case by applying $\Cut$ to (11) and (*).\\

\noindent Case 4. Now suppose the last inference was $(b\forall L)$, so $(\forall x\in a_i)F(x)\in\GA[\bar a]$ and the final inference looks like
\begin{prooftree}
\Axiom$\fCenter\GA[\bar a],b\in a_i\rightarrow F(b)\RI\DE[\bar a]$
\LeftLabel{$(b\forall L)$}
\UnaryInf$\fCenter\GA[\bar a]\RI\DE[\bar a].$
\end{prooftree}
If $b$ does occur in $\bar a$, without loss of generality we may assume $b\equiv a_1$. Regardless of whether $b$ is present in $\bar a$, by the induction hypothesis we have a $k_0<\omega$ such that
\begin{equation*}
\tag{13}\CH[\bbeta]\;\prov{\OO\cdot\omega^{k_0}}{\OO+k_0}{\sbar\in\Ebb,\GA[\sbar],s_1\in s_i\rightarrow F(s_1)\RI\DE[\sbar].}
\end{equation*}
Claim:
\begin{equation*}
\tag{**}\Vdash_\OO(\forall x\in s_i)F(x)\RI s_1\in s_i\rightarrow F(s_1).
\end{equation*}
To prove the claim we first note that by axiom (E6) we have
\begin{equation*}
\tag{14}\CH[\bbeta]\;\prov{0}{0}{\sbar\in\Ebb,s_1\in s_i\RI s_1\in\mathbb{E}_{\beta_i}}
\end{equation*}
Then we have the following template for derivations in $\IRSOE$.
\begin{prooftree}
\Axiom$\fCenter\quad\quad\quad\text{(E1)}$
\UnaryInf$\fCenter\Vdash s_1\in s_j\RI s_1\in s_j$
\Axiom$\fCenter\text{Lemma \ref{elogic}}$
\UnaryInf$\fCenter\Vdash_\OO F(s_1)\RI F(s_1)$
\LeftLabel{$(\rightarrow L)$}
\BinaryInf$\fCenter\Vdash_\OO s_1\in s_j\rightarrow F(s_1),s_1\in s_j\RI F(s_1)$
\Axiom$\fCenter\text{(14)}$
\LeftLabel{$b\forall L)$}
\BinaryInf$\fCenter\Vdash_\OO(\forall x\in s_i)F(x),s_1\in s_j\RI F(s_1)$
\LeftLabel{$(\rightarrow R)$}
\UnaryInf$\fCenter\Vdash_\OO(\forall x\in s_i)F(x)\RI s_1\in s_i\rightarrow F(s_1)$
\end{prooftree}
Finally we may apply $\Cut$ to (13) and (**) to complete this case.\\

\noindent Case 5. Now suppose the last inference was $(\forall L)$, so $\forall xF(x)\in\GA[\bar a]$ and the final inference looks like
\begin{prooftree}
\Axiom$\fCenter\GA[\bar a], F(b)\RI\DE[\bar a]$
\LeftLabel{$(\forall L)$}
\UnaryInf$\fCenter\GA[\bar a]\RI\DE[\bar a].$
\end{prooftree}
If $b$ is a member of $\bar a$, without loss of generality let us assume $b\equiv a_1$. By the induction hypothesis we have a $k_0<\omega$ such that
\begin{equation*}
\tag{15}\CH[\bbeta]\;\prov{\OO\cdot\omega^{k_0}+1}{\Omega+ k_0}{\sbar\in\Ebb,\GA[\sbar],F(s_1)\RI\DE[\sbar].}
\end{equation*}
If $b$ is not a member of $\bar a$ we can in fact still conclude (15) from the induction hypothesis. Now as an instance of axiom (E1) we have
\begin{equation*}
\tag{16}\CH[\bbeta]\;\prov{0}{0}{\sbar\in\Ebb\RI s_1\in\mathbb{E}_{\beta_1}}.
\end{equation*}
So applying $(\forall L)$ gives the desired result.\\

\noindent Case 6. Now suppose the last inference was $(\forall R)$, then $\{\forall xF(x)\}\equiv\DE[\bar a]$ and the final inference looks like
\begin{prooftree}
\Axiom$\fCenter\GA[\bar a]\RI F(b)$
\LeftLabel{$(\forall L)$}
\UnaryInf$\fCenter\GA[\bar a]\RI\forall xF(x)$
\end{prooftree}
with $b$ not present in $\bar a$. By the induction hypothesis we have a $k_0<\omega$ such that
\begin{equation*}
\CH[\bbeta,\gamma]\;\prov{\OO\cdot\omega^{k_0}}{\Omega+ k_0}{\sbar\in\Ebb, p\in\Eg,\GA[\sbar]\RI F(p)}
\end{equation*}
for all $p$ and all $\gamma<\OO$. Applying $(\forall R)_\infty$ gives the desired result.\\

\noindent Case 7. Suppose the last inference was $\Cut$ then the derivation looks like
\begin{prooftree}
\Axiom$\fCenter\GA[\bar a],B(\bar a,\bar b)\RI\DE[\bar a]$
\Axiom$\fCenter\GA[\bar a]\RI B(\bar a,\bar b)$
\BinaryInf$\fCenter\GA[\bar a]\RI\DE[\bar a]$
\end{prooftree}
where each member of $\bar b$ is distinct from the members of $\bar a$. By the induction hypothesis we get $k_0,k_1\in\omega$ such that
\begin{align*}
\tag{17}\CH[\bbeta]\;\prov{\OO\cdot\omega^{k_0}}{\Omega+ k_0}{\sbar\in\Ebb,\mathbb{E}_0\in\mathbb{E}_1,\GA[\sbar],B(\sbar,\bar{\mathbb{E}_0})\RI\DE[\sbar]}\\
\tag{18}\CH[\bbeta]\;\prov{\OO\cdot\omega^{k_1}}{\Omega+ k_1}{\sbar\in\Ebb,\mathbb{E}_0\in\mathbb{E}_1,\GA[\sbar]\RI B(\sbar,\bar{\mathbb{E}_0})}.
\end{align*}
Now since $\RI\mathbb{E}_0\in\mathbb{E}_1$ is an instance of axiom (E4) and $\sbar\in\Ebb\RI s_i\in\mathbb{E}_{\beta_i}$ is an instance of axiom (E1) we may apply $\Cut$ to (17) and (18) giving
\begin{equation*}
\tag{19}\CH[\bbeta]\;\prov{\OO\cdot\omega^{k}}{\Omega+ k}{\sbar\in\Ebb,\mathbb{E}_0\in\mathbb{E}_1,\GA[\sbar]\RI\DE[\sbar]}.
\end{equation*}
Finally applying $\Cut$ to (19) and $\CH[\bbeta]\;\prov{0}{0}{\mathbb{E}_0\in\mathbb{E}_1}$ we can complete this case.\\

\noindent All other cases can be treated in a similar manner to one of those above.
\end{proof}
\subsection{A relativised ordinal analysis of $\IKPE$}
\noindent Analogously to with $\IRSOP$ we will prove a soundness theorem for certain $\IRSOE$ derivable sequents in $E_{\psio{\varepsilon_{\OO+1}}}$. Again we need the notion of an assignment. Let $VAR_\mathcal{E}$ be the set of free variables of $\IRSOE$, an assignment is a map
\begin{equation*}
v:VAR_\mathcal{E}\longrightarrow E_{\psio{\varepsilon_{\OO+1}}}
\end{equation*}
such that $v(a_i^\al)\in E_{\al+1}$ for all $i<\omega$ and ordinals $\al$. Again an assignment canonically lifts to all $\IRSOE$ terms by setting
\begin{align*}
v(\mathbb{E}_\al)\:&=\:E_\al\\
v([x\in t\;|\;F(x, s_1,...,s_n)])\:&=\:\{x\in v(t)\;|\;F(x,v(s_1),...,v(s_n))\}.
\end{align*}
The difference between here and the case of $\IRSOP$ is that for a given term $t$, it is no longer possible to ascertain the location of $v(t)$ within the $E$-hierarchy solely by looking at the syntactic structure of $t$. It is however possible to place an upper bound on that location using the following function
\begin{align*}
m(\Ea):&=\al\\
m(a^\al_i):&=\al\\
m([x\in t\;|\;F(x,s_1,...,s_n)]):&=\text{max}(m(t),m(s_1),...,m(s_n))+1.
\end{align*}
It can be observed that $v(s)\in E_{m(s)+1}$ for any $s$, however in general $m(s)$ is only an upper bound on a term's position in the $E$-hierarchy.
\begin{theorem}[Soundness for $\IRSOE$]\label{esoundness}{\em\;
   Suppose $\GA[s_1,...,s_n]$ is a finite set of $\Pie$ formulae with max$\{rk(A)\;|\;A\in\GA\}\leq\OO$, $\DE[s_1,...,s_n]$ a set containing at most one $\Sigmae$ formula and
\begin{equation*}
\provx{\CH}{\al}{\rho}{\GA[\sbar]\RI\DE[\sbar]}\quad\text{for some operator $\CH$ and some $\al,\rho<\OO$.}
\end{equation*}
Then for any assignment $v$,
\begin{equation*}
E_{\psio{\varepsilon_{\OO+1}}}\models\bigwedge\GA[v(s_1),...,v(s_n)]\rightarrow\bigvee\DE[v(s_1),...,v(s_n)]
\end{equation*}
where $\bigwedge\GA$ and $\bigvee\DE$ stand for the conjunction of formulae in $\GA$ and the disjunction of formulae in $\DE$ respectively, by convention $\bigwedge\emptyset:=\top$ and $\bigvee\emptyset:=\bot$.
}\end{theorem}
\begin{proof}
The proof is by induction on $\al$. Note that the derivation $\provx{\CH}{\al}{\rho}{\GA[\sbar]\RI\DE[\sbar]}$ contains no inferences of the form $(\forall R)_\infty$, $(\exists L)_\infty$ or $\SRE$ and all cuts have $\Deltaoe$ cut formulae.\\

\noindent All axioms apart from (E6) and (E7) are clearly sound under the interpretation, the soundness of (E6) and (E7) follows from Lemma \ref{etran}.\\

\noindent Now suppose the last inference was $(\mathcal{E}b\exists R)$, so amongst other premises we have
\begin{equation*}
\provx{\CH}{\al_0}{\rho}{\GA[\sbar]\RI\text{fun}(t,s_i,s_j)\wedge A(t,\sbar)}\quad\text{for some $\al_0<\al$.}
\end{equation*}
Applying the induction hypothesis yields
\begin{equation*}
E_{\psio{\varepsilon_{\OO+1}}}\models \bigwedge\GA[v(\sbar)]\rightarrow [\text{fun}(v(t),v(s_i),v(s_j))\wedge A(v(t),\sbar)]\quad\text{where $v(\sbar):=v(s_1),...,v(s_n)$.}
\end{equation*}
Suppose $\GA[v(\sbar)]$ holds in $E_{\psio{\varepsilon_{\OO+1}}}$, so we have
\begin{equation*}
E_{\psio{\varepsilon_{\OO+1}}}\models\text{fun}(v(t),v(s_i),v(s_j))\wedge A(v(t),v(\sbar)).
\end{equation*}
It remains to note that the function space $^{v(s_i)}v(s_j)$ is a member of $E_{\psio{\varepsilon_{\OO+1}}}$ and thus
\begin{equation*}
E_{\psio{\varepsilon_{\OO+1}}}\models(\exists x\in\;\!\!^{v(s_i)}v(s_j))A(x,v(\sbar))
\end{equation*}
as required.\\

\noindent Now suppose the last inference was $(\mathcal{E}b\exists L)_\infty$, thus amongst other premises we have
\begin{equation*}
\tag{20}\provx{\CH}{\al_0}{\rho}{\GA[\sbar],\text{fun}(p,s_i,s_j)\wedge A(p,\sbar)\RI\DE[\sbar]}\quad\text{for all terms $p$ and some $\al_0<\al$.}
\end{equation*}
For the remainder of this case fix an arbitrary valuation $v_0$. Let $\beta_0:=m(s_i)$, $\beta_1:=m(s_j)$ and $\beta:=\text{max}(\beta_0,\beta_1)+2$. Choose $k$ such that $a^\beta_k$ does not occur in any of the terms in $\sbar$. As a special case of (20) we have
\begin{equation*}
\provx{\CH}{\al_0}{\rho}{\GA[\sbar],\text{fun}(a_k^\beta,s_i,s_j)\wedge A(a_k^\beta,\sbar)\RI\DE[\sbar]}.
\end{equation*}
Applying the induction hypothesis we get
\begin{equation*}
\tag{21}E_{\psio{\varepsilon_{\OO+1}}}\models \bigwedge\GA[v(\sbar)]\wedge [\text{fun}(v(a_k^\beta),v(s_i),v(s_j))\wedge A(v(a_k^\beta),v(\sbar))]\rightarrow\bigvee\DE[v(\sbar)]
\end{equation*}
for all valuations $v$. In particular (21) holds true for all valuations $v$ which coincide with $v_0$ on $\sbar$. By the choice of $a^\beta_k$ it follows that
\begin{equation*}
E_{\psio{\varepsilon_{\OO+1}}}\models \bigwedge\GA[v_0(\sbar)]\rightarrow\bigvee\DE[v_0(\sbar)]
\end{equation*}
as required.\\

\noindent All other cases may be treated in a similar manner to those above, using similar reasoning to Theorem \ref{psoundness}.
\end{proof}

\begin{lemma}\label{econc}{\em
Suppose $\IKPE\vdash\;\RI A$ for some $\Sigmae$ sentence $A$, then there exists an $n<\omega$, which we may compute from the derivation, such that
\begin{equation*}
\provx{\CH_\sigma}{\psio\sigma}{\psio\sigma}{\RI\;A}\quad\text{where $\sigma:=\omega_m(\OO\cdot\omega^m)$.}
\end{equation*}
}\end{lemma}
\begin{proof}
Suppose $\IKPE\vdash\;\RI A$, then by Theorem \ref{IKPEembed} we can explicitly calculate some $1\leq m<\omega$ such that
\begin{equation*}
\provx{\CH_0}{\OO\cdot\omega^m}{\OO+m}{\RI A}
\end{equation*}
Applying partial cut elimination for $\IRSOE$ \ref{epredce} we get
\begin{equation*}
\provx{\CH_0}{\omega_{m-1}(\OO\cdot\omega^m)}{\OO+1}{\RI A}.
\end{equation*}
Finally by applying collapsing for $\IRSOE$ \ref{ecollapsing} we get
\begin{equation*}
\provx{\CH_{\omega_m(\OO\cdot\omega^m)}}{\psio{\omega_m(\OO\cdot\omega^m)}}{\psio{\omega_m(\OO\cdot\omega^m)}}{\RI A}
\end{equation*}
as required.
\end{proof}

\begin{theorem}\label{econct}{\em If $A$ is a $\Sigmae$-sentence and $\IKPE\vdash\;\RI A$ then there is an ordinal term $\al<\psio{\varepsilon_{\OO+1}}$, which we may compute from the derivation, such that
\begin{equation*}
E_\al\models A.
\end{equation*}
}\end{theorem}
\begin{proof}
By Lemma \ref{econc} we can determine some $m<\omega$ such that
\begin{equation*}
\provx{\CH_\sigma}{\psio\sigma}{\psio\sigma}{\RI\;A}\quad\text{where $\sigma:=\omega_m(\OO\cdot\omega^m)$.}
\end{equation*}
Let $\al:=\psio\sigma$. Applying boundedness \ref{eboundedness} we get
\begin{equation*}
\provx{\CH}{\al}{\al}{\RI A^{\mathbb{E}_\al}}.
\end{equation*}
Now Theorem \ref{esoundness} yields
\begin{equation*}
E_{\psio{\varepsilon_{\OO+1}}}\models A^{\mathbb{E}_\al}.
\end{equation*}
It follows that
\begin{equation*}
E_\al\models A
\end{equation*}
as required.\\
\end{proof}
\begin{remark}{\em
Suppose $A\equiv\exists xC(x)$ is a $\Sigmae$ sentence and $\IKPE\vdash\;\RI A$. As in the case of $\IKPP$, as well as the ordinal term $\al$ given by Theorem \ref{econct}, it is possible to compute a specific $\IRSOE$ term $s$ such that $E_\al\models C(s)$. Moreover this process can be carried out inside $\IKPE$. These results will be verified in \cite{rathjen-EP}.
}\end{remark}
As in the foregoing cases we also have a conservativity result.

\begin{theorem}\label{cons3}{\em $\IKPE+\mbox{$\Sigma^{\mathcal E}$-Reflection}$ is conservative over
$\IKPE$ for $\Sigma^{\mathcal E}$-sentences. }\end{theorem}

\begin{remark}{\em An obvious question is whether the conservativity results of Theorems \ref{cons1}, \ref{cons2}, \ref{cons3} can be lifted to formulae with free variables? This would require ordinal analyses with set parameters. For classical Kripke-Platek set theory this has been carried out by the first author in \cite{diss}.
The second author thinks that this result can be lifted to the intuitionistic context. However it is likely
that this extension requires a fair amount of extra work since the linearity and decidability of the ordinal representation system would have to be sacrificed.}
\end{remark}

\end{document}